\documentclass[11pt]{amsart}

\usepackage{epsf,graphics,graphicx,amssymb,overpic,amscd,diagrams,psfrag,stmaryrd,euscript,times}

\usepackage{xr}
\newtheorem{thm}{Theorem}[section]
\newtheorem{prop}[thm]{Proposition}
\newtheorem{lemma}[thm]{Lemma}
\newtheorem{cor}[thm]{Corollary}

\newtheorem{claim}[thm]{Claim}

\numberwithin{equation}{subsection}
\numberwithin{thm}{subsection}

\theoremstyle{definition}
\newtheorem{defn}[thm]{Definition}

\theoremstyle{remark}

\newtheorem{rmk}[thm]{Remark}

\newtheorem{convention}[thm]{Convention}
\newtheorem{notation}[thm]{Notation}

\DeclareMathAlphabet{\mathpzc}{OT1}{pzc}{m}{it}
\newcommand{\hh}{\mathpzc{h}}

\newcommand{\D}{\mathbb{D}}
\renewcommand{\H}{\mathbb{H}}
\newcommand{\C}{\mathbb{C}}

\newcommand{\R}{\mathbb{R}}
\newcommand{\Z}{\mathbb{Z}}
\newcommand{\Q}{\mathbb{Q}}
\newcommand{\N}{\mathbb{N}}
\renewcommand{\P}{\mathbb{P}}
\newcommand{\bdry}{\partial}
\newcommand{\s}{\vskip.1in}
\newcommand{\n}{\noindent}

\newcommand{\F}{\mathbb{F}}

\newcommand{\ar}{\overrightarrow}
\newcommand{\bs}{\boldsymbol}

\newcommand{\be}{\begin{enumerate}}
\newcommand{\ee}{\end{enumerate}}
\newcommand{\op}{\operatorname}

\usepackage[refpage]{nomencl}
\makenomenclature

\newcommand{\nom}{\nomenclature}

\begin{document}

\title[HF=ECH via open book decompositions II]
{The equivalence of Heegaard Floer homology and embedded contact homology via open book decompositions II}

\author{Vincent Colin}
\address{Universit\'e de Nantes, 44322 Nantes, France}
\email{Vincent.Colin@univ-nantes.fr}

\author{Paolo Ghiggini}
\address{Universit\'e de Nantes, 44322 Nantes, France}
\email{paolo.ghiggini@univ-nantes.fr}
\urladdr{http://www.math.sciences.univ-nantes.fr/\char126 Ghiggini}

\author{Ko Honda}
\address{University of Southern California, Los Angeles, CA 90089}
\email{khonda@usc.edu} \urladdr{http://www-bcf.usc.edu/\char126 khonda}

\date{}



\thanks{VC supported by the Institut Universitaire de France, ANR Symplexe, ANR Floer Power, and ERC Geodycon. PG supported by ANR Floer Power and ANR TCGD. KH supported by NSF Grants DMS-0805352, DMS-1105432, and DMS-1406564.}

\begin{abstract}
This paper is the sequel to \cite{CGH-I} and is devoted to proving some of the technical parts of the HF=ECH isomorphism.
\end{abstract}

\maketitle

\setcounter{tocdepth}{2}
\tableofcontents

\section{Introduction}

This paper is the sequel to \cite{CGH-I} and is devoted to proving some of the technical parts of the isomorphism between $\widehat{HF}(-M)$ and $\widehat{ECH}(M)$. References from \cite{CGH-I} will be written as ``Section I.$x$'' to mean ``Section $x$'' of \cite{CGH-I}, for example.  The notation from \cite{CGH-I} carries over to this paper and is summarized in the Index of Notation, which appears at the beginning of \cite{CGH-I}.

In \cite{CGH-I} we defined the chain maps
$$\Phi: \widehat{CF}(S,{\bf a},\hh({\bf a}))\to PFC_{2g}(N),$$
$$\Psi: PFC_{2g}(N)\to \widehat{CF}(S,{\bf a},\hh({\bf a})).$$

In this paper we prove the following two results:

\begin{thm}[Quasi-isomorphism]
\label{thm: isomorphism}
The chain maps $\Phi$ and $\Psi$ are quasi-isomor\-phisms, provided the monodromy map $\hh$ does not have any elliptic periodic points of period $\leq 2g$ in $int(S)$.
\end{thm}

Theorem~\ref{thm: isomorphism} is quite involved and takes up most of this paper.  The condition on $h$ is a technical assumption which simplifies gluing and one should be able to remove this condition with more work.

\begin{thm}[Stabilization]\label{thm: stabilization}
$\widehat{ECH}(M)\simeq PFH_{2g} (N).$
\end{thm}

In view of Theorem~I.\ref{P1-thm: direct limit}, Theorem~\ref{thm: stabilization} would immediately follow from showing that the chain maps
$$\mathfrak{K}_j:ECC_j(N,f_j\alpha)\to ECC_{j+1}(N,f_{j+1}\alpha),$$
defined in Section~I.\ref{P1-direct limits} (or rather its PFC variants) are quasi-isomorphisms for $j\geq 2g$. What we actually prove is slightly weaker, but sufficient: the similarly defined chain maps
$$ECC_{2j}(N,f_{2j}\alpha)\to ECC_{2j+2}(N,f_{2j+2}\alpha)$$
are quasi-isomorphisms for $2j\geq 2g$. Theorem~\ref{thm: stabilization} is proved in Section~\ref{section: stabilization}.

\begin{notation}[Sub/superscripts $*$]  \label{notation: sub superscripts}
In this paper, as in \cite{CGH-I}, $*$ is often used to denote a variety of possible subscripts/superscripts (e.g., intersection numbers $n^*(\overline{u})$ in Equation~\eqref{n and n alt}, curves $\overline v_*$ and their domains $\dot F_*$ as in Notation~\ref{notation}, and moduli spaces $\mathcal{M}^*(\star)$ as in Notation~\ref{notation: modifiers}).
\end{notation}

\s\n {\em Organization of this paper.}  In Sections~\ref{section: homotopy of cobordisms I} and ~\ref{section: homotopy of cobordisms II} we prove the chain homotopy between the chain maps $\Psi\circ\Phi$ and $id$, as well as the chain homotopy between the chain maps $\Phi\circ\Psi$ and $id$. The necessary Gromov-Witten type calculations are carried out in Section~\ref{section: Gromov-Witten computation}. Finally, Section~\ref{section: stabilization} is devoted to proving Theorem~\ref{thm: stabilization}.

\section{Gromov-Witten type computations}
\label{section: Gromov-Witten computation}

This section is devoted to Gromov-Witten type calculations which are used in the proof of Theorem~\ref{thm: isomorphism}.

After a brief review of relative Gromov-Witten invariants in Section~\ref{subsection: brief review of relative Gromov-Witten invariants}, we treat a slightly simpler model situation in Section~\ref{subsection: first GW calculation}.  We then tackle the specific situations of interest in this paper in Sections~\ref{subsection: second GW calculation} and ~\ref{subsection: third GW calculation}.

\subsection{A brief review of relative Gromov-Witten invariants}
\label{subsection: brief review of relative Gromov-Witten invariants}

We now briefly review relative Gromov-Witten invariants, following Ionel-Parker~\cite{IP1}.  The reader is referred to Ionel-Parker~\cite{IP1}, as well as to Li-Ruan~\cite{LR} and McDuff~\cite{M2}, for more complete discussions of relative Gromov-Witten invariants.

Let $(X,\omega)$ be a compact symplectic manifold and $J$ a compatible almost complex structure on $(X,\omega)$.

Let $M_{g,n}$ be the moduli space of genus $g$ curves with $n$ (ordered) marked points and let $\overline{M}_{g,n}$ be the Deligne-Mumford compactification.
Let $\overline{\mathcal{U}}_{g,n}\to \overline{M}_{g,n}$ be the universal curve over
$\overline{M}_{g,n}$, with a fixed embedding $\overline{\mathcal{U}}_{g,n}\subset \C\P^N$. Then let $\nu$ be a section of the bundle $$\op{Hom}^{0,1}(\pi_2^*T\C\P^N, \pi_1^*TX)\to X\times \C\P^N$$ of $\C$-anti-linear maps, where $\pi_1$ and $\pi_2$ are projections onto the first and second factors of $X\times \C\P^N$.

Let $(F,j,{\bf x}=(x_1,\dots,x_n))$ be a closed, connected (i.e. connected after gluing together all the components at the nodes), marked nodal Riemann surface, let $st(F,j,{\bf x})$ be the stable Riemann surface obtained by collapsing the unstable components to points and let $\phi_0: st(F,j,{\bf x})\to \overline{\mathcal{U}}_{g,n}$ be a biholomorphism onto a fiber of $\overline{\mathcal{U}}_{g,n}$.

\begin{rmk}
The notation ${\bf x}=(x_1,\dots,x_n)$ as well as its close cousins will only be used in Section~\ref{section: Gromov-Witten computation} and will not conflict with the usage as a generator of a Heegaard Floer chain group in the other sections of this paper.
\end{rmk}

\begin{defn}
A {\em $(J,\nu)$-holomorphic map} is a map
$$(u,\phi_0\circ st): (F,j,{\bf x})\to X\times \overline{\mathcal{U}}_{g,n}$$
which satisfies the perturbed $J$-holomorphic equation
\begin{equation} \label{eqn: perturbed J-holomorphic}
\overline{\partial}_{J} u=(u,\phi_0\circ st)^*\nu.
\end{equation}
\end{defn}

We will usually abbreviate Equation~\eqref{eqn: perturbed J-holomorphic} as $\overline{\partial}_J u=\nu$ and not mention $\phi_0\circ st$.

Now let $V$ be a codimension $2$ symplectic submanifold of $X$. Then $(J,\nu)$ is {\em $V$-compatible} if $TV$ is $J$-invariant, the component of $\nu$ in the direction normal to $V$ is zero, and certain tensors given in Definition 3.2 (b) and (c) of \cite{IP1} vanish.

{\em From now on we assume that all $(J,\nu)$ are $V$-compatible without further mention.}

\begin{defn}
A $(J,\nu)$-holomorphic stable map
$$u: (F,j,{\bf x}, {\bf x'})\to X,$$
with ${\bf x}=(x_1,\dots,x_n)$ and ${\bf x'}=(x_1',\dots,x_l')$, is said to be {\em $V$-regular} if no component of the domain, no marked point in ${\bf x}$, and no double point of $u$ is mapped to $V$. The points in ${\bf x'}$ are mapped to $V$ with multiplicities ${\bf s}=(s_1,\dots,s_l)$.  We write $l({\bf s})=l$ and $\deg({\bf s})=\sum_{i=1}^l s_i$.
\end{defn}

 Let $\mathcal{M}_{g,n+l({\bf s})}(X,A,J,\nu)$ be the moduli space of $(J,\nu)$-holomorphic maps
$$u:(F,j,{\bf x}\cup{\bf x'})\to X$$
in the class $A\in H_2(X)$ such that $(F,j)$ is a closed genus $g$ Riemann surface, modulo automorphisms of the domain. Next we let
$$\mathcal{M}=\mathcal{M}_{g,n,{\bf s}}^V(X,A,J,\nu)\subset \mathcal{M}_{g,n+l({\bf s})}(X,A,J,\nu)$$
be the subset of (equivalence classes of) $V$-regular $(J,\nu)$-holomorphic maps $u:(F,j,{\bf x},{\bf x'})\to X$ such that ${\bf x'}$ is mapped to $V$ with multiplicities ${\bf s}$.
We also use the modifier $gi$ to denote the subset of {\em generically injective} maps, i.e., maps $u$ such that $u^{-1}(u(x))=x$ for almost all points $x$.
Then, for a generic $(J,\nu)$, $\mathcal{M}^{gi}\subset \mathcal{M}_{g,n+l({\bf s})}(X,A,J,\nu)$ is transversely cut out by \cite[Lemma~4.2]{IP1}.

\begin{rmk} \label{rmk: irred}
In \cite{IP1}, generically injective maps are called ``irreducible''. In our series of papers an {\em irreducible map} $u: F\to X$ will always refer to a map whose domain cannot be written as the union of two or more curves glued together at nodes.
\end{rmk}

By \cite[Theorem~7.4]{IP1}, there is a compactification $\overline{\mathcal{M}}=\overline{\mathcal{M}}_{g,n,{\bf s}}^V(X,A,J,\nu)$ of $\mathcal{M}$ and,  under ideal conditions (e.g., all the curves of $\overline{\mathcal{M}}$ are generically injective), $\overline{\mathcal{M}}-\mathcal{M}$ is the union of (real) codimension $\geq 2$ strata, and $\overline{\mathcal{M}}$ carries a fundamental class. Here $\overline{\mathcal{M}}$ is the space of {\em $V$-stable maps} $u: (F,j,{\bf x}, {\bf x'})\to X$ as defined in \cite[Definition~7.2]{IP1}. A map $u\in \overline{\mathcal{M}}-\mathcal{M}$ can have irreducible components $F_i$ which are mapped to $V$. Such components come equipped with solutions $\xi_i$ of Cauchy-Riemann type equations on the pullback of the normal bundle $N_{X/V}$. The sections $\xi_i$ keep track of how a sequence of $V$-regular maps approaches $u|_{F_i}$.

We then consider the evaluation map:
\begin{equation} \label{eqn: ev map}
\overline{ev}: \overline{\mathcal{M}}\to X^n\times V^{l({\bf s})},
\end{equation}
$$(u:(F,j,{\bf x},{\bf x'})\to X)\mapsto (u({\bf x}), u({\bf x'})).$$
We write $\overline{ev}_{(J,\nu)}$ when we want to emphasize the pair $(J,\nu)$. We also write $ev$ for the restriction of $\overline{ev}$ to $\mathcal{M}$.

We can make the following definition, assuming ideal conditions:

\begin{defn}
The {\em relative Gromov-Witten invariant} $GW^V_{X,A,g,n,{\bf s}}$ is the cycle
$$\overline{ev}_*[\overline{\mathcal{M}}]\in H_*(X^n\times V^{l({\bf s})};\Q).$$
\end{defn}

\subsection{First relative Gromov-Witten calculation}
\label{subsection: first GW calculation}

Let $\Sigma$ be a surface of genus $g$.  We consider $(X,\omega)=\Sigma\times \C\P^1$ with a product symplectic form.  Let $\pi_1$ and $\pi_2$ be the projections of $X$ onto $\Sigma$ and $\C\P^1$, respectively.

The goal of this subsection is to compute the top-dimensional summand of
$$GW^V_{X,A,g,n,{\bf s}}\in H_*(X^n\times V^{l({\bf s})}),$$ where $V=\Sigma\times\{\infty\}$, $A=[\Sigma]+2g[\C\P^1]$, $g=g(\Sigma)$, $n=g+1$, ${\bf s}=(1,\dots,1)$, and $l({\bf s})=2g$.  In other words, we are computing the degree of $\overline{ev}$, which is an integer that we denote by $G_1$.

\begin{thm} \label{thm: count of G sub 1}
$G_1=\pm 1$.
\end{thm}

Since $X,A,g,n,{\bf s}$ are fixed in this subsection, we abbreviate
$$\mathcal{M}^V_{(J,\nu)}:= \mathcal{M}_{g,n,{\bf s}}^V(X,A,J,\nu),\quad \mathcal{M}_{(J,\nu)}:=\mathcal{M}_{g,n+l({\bf s})}(X,A,J,\nu).$$

In the proof of Theorem~\ref{thm: count of G sub 1}, instead of a generic $(J,\nu)$, we will use $(J,0)$ with $J=j_\Sigma\times j_{\C\P^1}$ a product complex structure to carry out our calculations. This is legitimate because of certain automatic transversality results which will be discussed in Section~\ref{subsubsection: regularity}.

{\em From now on assume that $J=j_\Sigma\times j_{\C\P^1}$ is a product complex structure and moreover $(\Sigma,j_\Sigma)$ has trivial automorphism group.}

\subsubsection{Regularity} \label{subsubsection: regularity}

Let $\mathcal{M}_g(X,A,J)$ be the moduli space of holomorphic maps
$$u: (F,j) \to (X,J)$$
in the class $A$, modulo automorphisms of the domain. Here $F$ is a closed surface of genus $g$ and $j$ ranges over all complex structures on $F$.

Recall Hofer-Lizan-Sikorav's automatic transversality theorem \cite[Theorem~1]{HLS} (first suggested by Gromov~\cite{Gr}), which states that the moduli space $\mathcal{M}_g(X,A,J)$ is regular in a neighborhood of $u$, provided $u$ is embedded and
$$c_1(A)=\langle c_1(TM),A\rangle>0.$$
If $N$ is the normal bundle of an embedded curve $u$, then $c_1(A)>0$ is equivalent to $c_1(N)> 2g-2$ by the adjunction formula $c_1(A)=c_1(T\Sigma)+c_1(N)$.

The following lemma is a consequence of automatic transversality:

\begin{lemma} \label{lemma: regularity and dimension}
The moduli space $\mathcal{M}_g(X,A,J)$ is regular and $$\dim_{\R}\mathcal{M}_g(X,A,J)=6g+2.$$
\end{lemma}

\begin{proof}
Any curve $u\in \mathcal{M}_g(X,A,J)$ is clearly embedded, since the projection $\pi_1\circ u$ is a biholomorphism. For an embedded curve $u$ in the class $A=[\Sigma]+ 2g[\C\P^1]$, $$c_1(N)=A\cdot A= 4g>2g-2.$$ This implies the regularity of $\mathcal{M}_g(X,A,J)$.

Now $c_1(A)=c_1(T\Sigma)+c_1(N)=2g+2,$ and the Fredholm index is given by:
$$\op{ind}(u)= -\chi(F)+ 2c_1(A)= (2g-2)+ 2(2g+2)= 6g+2.$$
By the regularity of $\mathcal{M}_g(X,A,J)$, $\dim \mathcal{M}_g(X,A,J)=\op{ind}(u)=6g+2$.
\end{proof}

Let $u\in \mathcal{M}_g(X,A,J)$.  Since $\pi_1\circ u: F \to \Sigma$ must have degree one, it follows that $(F,j)$ must be biholomorphic to our chosen $(\Sigma,j_\Sigma)$.  Now fix a particular identification of $(F,j)$ with $(\Sigma,j_\Sigma)$; this has the effect of eliminating automorphisms of the domain. Hence there is a one-to-one correspondence between $u\in \mathcal{M}_g(X,A,J)$ and meromorphic functions (i.e., branched covers)
$$v=\pi_2\circ u:\Sigma\to \C\P^1$$
of degree $2g$ and $u$ can be recovered by taking the graph of the function $v$.

\begin{cor} \label{aardvark}
The moduli space $\mathcal{M}_{(J,0)}$ is regular and the moduli space $\mathcal{M}^{V}_{(J,0)}$ is transversely cut out in $\mathcal{M}_{(J,0)}$.
\end{cor}

By definition, all curves of $\mathcal{M}^{V}_{(J,0)}$ are irreducible.

\begin{proof}
The moduli space $\mathcal{M}_{(J,0)}$ is regular by Lemma~\ref{lemma: regularity and dimension}, since we are just adjoining the variables ${\bf x}=(x_1,\dots,x_{g+1}), {\bf x'}=(x_1',\dots,x_{2g}')$ to each $u: (F,j)\to X$.  The second statement of the corollary is more precisely stated as:
\begin{equation} \label{condition}
ev(\mathcal{M}^{V}_{(J,0)})\pitchfork (X^n\times V^{l({\bf s})}) ~\mbox{ in }~ X^{n+l({\bf s})},
\end{equation}
where $ev: \mathcal{M}_{(J,0)}\to X^{n+l({\bf s})}$ is the evaluation map $ev(u,{\bf x},{\bf x}')=(u({\bf x}),u({\bf x}'))$.
If $(u,{\bf x},{\bf x'})\in \mathcal{M}^{V}_{(J,0)}$, then $u$ is transverse to $V$ and intersects it at $2g$ points by definition. Hence Condition~\eqref{condition} is easily attained at $(u,{\bf x},{\bf x'})$ by considering the variations of ${\bf x'}$ while keeping $(u,{\bf x})$ fixed.
\end{proof}

\subsubsection{The Riemann-Roch theorem}
\label{subsubsection: Riemann-Roch}

Let $\Sigma$ be a closed Riemann surface of genus $g$. The classical Riemann-Roch theorem states that:
$$l(D)-l(K-D) =\op{deg}(D)+1-g,$$
where $D$ is a divisor on $\Sigma$, $K$ is the canonical divisor, and $l(D)$ is the dimension of the space of meromorphic functions on $\Sigma$ with poles at most at $D$.

\begin{lemma} \label{example}
If $\deg(D)= 2g-1+m$ with $m\geq 0$, then $l(D)=g+m$.
\end{lemma}

\begin{proof}
Since $\deg(K)=2g-2$, we have $\op{deg}(K-D)<0$ and $l(K-D)=0$. Hence $l(D)=\op{deg}(D)+1-g=(2g-1+m)+1-g=g+m$.
\end{proof}

\subsubsection{Proof of Theorem~\ref{thm: count of G sub 1}}
\label{subsubsection: proof of theorem count of G sub 1}

Let $\widetilde{\bf x}=(\widetilde{x}_1,\dots,\widetilde{x}_{g+1})$ and $\widetilde{\bf x}'=(\widetilde{x}'_1,\dots,\widetilde{x}'_{2g})$ be ordered $(g+1)$- and $2g$-tuples in $\Sigma$ and let ${\bf y}=(y_1,\dots,y_{g+1})$ be an ordered $(g+1)$-tuple in $\C\P^1$.  We assume that $(\widetilde{\bf x},\widetilde{\bf x}')$ and ${\bf y}$ are in generic position; in particular, the points of $(\widetilde{\bf x},\widetilde{\bf x}')$ are distinct and the points of ${\bf y}$ are distinct and different from $\infty$. If $(u,{\bf x},{\bf x'})\in \mathcal{M}^{V}_{(J,0)}\cap \overline{ev}^{-1}((\widetilde{\bf x},{\bf y}),(\widetilde{\bf x}',\boldsymbol{\infty}))$,\footnote{By slight abuse of notation, $(\widetilde{\bf x},{\bf y})$ means $((\widetilde{x}_1,y_1),\dots,(\widetilde{x}_{g+1},y_{g+1}))$. Also $\boldsymbol{\infty}$ is a tuple of $\infty$'s.} where $u:(F,j)\to (X,J)$ is a holomorphic map, then ${\bf x}=\widetilde{\bf x}$ and ${\bf x}'=\widetilde{\bf x}'$, since $(F,j)$ is identified with the Riemann surface $\Sigma$. Henceforth, we write ${\bf x}$ and ${\bf x}'$ instead of $\widetilde{\bf x}$ and $\widetilde{\bf x}'$.

The moduli spaces $\mathcal{M}^{V}_{(J,0)}$ and $\mathcal{M}_{(J,0)}$ are regular by Corollary~\ref{aardvark}.  We have the evaluation map
$$ev_{(J,0)}: \mathcal{M}^V_{(J,0)}\to X^n \times V^{l({\bf s})}.$$
We then write ${\frak E}=(ev_{(J,0)})^{-1}(({\bf x},{\bf y}),({\bf x'},\boldsymbol{\infty}))$.

\begin{lemma}\label{lemma: no bubbling}
If $({\bf x},{\bf x}')$ and ${\bf y}$ are generic, then ${\frak E}$ is compact.
\end{lemma}

\begin{proof}
Arguing by contradiction, let $u_i\in {\frak E}$ be a sequence that converges to a limit curve $u:\Sigma'\to X$. Since $u$ is $(J,0)$-holomorphic, $\pi_1\circ u$ is holomorphic and one component of $\Sigma'$ must be biholomorphic to $\Sigma$. Hence $\Sigma'$ is obtained from $\Sigma$ by attaching spheres which are mapped by $u$ to $\{ x''_i \} \times \C \P^1$, $1\leq i\leq k$. Moreover, $k\leq 2g$ by the positivity of intersections, since $\langle u(\Sigma'), \Sigma\times\{pt\}\rangle=2g,$ where $\langle\cdot,\cdot\rangle$ is the intersection pairing.

\s\n (1) Suppose that $u(\Sigma)\not=V=\Sigma\times\{\infty\}$. Since $\langle u(\Sigma'),V\rangle=2g$ and the intersection points are $(x'_1,\infty),\dots,(x'_{2g},\infty)$, we must have ${\bf x}''=\{x_1'',\dots,x_k''\}\subset {\bf x}'$.\footnote{Here we are slightly abusing notation and viewing ${\bf x}'', {\bf x}'$ as sets.} We claim the map $u|_\Sigma$ cannot exist. Let us write $v=\pi_2\circ u|_\Sigma: \Sigma\to \C\P^1$. Then $v$ is a meromorphic function with poles at ${\bf x}'-{\bf x}''$, subject to $v(x_i)=y_i$ for all $i=1,\dots,g+1$. Hence $v\in l(D'-D'')$, where $D'=x_1'+\dots+x_{2g}'$ and $D''= x_1''+\dots+x_k''$. By Lemma~\ref{example}, $l(D'-D'')\leq g$. [Apply the lemma to $D'-\{x_1''\}$, which has degree $2g-1$. This gives us $l(D'-\{x_1''\})=g$.  Then observe that $l(D'-\{x_1''\})\geq l(D'-D'')$, more or less by definition.] On the other hand, there are more constraints $v(x_i)=y_i$ than there are linearly independent functions. Hence $v$ cannot exist provided ${\bf x}$ and ${\bf y}$ are generic, a contradiction.

\s\n (2) Now suppose that $u(\Sigma)=V$. Then $u$ consists of $u|_\Sigma$, together with $2g$ bubbles $\{ x_i''\} \times \C\P^1$, $i=1,\dots, 2g$. The set ${\bf x''}=\{x_1'',\dots,x_{2g}''\}$ contains ${\bf x}=\{x_1,\dots,x_{g+1}\}$ and may also contain $k\leq g-1$ elements of ${\bf x'}$. (Recall that $({\bf x},{\bf x}')$, ${\bf y}$ are in generic position.) We then apply the renormalization procedure of \cite[Proposition~6.6]{IP1} to the sequence $u_i$ restricted to a neighborhood of $V$. After choosing suitable restrictions and rescalings, the sequence $u_i$ converges to a  nonconstant meromorphic function $\xi:\Sigma \rightarrow \C \P^1$ which encodes how the curves $u_i (\Sigma)$ approach $V$. The function $\xi$ has poles at most at $D''=x_1''+\dots+x_{2g}''$ and zeros at $2g-k\geq g+1$ points of ${\bf x'}$. The details of this argument are left to the reader.
Since $l(D'')=g+1$ by Lemma~\ref{example}, we have at least as many constraints as linearly independent functions.  This implies that $\xi=0$, which is again a contradiction.
\end{proof}

By Corollary~\ref{aardvark} and Lemma~\ref{lemma: no bubbling}, $G_1$ is the number of meromorphic functions $v: \Sigma\to \C\P^1$ with poles at ${\bf x}'$ such that $v(x_i)=y_i$ for all $i=1,\dots,g+1$. By Lemma~\ref{example}, $l(x_1'+\dots+x_{2g}')=g+1$. On the other hand, since there are $g+1$ constraints $v(x_i)=y_i$, there is a unique solution for generic ${\bf x}$, ${\bf y}$. This shows that $G_1=\pm 1$.

\subsection{Second relative Gromov-Witten calculation}
\label{subsection: second GW calculation}

\subsubsection{Definitions}

Let $\Sigma=\overline{S}$, where $\overline{S}=S\cup D^2$ is obtained by capping off a page $S$ of an open book decomposition with connected binding as in Section~I.\ref{P1-subsubsection: overline W pm}. Also recall the point at infinity $z_\infty=\{\rho=0\}\in \overline{S}$.

Consider $(X,\omega)$, where $X=\Sigma\times \C\P^1$ and $\omega$ is a product symplectic form. Let $V=V_1\cup V_\infty$, where $V_*=\Sigma\times\{*\}$ and $*=1,\infty$.  We take $A=[\Sigma]+2g[\C\P^1]$, $g=g(\Sigma)=g(S)$, $n=1$, and ${\bf s}={\bf s}_1\cup {\bf s}_\infty$, where:
\begin{itemize}
\item ${\bf s}_1=(s_{1,1},\dots,s_{1,l})$,  $l({\bf s}_1)=l$, $\deg ({\bf s}_1)=2g$;
\item ${\bf s}_\infty=(1,\dots,1)$, $l({\bf s}_\infty)=2g$, $\deg ({\bf s}_\infty)=2g$; and
\item $l({\bf s})=l({\bf s}_1)+l({\bf s}_\infty)$.
\end{itemize}

Let $\mathcal{M}_{(J,\nu)}:=\mathcal{M}_{g,n+l({\bf s})}(X,A,J,\nu)$ and let
$$\mathcal{M}^{V_1,V_\infty}_{(J,\nu)}:= \mathcal{M}_{g,n,{\bf s}_1,{\bf s}_\infty}^{V_1,V_\infty}(X,A,J,\nu)\subset \mathcal{M}_{g,n,{\bf s}}^{V}(X,A,J,\nu)$$
be  the subset of $V$-regular maps $u:(F,j,x,{\bf x}'_1,{\bf x}'_\infty)\to X$, where
$${\bf x}'_1=(x'_{1,1},\dots,x'_{1,l({\bf s}_1)}),\quad {\bf x}'_\infty = (x'_{\infty,1},\dots,x'_{\infty,l({\bf s}_\infty)}),$$
and ${\bf x}'_*$, $*=1,\infty$, is mapped to $V_*$ with multiplicities ${\bf s}_*$.  Similarly, for $*=1,\infty$, let
$$\mathcal{M}^{V_*}_{(J,\nu)}:=\mathcal{M}_{g,n,{\bf s}_*}^{V_*}(X,A,J,\nu)$$
be the subset of $V_*$-regular maps $(u,{\bf x}'_*)$ such that ${\bf x}'_*$ is mapped to $V_*$ with multiplicities ${\bf s}_*$.

We define the evaluation map
$$\overline{ev}_{(J,\nu)}: \overline{\mathcal{M}}^{V_1,V_\infty}_{(J,\nu)} \to X\times V_1^{l({\bf s}_1)}\times V_\infty^{l({\bf s}_\infty)},$$
$$(u:(F,j,x,{\bf x}'_1,{\bf x}'_\infty)\to X)\mapsto(u(x), u({\bf x}'_1), u({\bf x}'_\infty)),$$
and, for $*=1,\infty$, we define the evaluation maps
$$\overline{ev}_{(J,\nu)}^{V_*}: \overline{\mathcal{M}}^{V_*}_{(J,\nu)} \to V_*^{l({\bf s}_*)}, \quad (u,{\bf x}'_*)\mapsto u({\bf x}'_*).$$
 Let $ev_{(J,\nu)}$ and $ev_{(J,\nu)}^{V_*}$ be the restrictions of $\overline{ev}_{(J,\nu)}$ and $\overline{ev}_{(J,\nu)}^{V_*}$ to $\mathcal{M}^{V_1,V_\infty}_{(J,\nu)}$ and $\mathcal{M}^{V_*}_{(J,\nu)}$, respectively.

Recall the closed curves $\overline{a}_i\subset \Sigma=\overline{S}$, $i=1,\dots,2g$, from Section~I.\ref{P1-coconut}.  Let $\overline{a}_i'$, $i=1\dots,2g$, be a simple closed curve homotopic to $\overline{a}_i$ and let us write $\overline{\bf a}'=(\overline{a}_1',\dots,\overline{a}'_{2g})$.

\begin{defn}
A pair $((z_0,y_0), {\bf z}'_1)$ consisting of $(z_0,y_0)\in X$ and ${\bf z}_1'\in \Sigma^{l({\bf s}_1)}$ is {\em generic with respect to $\overline{\bf a}'$} if $(z_0,y_0)$ is a generic point in $X$, $(z_0, {\bf z}'_1)$ is a generic $(l({\bf s}_1)+1)$-tuple in $\Sigma^{l({\bf s}_1)+1}$, and, in particular, $y_0\not = 1,\infty$ and the points of $(z_0, {\bf z}'_1)$ are disjoint and do not lie on any $\overline{a}'_i$.
\end{defn}

\begin{defn} \label{defn: G sub 2}
The relative Gromov-Witten invariant $G_2(\Sigma,{\bf s}_1)$ is defined as follows:
$$G_2(\Sigma,{\bf s}_1)=\langle (\overline{ev}_{(J,\nu)})_* ([\overline{\mathcal{M}}^{V_1,V_\infty}_{(J,\nu)}]),\{(z_0,y_0)\} \times {\bf z}'_1 \times \overline{a}'_1 \times \dots \times \overline{a}'_{2g}\rangle,$$
where $((z_0,y_0), {\bf z}_1')$ is generic with respect to $\overline{\bf a}'$. When we want to specify the constraints, we write $G_2(\Sigma,{\bf s}_1; \{(z_0,y_0)\} \times {\bf z}'_1 \times \overline{a}'_1 \times \dots \times \overline{a}'_{2g})$.
\end{defn}

The main theorem of this subsection is the following:

\begin{thm} \label{thm: calc of G sub 2}
If $\deg({\bf s}_1)=2g$, then $G_{2}(\Sigma,{\bf s}_1)$ does not depend on ${\bf s}_1$ and $G_{2}(\Sigma,{\bf s}_1)=\pm 1$.
\end{thm}

\subsubsection{First steps in the calculation of $G_2$}

We will use $(J,\nu)=(J,0)$, where $J=j_\Sigma\times j_{\C\P^1}$ is a product complex structure on $\Sigma\times \C\P^1$.

There is a one-to-one correspondence between $(u,x,{\bf x}'_1,{\bf x}'_\infty)\in \mathcal{M}^{V_1,V_\infty}_{(J,0)}$ and meromorphic functions
$$v: (F,j,x,{\bf x}'_1,{\bf x}'_\infty) \to \C\P^1$$
satisfying $\Sigma=(F,j)$, $v({\bf x}'_1)=1$ and $v({\bf x}'_\infty)=\infty$, where the correspondence is given by $v=\pi_2 \circ u$. Moreover, if
$$ev_{(J,0)}(u,x,{\bf x}'_1,{\bf x}'_\infty)\in \{(z_0,0)\} \times {\bf z}'_1 \times \overline{a}'_1 \times \dots \times \overline{a}'_{2g},$$
then $(x,{\bf x}'_1)=(z_0,{\bf z}'_1)$ and $x'_{\infty,i}\in \overline{a}'_i$, $i=1,\dots,2g$.  Henceforth we identify $(x,{\bf x}'_1)=(z_0,{\bf z}'_1)$.  Since we have the freedom to postcompose $v$ by a fractional linear transformation, we may (and we will) assume that $y_0=0$.

The analog of Corollary~\ref{aardvark} is the following:

\begin{lemma} \label{aardvark2}
The moduli space $\mathcal{M}^{V_1}_{(J,0)}$ is regular.
\end{lemma}

\begin{proof}
Let $\mathcal{M}$ be the moduli space of $V$-regular $(J,0)$-holomorphic maps  from a closed genus $g$ surface $(F,j)$ to $X$ in the class $A$. $\mathcal{M}$ is regular by Lemma~\ref{lemma: regularity and dimension}.  We consider the map
$$\pi:\mathcal{M}\to \op{Sym}^{2g}(\Sigma),$$
which sends $u$ to the intersection of $u$ with $V_1$, counted with multiplicity. Each fiber $\pi^{-1}(D)$ can be viewed as a dense open subset of the space $H^0(\Sigma,\mathcal{O}_D)$ of meromorphic sections of $\Sigma$ with poles at most at $D$.  By Lemma~\ref{example},  $H^0(\Sigma,\mathcal{O}_D)$ is a complex vector space of dimension $l(D)=g+1$ for each $D\in \op{Sym}^{2g}(\Sigma)$; the denseness of $\pi^{-1}(D)$ is a consequence of the fact that $l(D')= g$ if $\deg(D')=2g-1$ and $D'\leq D$ (i.e., $D-D'$ is effective). Hence $\pi$ is the restriction of a holomorphic vector bundle
$$\widetilde\pi: \mathcal{E}\to \op{Sym}^{2g}(\Sigma)$$
to the dense open subset $\mathcal{M}$, where $\widetilde\pi^{-1}(D)=H^0(\Sigma,\mathcal{O}_D)$. In particular, $\pi$ is a submersion.

Let $Z_{{\bf s}_1}\subset \op{Sym}^{2g}(\Sigma)$ be the submanifold consisting of divisors of type $D=\sum_{i=1}^{l({\bf s}_1)} s_{1,i} w'_{1,i}$, $\deg(D)=\deg({\bf s}_1)=2g$, with pairwise distinct $w'_{1,i}$. Then $\pi$ is transverse to $Z_{{\bf s}_1}\subset \op{Sym}^{2g}(\Sigma)$ since $\pi$ is a submersion. This implies the regularity of $\mathcal{M}^{V_1}_{(J,0)}$.
\end{proof}

Consider the following genericity conditions:
\begin{enumerate}
\item[(${\frak G}_1$)] $((z_0,0),{\bf z}'_1)$ is generic with respect to $\overline{\bf a}'=\overline{\bf a}$.
\item[(${\frak G}_2$)] $(z_0,0)=(z_\infty,0)$,\footnote{Here $z_\infty$ is the point at infinity of $\Sigma=\overline{S}$.} ${\bf z}'_1$ is generic, and $\overline{\bf a}'$ is generic subject to the condition that each $\overline{a}'_i$ pass through $z_\infty$; in particular, the points of ${\bf z}'_1$ are disjoint and do not lie on any $\overline{a}_i'$.
\end{enumerate}
Let us write
$$ {\frak E}=({ev}_{(J,0)})^{-1}(\{(z_0,0)\} \times {\bf z}'_1 \times \overline{a}'_1 \times \dots \times \overline{a}'_{2g}),$$
The following lemma is analogous to Lemma~\ref{lemma: no bubbling}:

\begin{lemma}\label{lemma: no bubbling2}
If (${\frak G}_1$) or (${\frak G}_2$) holds, then ${\frak E}$ is compact.
\end{lemma}

\begin{proof}
Arguing by contradiction, let $u_i\in {\frak E}$ be a sequence that converges to a limit curve $u:\Sigma'\to X$. As argued in Lemma~\ref{lemma: no bubbling}, $\Sigma'$ is obtained from $\Sigma$ by attaching spheres and a nonconstant sphere bubble must be mapped to some $\{w\} \times \C\P^1$. We have the homological constraint:
\begin{equation}\label{eqn: 2g}
\langle u(\Sigma'),V_1\rangle = \langle u(\Sigma'),V_\infty\rangle=2g.
\end{equation}

\s\n (1) Suppose (${\frak G}_1$) holds. If $w\not\in {\bf z}'_1$, then $V_1$ must be a component of $u(\Sigma')$; otherwise, $\langle u(\Sigma'), V_1\rangle \geq 2g+1$, which contradicts Equation~\eqref{eqn: 2g}.  Hence $u(\Sigma')$ consists of $V_1$ and $2g$ bubbles $\{z'_{\infty,i}\}\times \C\P^1$, where $z'_{\infty,i}\in \overline{a}_i$, $i=1,\dots,2g$, and $w$ is equal to some $z'_{\infty,i}$. This contradicts the fact that $u(\Sigma')$ passes through $(z_0,0)$.

On the other hand, if $w\in {\bf z}'_1$, then $w\not \in \overline{\bf a}$.  Again, by Equation~\eqref{eqn: 2g}, $u(\Sigma')$ consists of $V_\infty$ and bubbles $\{z'_{1,i}\}\times \C\P^1$, $i=1,\dots,l({\bf s}_1)$, of total multiplicity $2g$, and this contradicts the fact that $u(\Sigma')$ passes through $(z_0,0)$.

\s\n (2) Suppose (${\frak G}_2$) holds. If $w\in {\bf z}'_1$, then the proof is the same as (1).

On the other hand, if $w\not\in {\bf z}'_1$, then $u(\Sigma')$ consists of $V_1$ and $2g$ bubbles $\{z'_{\infty,i}\}\times \C\P^1$, $z'_{\infty,i}\in \overline{a}'_i$. Since $u(\Sigma')$ passes through $(z_\infty,0)$, we must have $z'_{\infty,i}=z_\infty$ for one or more $i=1,\dots,2g$. Let $D(\overline{\bf a}')$ be the set of divisors $z'_{\infty,1}+\dots+z'_{\infty,2g}$, $z'_{\infty,i}\in \overline{a}'_i$, such that $z'_{\infty,i}=z_\infty$ for at least one $i$. Applying the renormalization procedure of \cite[Proposition 6.6]{IP1} as in Lemma~\ref{lemma: no bubbling}, we obtain a meromorphic function $\xi: \Sigma\to \C\P^1$ with divisor $D(\xi)=D_2- D_1$, where $D_1\in D(\overline{\bf a}')$ and $D_2\in Z_{{\bf s}_1}$.

By the Abel-Jacobi theorem, the Picard group $\op{Pic}^{2g}(\Sigma)$ of divisors of degree $2g$ modulo linear equivalence is isomorphic to a $g$-dimensional complex torus. Since $\dim_\R D(\overline{\bf a}')<2g$, its image in $\op{Pic}^{2g}(\Sigma)$ has $\op{dim}_\R<2g$.  Hence, a generic $D_2\in Z_{{\bf s}_1}$ is not linearly equivalent to any $D_1\in D(\overline{\bf a}')$, provided $\overline{\bf a}'$ is generic. This contradicts the existence of $\xi$.
\end{proof}

\begin{lemma} \label{lemma: count}
If (${\frak G}_1$) or (${\frak G}_2$) holds, then the invariant $G_2(\Sigma,{\bf s}_1)$  is the count of holomorphic maps
$$v: (\Sigma,z_0,{\bf z}'_1,{\bf z}'_\infty) \to \C\P^1,$$
such that $v(z_0)=0$, $v(z'_{1,i})=1$ with multiplicity $s_{1,i}$ for $i=1,\dots,l({\bf s}_1)$, $v({\bf z}'_\infty)=\infty$, and $z'_{\infty,i}\in \overline{a}'_i$ for $i=1,\dots,2g$.
\end{lemma}

\begin{proof}
This follows from Lemmas~\ref{aardvark2} and \ref{lemma: no bubbling2}.
\end{proof}

\begin{lemma} \label{lemma: inv}
$G_{2}(\Sigma,{\bf s}_1)$ does not depend on ${\bf s}_1$.
\end{lemma}

\begin{proof}
Let us write $ev^{V_1,{\bf s}_1}_{(J,0)}$ and $\mathcal{M}^{V_1,{\bf s}_1}_{(J,0)}$ for $ev^{V_1}_{(J,0)}$ and $\mathcal{M}^{V_1}_{(J,0)}$ corresponding to ${\bf s}_1$. By the proof of Lemma~\ref{aardvark2}, the moduli spaces $\mathcal{M}^{V_1,{\bf s}_1}_{(J,0)}$ for all the ${\bf s}_1$ can be smoothly glued into $\pi: \mathcal{M}\to \op{Sym}^{2g}(\Sigma)$. The lemma then follows from Lemma~\ref{lemma: count} by intersecting $\pi^{-1}(D)$ with generic $(z_0,0)$ and generic $(\overline{a}'_i,\infty)$, $i=1,\dots,2g$, and observing that the count is locally constant on a neighborhood of $D\in \op{Sym}^{2g}(\Sigma)$.
\end{proof}

In view of Lemma~\ref{lemma: inv}, from now on we may assume that ${\bf s}_1=(1,\dots,1)$ with $l({\bf s}_1)=2g$. For simplicity we write $G_2(\Sigma)$ for $G_2(\Sigma,{\bf s}_1)$.

The proof strategy for Theorem~\ref{thm: calc of G sub 2} is to first compute $G_2(T^2)$ and then reduce the higher genus case to $g=1$.

\subsubsection{Case of $T^2$}

We consider the situation where $\Sigma=T^2$. Suppose (${\frak G}_1$) holds with $g=1$.

\begin{lemma} \label{lemma: application of Riemann-Roch}
If $z_1'\not= z_2'$, then every meromorphic function on $T^2$ with poles at $z'_1$ and $z'_2$ of order $1$ can be written as $a_1+a_2 f$, where $a_1,a_2\in \C$, $a_2\not=0$, and $f$ is a function on $T^2$ with poles at $z'_1$ and $z'_2$ of order $1$.
\end{lemma}

In other words, $f$ is unique up to a postcomposition by a fractional linear transformation $\C\P^1 \to \C\P^1$ which fixes $\infty$.

\begin{proof}
By Lemma~\ref{example}, $l(z'_i)=1$ and is given by the constants, and $l(z'_1+z'_2)=2$ and one dimension is taken by the constants.  The lemma follows.
\end{proof}

\begin{lemma} \label{lemma: G two for T two}
$G_2(T^2)=\pm 1$.
\end{lemma}

\begin{proof}
Suppose (${\frak G}_1$) holds with $g=1$. In view of Lemmas~\ref{lemma: count} and \ref{lemma: inv}, $G_2(T^2)$ is the count of irreducible curves $u\in \mathcal{M}_{g=1}(X,A,J)$ that pass through $(z_0,0)$, $(z'_1,1)$, $(z'_2,1)$, $\overline{a}'_1\times\{\infty\}$, and $\overline{a}'_2\times\{\infty\}$, where $u$ intersects $\overline{a}'_1\times\{\infty\}$ and $\overline{a}'_2\times\{\infty\}$ at distinct points.

\begin{figure}[ht]
\begin{center}
\psfragscanon
\psfrag{a}{\small $z_0$}
\psfrag{b}{\small $z_1'$}
\psfrag{c}{\small $z_2'$}
\psfrag{d}{\small $\overline{a}'_1$}
\psfrag{e}{\small $\overline{a}'_2$}
\psfrag{0}{\small $0$}
\psfrag{1}{\small $1$}
\psfrag{3}{\small $y_1$}
\psfrag{4}{\small $y_2$}
\includegraphics[width=4cm]{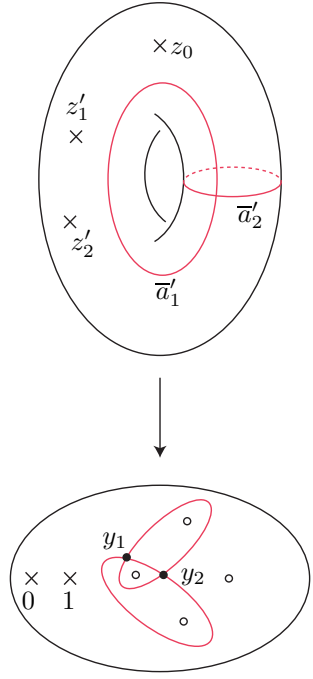}
\end{center}
\caption{The branched double cover $v$ of the torus $T^2$. The branch points $w_1,\dots,w_4$ of $v$ are indicated by $\circ$.} \label{fig: hyperelliptic2}
\end{figure}

According to Lemma~\ref{lemma: application of Riemann-Roch}, a branched double cover $v: T^2\to \C\P^1$ satisfying $v(z_0)=0$ and $v(z_1')=v(z_2')=1$ is determined up to postcomposition by $\eta\in PSL(2,\C)$ which fixes $0$ and $1$; see Figure~\ref{fig: hyperelliptic2}. Let $w_1,\dots,w_4\in \C\P^1$ be the branch points of $v$.  Since $z_1'$ and $z_2'$ are generic, we may assume that $w_i\not=0,1$ for $i=1,\dots,4$. Let $C_{ij}$ be a simple closed curve which bounds a small neighborhood of an arc from $w_i$ to $w_j$.  We may assume that $C_{12}$ and $C_{23}$ intersect only at $2$ points $y_1,y_2$.  Let $\overline{a}'_1,\overline{a}'_2$ be lifts of $C_{12},C_{23}$ such that $\#(\overline{a}'_1\cap \overline{a}'_2)=1$ and $v(\overline{a}'_1\cap \overline{a}'_2)=y_1$.  We now impose the condition $\eta\circ v(z_3')=\eta \circ v(z_4')=\infty$, where $z_3'\in \overline{a}'_1$, $z_4'\in \overline{a}'_2$, and $\eta$ is a fractional linear transformation which fixes $0,1$. This is possible only if $z_3',z_4'\in v^{-1}(y_i)$ for $i=1$ or $2$. On the other hand, since $y_1,y_2$ are not branch points, $z_3'$ and $z_4'$ must be distinct. Hence the only possibility is $\{z_3',z_4'\}=v^{-1}(y_2)$ and $\eta\in PSL(2,\C)$ is the uniquely determined by $\eta(0)=0$, $\eta(1)=1$, $\eta(y_2)=\infty$. This completes the proof of the lemma.
\end{proof}

\subsubsection{Proof of Theorem~\ref{thm: calc of G sub 2}}
\label{subsubsection: reduction for G 2}

We now reduce to the torus case and prove Theorem~\ref{thm: calc of G sub 2}. Suppose that $g(\Sigma)=2$; the general case is similar.

The strategy is to degenerate the fiber $\Sigma$ into a nodal surface $\Sigma^\infty$ consisting of two irreducible components of genus one. More precisely, let $\Sigma^\tau=(\overline{S},\overline{i}_\tau)$, $\tau\in[0,\infty)$, be a $1$-parameter family of $g=2$ Riemann surfaces which degenerates into a nodal Riemann surface
\begin{equation*}
(\Sigma^\infty,{\frak n})  = \left((\Sigma_-,{\frak n}_-)\sqcup (\Sigma_+,{\frak n}_+)\right)/ {\frak n}_-\sim {\frak n}_+,
\end{equation*}
where $\Sigma_-, \Sigma_+$ have genus one and ${\frak n}=\{{\frak n}_-,{\frak n}_+\}$. The nodal surface $\Sigma^\infty$ is obtained topologically as follows: Let $\gamma$ be a separating curve in $\overline{S}$ and let $\sim_0$ be an equivalence relation on $\overline{S}$ such that $x\sim_0 y$ if $x,y\in\gamma$. Then $\Sigma^\infty=\overline{S}/\sim_0$.  We also set $\widetilde{\bf a}'=\overline{\bf a}'/\sim_0$. Here $\overline{\bf a}'$ and $\gamma$ are fixed, while the complex structure $\overline{i}_\tau$ varies as $\tau\to \infty$.

Since $G_2(\Sigma)$ only depends on the homology classes of the $\overline{a}'_j$, $j=1,\dots, 4$, we may assume that the following hold:
\begin{enumerate}
\item[(i)] $\gamma$ is disjoint from the $\overline{a}'_j$, $j=1,\dots,4$, $\widetilde{a}'_1,\widetilde{a}'_2\subset \Sigma_+$, and $\widetilde{a}'_3,\widetilde{a}'_4\subset \Sigma_-$.
\item[(ii)] The points $z_0,z'_1,\dots,z'_4\in \overline{S}-\gamma$ are fixed for all $\tau$ and in the limit $z_0,z'_1,z_2'\in \Sigma_+$ and $z'_3,z'_4\in\Sigma_-$.
\item[(iii)] $\overline{a}'_1,\dots,\overline{a}'_4$ are generic and $z_0,z_1',\dots,z_4'$ are generic.
\end{enumerate}

Let $u^\tau: \Sigma^\tau\to\Sigma^\tau\times\C\P^1$, $\tau\to \infty$, be a sequence of curves in the class $A=[\Sigma^\tau\times\{pt\}]+4[\{pt\}\times\C\P^1]$ that pass through
$$(z_0,0),(z'_1,1),\dots,(z'_4,1), \overline{a}'_1\times\{\infty\},\dots,\overline{a}'_4\times\{\infty\}.$$
By Lemma~\ref{lemma: no bubbling2}, no bubbling can occur at finite $\tau$.  Let $u=u_-\cup u_0\cup u_+$ be the limit of $u^\tau$, where $\op{Im}(u_\pm)\subset \Sigma_\pm\times\C\P^1$, $u_\pm$ has no components of $\{{\frak n}\}\times \C\P^1$, and $u_0$ maps to $\{{\frak n}\}\times \C\P^1$.  If $u_\pm$ is irreducible, then we write $v_\pm=\pi_2\circ u_\pm$, where $\pi_2$ is the projection to $\C\P^1$.

\begin{lemma} \label{lemma: limit pigeon}
The levels $u_+$ and $u_-$ are irreducible and $u_0=\varnothing$.
\end{lemma}

\begin{proof}
The proof is similar to that of Lemma~\ref{lemma: no bubbling2}. Arguing by contradiction, suppose $u_+$ has a bubble component. The bubble is necessarily of the form $\{z\} \times \C\P^1$.  When $\tau$ is finite, $\op{Im}(u^\tau) \cap (\Sigma^\tau\times\{1\})=\{(z_1',1),\dots,(z_4',1)\}$.

We claim that $\Sigma_+ \times \{1\}$ or $\Sigma_+\times\{\infty\}$ must be a component of $u_+$. Otherwise, suppose that $\Sigma_+\times\{w\}$ is not a component of $u_+$ for $w=1$ or $\infty$. Since $z$ cannot be in $\{z_1',z_2'\}$ and $\widetilde{a}'_1\cup\widetilde{a}'_2$ at the same time, $\langle u_+, \Sigma_+\times\{w\}\rangle \geq 3$ for $w\in \C\P^1$. We also have $\langle u_-,\Sigma_-\times\{w\}\rangle\geq 2$ and $\langle u_0,\Sigma_\pm\times\{w\}\rangle\geq 0$. Hence $\langle u^\tau,\Sigma^\tau\times\{w\}\rangle>4$ for $w\in \C\P^1$, which is a contradiction. This implies the claim.

Now, the claim implies that $u_+$ consists of $\Sigma_+\times\{w\}$, $w=1$ or $\infty$, and $3$ bubbles: $\{z_0\} \times \C\P^1$ and either $\{z_1',z_2'\} \times \C\P^1$ or $\{z_3',z_4'\} \times \C\P^1$, where $z_{j+2}' \in \widetilde{a}'_j$, $j=1,2$. Since $\langle u_-,\Sigma_-\times\{w\}\rangle\geq 2$ and $\langle u_0,\Sigma_\pm\times\{w\}\rangle\geq 0$, we again obtain $\langle u^\tau,\Sigma^\tau\times\{w\}\rangle>4$ for $w\in \C\P^1$.
Hence $u_+$ does not have bubble components and is irreducible.

We can similarly prove that $u_-$ is irreducible and $u_0=\varnothing$.
\end{proof}

By Lemma~\ref{lemma: limit pigeon}, $u_+$ and $u_-$ are irreducible. Hence we are counting pairs $(v_-,v_+)$ such that $v_+({\frak n}_+)=v_-({\frak n}_-)$.  By Lemma~\ref{lemma: G two for T two},
$$G_2(\Sigma_+)=G_2(\Sigma_+;\{(z_0,0)\},z_1',z_2',\widetilde{a}'_1,\widetilde{a}'_2)=\pm 1$$
and is represented by a unique map $v_+$. Now the point $({\frak n}_-,v_+({\frak n}_+))\in \Sigma_-\times\C\P^1$ plays the role of $(z_0,0)$ for $v_-$. Hence, by Lemma~\ref{lemma: G two for T two},
$$G_2(\Sigma_-)=G_2(\Sigma_-; \{({\frak n}_-,v_+({\frak n}_+))\},z'_3,z'_4,\widetilde{a}'_3,\widetilde{a}'_4)=\pm 1.$$

 It remains to consider the gluing of $u_+$ and $u_-$. Note that the gluing can be viewed as an SFT gluing: we can consider
$$((\Sigma_- - \{{\frak n}_+\})\times \C\P^1) \sqcup ((\Sigma_+ -\{{\frak n}_-\})\times \C\P^1)$$
as a $2$-level building with cylindrical ends and the simply-covered stable Hamiltonian orbits come in a $\C\P^1$-Morse-Bott family. The curves $u_+$ and $u_-$, viewed as SFT curves $\widetilde u_+$ and $\widetilde u_-$, are simple and the gluing conditions for $u_+$ and $u_-$ are transverse by (iii). Hence the usual SFT gluing procedure with weights applies and we obtain
\begin{equation} \label{eqn: G sub 2}
G_2(\Sigma)= G_2(\Sigma_+)\times G_2(\Sigma_-)=\pm 1.
\end{equation}

This completes the proof of Theorem~\ref{thm: calc of G sub 2}.

\subsection{Third relative Gromov-Witten calculation}
\label{subsection: third GW calculation}

\subsubsection{Definitions}
\label{subsubsection: third GW calculation definitions}

Let $D^2$ be the closed unit disk $\{|z|\leq 1\}\subset \C$ and let $\Sigma=\overline{S}$. We consider $(X,\omega)$, where $X=\Sigma\times D^2$ and $\omega$ is a product symplectic form. Recall the Lagrangian submanifold $\widehat{\bf a}\subset \overline{S}$ as well as its components $\widehat{a}_i$, $i=1,\dots,2g$, from Section~I.\ref{P1-coconut}. Then $L_{\widehat{\bf a}}=\widehat{\bf a}\times \bdry D^2\subset \bdry X$ is a Lagrangian submanifold of $(X,\omega)$. We write $\pi_i$, $i=1,2$, for the projection of $X$ onto the $i$th factor. Let $A$ be the relative homology class
$$A=[\Sigma\times\{pt\}]+ 2g\cdot[\{pt\}\times D^2]\in H_2(X,\bdry X),$$
and let $(F,j,x,{\bf x}')$ be a compact Riemann surface of genus $g$ with $2g$ boundary components, an interior marked point $x$, and boundary marked points ${\bf x}'=(x'_1,\dots,x'_{4g})$, where $x_i',x'_{2g+i}$ are on the $i$th boundary component $\bdry_iF$.

Let $\mathcal{J}$ be the space of compatible almost complex structures on $(X,\omega)$ and let $\mathcal{J}^{prod}\subset \mathcal{J}$ be the subset of product complex structures on $X$. Given $J\in \mathcal{J}$, let $\mathcal{M}(X,A,J)$ be the moduli space of holomorphic maps $$u: (F,j,x,{\bf x}')\to (X,J)$$ in the class $A$ such that $u(\bdry_i F)\subset L_{\widehat{a}_i}$, modulo automorphisms of the domain.

Consider the evaluation map:
\begin{equation}
\label{eqn: ev map 2}
ev: \mathcal{M}(X,A,J)\to \mathfrak{X}= X\times L_{\widehat{a}_1}\times\dots\times L_{\widehat{a}_{2g}} \times L_{\widehat{a}_1}\times\dots\times L_{\widehat{a}_{2g}},
\end{equation}
$$(u,x,{\bf x'})\mapsto (u(x),u(x'_1),\dots,u(x'_{4g})).$$
Let ${\bf z}=(z_1,\dots,z_{4g})$ be a $4g$-tuple of points such that $z_i, z_{2g+i}\in \widehat{a}_i$, $i=1,\dots,2g$. Then we write
$$\mathcal{M}_J(z_\infty,{\bf z}):=ev^{-1}((z_\infty,0),(z_1,1),\dots,(z_{2g},1),(z_{2g+1},-1),\dots,(z_{4g},-1)),$$
where $z_\infty$ is the point at infinity.

\begin{defn}
The relative Gromov-Witten type invariant $G_3(\Sigma; J;z_\infty ,{\bf z})$ is the count of equivalence classes $(u,x,{\bf x}') \in \mathcal{M}_J(z_\infty,{\bf z})$, where $J\in \mathcal{J}^{prod}$ and ${\bf z}$ are generic.
\end{defn}

\begin{thm} \label{thm: calc of G sub 3}
If $J\in\mathcal{J}^{prod}$ and ${\bf z}$ is generic, then $\mathcal{M}_J(z_\infty,{\bf z})$ is regular and compact. The count $G_3(\Sigma; J;z_\infty ,{\bf z})$ is independent of the choice of $J\in\mathcal{J}^{prod}$ and ${\bf z}$ generic and is equal to $\pm 1$.
\end{thm}

\begin{proof}
Recall that $\{z_\infty\}= \overline{a}_i-\widehat{a}_i$ for each $i=1,\dots,2g$. Observe that if $(u,x,{\bf x}')\in \mathcal{M}_J(z_\infty,{\bf z})$, then $u$ is transverse to $\bdry X$; this is easily seen by projecting to the base $D^2$ and observing that each $L_{\widehat{a}_i}$ is used exactly once.

The first statement is proved in Lemmas~\ref{berkeley} and \ref{lemma: compact}. The independence of $G_3(\Sigma; J;z_\infty ,{\bf z})$ is proved in Lemma~\ref{berkeley2} and $G_3(\Sigma; J;z_\infty ,{\bf z})$ is computed in Sections~\ref{subsubsection: reduction to torus for G 3}--\ref{subsubsection: reduction to sphere}.
\end{proof}

\subsubsection{Transversality}

We recall the following automatic transversality result from \cite{HLS}; see Theorem~2 and Remark (1) following the statement of Theorem~2 in \cite{HLS}:

\begin{thm}[Hofer-Lizan-Sikorav]
\label{thm: automatic transversality HLS}
Let $(F,j)$ be a compact Riemann surface with nonempty boundary, $(M,J)$ be an almost complex $4$-manifold, and $L\subset (M,J)$ be a totally real surface. Let
$$u_0:(F,\bdry F, j)\to (M,L,J)$$
be a holomorphic map which sends $\bdry F$ to $L$.  If $u_0$ is immersed and the sum of the Maslov indices of $u_0|_{\bdry F}$ with respect to any unitary trivialization of $u_0^*TM$ is positive, then the set of holomorphic maps $u: (F,\bdry F, j)\to (M,L,J)$ near $u_0$ is regular.
\end{thm}

Using Theorem~\ref{thm: automatic transversality HLS}, we prove the following key result:

\begin{lemma}
\label{lemma: moduli prime for D regular and dim 4g plus 2}
If $J\in \mathcal{J}^{prod}$, then the moduli space $\mathcal{M}(X,A,J)$ is regular and
$$\dim_{\R} \mathcal{M}(X,A,J)= 8g+4.$$
\end{lemma}

\begin{proof}
Let $(u,x,{\bf x})\in\mathcal{M}(X,A,J)$, where $u:(F,j)\to (X,J)$. Since there are no branch points of $\pi_2\circ u$ along $\bdry F$, the curve $u$ is immersed near $\bdry F$. On the other hand, the restriction of $\pi_1\circ u$ to $int(F)$ is a biholomorphism onto its image. Hence $u$ is an immersion on all of $F$.

The total Maslov index $\mu(u)$ relative to the Lagrangian $L_{\widehat{\bf a}}$, with respect to any unitary trivialization of $u^*TX$ is:
$$\mu(u)= 2g\cdot \mu(\{pt\}\times D^2) + 2c_1(T\Sigma),$$
where $\{pt\}\times D^2$ has boundary on $L_{\widehat{\bf a}}$ and $\mu(\{pt\}\times D^2)$ is computed with respect to a trivialization of $TX|_{\{pt\}\times D^2}$.
An easy calculation gives $\mu(\{pt\}\times D^2)=2$ and $c_1(T\Sigma)=\chi(\Sigma)=2-2g.$ Hence
\begin{equation}
\mu(u)=2g (2)+2(2-2g)= 4>0.
\end{equation}
The regularity of $\mathcal{M}(X,J,A)$ then follows from Theorem~\ref{thm: automatic transversality HLS}. [Strictly speaking, Theorem~\ref{thm: automatic transversality HLS} applies to the case where $(M,J)$ has no boundary.  However, we note that Theorem~\ref{thm: automatic transversality HLS} is a local result, i.e., only uses an open neighborhood of the image of $u_0$.  The complex manifold $(X=\Sigma\times D^2,J)$ can then be slightly enlarged to a product complex manifold $\Sigma\times D^2_{1+\varepsilon}$, where $D^2_{1+\varepsilon}$ is a disk of radius $1+\varepsilon$, so that the theorem applies.  We then note that any nearby curve in the enlargement with boundary on $L_{\widehat{\bf a}}$ has image in $X$ by considering the projection to $D^2_{1+\varepsilon}$.]

Using the usual index formula for holomorphic curves with Lagrangian boundary, we obtain:
\begin{align*}
\dim_{\R}\mathcal{M}(X,A,J) &=  -\chi(F) +\mu(u) +(4g+2)\\
& = -(2-4g) +4 +(4g+2)= 8g+4,
\end{align*}
where the summand $4g+2$ comes from the marked points $x$ and ${\bf x}'$. This completes the proof of Lemma~\ref{lemma: moduli prime for D regular and dim 4g plus 2}.
\end{proof}

\begin{lemma} \label{berkeley}
If $J\in \mathcal{J}^{prod}$, then $\mathcal{M}_J(z_\infty,{\bf z})$ is regular for generic ${\bf z}$.
\end{lemma}

\begin{proof}
This is basically Sard's theorem, except that $z_\infty$ is nongeneric.  The modification is as follows: Let $(u,x,{\bf x}')\in \mathcal{M}_J(z_\infty,{\bf z})$. Then $ev_*(u,x,{\bf x}')$ composed with the projection onto $T_{(z_\infty,0)} X$ is surjective due to two inputs: (i) postcomposition of $u$ by a fractional linear transformation and (ii) variation of the point $x$. Hence $\mathcal{M}'=\mathcal{M}(X,A,J)\cap \{u(x)=(z_\infty,0)\}$ is transversely cut out.  The rest of the proof follows from applying Sard's theorem to
$$ev':\mathcal{M}'\to L_{\widehat{a}_1}\times\dots\times L_{\widehat{a}_{2g}}\times L_{\widehat{a}_1}\times\dots\times L_{\widehat{a}_{2g}},$$
$$(u,x,{\bf x}')\mapsto (u(x_1'),\dots,u(x'_{4g})).$$
\end{proof}

\subsubsection{Compactness}

\begin{lemma}\label{lemma: compact}
If $J\in\mathcal{J}^{prod}$, then $\mathcal{M}_J(z_\infty,{\bf z})$ is compact for generic ${\bf z}$.
\end{lemma}

\begin{proof}
Let $u_i: F_i\to X$ be a sequence of curves in $\mathcal{M}_J (z_\infty,{\bf z})$ that limits to $u: F\to X$. We will eliminate all possible degenerations in the limit to show that $u\in \mathcal{M}_J(z_\infty,{\bf z})$.  Note in particular that we are assuming that all the points of ${\bf z}$ are distinct.

\s\n
(i) We first claim that $u$ can have no disk bubbles. 
Arguing by contradiction, suppose that $u$ has a disk bubble $u'$. Since $\pi_1\circ u_i|_{int(F_i)}$ is a biholomorphism onto its image for all $i$, $\op{Im}(u')=\{z \} \times D^2$ for some $z$.

The point $z$ must be an element of $\{z_1,\dots,z_{2g}\}$: For each $u_i$ with $i\gg 0$, there exists $w_i\in int(D^2)$ close to $1$ such that $u_i$ intersects $\Sigma\times \{w_i\}$ near $\{z_1,\dots,z_{2g}\}\times \{w_i\}$.  On the other hand, $u'$ intersects $\Sigma\times \{w_i\}$ at $(z,w_i)$. This implies that $z\in \{z_1,\dots,z_{2g}\}$.

Using the same argument as in the previous paragraph, we can show that $z\in \{z_{2g+1},\dots,z_{4g}\}$.  This is a contradiction.

\s\n
(ii) Next we claim that $u(\bdry F)$ does not intersect $\{z_\infty\}\times \bdry D^2$.  Arguing by contradiction, suppose that $u(\bdry F)\cap (\{z_\infty\}\times \bdry D^2)\not=\varnothing$. Since $u$ cannot have a disk bubble $\{ z_\infty \}\times D^2$ by (i), there exists a point $z\in\Sigma-\overline{\bf a}$ close to $z_\infty$, such that $u$ intersects $\{z\}\times D^2$ at two points: once near $(z_\infty,0)$ and once near $\{z_\infty\}\times \bdry D^2$. Hence $\langle u, \{z\}\times D^2\rangle\geq 2$, which is a contradiction.

A similar argument also implies that $u$ has no nodes along the boundary: if there are two components of $\bdry F$ that map to the same $\widehat{a}_i\times \bdry D^2$, then $\langle u, \{z\}\times D^2\rangle\geq 2$  for $z\in \Sigma-\overline{\bf a}$, which is a contradiction.

\s\n
(iii) Finally we eliminate interior nodes.  Since $u$ has no disk bubbles and sphere bubbles (the latter easily follows by projecting to $D^2$), $u$ must have a component $u':F'\to X$ of genus $<g$ if $u$ has interior nodes. Since $\pi_1\circ u'|_{int(F')}$ cannot map onto the genus $g$ surface $\Sigma-(\widehat{\bf a}-N(z_\infty))$, we have a contradiction.
\end{proof}

\begin{lemma} \label{berkeley2}
The count $G_3(\Sigma; J;z_\infty ,{\bf z})$ is independent of the choice of $J\in\mathcal{J}^{prod}$ and ${\bf z}$ generic.
\end{lemma}

\begin{proof}
We prove the compactness of $\coprod_{t\in[0,1]}\mathcal{M}_{J_t}(z_\infty,{\bf z}_t)$, where $(J_t,{\bf z}_t)$, $t\in[0,1]$, is a generic $1$-parameter family and we are writing ${\bf z}_t=((z_1)_t,\dots,(z_{4g})_t)$. The proof is similar to that of Lemma~\ref{lemma: compact}.  The only difference is that two of the points of ${\bf z}_{t_0}$ may come together at isolated $t_0\in[0,1]$. Hence, compared to Lemma~\ref{lemma: compact}(i), it is possible that a limit holomorphic curve $u$ consists of a disk bubble $u'$ and a component $u'': F''\to X$ where $F''$ is a compact surface of genus $g$ with $2g-1$ boundary components. By a calculation similar to that of Lemma~\ref{lemma: moduli prime for D regular and dim 4g plus 2}, $u''$ is regular and $\op{ind}(u'')=8g-1$. On the other hand, passing through $z_\infty$ and remaining $4g-2$ points of ${\bf z}_{t_0}$ (which we may assume is generic since ${\bf z}_t$, $t\in[0,1]$, is generic) is a constraint of dimension $4+2(4g-2)=8g$, which is overdetermined.  Hence $u$ must be irreducible, and the proof of Lemma~\ref{lemma: compact} carries over.

The regularity of $\coprod_{t\in[0,1]}\mathcal{M}_{J_t}(z_\infty,{\bf z}_t)$ follows from Lemma~\ref{berkeley} adapted to a generic $1$-parameter family. The compactness and regularity then imply the lemma.
\end{proof}

\subsubsection{Reduction to a torus}
\label{subsubsection: reduction to torus for G 3}

We now explain how to reduce to the case of a torus. Suppose that $g(\Sigma)=2$; the general case is similar. As in Section~\ref{subsubsection: reduction for G 2}, we degenerate $\Sigma^\tau=(\Sigma,i_\tau)$, $\tau\in [0,\infty)$, into a nodal Riemann surface
$$(\Sigma^\infty,{\frak n})=\left((\Sigma_-,{\frak n}_-)\sqcup (\Sigma_+,{\frak n}_+)\right)/ {\frak n}_-\sim {\frak n}_+$$
by pinching along a separating curve $\gamma$. Let $(u^\tau,x^\tau,({\bf x}')^\tau)\in \mathcal{M}_{J_\tau}(z_\infty,{\bf z})$ be a sequence where
$$u^\tau:(F_\tau,j_\tau)\to (\Sigma^\tau\times D^2,J_\tau)$$
and $J_\tau$ is the product complex structure $(i_\tau, i_{D^2})$. Here we assume the following:
\begin{enumerate}
\item[(i)] $z_\infty$, ${\bf z}$, and $\widehat{\bf a}$ are fixed for all $\tau\in[0,\infty)$;
\item[(ii)] $\gamma$ is disjoint from ${\bf z},\widehat{a}_1,\widehat{a}_2$, but intersects $\widehat{a}_3,\widehat{a}_4$ at two points each;
\item[(iii)] in the limit $z_\infty,z_1,z_2,z_5,z_6\in \Sigma_+$ and $z_3,z_4,z_7,z_8\in \Sigma_-$.
\end{enumerate}
Let $\widetilde{\bf a}$ be the quotient of $\widehat{\bf a}$ at $\tau=\infty$. Let $v^\tau=\pi_2\circ u^\tau$, where $\pi_2$ is the projection to $D^2$. We also write ${\bf z}_1=(z_1,z_2,z_5,z_6)$ and ${\bf z}_2=(z_3,z_4,z_7,z_8)$.

We first note that $u^\tau$ cannot degenerate as $\tau$ approaches a finite $\tau_0$. This is as argued in Lemma~\ref{lemma: compact}.

Let $u=u_-\cup u_0\cup u_+$ be the limit of the sequence $u^\tau$ as $\tau\to \infty$, where $\op{Im}(u_\pm) \subset \Sigma_\pm \times D^2$, $u_\pm$ has no components of $\{{\frak n}\}\times D^2$, and $u_0$ maps to $\{{\frak n}\}\times D^2$, and let $v_\pm=\pi_2\circ u_\pm$.  By the argument of Lemma~\ref{lemma: compact}(i), disks $\{z\}\times D^2$ cannot bubble off, since otherwise they would introduce extra intersection points with
$$\Sigma^\infty\times\{w\}=\left((\Sigma_+\sqcup \Sigma_-)/\sim\right)\times \{w\},$$
where $w\in int(D^2)$.  In particular, this implies that $u_0=\varnothing$.

Let $F_\pm$ be the domain of $u_\pm$. By Lemma~\ref{lemma: compact}(ii), $u_+(\bdry F_+)$ does not intersect $\{z_\infty\}\times \bdry D^2$. Let $\#(\bdry F_\pm)$ be the number of boundary components of $F_\pm$.

We claim that $\#(\bdry F_+)=\#(\bdry F_-)=2$. First observe that $\# (\bdry F_+)\geq 2$, since two boundary components are needed to map to $\widetilde{a}_i$, $i=1,2$; similarly $\#(\bdry F_-)\geq 2$  since $z_3,z_4,z_7,z_8\in \Sigma_-$. The restriction of $v_+$ to each component $C$ of $\bdry F_+$ is either a positive degree map $C\to \bdry D^2$ or is a constant map to a point $w\in\bdry D^2$. If $v_+$ maps $C$ to a point $w\in \bdry D^2$, then $u_+$ maps an irreducible component of $F_+$ to a fiber $\Sigma_+\times\{w\}$. This in turn implies that $u_+$ has disk bubbles, a contradiction. Hence the restriction of $v_+$ to each $C$ is a positive degree map. If $\#(\bdry F_+)> 2$, then $\op{deg}(v_+|_{\bdry F_+})>2$. Similarly we obtain that $\op{deg}(v_-|_{\bdry F_-})\geq 2$. This implies that $\op{deg}((v_+\cup v_-)|_{\bdry F})>4$, which is a contradiction. Hence $\#(\bdry F_+)=\#(\bdry F_-)=2$.

The above claim implies that there exist $x_i,{\bf x}'_i$, $i=1,2$, such that
$$(u_+,x_1,{\bf x}'_1)\in \mathcal{M}_{J_1}(z_\infty,{\bf z}_1), \quad (u_-,x_2,{\bf x}'_2)\in \mathcal{M}_{J_2}({\frak n}_-,{\bf z}_2).$$

By the usual gluing argument we obtain the following:

\begin{lemma}
$$G_3(\Sigma^\tau;J^\tau;z_\infty,{\bf z})=G_3 (\Sigma_+;z_\infty,{\bf z}_1)\times  G_3(\Sigma_-;{\frak n}_-,{\bf z}_2).$$
\end{lemma}

\subsubsection{Two calculations on $\C\P^1\times D^2$}
\label{subsubsection: two calculations}

We now calculate two model situations which are key ingredients in the proof of Lemma~\ref{lemma: G sub 3 for torus} below. Note that all the holomorphic curves on $\C\P^1\times D^2$ that are considered below satisfy automatic transversality; see Hofer-Lizan-Sikorav~\cite[Theorem~$2'$]{HLS}.

Fix real numbers $a>b>0$. Let $\mathcal{S}_1$ be the set of pairs $(v_1,w)$, where $v_1$ is a degree $1$ holomorphic map $D^2\to \C\P^1$  (more precisely, is a biholomorphism onto its image when restricted to $int(D^2)$) such that $v_1(\bdry D^2)\subset \R^+$, $v_1(1)=a$, $v_1(-1)=b$, $v_1(0)=\infty$, and $w$ is a point in $D^2$ such that $v_1(w)=0$, and let $C_1$ be the set of points $w$ for which there is some $v_1$ with $(v_1,w)\in\mathcal{S}_1$. Similarly, let $\mathcal{S}_2$ be the set of pairs $(v_2,w)$, where $v_2$ is a degree $1$ holomorphic map $D^2\to \C\P^1$ such that $v_2(\bdry D^2)\subset \R$, $v_2(1)=a$, $v_2(-1)=b$, $v_2(0)=\infty$, and $w$ is a point on $D^2$ such that $v_2(w)=-i$, and let $C_2$ be the set of points $w$ for which there is some $v_2$ with $(v_2,w)\in\mathcal{S}_2$.

Let $\mathcal{R}_\theta$ be the restriction to $D^2$ of the radial ray which passes through $0\in D^2$ and makes an angle of $\theta$ with the positive real axis.

\begin{lemma} \label{C one}
There exists $0<\theta_0< {\pi\over 2}$ such that $C_1\subset D^2$ can be written as the image of a curve which is parametrized by $\theta\in (\pi-\theta_0,\pi+\theta_0)$, $C_1(\theta)= C_1\cap int(\mathcal{R}_\theta)$ (here we are abusing notation and using $C_1$ for both the curve and its image), and
$$\lim_{\theta\to (\pi +\theta_0)^-} C_1(\theta)= e^{i(\pi+\theta_0)},\quad\displaystyle\lim_{\theta\to (\pi -\theta_0)^+} C_1(\theta)= e^{i(\pi-\theta_0)}.$$
\end{lemma}

See Figure~\ref{fig: disks} for a rough picture of the curve $C_1$.

\begin{figure}[ht]
\begin{overpic}[width=4.5cm]{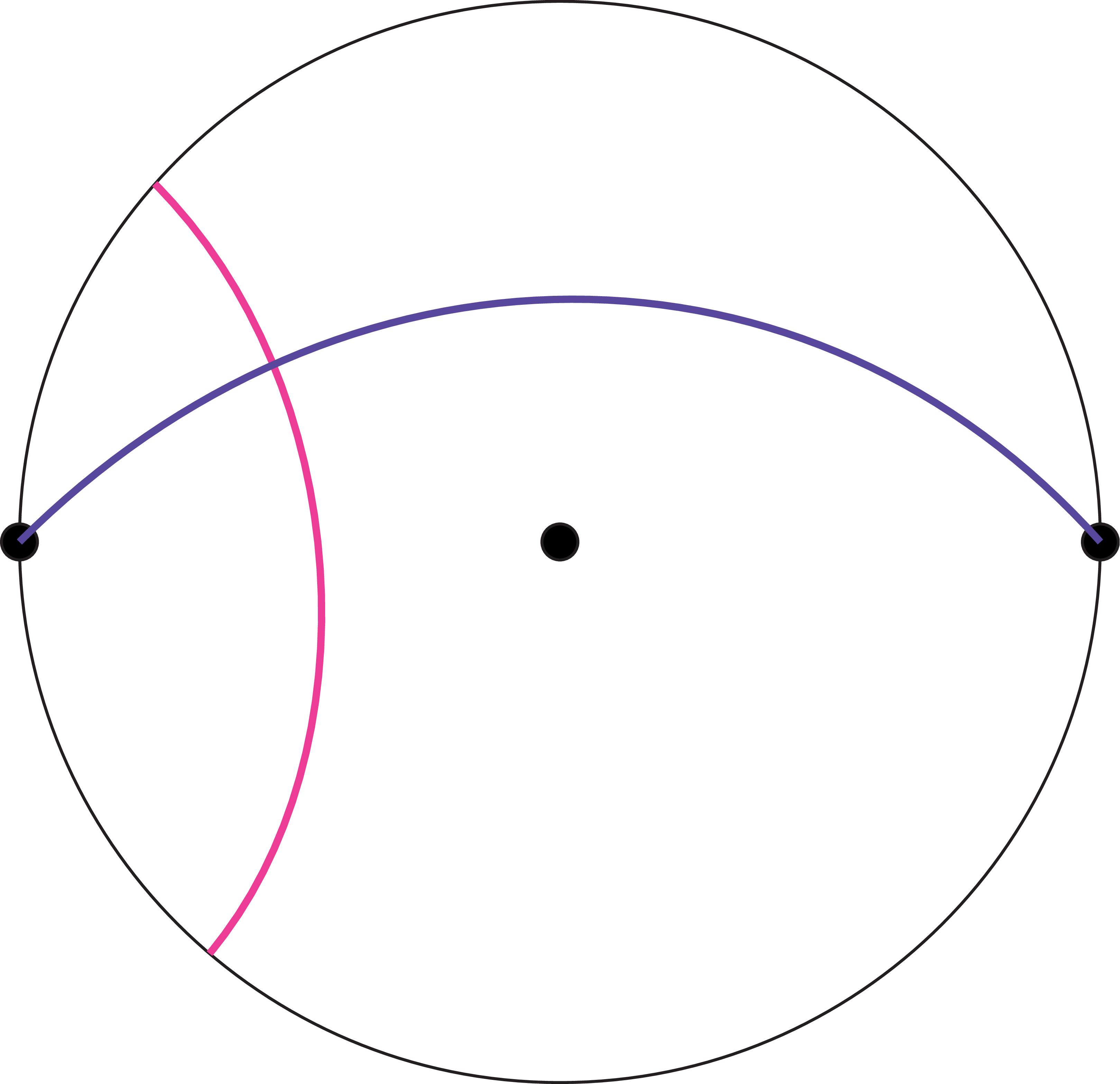}
\put(19,32){\tiny $C_1$} \put(48,74.4){\tiny $C_2$}
\end{overpic}
\caption{} \label{fig: disks}
\end{figure}

\begin{proof}
By the Schwarz reflection principle, a map $v_1$ with $(v_1,w)\in \mathcal{S}_1$ extends to a degree $2$ branched cover $\C\P^1\to \C\P^1$ with two branch points which lie on $\R^+\subset \C\P^1$. Hence $v_1$ admits a factorization
$$D^2\stackrel{f}\to \H \stackrel{\tilde v_1}\to \C\P^1\stackrel{g}\to\C\P^1,$$
where $\H=\{\op{Im}(z)\geq 0\}$ is the upper half plane, $f:D^2\stackrel\sim\to \H$ is a fractional linear transformation, $\tilde v_1(z)=z^2$, and $g\in PSL(2,\R)$.  In order for $v_1(0)=\infty$ and $v_1(w)=0$ to hold, $f$ must map the ray $\mathcal{R}_\theta$ through $0$ and $w$ to the line $\{\op{Re}(z)=0\}\subset \H$. We may set $f(e^{i\theta} )=0$ and $f(0)=i$; these conditions uniquely determine $f=f_\theta$. We leave it to the reader to verify that for each $\theta$, there is a one-to-one correspondence between $(v_1,w)\in \mathcal{S}_1$ with $w\in \mathcal{R}_\theta$ and $w\in \mathcal{R}_\theta$ satisfying the following equality of cross ratios:
\begin{equation} \label{cross ratio}
(\tilde v_1\circ f_\theta(w), \tilde v_1 \circ f _\theta (1); \tilde v_1 \circ f _\theta (-1), -1 )=(0,b;a,\infty).
\end{equation}
Note that $\tilde v_1\circ f_\theta(\mathcal{R}_\theta)=[-1,0]$ and that there is at most one $w\in \mathcal{R}_\theta$ such that Equation~\eqref{cross ratio} holds.

When $\theta=\pi$,  $\tilde v_1 \circ f_\pi (1)=\infty$ and $\tilde v_1 \circ f_\pi (-1) =0$. Hence there is a unique $w\in \mathcal{R}_\pi$ satisfying Equation~\eqref{cross ratio}.  As $\theta$ moves from $\pi$ to ${3\pi\over 2}$, the points $\tilde v_1 \circ f_\theta (1)$ and $\tilde v_1 \circ f_\theta (-1)$ approach each other and become equal when $\theta={3\pi\over 2}$. Hence there exists $0<\theta_0<{\pi\over 2}$ such that there is no $w\in \mathcal{R}_\theta$ for $\pi+\theta_0 <\theta <{3\pi\over 2}$ and there is a unique $w\in \mathcal{R}_\theta$ for $\pi\leq \theta < \pi+\theta_0$. The situation of $\theta\in ({\pi\over 2},\pi]$ is symmetric. The lemma then follows.
\end{proof}

\begin{lemma} \label{C two}
There exists a parametrization of $C_2\subset D^2$ by $\theta\in ({\pi\over 2},{3\pi\over 2})$ such that
$$\lim_{\theta\to {\pi\over 2}+}C_2(\theta) = 1, \quad \lim_{\theta\to {3\pi\over 2}-}C_2(\theta)=-1.$$
\end{lemma}

See Figure~\ref{fig: disks} for a rough picture of the curve $C_2$.

\begin{proof}
As in the proof of Lemma~\ref{C one}, $v_2$ can be factored as
$$D^2\stackrel{f_\theta}\to \H\stackrel{\tilde v_2}\to \C\P^1\stackrel{g}\to \C\P^1,$$
where $\tilde v_2(z)=z^2$ and $f_\theta$ maps $0$ to $i$ and $\mathcal{R}_\theta$ to $\{\op{Re}(z)=0, 0\leq \op{Im}(z)\leq 1\}\subset \H$. Note that there is a unique fractional linear transformation $g$ such that $g(-1)=\infty$, $g(0)=b$, and $g(\infty)=a$.  For each $\theta\in ({\pi\over 2},{3\pi\over 2})$, there exists $(v_2,C_2(\theta))\in\mathcal{S}_2$ such that $C_2(\theta)$ is in one of the half-disks of $D^2$ divided by $\mathcal{R}_\theta\cup \mathcal{R}_{\theta+\pi}$. As a reference point, $C_2(\pi)$ is in the upper half-disk.  As $\theta$ approaches ${\pi\over 2}$ from above, the corresponding $v_2$ sends $-1$ and $1$ to arbitrarily close points.  Hence $\displaystyle\lim_{\theta\to {\pi\over 2}+}C_2(\theta) = 1$. Similarly, $\displaystyle\lim_{\theta\to {3\pi\over 2}-}C_2(\theta)=-1$.
\end{proof}

The following lemma is immediate from Lemmas~\ref{C one} and \ref{C two}; the key ingredient is that the endpoints of $C_1$ and $C_2$ alternate on $\bdry D^2$:

\begin{lemma} \label{lemma: intersection C one and C two}
The total signed count of intersections between $C_1$ and $C_2$ is $\pm 1$.
\end{lemma}

\subsubsection{Reduction to $\C\P^1\times D^2$} \label{subsubsection: reduction to sphere}

We now explain how to further reduce from $T^2$ to $S^2$. We pinch $T^2$ along three parallel, disjoint, essential closed curves $\gamma_1,\gamma_2,\gamma_3$ to obtain a ``sausage''
\begin{equation}\label{eqn: sausage}
\left((\Sigma_1,w_1,w_2') \sqcup (\Sigma_2,w_2,w_3') \sqcup (\Sigma_3,w_3,w_1')\right)/\sim,
\end{equation}
where $\Sigma_i\simeq S^2$, $i=1,2,3$, and $w_i\sim w_i'$, $i=1,2,3$. More precisely, pick an oriented identification $T^2\simeq \R^2/\Z^2$ with coordinates $(x,y)$ so that $\overline{a}_1=\{y=0\}$ and $\overline{a}_2=\{x=0\}$. Then $\gamma_1=\{x={1\over 4}\}$, $\gamma_2=\{x={1\over 2}\}$, and
$\gamma_3$ is obtained from $\{x={3\over 4}\}$ by applying a finger move along the arc $[{3\over 4},1+\varepsilon]\times\{{1\over 2}\}$ so that $\gamma_3$ has two intersections $y={1\over 2}-\varepsilon, {1\over 2}+\varepsilon$ with $\widehat{a}_2$. We also assume that $z_1,z_3
\in \widehat{a}_1$ lie on $\{{1\over 4} < x <{1\over 2}\}$ and $z_2,z_4\in \widehat{a}_2$ lie on $\{{1\over 2}-\varepsilon < y < {1\over 2}+\varepsilon\}$.  Then $\Sigma_1$ is obtained from the closure of the connected component of $T^2-\cup_{i=1}^3\gamma_i$ which is bounded by (copies of) $\gamma_1$ and $\gamma_2$, by identifying all of $\gamma_1$ to $w_1$ and all of $\gamma_2$ to $w_2'$. The other components
$\Sigma_2$ and $\Sigma_3$ are defined similarly.  In particular, $z_\infty \in \Sigma_3$. See Figure~\ref{fig: torus-degeneration}.

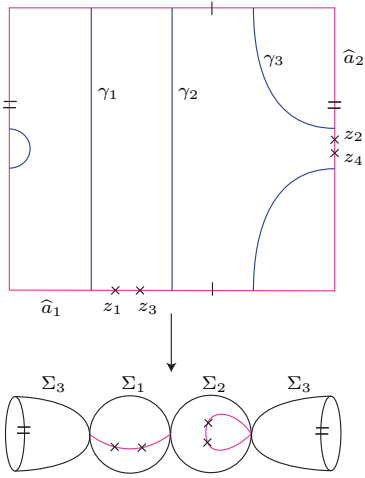
\begin{figure}[ht]
\begin{overpic}[width=4.5cm]{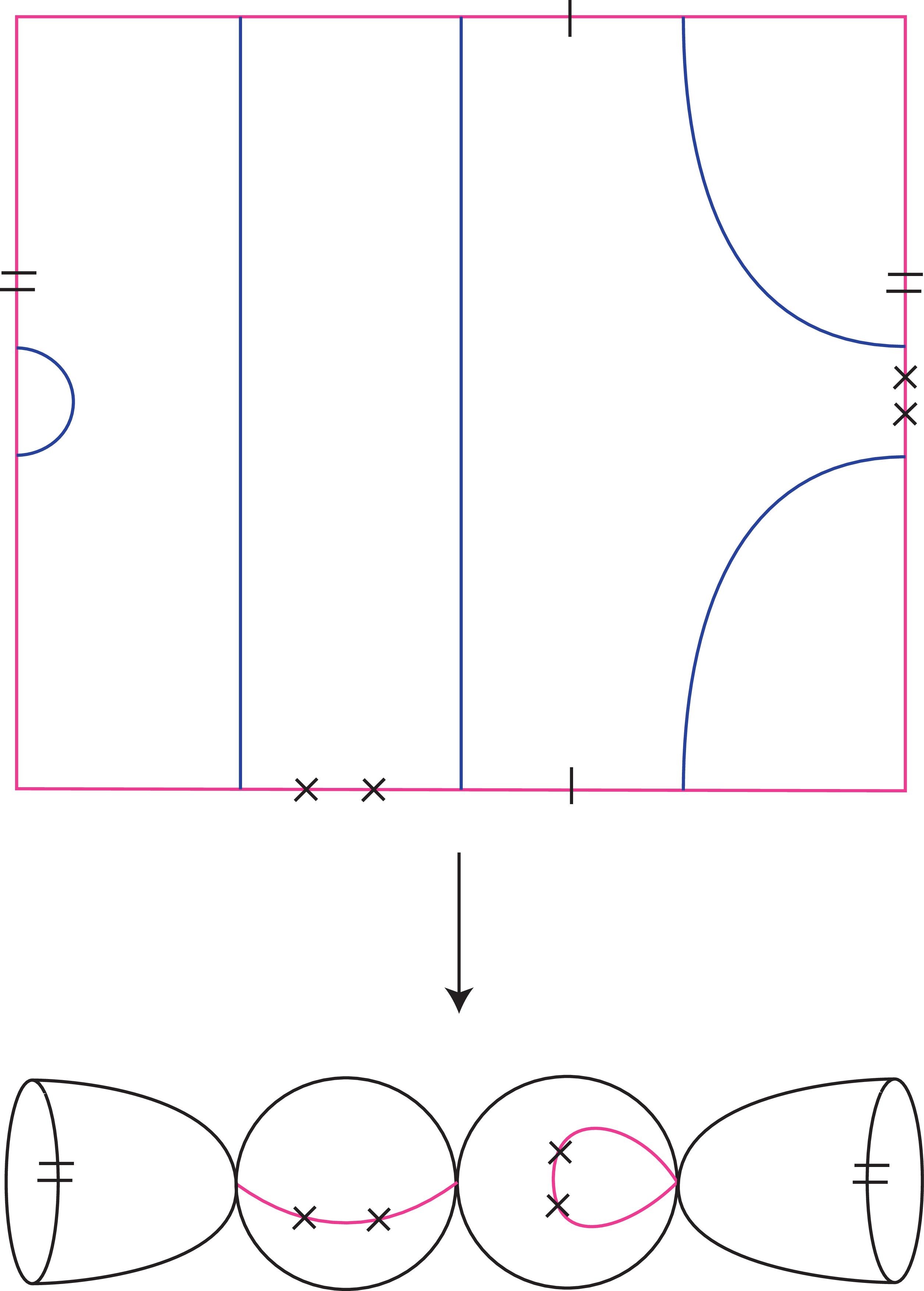}
\put(20,80){\tiny $\gamma_1$} \put(37,80){\tiny $\gamma_2$} \put(55,87){\tiny $\gamma_3$} \put(21,34.5){\tiny $z_1$} \put(28.5,34.5){\tiny $z_3$} \put(72,71){\tiny $z_2$} \put(72,66){\tiny $z_4$} \put(25,18){\tiny $\Sigma_1$}\put(42,18){\tiny $\Sigma_2$} \put(60,18){\tiny $\Sigma_3$} \put(8,18){\tiny $\Sigma_3$} \put(8,34){\tiny $\widehat{a}_1$} \put(72,87){\tiny $\widehat{a}_2$}
\end{overpic}
\caption{The top diagram is the torus $T^2$, where the sides are identified and the top and the bottom are identified. The arrow indicates the projection of $T^2$ onto $\Sigma_1\sqcup\Sigma_2\sqcup\Sigma_3/\sim$, given in \eqref{eqn: sausage}. The components of $\widetilde{a}_i$ which do not contain points in the set $\{z_1,\dots,z_4\}$ are not drawn in $\Sigma_i$. Also the two disks on the left and the right of the bottom diagram are glued into $\Sigma_3$.}
\label{fig: torus-degeneration}
\end{figure}

\begin{lemma}
\label{lemma: G sub 3 for torus}
$G_3(T^2;J;z_\infty,{\bf z})=\pm 1$, where ${\bf z}=\{z_1,\dots,z_4\}$.
\end{lemma}

\begin{proof}
We degenerate the Riemann surfaces $\Sigma^\tau=(T^2,i_\tau)$, $\tau\to \infty$, by pinching along $\gamma_1\cup\gamma_2\cup \gamma_3$. Then a sequence of holomorphic maps
$$u^\tau: (F,j_\tau)\to (\Sigma^\tau\times D^2,J_\tau)$$
in $\mathcal{M}_{J_\tau}(z_\infty,{\bf z})$ converges to a cusp curve $(u_1,u_2,u_3)$, where
$$u_i: F_i\to \Sigma_i\times D^2, \quad i=1,2,3,$$
and $F_1=F_1=D^2$, $F_3=\C\P^1$. This is because $z_1,z_3\in \Sigma_1$ and $z_2,z_4\in \Sigma_2$, and
the total number of boundary components $\sum_{i=1}^3\#(\bdry F_i)$ is equal to two by the argument in Section~\ref{subsubsection:
reduction to torus for G 3}.  Now, $u_3$ must have image
$\Sigma_3\times\{0\}$ since $z_\infty\in \Sigma_3$. The sets $\mathcal{S}_1$ and $\mathcal{S}_2$
for $v_i=\pi_2\circ u_i$, $i=1,2$, were determined in
Section~\ref{subsubsection: two calculations}.  The gluing of
intersecting curves $u_1$ and $u_2$ is given by the signed
intersection number of $C_1$ and $C_2$, which is $\pm 1$ by Lemma~\ref{lemma: intersection C one and C two}.
\end{proof}

\section{Homotopy of cobordisms I}
\label{section: homotopy of cobordisms I}

In this section and the next we prove Theorem~\ref{thm: isomorphism}. The chain homotopies that appear in the proof of Theorem~\ref{thm: isomorphism} are induced by homotopies of cobordisms $\overline{W}^\pm_\tau$ and $\overline{W}^\mp_\tau$ which are parametrized by $\tau\in\R$.  In this section we treat $\overline{W}^\pm_\tau$, leaving $\overline{W}^\mp_\tau$ for Section~\ref{section: homotopy of cobordisms II}. If $\pm$ is understood (as it will be in the rest of this section), then it will be omitted.

We now give a brief description of $\overline{W}_\tau$, leaving precise definitions for later. The base $B_\tau$ of $\overline{W}_\tau$ is biholomorphic to an annulus with one puncture on each boundary component; the neighborhoods of the punctures are viewed as strip-like ends.  As $\tau\to +\infty$, $\overline{W}_\tau$ degenerates to the stacking of
$\overline{W}_+$ ``on top of'' $\overline{W}_-$, where $\overline{W}_+$ and $\overline{W}_-$ are used in the definitions of the chain maps $\Phi$ and $\Psi$. 
(This is the reason why there is a $\pm$ in $\overline{W}^\pm_\tau$.)  On the other hand, as $\tau\to -\infty$, $\overline{W}_\tau$ degenerates to $\overline{W}_{-\infty}$, whose base $B_{-\infty}$ is (more or less) given by:
\begin{equation}
\label{eqn: base B sub D} B_{-\infty}=\left((\R\times[0,1])\sqcup D\right)/\sim,
\end{equation}
where $D=\{|z|\leq 1\}\subset \C$ and $\sim$ identifies $(0,1)\in \R\times[0,1]$ with $1\in D$ and $(0,0)\in \R\times[0,1]$ with $-1\in D$.

\subsection{Construction of the homotopy of cobordisms for $\Psi\circ \Phi$}

\subsubsection{Recollections}
\label{subsubsection: recollections}

In this subsection we recall some notation from \cite{CGH-I}.

Recall that $S$ is a compact oriented surface of genus $g$ with connected boundary (a page of an open book $(S,\hh)$), $\overline{S}=S\cup D^2$ is a closed surface obtained by capping off $S$, $\overline{\hh}=\overline{\hh}_m: \overline{S}\stackrel\sim\to \overline{S}$ is an extension of $\hh$ which is dependent on the integer $m\gg 0$ as in Section~I.\ref{P1-subsubsection: overline W pm}, and $\overline\omega$ is the area form on $\overline{S}$ from Section~I.\ref{P1-subsubsection: overline W pm} which is invariant under $\overline{\hh}$.  Also $z_\infty$ is the origin $\rho=0$ of $D^2=\{\rho\leq 1\}$ with polar coordinates $(\rho,\phi)$.

The mapping tori
$$N=(S\times[0,2])/(x,2)\sim (\hh(x),0), \quad \overline{N}=(\overline{S}\times[0,2])/(x,2)\sim (\overline{\hh}(x),0)$$
were defined in Section~I.\ref{P1-subsection: symplectic cobordisms}. Let $W=\R\times[0,1]\times S$, $\overline{W}=\R\times[0,1]\times\overline{S}$, $W'=\R\times N$, and $\overline{W'}=\R\times \overline{N}$; they admit symplectic fibrations with fibers diffeomorphic to $S$, $\overline{S}$, $S$, and $\overline{S}$, respectively. We also have the symplectic fibrations $\pi_{B_+}: W_+\to B_+$, $\overline\pi_{B_+}: \overline{W}_+\to B_+$, and $\overline{\pi}_{B_-}:\overline{W}_-\to B_-$ from Sections~I.\ref{P1-acorn} and I.\ref{P1-subsubsection: overline W pm}, with fibers diffeomorphic to $S$, $\overline{S}$, and $\overline{S}$, respectively.

The fibration $W$ (or $\overline{W}$) was used in the definition of $\widehat{CF}(S,{\bf a},\hh({\bf a}))$ and the fibration $W'$ (or $\overline{W'}$) in the definition of $PFC_{2g}(N)$. The positive end of $\overline{W}_+$ and the negative end of $\overline{W}_-$ agree with those of $\overline{W}$ and the negative end of $\overline{W}_+$ and the positive end of $\overline{W}_-$ agree with those of $\overline{W'}$. The fibrations $\pi_{B_+}$ and $\overline{\pi}_{B_-}$ were used in the definitions of the chain maps
$$\Phi: \widehat{CF}(S,{\bf a},\hh({\bf a}))\to PFC_{2g}(N),$$
$$\Psi: PFC_{2g}(N)\to \widehat{CF}(S,{\bf a},\hh({\bf a})).$$

\subsubsection{Definition of the family $\overline{W}_\tau$}
\label{defn of family 1}

For each $r\in[2,\infty)$, consider the fibration
$$\pi_r : \R \times\overline{N}_r \rightarrow \R \times (\R/r\Z),$$
where
$$\overline{N}_r=(\overline{S}\times[0,r])/(x,r)\sim (\overline{\hh}(x),0)$$
and $(s,t)$ are coordinates on $\R \times (\R/r\Z)$. For each $l,r\in[2,\infty)$, define $\overline{W}_{l,r} =\pi^{-1}_r (B_{l,r} ),$ where the base $B_{l,r}$ is obtained by smoothing the corners of
$$\{-l\leq s\leq l\} \cup \{ 0 \leq t\leq 1\}\subset \R \times
(\R/r\Z).$$

Next choose a function
\begin{equation} \label{l and r}
\eta=(l,r): \R\to[2,\infty)\times[2,\infty),
\end{equation}
which is obtained by smoothing
$$\eta_0(\tau)=\left\{ \begin{array}{cc} (\tau+2,2), & \mbox{for } \tau\geq 0;\\
(2, 2-\tau), & \mbox{for } \tau\leq 0;\end{array}\right.$$
near $\tau=0$. We then let $\overline{W}_\tau= \overline{W}_{\eta(\tau)}$ and $B_\tau=B_{\eta(\tau)}$.  Let $\pi_{B_\tau}:\overline{W}_\tau\to B_\tau$ be the projection along $\{(s,t)\}\times\overline{S}$.
\nom[1p]{$\pi_{B_\tau}: \overline{W}_\tau \to B_\tau$}{Symplectic fibration with fiber $S$ used in the definition of the chain homotopy}
\nom[B]{$B_\tau$}{Base of the projection $\pi_{B_\tau}: \overline{W}_\tau \to B_\tau$}
\nom[B]{$B_{-\infty,i}$, $i=1,2$}{Base of $\overline{W}_{-\infty,i}$}
\nom[W]{$\overline{W}_\tau$}{Total space of the projection $\pi_{B_\tau}: \overline{W}_\tau \to B_\tau$}
\nom[W]{$\overline{W}_{-\infty}= \overline{W}_{-\infty,1}\cup \overline{W}_{-\infty,2}$}{Limit of $\overline{W}_\tau$ as $\tau\to -\infty$}

As $\tau\to +\infty$, the cobordism $\overline{W}_\tau$ approaches the concatenation of $\overline{W}_+$ and $\overline{W}_-$; see Figure~\ref{figure: base degeneration 1}. On the other hand, as $\tau\to -\infty$, the cobordism $\overline{W}_\tau$ degenerates to a $2$-component building $\overline{W}_{-\infty}=\overline{W}_{-\infty,1}\cup\overline{W}_{-\infty,2}$, which we describe now; see Figure~\ref{figure: winfty}. [The decomposition as $\tau\to-\infty$ into a $2$-component building is a bit arbitrary.  What we mean to do is introduce a marked point $\overline{\frak m}(\tau)\in \overline{W}_\tau$ for each finite $\tau$ in Section~\ref{subsubsubsection: marked points} below and declare that $\overline{W}_\tau$ decomposes into $\overline{W}_{-\infty,1}\cup\overline{W}_{-\infty,2}$ as $\tau\to -\infty$, where $\overline{\frak m}(\tau)$ is used to introduce local coordinates centered at this marked point so that $\overline{\frak m}(\tau)$ limits to $\overline{\frak m}(-\infty)=(\overline{\frak m}^b(-\infty),\overline{\frak m}^f(-\infty))$ as $\tau\to-\infty$ and $\overline{\frak m}^b(-\infty)=(0,0)\in B_{-\infty,2}$ makes sense.]

The base of $\overline{W}_{-\infty}=\overline{W}_{-\infty,1}\cup\overline{W}_{-\infty,2}$ is $B_{-\infty}=B_{-\infty,1}\cup B_{-\infty,2}$, where $B_{-\infty,1}$ is obtained from $\{-2\leq s\leq 2\}\cup\{0\leq t\leq 1\} \subset \R^2$ by smoothing the corners and $B_{-\infty,2}=[-2,2]\times\R$. Here both $\R^2$ and $[-2,2]\times\R$ have coordinates $(s,t)$.  The component $B_{-\infty,1}$ has four strip-like ends: the ends $s\to +\infty$, $s\to-\infty$, $t\to +\infty$, $t\to-\infty$ will be referred to as the top, bottom, left, and right ends. The component $B_{-\infty,2}$ has two strip-like ends: the ends $t\to +\infty$ and $t\to -\infty$ will be referred to as the left and right ends. As usual, $B_{-\infty}$ is endowed with identifications of the compactifications of the strip-like ends.  More precisely, if the compactification $\check B_{-\infty,1}$ is obtained from $B_{-\infty,1}$ by attaching $\{\pm \infty\}\times [0,1]$ and $[-2,2]\times\{\pm \infty\}$ and the compactification $\check B_{-\infty,2}$ is obtained from $B_{-\infty,2}$ by attaching $[-2,2]\times\{\pm \infty\}$, then we identify $(s,\pm \infty)\in \check B_{-\infty,1}$ with $(s,\mp \infty)\in \check B_{-\infty,2}$.

Let $\overline{W}_{-\infty,i}=B_{-\infty,i}\times \overline{S}$ for $i=1,2$. The $2$-component building $\overline{W}_{-\infty}=\overline{W}_{-\infty,1}\cup\overline{W}_{-\infty,2}$ is endowed with identifications of compactifications of the ends. The compactification $\check {\overline{W}}_{-\infty,1}$ is obtained from $\overline{W}_{-\infty,1}$ by attaching $\{\pm\infty\}\times[0,1]\times \overline{S}$ and $[-2,2]\times\{\pm \infty\}\times\overline{S}$, the compactification $\check {\overline{W}}_{-\infty,2}$ is obtained from $\overline{W}_{-\infty,2}$ by attaching $[-2,2]\times\{\pm \infty\}\times \overline{S}$, and we identify $(s,+\infty,x)\in \check{\overline{W}}_{-\infty,1}$ with $(s,-\infty,x)\in \check{\overline{W}}_{-\infty,2}$ and $(s,+\infty,x)\in\check{\overline{W}}_{-\infty,2}$ with $(s,-\infty,\overline{\hh}(x))\in\check{\overline{W}}_{-\infty,1}$. We write $\pi_{B_{-\infty,i}}:\overline{W}_{-\infty,i}\to B_{-\infty,i}$ for the projection along $\overline{S}$.

\begin{figure}[ht]
\begin{overpic}[width=12.5cm]{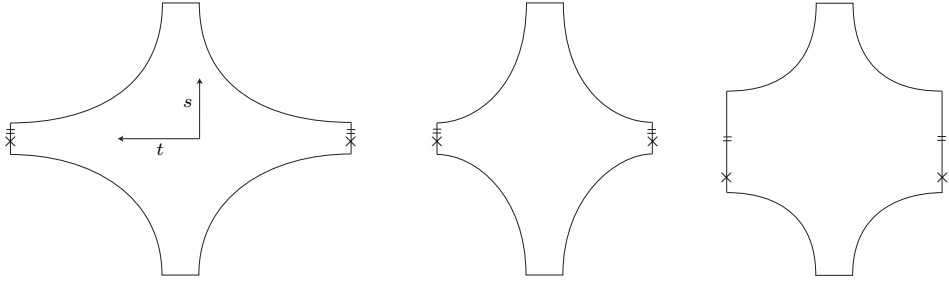}
\put(18.8,18){\tiny $s$} \put(16,12.8){\tiny $t$}
\put(-0.4,13.8){\tiny $\times$} \put(35.8,13.8){\tiny $\times$}
\put(45,13.8){\tiny $\times$} \put(67.92,13.8){\tiny $\times$}
\put(75.73,10){\tiny $\times$} \put(98.75,10){\tiny $\times$}
\end{overpic}
\caption{The bases of the family $\overline{W}_\tau$. The parameter $\tau$ increases as we go to the right. The sides are identified in this picture, as indicated. The location of $\overline{\frak m}^b(\tau)$ is indicated by $\times$.}
\label{figure: base degeneration 1}
\end{figure}

\begin{figure}[ht]
\begin{center}
\psfragscanon
\psfrag{a}{\small $\overline{\bf a}$}
\psfrag{b}{\small $\overline{\bf b}$}
\psfrag{c}{\small $\overline{\hh}(\overline{\bf b})$}
\psfrag{d}{\small $\overline{\hh}(\overline{\bf a})$}
\psfrag{x}{\small $\times$}
\psfrag{y}{\small $B_{-\infty,2}$}
\psfrag{z}{\small $B_{-\infty,1}$}
\includegraphics[width=10cm]{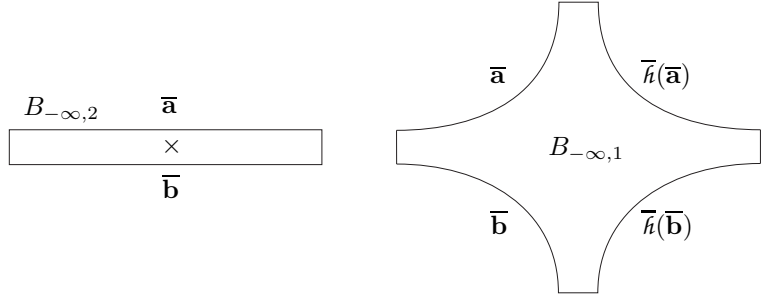}
\end{center}
\caption{The bases of $\overline{W}_{-\infty,2}$ to the left and $\overline{W}_{-\infty,1}$ to the right.}
\label{figure: winfty}
\end{figure}

We write $cl(B_\tau)$, $cl(B_+)$, $cl(B_-)$ to denote the compactifications of $B_\tau$, $B_+$, $B_-$, obtained by adjoining a point at infinity for each end $s=\pm \infty$. Similarly, we write $cl(B_{-\infty,1})$ for the compactification of $B_{-\infty,1}$, obtained by adding $4$ points $s=\pm \infty$ and $t=\pm \infty$, and $cl(B_{-\infty,2})$ for the compactification of $B_{-\infty,2}$, obtained by adding $2$ points $t=\pm \infty$.

\subsubsection{Marked points} \label{subsubsubsection: marked points}

We choose a $1$-parameter family of
\nom[m]{$\overline{\frak m}(\tau)=(\overline{\frak m}^b(\tau),\overline{\frak m}^f(\tau))$}{$1$-parameter family of marked points on $\overline{W}_\tau$}
\nom[m]{$\overline{\frak m}(+\infty)$, $\overline{\frak m}(-\infty)$}{Marked points on $\overline{W}_-$ and $\overline{W}_{-\infty,2}$, respectively}
marked points
$$\overline{\frak m}(\tau)=(\overline{\frak m}^b(\tau),\overline{\frak m}^f(\tau))=((-l(\tau)+2,(r(\tau)+1)/2),z_\infty)\in \overline{W}_\tau$$
for $\tau\in\R$, such that the following hold:
\begin{itemize}
\item[(i)] as $\tau\to +\infty$, $\overline{\frak m}(\tau)$ limits to $\overline{\frak m}(+\infty)=(\overline{\frak m}^b(+\infty),\overline{\frak m}^f(+\infty))$, where $\overline{\frak m}^b(+\infty)=(0,{3\over 2})\in B_-$ and $\overline{\frak m}^f(+\infty)=z_\infty$;
\item[(ii)] as $\tau\to -\infty$, $\overline{\frak m}(\tau)$ limits to $\overline{\frak m}(-\infty)=(\overline{\frak m}^b(-\infty),\overline{\frak m}^f(-\infty))$, where $\overline{\frak m}^b(-\infty)=(0,0)\in B_{-\infty,2}$ and $\overline{\frak m}^f(-\infty)=z_\infty$.
\end{itemize}

\begin{convention}
In this section and the next, $\overline{\frak m}$ will denote a $1$-parameter family (as opposed to a single marked point).
\end{convention}

Let us write $\mathcal{L}_{t_0}$ for the locus $\{t=t_0\}$, viewed as a subset of $B_+$, $B_-$, $B_\tau$, as appropriate.
\nom[L]{$\mathcal{L}_{t_0}$}{The locus $\{t=t_0\}$, viewed as a subset of $B_+$, $B_-$, $B_\tau$, as appropriate}
Of particular interest is $\mathcal{L}_{(r(\tau)+1)/2}$, which passes through $\overline{\frak m}^b(\tau)\in B_\tau$ or $\overline{\frak m}^b(+\infty)\in B_-$.

\subsubsection{Stable Hamiltonian structures and symplectic forms} \label{subsubsection: stable Hamiltonian}

We first consider $\overline{W}_\tau$.  The stable Hamiltonian structure on $\overline{N}_{r(\tau)} = (\overline{S} \times[0,r(\tau)])/ \sim$ is obtained from $(dt,\overline{\omega})$ on $\overline{S}\times[0,r(\tau)]$ by passing to the quotient, where $\overline\omega$ is the area form on $\overline{S}$ from Section~I.\ref{P1-subsubsection: overline W pm}.  The $2$-plane field is $\xi_\tau=\ker dt=T\overline{S}$ and the Hamiltonian vector field is $\overline{R}_\tau=\bdry_t$. The symplectic form $\overline{\Omega}_\tau$ is obtained from the symplectic form $ds\wedge dt+\overline\omega$ on $\R\times \overline{S}\times[0,r(\tau)]$ by passing to the quotient $\R\times\overline{N}_{r(\tau)}$ and then restricting to $\overline{W}_\tau$.

Next we consider $\overline{W}_{-\infty}=\overline{W}_{-\infty,1}\cup\overline{W}_{-\infty,2}$. Let $\omega_{-\infty,1}$ be the restriction of the area form $ds\wedge dt$ on $\R^2$ to $B_{-\infty,1}$ and let $\omega_{-\infty,2}=ds\wedge dt$ on $B_{-\infty,2}=[-2,2]\times \R$. Then we set $\overline{\Omega}_{-\infty,i}=\omega_{-\infty,i}+\overline{\omega}$. The stable Hamiltonian structure at the $s\to \pm \infty$ ends of $\overline{W}_{-\infty,1}$ are given by $(dt,\overline{\omega})$ and the stable Hamiltonian structure at the $t\to\pm \infty$ ends of $\overline{W}_{-\infty,1}$ are given by $(ds,\overline{\omega})$.

\subsection{Holomorphic curves and moduli spaces}

\subsubsection{Lagrangian boundary conditions}

Recall that the monodromy map $\overline{\hh}=\overline{\hh}_m: \overline{S}\to \overline{S}$ depends on the integer $m\gg 0$. Also $\overline{\mathbf{a}}=\{\overline{a}_1,\dots,\overline{a}_{2g}\}$ is the extension of the basis $\mathbf{a}=\{a_1,\dots,a_{2g}\}$ to $\overline{S}$ so that $\overline{a}_i=a_i\cup \overline{a}_{i,0,}\cup \overline{a}_{i,1}$, as described in Section~I.\ref{P1-coconut}. We also note that $\overline{\mathbf{a}}$ depends on $m$.

We first describe the pushoff $\overline{\bf b}$ of $\overline{\bf a}$ which also depends on $m$:
\nom[b]{$\overline{\bf b}=\{\overline{b}_1,\dots,\overline{b}_{2g}\}$}{Pushoff of $\overline{\bf a}=\{\overline{a}_1,\dots,\overline{a}_{2g}\}$}
Let $\varepsilon_0=\varepsilon_0(m)>0$ be sufficiently small and let $\overline{b}_i$ be a $\varepsilon_0$-close transverse pushoff of $\overline{a}_i$ which satisfies the following:
\begin{itemize}
\item in a neighborhood of $z_\infty$, $\overline{b}_i$ is obtained from $\overline{a}_i$ by a $-{2\pi\over m K(m)}$-rotation, where $K(m)$ is a positive integer such that $\lim_{m\to\infty}K(m)=\infty$; and
\item $\overline{a}_i$ and $\overline{b}_i$ intersect at three points $x_{i1}^\#$, $x_{i2}^\#$ and $x_{i3}^\#$
\nom[x]{$x_{i1}^\#$, $x_{i2}^\#$, $x_{i3}^\#$}{Intersection points of $\overline{a}_i$ and $\overline{b}_i$ besides $z_\infty$}
(besides at $z_\infty$); $x_{i2}^\#\in int(S)$ and $x_{i1}^\#,x_{i3}^\#\in \overline{S}-S$.
\end{itemize}
See Figure~\ref{figure: aandb}. We also write $\overline{b}_{i,j}$ for the portion of $\overline{b}_i$ analogous to $\overline{a}_{i,j}$. When we want to signal that $z_\infty$ is an intersection point of $\overline{a}_i$ and $\overline{b}_i$, then we write it as $z_\infty^\#$.

Let us write $\phi(\overline{a}_{i,j})$ for the $\phi$-coordinate of the portion of $\overline{a}_{i,j}$ near $z_\infty$, subject to the condition $0\leq \phi(\overline{a}_{i,j})<2\pi$; similarly define $\phi(\overline{b}_{i,j})$, etc.  We additionally assume that
$$0<\phi(\overline{a}_{i,j}),\phi(\overline{b}_{i,j}),\phi(\overline{\hh}(\overline{a}_{i,j})),\phi(\overline{\hh}(\overline{b}_{i,j}))\leq c(m)$$
where $c(m)$ is a function which satisfies $c(m)\to 0$ as $m\to \infty$; cf.\ Section~I.\ref{P1-coconut}.

\begin{figure}[ht]
\begin{center}
\psfragscanon
\psfrag{a}{\tiny $\overline{a}_i$}
\psfrag{b}{\tiny $\overline{b}_i$}
\psfrag{c}{\tiny $\overline{\hh}(\overline{a}_i)$}
\psfrag{d}{\tiny $\overline{\hh}(\overline{b}_i)$}
\psfrag{x}{\tiny $x^\#_{i1}$}
\psfrag{y}{\tiny $x^\#_{i2}$}
\psfrag{z}{\tiny $z_\infty$}
\includegraphics[width=2cm]{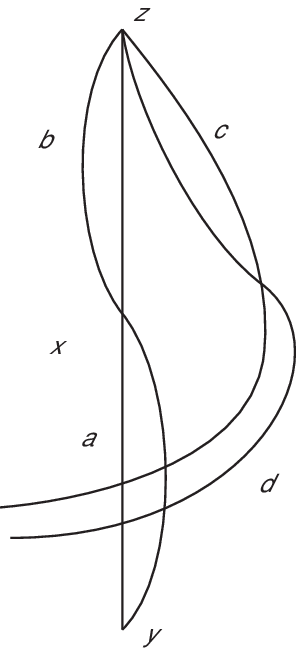}
\end{center}
\caption{The arcs $\overline{a}_i$, $\overline{b}_i$, $\overline{\hh}(\overline{a}_i)$, and $\overline{\hh}(\overline{b}_i)$ near $z_\infty$.}
\label{figure: aandb}
\end{figure}

\begin{rmk} \label{comparing rates}
In view of the choices of $\overline{a}_{i,j}$ and $\overline{\hh}(\overline{a}_{i,j})$ from Section~I.\ref{P1-coconut} and the choices of $\overline{b}_{i,j}$ above, as $m\to\infty$:
\begin{itemize}
\item $\phi(\overline{a}_{i,j})-\phi(\overline{b}_{i,j})\to 0$ the fastest;
\item $\phi(\overline{\hh}(\overline{a}_{i,j}))-\phi(\overline{a}_{i,j})\to 0$ the next fastest;
\item $\phi(\overline{a}_{i,j})-\phi(\overline{a}_{i',j'})\to 0$ the slowest if $(i,j)\not=(i',j')$.
\end{itemize}
\end{rmk}

The symplectic fibration
$$\pi_{B_\tau}:(\overline{W}_\tau,\overline{\Omega}_\tau)\to (B_\tau,ds\wedge dt)$$
induces a symplectic connection, defined
as the $\overline{\Omega}_\tau$-orthogonal of the tangent plane to the fibers. We place a copy of $\overline{\mathbf{a}}$ on the fiber $\pi^{-1}_{B_\tau}(s,1)$ with $s> l(\tau)$ and use the symplectic connection to parallel transport $\overline{\bf a}$ along $\bdry \overline{W}_\tau$. (Note that $\bdry \overline{W}_\tau$ is equal to the {\em vertical boundary} $\bdry_v \overline{W}_\tau:= \pi^{-1}_{B_\tau}(\bdry B_\tau)$.) This gives us a singular Lagrangian submanifold $L^{\tau,+}_{\overline{\mathbf{a}}}$.  (Note that $\overline{\bf a}$ is a singular Lagrangian submanifold of $\overline{S}$ with singularity $z_\infty$ and is a union of radial rays in a neighborhood of $z_\infty$. Hence the singular set of $L^{\tau,+}_{\overline{\mathbf{a}}}$ and its neighborhood in $L^{\tau,+}_{\overline{\mathbf{a}}}$ are obtained from those of $\overline{\bf a}$ by multiplying with the upper boundary of $B_\tau$.) Similarly, we place a copy of $\overline{\mathbf{b}}$ on the fiber $\pi^{-1}_{B_\tau}(s,1)$ with $s<-l(\tau)$ and use the symplectic connection to parallel transport $\overline{\bf b}$ along $\bdry \overline{W}_\tau$ to construct the singular Lagrangian submanifold $L^{\tau,-}_{\overline{\mathbf{b}}}$. The Lagrangian submanifolds $L^{\tau,+}_{\widehat{\bf a}}$, $L^{\tau,+}_{\widehat{a}_i}$, $L^{\tau,-}_{\widehat{\bf b}}$, $L^{\tau,-}_{\widehat{b}_i}$, etc.\ are defined similarly, where $\widehat{\bf a}=\{\widehat{a}_1,\dots,\widehat{a}_{2g}\}$ and $\widehat{a}_i=\overline{a}_i-\{z_i\}$ were defined in Section~I.\ref{P1-coconut} and $\widehat{\bf b}$ and $\widehat{b}_i$ are defined analogously.
\nom[L]{$L^{\tau,+}_{\widehat{\bf a}}$, $L^{\tau,-}_{\widehat{\bf b}}$}{Lagrangian submanifolds on $\overline{W}_\tau$}
\nom[L]{$L^{-\infty,1}_{\overline{\bf a},1}$, $L^{-\infty,1}_{\overline{\bf b},2}$, $L^{-\infty,1}_{\overline{\hh}(\overline{\bf b}),3}$, $L^{-\infty,1}_{\overline{\hh}(\overline{\bf a}),4}$}{Lagrangian submanifolds on $\overline{W}_{-\infty,1}$}
\nom[L]{$L^{-\infty,2}_{\overline{\bf a},+}$, $L^{-\infty,2}_{\overline{\bf b},-}$}{Lagrangian submanifolds on $\overline{W}_{-\infty,2}$}

On $(\overline{W}_{-\infty},\overline{\Omega}_{-\infty})$, we define the singular Lagrangian submanifolds as follows: Let us write
$$\bdry B_{-\infty,1}=\sqcup_{i=1}^4 \bdry_i B_{-\infty,1},$$ where the boundary components, in order from $i=1$ to $i=4$, satisfy
$$s>0, t>1/2; \ \ s<0, t>1/2; \ \  s<0,t<1/2;\ \mbox{ and } s>0, t<1/2.$$
Also let $\bdry_i \overline{W}_{-\infty,1}$ be the component of $\bdry\overline{W}_{-\infty,1}$ corresponding to $\bdry_i B_{-\infty,1}$. Then we define:
\begin{itemize}
\item $L^{-\infty,1}_{\overline{\bf a},1}=\bdry_1 B_{-\infty,1}\times  \overline{\mathbf{a}}$ on $\bdry_1 \overline{W}_{-\infty,1}$;
\item $L^{-\infty,1}_{\overline{\bf b},2}=\bdry_2 B_{-\infty,1}\times  \overline{\mathbf{b}}$ on $\bdry_2\overline{W}_{-\infty,2}$;
\item $L^{-\infty,1}_{\overline{\hh}(\overline{\bf b}),3}=\bdry_3 B_{-\infty,1}\times  \overline{\hh}(\overline{\mathbf{b}})$ on $\bdry_3 \overline{W}_{-\infty,1}$;
\item $L^{-\infty,1}_{\overline{\hh}(\overline{\bf a}),4}= \bdry_4 B_{-\infty,1}\times \overline{\hh}(\overline{\mathbf{a}})$ on $\bdry_4 \overline{W}_{-\infty,1}$;
\item $L^{-\infty,2}_{\overline{\bf a},+}=\{2\}\times \R \times \overline{\bf a}$ on $\bdry \overline{W}_{-\infty,2}$;
\item $L^{-\infty,2}_{\overline{\bf b},-}=\{-2\}\times \R \times \overline{\bf b}$ on $\bdry \overline{W}_{-\infty,2}$.
\end{itemize}

\subsubsection{Almost complex structures}
\label{subsubsection: almost complex structures for W tau}

Recall the space $\mathcal{J}_{\overline{W}}$ of admissible almost complex structures $\overline{J}$ on $\overline{W}$ from Definition~I.\ref{P1-defn: almost complex structures on overline W}, the space $\mathcal{J}_{\overline{W'}}$ of adapted almost complex structures $\overline{J'}$ on $\overline{W'}$ from Definition~I.\ref{P1-defn: almost complex structures on R times overline N}, and the spaces $\mathcal{J}_{\overline{W}_+}$ and $\mathcal{J}_{\overline{W}_-}$ of admissible almost complex structures $\overline{J}_+$ and $\overline{J}_-$ on $\overline{W}_+$ and $\overline{W}_-$ from Definition~I.\ref{P1-defn: admissible J for W plus}.

\begin{defn}
An almost complex structure $\overline{J}_{-\infty,2}$ on $\overline{W}_{-\infty,2}$ is {\em admissible} if the following hold:
\begin{enumerate}
\item $\overline{J}_{-\infty,2}$ is $t$-invariant, $\overline{J}_{-\infty,2}(\bdry_t)=-\bdry_s$, and $\overline{J}_{-\infty,2}(T\overline{S})=T\overline{S}$; and
\item there exists $\varepsilon>0$ such that $\overline{J}_{-\infty,2}$ restricts to the standard complex structure on the subsurface $D^2_\varepsilon=\{\rho\leq \varepsilon\}\subset \overline{S}$ of each fiber.
\end{enumerate}
The space of all admissible $\overline{J}_{-\infty,2}$ will be denoted by $\mathcal{J}_{\overline{W}_{-\infty,2}}$.
\end{defn}

\begin{defn} \label{adm2}
An almost complex structure $\overline{J}_{-\infty,1}$ on $\overline{W}_{-\infty,1}$ is {\em admissible} if the following hold:
\begin{enumerate}
\item the projection $\pi_{B_{-\infty,1}}$ is $(\overline{J}_{-\infty,1},j_{-\infty,1})$-holomorphic with respect to the standard complex structure  $j_{-\infty,1}$ on $B_{-\infty,1}$;
\item there exists $\varepsilon>0$ such that $\overline{J}_{-\infty,1}$ restricts to the standard complex structure on the subsurface $D^2_\varepsilon=\{\rho\leq \varepsilon\}\subset \overline{S}$ of each fiber;
\item there exist $\overline{J}\in\mathcal{J}_{\overline{W}}$ and $\overline{J}_{-\infty,2}\in \mathcal{J}_{\overline{W}_{-\infty,2}}$ such that $\overline{J}_{-\infty,1}$ agrees with $\overline{J}$ on $\overline{W}=\R\times[0,1]\times\overline{S}$, with $\overline{J}_{-\infty,2}$ on $[-2,2]\times[3,+\infty)\times\overline{S}$, and with $(id\times\overline{\hh})_*(\overline{J}_{-\infty,2})$ on $[-2,2]\times[-2,-\infty)\times\overline{S}$.
\end{enumerate}
If (3) holds, we say that $\overline{J}_{-\infty,1}$ is {\em compatible} with $\overline{J}\in\mathcal{J}_{\overline{W}}$ and $\overline{J}_{-\infty,2}\in \mathcal{J}_{\overline{W}_{-\infty,2}}$. The space of admissible $\overline{J}_{-\infty,1}$ will be denoted by $\mathcal{J}_{\overline{W}_{-\infty,1}}$.
\end{defn}

\nom[J1]{$\overline{J}_{-\infty}=\overline{J}_{-\infty,1}\cup\overline{J}_{-\infty,2}$}{Almost complex structure on $\overline{W}_{-\infty}=\overline{W}_{-\infty,1}\cup\overline{W}_{-\infty,2}$}
\nom[J]{$\mathcal{J}_{\overline{W}_{-\infty}}$}{Space of admissible $\overline{J}_{-\infty}$}

\begin{defn}
An almost complex structure $\overline{J}_{-\infty}=\overline{J}_{-\infty,1}\cup\overline{J}_{-\infty,2}$ on $\overline{W}_{-\infty}=\overline{W}_{-\infty,1}\cup\overline{W}_{-\infty,2}$ is {\em admissible} if $\overline{J}_{-\infty,i}\in \mathcal{J}_{\overline{W}_{-\infty,i}}$ for $i=1,2$ and $\overline{J}_{-\infty,1}$ is compatible with $\overline{J}_{-\infty,2}$. The space of admissible $\overline{J}_{-\infty}$ will be denoted by $\mathcal{J}_{\overline{W}_{-\infty}}$.
\end{defn}

\nom[J1]{$\overline{J}_\tau$}{Almost complex structure on $\overline{W}_\tau$}
\nom[J]{$\mathcal{J}_{\overline{W}_\tau}$}{Space of all admissible $\overline{J}_\tau$}

\begin{defn} \label{adm1}
An almost complex structure $\overline{J}_\tau$ on $\overline{W}_\tau$ is {\em admissible} if the following hold:
\begin{enumerate}
\item the projection $\pi_{B_\tau}$ is $(\overline{J}_\tau,j_\tau)$-holomorphic with respect to the standard complex structure $j_\tau$ on $B_\tau$;
\item there exists $\varepsilon>0$ such that $\overline{J}_\tau$ restricts to the standard complex structure on the subsurface $D^2_\varepsilon=\{\rho\leq \varepsilon\}\subset \overline{S}$ of each fiber;
\item if $\tau\geq 0$, then $\overline{J}_\tau$ is the restriction of some $\overline{J'}\in \mathcal{J}_{\overline{W'}}$;
\item if $\tau\leq 0$, then $\overline{J}_\tau$ agrees with some $\overline{J}\in\mathcal{J}_{\overline{W}}$ on $\overline{W}=\R\times[0,1]\times\overline{S}$ and with some $\overline{J}_{-\infty,2}\in\mathcal{J}_{\overline{W}_{-\infty,2}}$ on $[-2,2]\times[3,r(\tau)-2]\times\overline{S}$.
\end{enumerate}
If (3) holds, we say that $\overline{J}_\tau$ is {\em compatible} with $\overline{J'}\in \mathcal{J}_{\overline{W'}}$ and if (4) holds, we say that $\overline{J}_\tau$ is {\em compatible} with $\overline{J}\in\mathcal{J}_{\overline{W}}$ and $\overline{J}_{-\infty,2}\in\mathcal{J}_{\overline{W}_{-\infty,2}}$. The space of all admissible $\overline{J}_\tau$ on $\overline{W}_\tau$ will be denoted by $\mathcal{J}_{\overline{W}_\tau}$.
\end{defn}

\begin{defn} \label{adm family}
A family $\{\overline{J}_\tau\in \mathcal{J}_{\overline{W}_\tau}\}_{\tau\in\R}$ of almost complex structures is {\em admissible} if there exist $\overline{J'}\in \mathcal{J}_{\overline{W'}}$, $\overline{J}\in\mathcal{J}_{\overline{W}}$, $\overline{J}_+\in \mathcal{J}_{\overline{W}_+}$, $\overline{J}_-\in\mathcal{J}_{\overline{W}_-}$ and $\overline{J}_{-\infty}=\overline{J}_{-\infty,1} \cup\overline{J}_{-\infty,2}\in \mathcal{J}_{\overline{W}_{-\infty}}$ such that the following hold:
\begin{enumerate}
\item $\overline{J}_\tau$ converges to $\overline{J}_{-\infty}$ as $\tau\to-\infty$;
\item $\overline{J}_\tau$ converges to $\overline{J}_+$ and $\overline{J}_-$ as $\tau\to +\infty$;
\item $\overline{J}_+$ and $\overline{J}_-$ are compatible with $\overline{J'}$ and $\overline{J}$; and
\item $\overline{J}_\tau$ is compatible with $\overline{J}$ and $\overline{J}_{-\infty,2}$ for $\tau \leq 0$ and with $\overline{J'}$ for $\tau \geq 0$.
\end{enumerate}
The space of all admissible $\{\overline{J}_\tau\in \mathcal{J}_{\overline{W}_\tau}\}_{\tau\in\R}$ will be denoted by $\overline{\mathcal{I}}$.
\end{defn}

\nom[I]{$\overline{\mathcal{I}}$}{Space of all admissible $\{\overline{J}_\tau\in \mathcal{J}_{\overline{W}_\tau}\}_{\tau\in\R}$}

\subsubsection{Some notation and conventions} \label{subsubsection: convention bambi}

We now collect some notation and conventions.

\begin{notation} [Tuples and orbit sets]  \label{notation: tuples and orbit sets}
When we write a tuple of $\overline{\bf a}\cap \overline{\hh}(\overline{\bf a})$ as ${\bf y}$ or an orbit set of $\overline{N}$ as $\bs\gamma$ (with possible superscripts, subscripts and other decorations), it is assumed that ${\bf y}\subset S$ and $\bs\gamma\subset N$.  In particular, ${\bf y}$ and $\bs\gamma$ do not contain any multiples of $z_\infty$ or $\delta_0$.
\end{notation}

\begin{notation} [Sections at $\infty$] \label{notation: section at infty}
The sections $\{\rho=0\}$ of $\overline{W}$, $\overline{W'}=\R\times\overline{N}$, $\overline{W}_\tau$, $\overline{W}_+$, $\overline{W}_-$ and $\overline{W}_{-\infty,i}$ are holomorphic with respect to almost complex structures in $\mathcal{J}_{\overline{W}}$, $\mathcal{J}_{\overline{W'}}$, $\mathcal{J}_{\overline{W}_\tau}$, etc. They are called {\em sections at $\infty$} and are denoted by $\sigma_\infty$, $\sigma_\infty'$, $\sigma_\infty^\tau$, $\sigma_\infty^+$, $\sigma_\infty^-$ and $\sigma_{\infty}^{-\infty,i}$.
\nom[1]{$\sigma_\infty^\tau$, $\sigma_\infty^{-\infty,i}$}{Sections at infinity of $\overline{W}_\tau$ and $\overline{W}_{-\infty,i}$}
\end{notation}

\begin{notation}[The intersection numbers $n^*(\overline{u})$ and $n^{*,alt}(\overline{u})$]  \label{notation: intersection numbers}
Let $\delta_{\rho_0,\phi_0}$ be a closed orbit of the Hamiltonian vector field $\bdry_t$ which lies on the torus
$$\{\rho=\rho_0\}\subset  \overline{N}_r= (\overline{S}\times[0,r])/(x,1)\sim(\overline{\hh}(x),0)$$
for appropriate $r$ and $\rho_0>0$ sufficiently small and which passes through the point $(t,\rho,\phi)=(0,\rho_0,\phi_0)$. Since $\overline\hh=\overline\hh_m$ is a ${2\pi\over m}$-rotation on $D^2_{1/2}=\{\rho\leq 1/2\}\subset D^2$, the orbit $\delta_{\rho_0,\phi_0}$ winds $m$ times in the longitudinal direction and once in the meridian direction. The point $(0,\rho_0,\phi_0)$ is with respect to balanced coordinates on $\overline{N}_r$; see Section~I.\ref{P1-subsubsection: overline W pm}.  We assume additionally that $\delta_{\rho_0,\phi_0}$ does not intersect the projections of the Lagrangians of $\overline{W}_\tau$, $\overline{W}_+$ and $\overline{W}_-$ to $\overline{N}_r$.

Let us write $\phi(\overline{a}_{i,j})$ for the $\phi$-coordinate of $\overline{a}_{i,j}$ such that $0\leq \phi(\overline{a}_{i,j})<2\pi$. Also let $\varepsilon_0={2\pi\over Km}$ be the constant appearing in the definition of $\overline{\bf b}$ and let $\varepsilon_1$ be a constant satisfying  $0<\varepsilon_1<\varepsilon_0$. We consider two possibilities for $\phi_0$:
$$\phi^\pm_0= \phi(\overline{a}_{i,j})\pm \varepsilon_1.$$
Hence $\phi(\overline{b}_{i,j})<\phi_0^-<\phi(\overline{a}_{i,j})<\phi_0^+$.

We write $(\sigma_\infty^*)^{\dagger,\pm}$ for the restriction of $\R\times \delta_{\rho_0,\phi_0^\pm}$ to $\overline{W}_*$, where $*=\varnothing$, $'$, $\tau$, $+$, or $-$. For $\overline{W}_{-\infty,i}$, we write
$$(\sigma_\infty^{-\infty,i})^{\dagger,\pm}=B_{-\infty,i}\times\{\rho=\rho_0,\phi=\phi_0^\pm+2\pi k/m,k\in \Z\},$$
for some $\phi_0$. Finally we define:
\begin{equation} \label{n and n alt}
n^*(\overline{u}) =\langle\overline{u},(\sigma^*_\infty)^{\dagger,+}\rangle, \quad
n^{*,alt}(\overline{u}) = \langle\overline{u},(\sigma^*_\infty)^{\dagger,-}\rangle,
\end{equation}
where $*=\varnothing$, $'$, $\tau$, $+$, $-$, or $(-\infty,i)$.  The two quantities $n^*$ and $n^{*,alt}$ can be used interchangeably, except when $\tau=-\infty$; for the most part we will use $n^*$.
\nom[n]{$n^*(\overline{u})$, $n^{*,alt}(\overline{u})$}{Intersection numbers defined in Equation~\eqref{n and n alt}}
\end{notation}

\begin{notation}[Components of a holomorphic curve $\overline{u}$] \label{notation: overline u}
Given a holomorphic curve $\overline{u}$ in $\overline{W}$, $\overline{W'}$, etc., we write
$$\overline{u}=\overline{u}'\cup\overline{u}''=\overline{u}'\cup \overline{u}^\sharp\cup\overline{u}^\flat\cup\overline{u}^f,$$
where
\begin{itemize}
\item $\overline{u}'$ is a possibly disconnected branched cover of $\sigma_\infty^*$;
\item $\overline{u}''$ is the union of irreducible components which do not branch cover $\sigma_\infty^*$;
\item $\overline{u}^\sharp$ is the union of components of $\overline{u}''$ which are asymptotic to a multiple of $\delta_0$ or $z_\infty$ at one or more ends;
\item $\overline{u}^\flat$ is the union of the remaining non-fiber components of $\overline{u}''$; and
\item $\overline{u}^f$ is the union of fiber components of $\overline{u}$, including ghosts.
\end{itemize}
If $\overline{u}$ is a multisection, then $\deg{\overline{u}}$ is the degree of $\overline{u}$ as a multisection.
\end{notation}

\s\n
{\em The choice of hyperbolic orbit.}  By the definition of the monodromy map $\hh$, $\bdry N$ is a negative Morse-Bott torus; we denote the negative Morse-Bott family of simple orbits on $\bdry N$ by $\mathcal{N}$. Let $\phi_\gamma$ be the $\phi$-coordinate of $\gamma\in \mathcal{N}$. Also recall the orbit $\delta_0= \{z_\infty\}\times[0,2]/\sim$ of $\overline{N}$.

Let $\overline{J'}\in \mathcal{J}_{\overline{W'}}$.
Without loss of generality, we may assume that there is only one holomorphic cylinder $Z_\gamma$ in $(\overline{W'}=\R\times\overline{N},\overline{J'})$ from $\delta_0$ to any orbit $\gamma\in \mathcal{N}$, modulo $\R$-translation. Each $Z_\gamma$ corresponds to a radial ray $\mathcal{R}_{\phi_\gamma}=\{\phi=\phi_\gamma, \rho\geq 0\}\subset D^2$, which is the asymptotic direction of $\pi_{D^2}(Z_\gamma)$ at the positive end. Here $\pi_{D^2}:\overline{N}-int(N)\to D^2$ is the projection with respect to the balanced coordinates; see Section~I.\ref{P1-subsubsection: overline W pm}.

We now choose a hyperbolic orbit $h$ and an elliptic orbit $e$ in $\mathcal{N}$.
The choice of $h=\gamma_{\phi_h}$ is the same as that of Convention~I.\ref{P1-convention for h}:  $h$ is generic and $\phi_h$ is close to $-{2\pi \over m}$, where the integer $m$ which appears in the definition of $\overline{\hh}_m$ additionally satisfies the conditions of Section~I.\ref{P1-coconut}. In particular, the radial ray $\mathcal{R}_{\phi_h}$ does not lie on the thin wedges from $\overline{a}_i$ to $\overline{\hh}(\overline{a}_i)$ for all $i$. There are no restrictions on $e$ except that $e\not=h$.

\subsubsection{Holomorphic maps to $\overline{W}_\tau$}

Let $(F,j)$ be a compact Riemann surface, possibly disconnected, with two $k$-tuples of boundary punctures $\mathbf{q}^+=\{q^+_1,\dots,q^+_{k}\}$ and $\mathbf{q}^-=\{q^-_1,\dots,q^-_{k}\}$ on $\bdry F=\bdry_+ F \sqcup \bdry _- F$, such that:
\begin{itemize}
\item[(i)] each component of $F$ nontrivially intersects $\bdry_+ F$ and $\bdry_-F$;
\item[(ii)] each of $\bdry_+ F$ and $\bdry_- F$ is a union of connected components of $\bdry F$; and
\item[(iii)] on each component of $\bdry_+ F$ (resp.\ $\bdry_- F$) there is at least one puncture from $\mathbf{q}^+$ (resp.\ $\mathbf{q}^-$) and none from $\mathbf{q}^-$ (resp.\ $\mathbf{q}^+$).
\end{itemize}
We write $\dot F=F-\mathbf{q}^+-\mathbf{q}^-$, $\bdry_+ \dot F= \bdry_+ F- \mathbf{q}^+$ and $\bdry_- \dot F = \bdry_-F - \mathbf{q}^-$.

Let ${\bf z}=\{z_\infty^p(\overrightarrow{\mathcal{D}})\}\cup {\bf y}$, $p\geq 0$, be a $k$-tuple of points of $\overline{\bf a}\cap \overline{\hh}(\overline{\bf a})$, where $k\leq 2g$, $z_\infty$ has multiplicity $p$, and $\overrightarrow{\mathcal{D}}$ is the data at $z_\infty^p$ with respect to $\overline{\bf a}\cap \overline{\hh}(\overline{\bf a})$. The definition of ${\bf z}$ and the notion of data at $z_\infty^p$ are given in Section~I.\ref{P1-subsection: modified indices at z infty}. In particular, by definition, each arc of $\{\overline{a}_i, \overline{\hh}(\overline{a}_i)\}_{i=1}^{2g}$ is used at most once.  Also let ${\bf z'}=\{z_\infty^q(\overrightarrow{\mathcal{D}}')\}\cup {\bf y'}$, $q\geq 0$, be a $k$-tuple of points of $\overline{\bf b}\cap\overline{\hh}(\overline{\bf b})$, where $z_\infty$ has multiplicity $q$ and $\overrightarrow{\mathcal{D}}'$ is the data at $z_\infty^q$ with respect to $\overline{\bf b}\cap\overline{\hh}(\overline{\bf b})$.

Let $\overline{J}_\tau\in \mathcal{J}_{\overline{W}_\tau}$.  If $\overline{u}': \dot F\to (\overline{W}_\tau,\overline{J}_\tau)$ is a branched cover of $\sigma_\infty^\tau$, then $\overline{u}'$ comes equipped with data $\mathcal{C}$ (cf.\ Definition~I.\ref{P1-defn: data C}), which is a map from $\pi_0(\bdry_+\dot F)$ (resp.\ $\pi_0(\bdry_-\dot F)$) to the set of arcs $\overline{a}_{i,j}$ (resp.\ $\overline{b}_{i,j}$). In words, $\overline{u}'$ is viewed as mapping each component of $\bdry_+ \dot F$ (resp.\ $\bdry_-\dot F$) to some $L^{\tau,+}_{\overline{a}_{i,j}}$ (resp.\ $L^{\tau,-}_{\overline{b}_{i,j}}$). Then $\mathcal{C}$ determines the data $\overrightarrow{\mathcal{D}}$ and $\overrightarrow{\mathcal{D}}'$ at the positive and negative ends.

\s
We then make the following definition:

\begin{defn} \label{defn: W tau curve}
Let $\overline{J}_\tau\in \mathcal{J}_{\overline{W}_\tau}$, ${\bf z}=\{z_\infty^p(\overrightarrow{\mathcal{D}})\}\cup {\bf y}$ be a $k$-tuple of $\overline{\bf a}\cap \overline{\hh}(\overline{\bf a})$ and ${\bf z'}=\{z_\infty^q(\overrightarrow{\mathcal{D}}')\}\cup {\bf y'}$ be a $k$-tuple of $\overline{\bf b}\cap \overline{\hh}(\overline{\bf b})$.

A {\em degree $k$ multisection of $(\overline{W}_\tau,\overline{J}_\tau)$ from ${\bf z}$ to ${\bf z'}$} is a pair $(\overline{u},\mathcal{C})$ consisting of a holomorphic map
$$\overline{u}=\overline{u}'\cup\overline{u}'': (\dot F=\dot F'\sqcup \dot F'',j)\to (\overline{W}_\tau, \overline{J}_\tau)$$
which is a degree $k$ multisection of $\pi_{B_\tau}: \overline{W}_\tau\to B_\tau$ and data $\mathcal{C}$ for $\overline{u}'$, and which additionally satisfies the following:
\begin{enumerate}
\item $\overline{u}''(\bdry_+ \dot{F}'' )\subset L^{\tau,+} _{\widehat{\bf a}}$ and $\overline{u}'' (\bdry_-\dot F'') \subset L^{\tau,-}_{\widehat{\bf b}}$;
\item $\overline{u}$ maps each connected component of $\bdry_+\dot F$ to a different $L^{\tau,+}_{\overline{a}_i}$ and each connected component of $\bdry_-\dot F$ to a different $L^{\tau,-}_{\overline{b}_i}$ (here we are using $\mathcal{C}$ to assign some $L^{\tau,+}_{\overline{a}_i}$ or $L^{\tau,-}_{\overline{b}_i}$ to each component of $\bdry_\pm\dot F'$);
\item $\displaystyle{\lim_{w\rightarrow q_i^+}\pi_{\R}\circ \overline{u}(w) =+\infty}$ and $\displaystyle{\lim_{w\rightarrow q_i^-} \pi_{\R}\circ \overline{u}(w) =-\infty}$;
\item $\overline{u}$ converges to a strip over $[0,1]\times {\bf z}$ near ${\bf q}^+$ and to a strip over $[0,1]\times {\bf z'}$ near ${\bf q}^-$;
\item the positive and negative ends of $\overline{u}$ which limit to $z_\infty$ are described by $\overrightarrow{\mathcal{D}}$ and $\overrightarrow{\mathcal{D}}'$.
\end{enumerate}
Here $\pi_{\R}:\overline{W}_\tau\to\R$ is the projection to the $s$-coordinate.

A {\em $(\overline{W}_\tau,\overline{J}_\tau)$-curve from ${\bf y}$ to ${\bf y'}$} is a degree $2g$ multisection of $(\overline{W}_\tau,\overline{J}_\tau)$ satisfying $n^*(\overline{u})=m$. (Recall that the integer $m$ is the integer on which the monodromy map $\overline{\hh}=\overline{\hh}_{m}$ depends.)
\end{defn}

Let $\mathcal{M}_{\overline{J}_\tau}(\mathbf{z},\mathbf{z}')$ be the moduli space of degree $k$ multisections of $(\overline{W}_\tau,\overline{J}_\tau)$ from ${\bf z}$ to ${\bf z'}$ and let $\mathcal{M}_{\overline{J}_\tau}(\mathbf{z},\mathbf{z}';\overline{\frak m}(\tau))\subset \mathcal{M}_{\overline{J}_\tau} (\mathbf{z},\mathbf{z}')$ be the subset consisting of $\overline{W}_\tau$-curves that pass through $\overline{\frak m}(\tau)$.  We write
$$\mathcal{M}_{\{\overline{J}_\tau\}}(\mathbf{z},\mathbf{z}'):=\coprod_{\tau\in\R} \mathcal{M}_{\overline{J}_\tau}(\mathbf{z},\mathbf{z}'),$$
$$\mathcal{M}_{\{\overline{J}_\tau\}}(\mathbf{z},\mathbf{z}';\overline{\frak m}):=\coprod_{\tau\in\R} \mathcal{M}_{\overline{J}_\tau} (\mathbf{z},\mathbf{z}';\overline{\frak m}(\tau)).$$

\begin{notation}[Modifiers] \label{notation: modifiers}
For any moduli space $\mathcal{M}_{\star_1}(\star_2)$, we may place modifiers $*$ as in $\mathcal{M}^*_{\star_1}(\star_2)$ to denote the subset of $\mathcal{M}_{\star_1}(\star_2)$ satisfying $*$.  Typical self-explanatory modifiers are $I=i$, $n^*=m$, and $\op{deg}=k$. Note that the degree can be inferred from $\star_2$.

The following is a list of non-self-explanatory modifiers:
\begin{enumerate}
\item[$\dagger$] $=$ no component of $\overline{u}$ branch covers $\sigma_\infty^*$ with possibly empty branch locus.
\item[$s$] $=$ all the components of $\overline{u}$ are simply covered.
\item[$irr$] $=$ the curve $\overline{u}$ is irreducible.
\end{enumerate}
\end{notation}

\subsubsection{Holomorphic maps to $\overline{W}_{-\infty}$}
\label{subsubsection: widehat W sub D}

We first discuss holomorphic curves without ends at $z_\infty$.

\begin{defn} \label{defn: bambi 1}
Let $\overline{J}_{-\infty,1}\in \mathcal{J}_{\overline{W}_{-\infty,1}}$, ${\bf y}_1\in \mathcal{S}_{{\bf a}, \hh({\bf a})}$, ${\bf y}_2\in \mathcal{S}_{{\bf b}, {\bf a}}$, ${\bf y}_3\in \mathcal{S}_{{\bf b}, \hh({\bf b})}$ and ${\bf y}_4\in \mathcal{S}_{\hh({\bf a}), \hh({\bf b})}$.  (Recall that an element of $\mathcal{S}_{{\bf a}, \hh({\bf a})}$ is a tuple of intersection points $a_i\cap \hh(a_j)$, where each $a_i$ and each $\hh(a_j)$ is used at most once; the other $\mathcal{S}_{*,*}$ are defined analogously. In this definition all the tuples are $k$-tuples.)

A {\em degree $k\leq 2g$ multisection $\overline{u}$ of $(\overline{W}_{-\infty,1},\overline{J}_{-\infty,1})$ with ends ${\bf y}_1, {\bf y}_2, {\bf y}_3, {\bf y}_4$} is a holomorphic map
$$\overline{u}: (\dot F,j)\to (\overline{W}_{-\infty,1},\overline{J}_{-\infty,1})$$
which is degree $k$ multisection of $\pi_{B_{-\infty,1}}: \overline{W}_{-\infty,1}\to B_{-\infty,1}$ and which additionally satisfies the following:
\begin{enumerate}
\item $\overline{u}(\bdry \dot F)\subset L^{-\infty,1}_{\overline{\bf a},1}\cup L^{-\infty,1}_{\overline{\bf b},2}\cup L^{-\infty,1}_{\overline{\hh}(\overline{\bf b}),3}\cup L^{-\infty,1}_{\overline{\hh}(\overline{\bf a}),4}$ and $\overline{u}$ maps each component of $\bdry \dot F$ to a different $L^{-\infty,1}_{\overline{a}_i,1}$, $L^{-\infty,1}_{\overline{b}_i,2}$, $L^{-\infty,1}_{\overline{\hh}(\overline{b}_i),3}$, or $L^{-\infty,1}_{\overline{\hh}(\overline{a}_i),4}$;
\item $\overline{u}$ converges to a strip over $[0,1]\times{\bf y}_1$ as $s\to +\infty$; $[-2,2]\times {\bf y}_2$ as $t\to +\infty$;  $[0,1]\times{\bf y}_3$ as $s\to-\infty$; and $[-2,2]\times{\bf y}_4$ as $t\to -\infty$.
\end{enumerate}

A {\em $(\overline{W}_{-\infty,1},\overline{J}_{-\infty,1})$-curve with ends ${\bf y}_1,\dots,{\bf y}_4$} is a degree $2g$ multisection of $(\overline{W}_{-\infty,1},\overline{J}_{-\infty,1})$ satisfying $n^*(\overline{u})=0$.
\end{defn}

We will use the convention to list the ends of a multisection $\overline{u}$ of $\overline{W}_{-\infty,1}$ in counterclockwise order, starting with the top end.

\begin{defn}
Let $\overline{J}_{-\infty,2}\in \mathcal{J}_{\overline{W}_{-\infty,2}}$ and ${\bf y}, {\bf y'}\in \mathcal{S}_{{\bf b},{\bf a}}$. A {\em degree $k\leq 2g$ multisection $\overline{u}$ of $(\overline{W}_{-\infty,2},\overline{J}_{-\infty,2})$ with ends ${\bf y}$ and ${\bf y'}$} is a holomorphic map
$$\overline{u}: (\dot F,j)\to (\overline{W}_{-\infty,2},\overline{J}_{-\infty,2})$$
which is degree $k$ multisection of $\pi_{B_{-\infty,2}}: \overline{W}_{-\infty,2}\to B_{-\infty,2}$ and which additionally satisfies the following:
\begin{enumerate}
\item $\overline{u}(\bdry \dot F)\subset L^{-\infty,2}_{\overline{\bf a},+}\cup L^{-\infty,2}_{\overline{\bf b},-}$ and $\overline{u}$ maps each component of $\bdry \dot F$ to a different $L^{-\infty,2}_{\overline a_i,+}$ or a different $L^{-\infty,2}_{\overline b_i,-}$;
\item $\overline{u}$ converges to a cylinder over $[-2,2]\times {\bf y}$ as $t\to +\infty$ and to a cylinder over $[-2,2]\times{\bf y'}$ as $t\to -\infty$.
\end{enumerate}

A {\em $(\overline{W}_{-\infty,2},\overline{J}_{-\infty,2})$-curve with ends ${\bf y}$ and ${\bf y'}$} is a degree $2g$ multisection of $(\overline{W}_{-\infty,2},\overline{J}_{-\infty,2})$ satisfying $n^*(\overline{u})=m$.
\end{defn}

Note that the definition is not symmetric in ${\bf y}$ and ${\bf y'}$; the $t=+\infty$ end is always written first.

Next we discuss holomorphic curves with ends at $z_\infty$. For $s=1,\dots,4$, the data $\overrightarrow{\mathcal{D}}_s$ at $z_\infty^{p_s}$  (cf.\ Section~I.\ref{P1-subsection: modified indices at z infty}) corresponding to the $s^{\mbox{\tiny th}}$ end is given by
$$\overrightarrow{\mathcal{D}}_s=\{(i_{s,\ell}',j_{s,\ell}')\to (i_{s,\ell},j_{s,\ell})\}_{\ell=1}^{p_s}.$$
When $s=1$, $\mathcal{D}_{s=1}^{from}=\{(i_{1,\ell}',j_{1,\ell}')\}_{\ell=1}^{p_1}$ and $\mathcal{D}^{to}_{s=1}=\{(i_{1,\ell},j_{1,\ell})\}_{\ell=1}^{p_1}$ specify the initial points on $\overline{\hh}(\overline{a}_{i_{1,\ell}',j_{1,\ell}'})$ and terminal points on $\overline{a}_{i_{1,\ell},j_{1,\ell}}$, respectively; the cases $s=2,3,4$ are analogous.

We then extend the definition of a degree $k$ multisection $\overline{u}$ of $\overline{W}_{-\infty,1}$ to include those with ends ${\bf z}_s=\{z_\infty^{p_s}(\overrightarrow{\mathcal{D}}_s)\}\cup {\bf y}_s$, $s=1,\dots,4$, by attaching data $\mathcal{C}$ to $\overline{u}'$ (cf.\ Definition~I.\ref{P1-defn: data C}) and modifying Definition~\ref{defn: bambi 1} in the same way Definition~I.\ref{P1-overline W curve, version 2} modifies Definition~I.\ref{P1-HF-curve}. Degree $k$ multisections of $\overline{W}_{-\infty,2}$ with ends ${\bf z}=\{z_\infty^p(\overrightarrow{\mathcal{D}})\}\cup {\bf y}$ and ${\bf z'}=\{z_\infty^q(\overrightarrow{\mathcal{D}}')\}\cup{\bf y'}$ are defined similarly.

Let $\mathcal{M}_{\overline{J}_{-\infty,1}}({\bf z}_1,{\bf z}_2,{\bf z}_3,{\bf z}_4)$ be the moduli space of degree $k$ multisections of $(\overline{W}_{-\infty,1},\overline{J}_{-\infty,1})$ with ends ${\bf z}_s$, $s=1,\dots,4$, and let $\mathcal{M}_{\overline{J}_{-\infty,2}}({\bf z},{\bf z'})$ be the moduli space of degree $k$ multisections of $(\overline{W}_{-\infty,2},\overline{J}_{-\infty,2})$ with ends ${\bf z}$ and ${\bf z'}$. Also let
$$\mathcal{M}_{\overline{J}_{-\infty,2}}({\bf z},{\bf z'};\overline{\frak m}(-\infty))\subset \mathcal{M}_{\overline{J}_{-\infty,2}}({\bf z},{\bf z'})$$
be the subset of curves which pass through $\overline{\frak m}(-\infty)$. We then define the extended moduli spaces $\mathcal{M}^{\dagger,ext}_{\overline{J}_{-\infty,1}}({\bf z}_1,{\bf z}_2,{\bf z}_3,{\bf z}_4)$ and $\mathcal{M}^{\dagger,ext}_{\overline{J}_{-\infty,2}}({\bf z},{\bf z'})$ in a manner similar to that of Section~I.\ref{P1-subsubsection: more moduli spaces}.  The precise definitions will be omitted.

\subsubsection{Indices}
\label{subsubsection: indices part 1}

We now discuss the Fredholm index $\op{ind}(\overline{u})$ and the ECH index $I(\overline{u})$ of a $\overline{W}_\tau$-curve $\overline{u}: \dot F \to \overline{W}_\tau$ from ${\bf y}$ to ${\bf y}'$ (i.e., when $\overline{u}'=\varnothing$). The discussion will be brief since all the key ingredients have already been discussed in Sections~I.\ref{P1-subsection: the Fredholm index W plus W minus} and I.\ref{P1-subsection: ECH index W plus minus}.

We remark that, once again, $\op{ind}(\overline{u})$ and $I(\overline{u})$ do not take into account the point constraint $\overline{\frak m}(\tau)$ and that the condition ``passing through $\overline{\frak m}(\tau)$'' is a codimension $2$ condition.

Let $\check{\overline{W}}_\tau=\overline{W}_\tau-\{s>l(\tau)+1\}-\{s< -l(\tau)-1\}$, where $l(\tau)$ is given in Equation~\eqref{l and r}.\footnote{Note that the definition of $\check{\overline{W}}_\tau$ is slightly different from that of Section~\ref{defn of family 1}.} Let $\check{\overline{u}}: \check F\to \check{\overline{W}}_\tau$ be the compactification of $\overline{u}$, where $\check F$ is obtained by performing a real blow-up of $F$ at its boundary punctures. We also define
\begin{align*}
Z_{{\bf y}, {\bf y'}}&= (\{l(\tau)+1\}\times[0,1]\times{\bf y}) \cup (\{-l(\tau)-1\}\times [0,1]\times {\bf y'})\\
&  \qquad \qquad \cup ((L^{\tau,+}_{\widehat{\bf a}} \cup L^{\tau,-}_{\widehat{\bf b}})\cap \check{\overline{W}}_\tau).
\end{align*}

The trivialization $\tau^\star$ of $T\overline{S}$ along $Z_{{\bf y}, {\bf y'}}$ is defined in a manner similar to that of Section~I.\ref{P1-subsubsection: Fredholm index second version}:\footnote{Here we are write $\tau^\star$ instead of $\tau$ due to the notational conflict with $\tau\in\R$.} First we define the trivialization $\tau^\star$ of $T\overline{S}|_{\pi^{-1}_{B_\tau}(l(\tau)+1,1)}$ (resp.\ $T\overline{S}|_{\pi^{-1}_{B_\tau}(-l(\tau)-1,1)}$) along the $\widehat{a}_i$ (resp.\ $\widehat{b}_i$) by choosing a nonsingular tangent vector field along $\widehat{a}_i$ (resp.\ $\widehat{b}_i$).  We then parallel transport $\tau^\star$ along $\bdry \overline{W}_\tau$ and extend $\tau^\star$ arbitrarily to $\{l(\tau)+1\}\times[0,1]\times{\bf y}$ and $\{-l(\tau)-1\}\times [0,1]\times {\bf y'}$.

Let $Q_{\tau^\star}(\check{\overline u})$ be the relative intersection form given by intersecting $\check{\overline u}$ and a pushoff of $\check{\overline{u}}$ in the direction of $J\tau^\star$ along $\bdry \check{\overline{W}}_\tau$. Then
\begin{equation}
I({\overline u})=c_1(\check{\overline
u}^*T\check{\overline{W}}_\tau,(\tau^\star,\bdry_t))
+Q_{\tau^\star}(\check{\overline u})+ \mu_{\tau^\star}({\bf
y})-\mu_{\tau^\star}({\bf y'})-2g,
\end{equation}
\begin{equation}
\op{ind}(\overline{u})=-\chi(\dot F)+ 2c_1(\check{\overline
u}^*T\check{\overline{W}}_\tau,(\tau^\star,\bdry_t))+
\mu_{\tau^\star}({\bf y})-\mu_{\tau^\star}({\bf y'})-2g,
\end{equation}
and the index inequality holds as usual:
\begin{equation}
\op{ind}(\overline{u})+2\delta(\overline{u})\leq I(\overline{u}),
\end{equation}
where $\delta(\overline{u})\geq 0$ and equals zero if and only if $\overline{u}$ is an embedding.

In the general case when $\overline{u}'\not=\varnothing$, modifications can be made as in Section~I.\ref{P1-lerici2} and one can easily verify the following:
\begin{equation} \label{index for sigma infty plus}
I(\sigma_\infty^+)=\op{ind}(\sigma_\infty^+)=-1;
\end{equation}
\begin{equation} \label{index for sigma infty minus}
I(\sigma_\infty^-)=\op{ind}(\sigma_\infty^-)=0;
\end{equation}
\begin{equation} \label{index for sigma infty tau}
I(\sigma_\infty^\tau)=\op{ind}(\sigma_\infty^\tau)=-1.
\end{equation}

The Fredholm and ECH indices for $\overline{W}_{-\infty,i}$-curves can be defined and computed similarly. We now prove the following:

\begin{lemma} \label{lemma: ECH for W minus infinity 2}
If there exists a $\overline{W}_{-\infty,2}$-curve $\overline{u}$ with $I=2$ and ends ${\bf y}$ and ${\bf y'}$, then ${\bf y}=\{ x_{ij(i)}^\#\}_{i=1}^{2g}$ and ${\bf y'}=\{x_{ik(i)}^\#\}_{i=1}^{2g}$, where $j(i)$ is odd and $k(i)$ is even for all $i$.
\end{lemma}

In other words, ${\bf y}$ is a summand of the top generator $\Theta_{\overline{\bf a},\overline{\bf b}}\in \widehat{HF}(\overline{\bf a},\overline{\bf b})$ and ${\bf y'}$ is a summand of the top generator $\Theta_{\overline{\bf b},\overline{\bf a}}\in \widehat{HF}(\overline{\bf b},\overline{\bf a})$. We remind the reader that our convention is that ${\bf y}$ and ${\bf y'}$ do not contain $z_\infty$.

\begin{proof}
The proof is similar to the index calculation of Lemma~\ref{lemma: regularity and dimension}. Let $\overline{u}$ be a $\overline{W}_{-\infty,2}$-curve; note that $n^*(\overline{u})=m$ by definition. First consider the situation where $\overline{u}$ is in the homology class consisting of a copy $\{pt\}\times \overline{S}$ of the fiber and $2g$ trivial strips.  Then $$\op{ind}(\{pt\}\times\overline{S})=-\chi(\overline{S})+2c_1(T\overline{S})=-(2-2g)+2(2-2g)=2-2g.$$
The strips contribute $0$ to $\op{ind}$ and there are $2g$ intersection points.  Hence
$$I(\overline{u})=(2-2g)+ 0 +2(2g)=2g+2$$
when ${\bf y'}={\bf y}$.  The only way to lower $I(\overline{u})$ to $2$ is to take ${\bf y}=\{ x_{ij(i)}^\#\}_{i=1}^{2g}$ and ${\bf y'}=\{x_{ik(i)}^\#\}_{i=1}^{2g}$ so that all the $j(i)$ are odd and all the $k(i)$ are even.
\end{proof}

\subsubsection{Regularity} \label{subsubsection: regularity tomodachi}

We now discuss the regularity of the family $\{\overline{J}_\tau\}_{\tau\in\R}$.

\begin{defn} $\mbox{}$
\begin{enumerate}
\item $\overline{J}_{-\infty,2}\in \mathcal{J}_{\overline{W}_{-\infty,2}}$ is {\em regular} if all the moduli spaces  $\mathcal{M}^{\dagger,ext}_{\overline{J}_{-\infty,2}}({\bf z},{\bf z'})$ are transversely cut out.
\item $\overline{J}_{-\infty,1}\in \mathcal{J}_{\overline{W}_{-\infty,1}}$ is {\em regular} if all the moduli spaces $\mathcal{M}^{\dagger,ext}_{\overline{J}_{-\infty,1}}({\bf z}_1,\dots,{\bf z}_4)$ are transversely cut out and the restrictions $\overline{J}_{-\infty,2}$ and $\overline{J}$ of $\overline{J}_{-\infty,1}$ to the ends are regular.
\item $\overline{J}_{-\infty}=\overline{J}_{-\infty,1}\cup \overline{J}_{-\infty,2}\in \mathcal{J}_{\overline{W}_{-\infty}}$ is {\em regular} if $\overline{J}_{-\infty,i}$ are regular for $i=1,2$.
\end{enumerate}
\end{defn}

\begin{defn} \label{defn: regularity for W tau part 1}
The family $\{\overline{J}_\tau\}_{\tau\in\R}\in \overline{\mathcal{I}}$ is {\em regular} if:
\begin{enumerate}
\item all the moduli spaces $\mathcal{M}_{\{\overline{J}_\tau\}}^{\dagger,ext}({\bf z},{\bf z'})$ are transversely cut out;
\item the restriction $\overline{J}$ of $\overline{J}_\tau$ to the positive and negative ends is regular;
\item $\overline{J}_+$ and $\overline{J}_-$ in the limit $\tau\to +\infty$ are regular; and
\item $\overline{J}_{-\infty}$ in the limit $\tau\to -\infty$ is regular.
\end{enumerate}
\end{defn}

Let $\overline{\mathcal{I}}^{reg}$ be the space of regular $\{\overline{J}_\tau\}\in \overline{\mathcal{I}}$. As usual, we have:

\begin{lemma}
The generic $\{\overline{J}_\tau\}\in \overline{\mathcal{I}}$ is regular. 
\end{lemma}

Next we discuss the regularity of moduli spaces passing through $\overline{\frak m}$. Since the point constraints $\overline{\frak m}(\tau)$ are nongeneric, we need to introduce a perturbation of the family $\{\overline{J}_\tau\}$:

\begin{defn} \label{defn: diamond}
Let $\varepsilon>0$ and let $\{U_\tau\}_{\tau\in\R}$, $U_\tau\subset\overline{W}_\tau$, be a family of open sets such that $U_\tau\not\ni \overline{\frak m}(\tau)$. Then a family $\{\overline{J}_\tau^\Diamond\}_{\tau\in \R}$ of almost complex structures on $\{\overline{W}_\tau\}$ is {\em $(\varepsilon,\{U_\tau\})$-close} to a regular $\{\overline{J}_\tau\}$ if:
\begin{itemize}
\item $\overline{J}_\tau^\Diamond=\overline{J}_\tau$ on $\overline{W}_\tau-U_\tau$;
\item $\overline{J}_\tau^\Diamond$ is $\varepsilon$-close to $\overline{J}_\tau$ on $U_\tau$; and
\item $\nabla \overline{J}_\tau^\Diamond$ is $\varepsilon$-close to $\nabla\overline{J}_\tau$ on $U_\tau$.
\end{itemize}
\end{defn}

Here the $\varepsilon$-closeness is measured with respect to a family $\{g_\tau\}_{\tau\in\R}$ of Riemannian metrics which is defined as follows: Let $h$ be an $s$-invariant metric on $\R\times \overline{N}_2$ such that $h$, viewed as a metric on $\R\times[0,2]\times\overline{S}$, is also $t$-invariant on $\R\times[1,2]\times\overline{S}$. There is an extension of $h$ to $h_\tau$ on $\R\times[0,r(\tau)]\times\overline{S}$ which is $s$- and $t$-invariant on $\R\times[2,r(\tau)]\times\overline{S}$. We then view $h_\tau$ as a metric on $\R\times\overline{N}_{r(\tau)}$ and define $g_\tau$ as the restriction of $h_\tau$ to $\overline{W}_\tau\subset \R\times\overline{N}_{r(\tau)}$.

 Let ${\frak p}(\tau)\subset int(B_\tau)$, $\tau\in[-\infty,\infty]$, be a family of points, where:
\begin{itemize}
\item the cardinality $\#{\frak p}(\tau)$ is finite and independent of $\tau\in[-\infty,\infty]$;
\item ${\frak p}(\tau)$ is smooth for $\tau\in(-\infty,\infty)$;
\item $\lim_{\tau\to +\infty}{\frak p}(\tau)$ exists and equals ${\frak p}(+\infty)$;
\item $\lim_{\tau\to -\infty}{\frak p}(\tau)$ exists and equals ${\frak p}(-\infty)$.
\end{itemize}
In order for $\lim_{\tau\to +\infty}{\frak p}(\tau)$ to be defined we require the existence of $C>0$ such that, for all $\tau\gg 0$, ${\frak p}(\tau)$ is contained in a $C$-neighborhood of $\bdry B_\tau$. Then ${\frak p}(\tau)$ can be viewed as a subset of $B_+\cup B_-$ and we are asking $\lim_{\tau\to +\infty}{\frak p}(\tau)={\frak p}(+\infty)$ in $B_+\cup B_-$. $\lim_{\tau\to -\infty}{\frak p}(\tau)$ is defined analogously.
\begin{itemize}
\item ${\frak p}(+\infty)$ is a nontrivial intersection of points of $int(B_+)$ and $int(B_-)$; similarly, ${\frak p}(-\infty)$ is a nontrivial intersection of points of $int(B_{-\infty,1})$ and $int(B_{-\infty,2})$;
\item for each $\tau\in [-\infty,\infty]$, $\overline{\frak m}^b(\tau)\not\in {\frak p}(\tau)$.
\end{itemize}

We will use the following specific open sets $U_\tau$:  For $\tau\in \R$, let $U_\tau$ be an open $\delta$-neighborhood of $K_\tau=\pi^{-1}_{B_\tau}({\frak p}(\tau))-\{\rho<2\delta\}$, where $\delta>0$ is arbitrarily small. Then let $U_{\pm\infty}$ and $K_{\pm\infty}$ be the limits of $U_\tau$ and $K_\tau$ as $\tau\to \pm\infty$. When we want to emphasize $(\varepsilon,{U_\tau})$ or $(\varepsilon,\delta,{\frak p}(\tau))$, we write $\overline{J}_\tau^\Diamond(\varepsilon,U_\tau)$ or $\overline{J}_\tau^\Diamond(\varepsilon,\delta,{\frak p}(\tau))$ for $\overline{J}_\tau^\Diamond$, $U_{\varepsilon,\delta,{\frak p}(\tau)}$ for $U_\tau$, and $K_{{\frak p}(\tau),\delta}$ for $K_\tau$.

We define a degree $k$ {\em almost multisection $\overline{u}$ of $(\overline{W}_\tau,\overline{J}_\tau^\Diamond)$} in the same way as a degree $k$ multisection of $(\overline{W}_\tau,\overline{J}_\tau)$, except that $\overline{u}$ is a multisection of
$$\pi_{B\tau}: \overline{W}_\tau - \pi_{B_-}^{-1}(V_\tau)\to B_\tau - V_\tau,$$
where $V_\tau = \pi_{B_-}(U_\tau)$. The moduli spaces of almost multisections are defined in the same way as the moduli spaces of multisections, with $\overline{J}_\tau^\Diamond$ replacing $\overline{J}_\tau$.  If $\{K_\tau\not\ni \overline{\frak m}(\tau)\}$ is a family of compact sets of $\overline{W}_\tau$, then the modifier $\{K_\tau\}$ means that $\overline{u}$ in $\overline{W}_\tau$ passes through $K_\tau$.

Almost complex structures $\overline{J}_{-\infty,i}^\Diamond$, almost multisections on $(\overline{W}_{-\infty,i},\overline{J}_{-\infty,i}^\Diamond)$, and moduli spaces of almost multisections are defined similarly.

\begin{defn}
The family $\{\overline{J}_\tau^\Diamond\}$ is {\em $\{K_\tau\}$-regular with respect to $\overline{\frak m}$} if all the moduli spaces $\mathcal{M}^{\dagger,ext,\{K_\tau\}}_{\{\overline{J}_\tau^\Diamond\}} ({\bf z},{\bf z'};\overline{\frak m})$ are transversely cut out.
\end{defn}

\begin{lemma} \label{lemma: regularity of W tau family}
A generic family $\{\overline{J}_\tau^\Diamond\}$ is $\{K_{\tau}\}$-regular with respect to $\overline{\frak m}$.
\end{lemma}

\begin{proof}
The proof is similar to that of \cite[Theorem~3.1.7]{MS}, with modifications as in Proposition~I.\ref{P1-prop: J+ and J- regular}.
\end{proof}

The following can also be proved using a standard regularity argument:

\begin{lemma} \label{lemma: codimension one}
If $\{\overline{J}_\tau\}$ is a generic family, then for $\varepsilon,\delta>0$ sufficiently small, there exist a generic family $\{\overline{J}_\tau^\Diamond(\varepsilon,\delta,{\frak p}(\tau))\}$ which is $\{K_{{\frak p}(\tau),\delta}\}$-regular with respect to $\overline{\frak m}$ and disjoint finite subsets $\mathcal{T}_1,\mathcal{T}_2\subset \R$ with the following properties:
\begin{enumerate}
\item $\tau\in \mathcal{T}_1$ if and only if there exists $\overline v_\tau\in \mathcal{M}_{\overline{J}_\tau}^{\dagger,s,irr,\op{ind}=-1}({\bf z},{\bf z'})$ for some ${\bf z}$ and ${\bf z'}$.
\item $\tau\in \mathcal{T}_2$ if and only if there exists $\overline v_\tau\in \mathcal{M}^{\dagger,s,irr,\{K_\tau\},\op{ind}=1}_{\overline{J}_\tau^\Diamond(\varepsilon,\delta,{\frak p}(\tau))} ({\bf z},{\bf z'};\overline{\frak m})$ for some ${\bf z}$ and ${\bf z'}$.
\end{enumerate}
Moreover, for each $\tau\in\mathcal{T}_i$ there is a unique such irreducible curve $\overline v_\tau$.
\end{lemma}

By shrinking $\overline{\mathcal{I}}$, we assume that all $\{\overline{J}_\tau\}\in \overline{\mathcal{I}}$ satisfy Lemma~\ref{lemma: codimension one}.

\subsection{Proof of half of Theorem~\ref{thm: isomorphism}}
\label{subsection: proof of thm chain homotopy one}

In this subsection we prove that $\Psi\circ \Phi$ is an isomorphism on the level of homology.

In the next several paragraphs we briefly recall the chain complexes
$$(\widehat{CF'}(S,\mathbf{a},\hh(\mathbf{a})),\bdry'), \quad (\widetilde{CF}(S,{\bf a},\hh({\bf a})),\widetilde\bdry)$$
with isomorphic homology groups from Sections~I.\ref{P1-subsubsection: variant CF of S} and I.\ref{P1-subsection: variant widetide psi}, the quotient map
$$q:\widehat{CF'}(S,\mathbf{a},\hh(\mathbf{a}))\to \widehat{CF}(S,\mathbf{a},\hh(\mathbf{a})),$$
and the maps $\widetilde\Phi$ and $\Psi,\Psi'$ from Sections~I.\ref{P1-subsection: variant widetide psi} and I.\ref{P1-subsection: defn of psi}.

The chain complex $(\widehat{CF'}(S,\mathbf{a},\hh(\mathbf{a})),\bdry')$ is generated by the set $\mathcal{S}_{{\bf a}, \hh({\bf a})}$ of $2g$-tuples of intersection points of ${\bf a}$ and $\hh(\bf a)$, where each ${\bf y}\in \mathcal{S}_{{\bf a}, \hh({\bf a})}$ intersects $a_i$ and $\hh(a_i)$ exactly once and the differential $\bdry'$ counts $I=1$, degree $2g$ multisections of $W=\R\times[0,1]\times S$. The chain complex $\widehat{CF}(S,\mathbf{a},\hh(\mathbf{a}))$ is the quotient of $(\widehat{CF'}(S,\mathbf{a},\hh(\mathbf{a})),\bdry')$ under the equivalence relation $\sim$ which identifies ${\bf y}\sim {\bf y'}$ if ${\bf y'}$ can be obtained from ${\bf y}$ by successively replacing $x_i$ by $x_i'$ or $x_i'$ by $x_i$ for any $i=1,\dots,2g$; the map $q$ is the corresponding quotient map.  Here $x_i$ and $x_i'$ are intersection points of $a_i$ and $\hh(a_i)$ on $\bdry S$ given in Section~I.\ref{P1-subsubsection: Heegaard diagram compatible with S h}.

The chain complex $(\widetilde{CF}(S,{\bf a},\hh({\bf a})),\widetilde\bdry)$ is a variant of $\widehat{CF}(S,\mathbf{a},\hh(\mathbf{a}))$ (with isomorphic homology groups) generated by $2g$-tuples of intersection points $\{z_{\infty,i}\}_{i\in \mathcal{I}}\cup {\bf y'}$ (with $\mathcal{I}\subset\{1,\dots,2g\}$) of $\overline{\bf a}$ and $\overline{\hh}({\bf a})$, where:
\begin{itemize}
\item $z_\infty$ can be used more than once, but is always viewed as an intersection point of $\overline{a}_i$ and $\hh({\overline{a}}_i)$ and hence is written as $z_{\infty,i}$;
\item each $\{z_{\infty,i}\}_{i\in \mathcal{I}}\cup {\bf y'}$ intersects $\overline{a}_i$ and $\hh({\overline{a}}_i)$ exactly once.
\end{itemize}
The differential counts $I=1$, degree $2g$ multisections $\overline{u}=\overline{u}'\cup\overline{u}''$ of $\overline{W}$ with $n^*(\overline{u})\leq 1$ such that $\overline{u}'$ has empty branch locus; see Definition~I.\ref{P1-defn: variant chain cx}.

The map
$$\widetilde{\Phi}: \widetilde{CF}(S,{\bf a},\hh({\bf a}))\to PFC_{2g}(N)$$
is a variant of the map $\Phi:\widehat{CF}(S,{\bf a}, \hh({\bf a}))\to PFC_{2g}(N)$ and
$$\langle \widetilde{\Phi}(\{z_{\infty,i}\}_{i\in \mathcal{I}}\cup {\bf y'}),\bs\gamma\rangle,$$
counts $I=0$, degree $2g$ multisections of $\overline{W}_+$ with $n^*\leq |\mathcal{I}|$ from $\{z_{\infty,i}\}_{i\in \mathcal{I}}\cup {\bf y'}$ to $\bs \gamma$.

The map
$$\Psi':PFC_{2g}(N)\to \widehat{CF'}(S,\mathbf{a},\hh(\mathbf{a}))$$
counts $I=2$, degree $2g$ almost multisections of $(\overline{W}_-,\overline{J}_-^\Diamond)$ (cf.\ Definition~I.\ref{P1-defn: almost multisection}) with $n^*=m$ which pass through $\overline{\frak m}(+\infty)$. The map
$$\Psi: PFC_{2g}(N)\to \widehat{CF}(S,\mathbf{a},\hh(\mathbf{a}))$$
is then given as the composition $q\circ \Psi'$.

The following is proved in this subsection:

\begin{thm} \label{thm: chain homotopy part i}
Suppose $m\gg 0$. Then there exist maps
$$H', \Theta'_0: \widetilde{CF}(S,\mathbf{a},\hh(\mathbf{a}))\to \widehat{CF'} (S,\mathbf{b},\hh(\mathbf{b})),$$
$$\widetilde{V}:\widetilde{CF}(S,\mathbf{a},\hh(\mathbf{a}))\to \widetilde{CF} (S,\mathbf{b},\hh(\mathbf{b})),$$
which satisfy the following:
\begin{equation} \label{eqn chain homotopy part i}
\Psi'\circ \widetilde\Phi -\Theta'_0 = (\bdry' H'+H'\widetilde\bdry) + \widetilde\bdry_1 \circ \widetilde{V},
\end{equation}
where $\widetilde\bdry_1$, defined in Equation~(I.\ref{P1-eqn: bdry sub 1}),  satisfies:
$$\widetilde\bdry_1(\{z_{\infty,i}\}\cup {\bf y'})=\{x_i\}\cup {\bf y'} +\{x_i'\}\cup {\bf y'}$$
for all $i=1,\dots,2g$ and is zero for any other generator. (Here $z_{\infty,i}$, $x_i$, and $x_i'$ are for the basis ${\bf b}$.)
Postcomposing with $q$ (for the basis ${\bf b}$), we obtain the chain homotopy
\begin{equation} \label{eqn: chain homotopy part I variant}
\Psi\circ \widetilde\Phi -\Theta_0 = \bdry H+H\widetilde\bdry,
\end{equation}
where $\Theta_0=q\circ \Theta'_0$ induces an isomorphism on homology.
\end{thm}

Since $\Theta_0$ is an isomorphism on the level of homology, so is $\Psi\circ \widetilde\Phi$. In view of Corollary~I.\ref{P1-tigger}, $\Psi\circ \Phi$ is also an isomorphism on the level of homology.

\begin{proof}
We prove Theorem~\ref{thm: chain homotopy part i}, assuming the results of Sections~\ref{subsection: chain homotopy part 1}--\ref{subsection: additional degenerations III}.   In Steps 1--3 we consider the situation where the holomorphic curves $\overline{u}$ in $\overline{W}_\tau$ are asymptotic to some ${\bf y}\in \mathcal{S}_{{\bf a},\hh({\bf a})}$ at the positive end.  In Step 4 we describe the modifications needed for the situation where $\overline{u}$ is asymptotic to some ${\bf z}=\{z_{\infty,i}\}_{i\in\mathcal{I}}\cup {\bf y}$ at the positive end and $\mathcal{I}$ is a subset of $\{1,\dots,2g\}$ with $|\mathcal{I}|>0$. Steps 1--4 prove Equation~\eqref{eqn chain homotopy part i} and hence Equation~\eqref{eqn: chain homotopy part I variant}. In Step 5 we prove that $\Theta_0$ induces an isomorphism on homology by further degenerating $\overline{W}_{-\infty,1}$.

Suppose $m\gg 0$. Choose ${\frak p(\tau)}$ and $\{\overline{J}_\tau\}\in \overline{\mathcal{I}}^{reg}$. For sufficiently small $\varepsilon,\delta>0$ (which depend on the choices of $m$ and $\{\overline{J}_\tau\}$), there exists $\{\overline{J}_\tau^\Diamond(\varepsilon,\delta,{\frak p}(\tau))\}$ so that Lemma~\ref{lemma: codimension one} holds.

Fix ${\bf y}\in \mathcal{S}_{{\bf a},\hh({\bf a})}$, ${\bf y'}\in\mathcal{S}_{{\bf b},\hh({\bf b})}$ and abbreviate
$$\mathcal{M}=\mathcal{M}^{I=2,n^*=m}_{\{\overline{J}_\tau^\Diamond(\varepsilon,\delta,{\frak p}(\tau))\}}({\bf y},{\bf y'};\overline{\frak m}),~\mathcal{M}^{\{K_{{\frak p}(\tau),\delta}\}}=\mathcal{M}^{I=2,n^*=m,\{K_{{\frak p}(\tau),\delta}\}}_{\{\overline{J}_\tau^\Diamond(\varepsilon,\delta,{\frak p}(\tau))\}}({\bf y},{\bf y'};\overline{\frak m}).$$
Let $\overline{\mathcal{M}}$ be the SFT compactification of $\mathcal{M}$ and let $\bdry \mathcal{M}= \overline{\mathcal{M}}-\mathcal{M}$
be the boundary of $\mathcal{M}$. If $U\subset [-\infty,+\infty]$, then we write $\bdry_U \mathcal{M}$ for the set of $\overline{u}_\infty\in \bdry \mathcal{M}$ where $\overline{u}_\infty$ is a building which corresponds to some $\tau\in U$. By Lemma~\ref{lemma: regularity of W tau family}, we may take $\mathcal{M}^{\{K_{{\frak p}(\tau),\delta}\}}$ to be regular.

\s\n {\em Step 1 (Breaking at $+\infty$).}
Recall the definition of a {\em bad radial ray} $\mathcal{R}_\phi$ from Definition~I.\ref{P1-defn: radial rays}. We now enlarge the class of bad radial rays as follows: Let $(\overline{J}_+)_\infty$ be the limit of $(\overline{J}_+)_m$ as $m\to \infty$.  Let
$$\coprod_{{\bf y},\delta_0^{l_i}\bs\gamma'}\mathcal{M}_{(\overline{J}_+)_\infty}^{I=0,(l_i)}({\bf y},\delta_0^{l_i}\bs\gamma')=\{\mathcal{C}_1,\dots,\mathcal{C}_{r}\},$$
where the disjoint union is over all ${\bf y}$ and $\delta_0^{l_i}\bs\gamma'$ with $l_i>0$ and let $f_{i}: \R/2l_{i}\Z\to \C$ be the asymptotic eigenfunction corresponding to the ends $\delta_0^{l_{i}}$ of $\mathcal{C}_i$, where $(l_i)$ is the partition of $l_i$ corresponding to $\mathcal{C}_i$.
 (We remark that $\mathcal{C}_1,\dots, \mathcal{C}_r$ and $f_{i}$ also appear in Section~I.\ref{P1-subsubsection: asymptotic eigenfunction at an end}, but denote similar but different things.) We then add the radial rays which pass through
$$\{f_{i}(t)~|~ i=1,\dots,r; 0<t< 2l_{i}; t\equiv 3/2 \mbox{ mod } 2\}$$
to the class of bad radial rays.  We can still assume that $\mathcal{R}_\pi$ is a good radial ray.

Recall the set $\widehat{\mathcal{O}}_k$ of orbit sets constructed from $\widehat{\mathcal{P}}$ (the set of simple Reeb orbits in $int(N)$, together with $h$ and $e$) which intersect $\overline{S}\times\{0\}$ exactly $k$ times.  We then have the following:

\begin{lemma} \label{cherries4}
$\bdry_{\{+\infty\}} \mathcal{M}\subset A_1\sqcup A_2$, where
\begin{align*}
A_1&=\coprod_{\bs\gamma\in \widehat{\mathcal{O}}_{2g}}\left(\mathcal{M}^{I=0}_{J_+^\Diamond(\varepsilon,\delta,{\frak p}(+\infty))}({\bf y},\bs\gamma)\times \mathcal{M}^{I=2,n^*=m}_{\overline{J}_-^\Diamond(\varepsilon,\delta,{\frak p}(+\infty))}(\bs\gamma,{\bf y}';\overline{\frak m}(+\infty))\right);\\
A_2&=\coprod_{\delta_0\bs\gamma, \{z_\infty\}\cup {\bf y}''}\left( \mathcal{M}^{I=1,n^*=m-1, f_{\delta_0}}_{\overline{J}_+^\Diamond(\varepsilon,\delta,{\frak p}(+\infty))}({\bf y},\delta_0\bs\gamma)\times \mathcal{M}^{I=0,n^*=0}_{\overline{J}_-^\Diamond(\varepsilon,\delta,{\frak p}(+\infty))}(\delta_0\bs\gamma,\{z_\infty\}\cup {\bf y}'')\right.\\
&  \qquad \qquad \qquad \left.\times~~~ \mathcal{M}^{I=1,n^*=1}_{\overline{J}}(\{z_\infty\}\cup {\bf y}'',{\bf y}')\right),
\end{align*}
if ${\bf y}'=\{x_i^j\}\times{\bf y}''$ for some $x_i^j$ and $A_2=\varnothing$ otherwise.  The disjoint union for $A_2$ ranges over all $\delta_0\bs\gamma$ such that $\bs\gamma\in \widehat{\mathcal{O}}_{2g-1}$ and all $\{z_\infty\}\cup {\bf y}''$ such that ${\bf y}'$ can be written as $\{x_i^j\}\cup {\bf y}''$ for some $x_i^j$. Here we have omitted the potential contributions of connector components and we are writing $x_i^0:=x_i$ and $x_i^1:=x_i'$.
\end{lemma}

We will explain the moduli spaces that are involved in $A_2$: $f_{\delta_0}$ is a nonzero normalized asymptotic eigenfunction of $\delta_0$ at the negative end such that $f_{\delta_0}({3\over 2})$ lies on the good radial ray $\mathcal{R}_\pi$.  Used as a modifier, $f_{\delta_0}$ stands for ``the normalized asymptotic eigenfunction at the negative end $\delta_0$ is $f_{\delta_0}$''. If
$$\overline{u}=\overline{u}'\cup \overline{u}''\in \mathcal{M}^{I=0,n^*=0}_{\overline{J}_-^\Diamond(\varepsilon,\delta,{\frak p}(+\infty))} (\delta_0\bs\gamma,\{z_\infty\}\cup {\bf y}''),$$
then $\overline{u}$ consists of $\overline{u}'=\sigma_\infty^-$ and a curve $\overline{u}''$ from $\bs\gamma$ to ${\bf y}''$ which is arbitrarily close to a curve with image in $W_-$. If
$$\overline{u}\in \mathcal{M}^{I=1,n^*=1}_{\overline{J}}(\{z_\infty\}\cup {\bf y}'',{\bf y}'),$$
then ${\bf y}'= \{x_i^j\}\cup{\bf y}''$ for some $i,j$ and $\overline{u}$ consists of one thin strip from $z_\infty$ to $x_i^j$ and $2g-1$ trivial strips.

\begin{rmk}
As in the proof of the chain map property for $\Psi$ from Section~I.\ref{P1-subsection: outline of proof}, the point constraint of passing through $\overline{\frak m}(+\infty)$ is converted to an asymptotic constraint of the normalized asymptotic eigenfunction being $f_{\delta_0}$ when a section at infinity $\sigma_\infty^-$ is present.
\end{rmk}

Lemma~\ref{cherries4} will be proved in Section~\ref{subsection: chain homotopy part 1}. Gluing the pairs in $A_1$ using the Hutchings-Taubes gluing theorem~\cite{HT1,HT2} (see Section~I.\ref{P1-subsection: gluing for Phi}) accounts for the term $\Psi'\circ \widetilde\Phi$ in Equation~\eqref{eqn chain homotopy part i}.

Gluing the triples in $A_2$ accounts for the term $\bdry_1\circ \widetilde{V}$. This is similar to Section~I.\ref{P1-subsection: gluing for psi}, and the details will be omitted. For each triple $(\overline{v}_+,\overline{v}_-,\overline{v}_{-1,1})\in A_2$, where the nontrivial component of $\overline{v}_{-1,1}$ is a thin strip from $z_{\infty,i,0}$ to $x_i$, there is another triple $(\overline{v}_+,\overline{v}_-,\overline{v}^*_{-1,1})\in A_2$, where the nontrivial component of $\overline{v}^*_{-1,1}$ is a thin strip from $z_{\infty,i,1}$ to $x_i'$.

Since $\mathcal{M}^{\{K_{{\frak p}(\tau),\delta}\}}$ is regular but $\mathcal{M}-\mathcal{M}^{\{K_{{\frak p}(\tau),\delta}\}}$ is not a priori regular, it remains to verify the following:

\begin{claim} \label{elmwood}
For $\varepsilon,\delta>0$ sufficiently small, the restriction of $\bdry \mathcal{M}^{\{K_{{\frak p}(\tau),\delta}\}}$ to a neighborhood of $\tau=+\infty$ is $A_1\cup \widetilde{A}_2$, where $\# A_2\equiv \# \widetilde{A}_2 \mbox{ mod } 2$.
\end{claim}

\begin{proof}
If $\varepsilon,\delta>0$ are sufficiently small, then $\overline{v}_-$ passes through $K_{{\frak p}(\tau),\delta}$ whenever $(\overline{v}_+,\overline{v}_-)\in A_1$.  Hence if $\overline{u}\in \mathcal{M}$ is close to a curve in $A_1$, then $\overline{u}\in \mathcal{M}^{\{K_{{\frak p}(\tau),\delta}\}}$. This accounts for the term $A_1$ in the claim.

Next suppose that $\overline{u}\in \mathcal{M}-\mathcal{M}^{\{K_{{\frak p}(\tau),\delta}\}}$ for $\tau$ near $+\infty$.  By the argument of Lemma~I.\ref{P1-claim 1}, $\overline{u}$ is arbitrarily close to a building of type $A_2$ and $\mathcal{M}^{\{K_{{\frak p}(\tau),\delta}\}}$ is obtained from $\mathcal{M}$ by truncating ends that are close to $A_2$.  Then, by the adaptation of Theorem~I.\ref{P1-thm: transversality of ev map} to our case, $\widetilde{A}_2:= \bdry \mathcal{M}^{\{K_{{\frak p}(\tau),\delta}\}}-A_1$ satisfies $\# A_2\equiv \# \widetilde{A}_2 \mbox{ mod } 2$.
\end{proof}

\s\n {\em Step 2 (Breaking at $-\infty$).}

\begin{lemma} \label{cherries5}
$\bdry_{\{-\infty\}} \mathcal{M}\subset A_3$, where
\begin{align*}
A_3&=\coprod_{{\bf y}_2,{\bf y}_4} \left(\mathcal{M}_{\overline{J}_{-\infty,1}^\Diamond(\varepsilon,\delta,{\frak p}(-\infty))}^{I=0,n^*=0} ({\bf y},{\bf y}_2,{\bf y}',\overline{\hh}({\bf y}_4))\times \mathcal{M}_{\overline{J}_{-\infty,2}^\Diamond(\varepsilon,\delta,{\frak p}(-\infty))}^{I=2,n^*=m}({\bf y}_4,{\bf y}_2;\overline{\frak m}(-\infty))\right).
\end{align*}
Here the union is over all ${\bf y}_2$, ${\bf y}_4$ such that ${\bf y}_4=\{ x_{ij(i)}^\#\}_{i=1}^{2g}$ and ${\bf y}_2=\{x_{ik(i)}^\#\}_{i=1}^{2g}$, where $j(i)$ is odd and $k(i)$ is even.
\end{lemma}

Lemma~\ref{cherries5} will be proved in Section~\ref{subsection: chain homotopy part 3}. Gluing the pairs $(\overline{v}_1,\overline{v}_2)$ in $A_3$ accounts for the term $\Theta_0'$ in Equation~\eqref{eqn chain homotopy part i}. The map $\Theta_0'$ is given by:
$$\langle \Theta_0' ({\bf y}), {\bf y}'\rangle=\sum_{{\bf y}_2,{\bf y}_4} \# \mathcal{M}_{\overline{J}_{-\infty,1}^\Diamond(\varepsilon,\delta,{\frak p}(-\infty))}^{I=0,n^*=0} ({\bf y},{\bf y}_2,{\bf y}',\overline{\hh}({\bf y}_4)),$$
where ${\bf y}_2$, ${\bf y}_4$ are as in $A_3$, and Theorem~\ref{thm: calc of G sub 3} can be rephrased as (verification left to the reader):

\begin{thm}
Suppose ${\bf y}_4=\{ x_{ij(i)}^\#\}_{i=1}^{2g}$ and ${\bf y}_2=\{x_{ik(i)}^\#\}_{i=1}^{2g}$, where $j(i)$ is odd and $k(i)$ is even. Then $$\#\mathcal{M}_{\overline{J}_{-\infty,2}^\Diamond(\varepsilon,\delta,{\frak p}(-\infty))}^{I=2,n^*=m}({\bf y}_4,{\bf y}_2;\overline{\frak m}(-\infty))\equiv 1 \mbox{ mod } 2.$$
\end{thm}

The argument of Claim~\ref{elmwood} gives:

\begin{claim} \label{elmwood2}
For $\varepsilon,\delta>0$ sufficiently small, the restriction of $\bdry \mathcal{M}^{\{K_{{\frak p}(\tau),\delta}\}}$ to a neighborhood of $-\infty$ is $A_3$.
\end{claim}

\s\n {\em Step 3 (Breaking in the middle).}

\begin{lemma} \label{cherries3}
$\bdry_{(-\infty,+\infty)} \mathcal{M}\subset A_4\sqcup A_5$, where:
\begin{align*}
A_4&=\coprod_{{\bf y''}\in \mathcal{S}_{{\bf a},\hh({\bf a})}}\left(\mathcal{M}^{I=1}_{J}({\bf y},{\bf y''})\times \mathcal{M}^{I=1,n^*=m}_{\{\overline{J}_\tau^\Diamond(\varepsilon,\delta,{\frak p}(\tau))\}}({\bf y''},{\bf y'};\overline{\frak m})\right); \\
A_5&=\coprod_{{\bf y'''}\in \mathcal{S}_{{\bf b},\hh({\bf b})}}\left(\mathcal{M}^{I=1,n^*=m}_{\{\overline{J}_\tau^\Diamond(\varepsilon,\delta,{\frak p}(\tau))\}}({\bf y},{\bf y'''};\overline{\frak m})\times \mathcal{M}^{I=1}_{J}({\bf y'''},{\bf y}')\right).
\end{align*}
\end{lemma}

Lemma~\ref{cherries3} will be proved in Section~\ref{subsection: chain homotopy part 2}.
Using the technique of \cite[Prop. A.1 and A.2]{Li}, we can glue each of the pairs in $A_4$ and $A_5$. This gluing accounts for the term $\bdry' H'+H'\widetilde\bdry$ in Equation~\eqref{eqn chain homotopy part i}.

\begin{claim}\label{elmwood3}
For $\varepsilon,\delta>0$ sufficiently small, $$\bdry \mathcal{M}^{\{K_{{\frak p}(\tau),\delta}\}}=A_1\cup \widetilde{A}_2\cup A_3\cup A_4\cup A_5.$$
\end{claim}

\s\n
{\em Step 4 (Additional degenerations).} In this step we give the necessary modifications for
$$\mathcal{M}=\mathcal{M}^{I=2,n^*=m+|\mathcal{I}|}_{\{\overline{J}_\tau^\Diamond(\varepsilon,\delta,{\frak p}(\tau))\}}({\bf z},{\bf y}';\overline{\frak m}),$$
where ${\bf z}=\{z_{\infty,i}\}_{i\in\mathcal{I}}\cup {\bf y}$, ${\bf y}$ and ${\bf y}'$ are tuples in $S$, and $\mathcal{I}\subset\{1,\dots,2g\}$ with $|\mathcal{I}|>0$.

The following is proved in Section~\ref{subsection: additional degenerations I}:

\begin{lemma} \label{kyoho plus infty}
$\bdry_{\{+\infty\}}\mathcal{M}\subset A'_2$, where
\begin{align*}
A'_2&=\coprod_{\delta_0\bs\gamma, \{z_\infty\}\cup {\bf y}''}\left( \mathcal{M}^{I=1,n^*=m+|\mathcal{I}|-1, f_{\delta_0},\dagger}_{\overline{J}_+^\Diamond(\varepsilon,\delta,{\frak p}(+\infty))}({\bf z},\delta_0\bs\gamma)\times \mathcal{M}^{I=0,n^*=0}_{\overline{J}_-^\Diamond(\varepsilon,\delta,{\frak p}(+\infty))}(\delta_0\bs\gamma,\{z_\infty\}\cup {\bf y}'')\right.\\
&  \qquad \qquad \qquad \left.\times~~~ \mathcal{M}^{I=1,n^*=1}_{\overline{J}}(\{z_\infty\}\cup {\bf y}'',{\bf y}')\right),
\end{align*}
if ${\bf y}'=\{x_i^j\}\times{\bf y}''$ for some $x_i^j$ and $A'_2=\varnothing$ otherwise.  The disjoint union for $A'_2$ ranges over all $\delta_0\bs\gamma$ such that $\bs\gamma\in \widehat{\mathcal{O}}_{2g-1}$ and all $\{z_\infty\}\cup {\bf y}''$ such that ${\bf y}'$ can be written as $\{x_i^j\}\cup {\bf y}''$ for some $x_i^j$. Here we have omitted the potential contributions of connector components and we are writing $x_i^0:=x_i$ and $x_i^1:=x_i'$.
\end{lemma}

If $\mathcal{I}\not=\varnothing$ (i.e., $|\mathcal{I}|\geq 1$), then $\widetilde\Phi({\bf z})=0$ by Lemma~I.\ref{P1-lemma: value of widetilde Phi} and we have $\Psi\circ\widetilde\Phi({\bf z})=0$; this is consistent with the analog of $A_1$ being empty. On the other hand, gluing the triples in $A_2'$ accounts for the term $\widetilde \bdry_1\circ \widetilde V$.

Next, the following is proved in Section~\ref{subsection: additional degenerations II}:

\begin{lemma}\label{kyoho minus infty}
$\bdry_{\{-\infty\}}\mathcal{M}\subset A_3'$, where:
\begin{align*}
A_3'&=\coprod_{{\bf y}_2,{\bf y}_4} \left(\mathcal{M}_{\overline{J}_{-\infty,1}^\Diamond(\varepsilon,\delta,{\frak p}(-\infty))}^{I=0,n^*= |\mathcal{I}|} ({\bf z},{\bf y}_2,{\bf y}',\overline{\hh}({\bf y}_4))\times \mathcal{M}_{\overline{J}_{-\infty,2}^\Diamond(\varepsilon,\delta,{\frak p}(-\infty))}^{I=2,n^*=m}({\bf y}_4,{\bf y}_2;\overline{\frak m}(-\infty))\right)
\end{align*}
and the summation is over ${\bf y}_2$ and ${\bf y}_4$ as in Lemma~\ref{cherries5}.
\end{lemma}

The following is proved in Section~\ref{subsection: additional degenerations III}. The corresponding gluing accounts for the term $\bdry H +H\widetilde \bdry$ in Equation~\eqref{eqn: chain homotopy part I variant} when $|\mathcal{I}|\geq 1$.

\begin{lemma} \label{kyoho middle}
$\bdry_{(-\infty,+\infty)} \mathcal{M}\subset A_4'\sqcup A_5'$, where:
\begin{align*}
A_4'&=\coprod_{{\bf z}'}\left(\mathcal{M}^{I=1,n^*\leq |\mathcal{I}|}_{\overline{J}}({\bf z},{\bf z}')\times \mathcal{M}^{I=1,n^*\leq m+|\mathcal{I}|}_{\{\overline{J}_\tau^\Diamond(\varepsilon,\delta,{\frak p}(\tau))\}}({\bf z}',{\bf y}';\overline{\frak m})\right);\\
A_5'&=\coprod_{{\bf y}'''\in \mathcal{S}_{{\bf b},\hh({\bf b})}}\left(\mathcal{M}^{I=1,n^*=m+|\mathcal{I}|}_{\{\overline{J}_\tau^\Diamond(\varepsilon,\delta,{\frak p}(\tau))\}} ({\bf z},{\bf y}''';\overline{\frak m})\times \mathcal{M}^{I=1}_{J}({\bf y}''',{\bf y}')\right).
\end{align*}
Moreover, if $\overline{v}_{1,1}\in \mathcal{M}^{I=1,n^*\leq |\mathcal{I}|}_{\overline{J}}({\bf z},{\bf z}')$, then either (i) $\overline{v}_{1,1}^\sharp$ is a thin strip and $\overline{v}_{1,1}^\flat$ is a union of trivial strips, or (ii) $\overline{v}_{1,1}^\sharp=\varnothing$ and $\overline{v}_{1,1}^\flat$ has image in $W$.
\end{lemma}

As before, the analog of Claim~\ref{elmwood3} holds.

\s\n {\em Step 5 (The map $\Theta_0$).}
The map $\Theta_0$ is defined as follows:
$$\Theta_0: \widetilde{CF}({\bf a},\hh({\bf a}))\to \widehat{CF}({\bf b},\hh({\bf b})),$$
\begin{align*}
\langle \Theta_0 ({\bf z}), {\bf y}'\rangle &=\sum_{{\bf y}_2,{\bf y}_4} \# \mathcal{M}_{\overline{J}_{-\infty,1}}^{I=0,n^*=|\mathcal{I}|} ({\bf z},{\bf y}_2,{\bf y}',\overline{\hh}({\bf y}_4)).
\end{align*}

\begin{lemma} \label{theta zero is an isomorphism}
The map $\Theta_0$ induces an isomorphism on the level of homology.
\end{lemma}

\begin{proof}[Proof of Lemma~\ref{theta zero is an isomorphism}.]
We degenerate the base $B_{-\infty,1}$ as given in Figure~\ref{figure: base degeneration 3}. Slightly more precisely, we take $B_{-\infty,1,\tau'}$, $\tau'\in[0,+\infty)$, which is obtained from
$$\{-2\leq s \leq 2\}\cup \{ 0\leq t\leq 1, s\leq 2\} \cup \{\tau'\leq t \leq \tau'+1, s\geq -2\}\subset \R^2=\C$$
by smoothing the corners and use the complex structure $j_{-\infty,1,\tau'}$ induced from the standard complex structure on $\C$; then $B_{-\infty,1,\tau'=0}=B_{-\infty,1}$ and $j_{-\infty,1,\tau'=0}=j_{-\infty,1}$. Let $\overline{W}_{-\infty,1,\tau'}=B_{-\infty,1,\tau'}\times \overline{S}$ and define the almost complex structures $\overline{J}_{-\infty,1,\tau'}$ on $\overline{W}_{-\infty,1,\tau'}$ in the same way as on $\overline{W}_{-\infty,1}$ with $j_{-\infty,1}$ replaced by $j_{-\infty,1,\tau'}$.
\begin{figure}[ht]
\begin{center}
\psfragscanon
\psfrag{a}{\small $\overline{\bf a}$}
\psfrag{b}{\small $\overline{\hh}(\overline{\bf a})$}
\psfrag{c}{\small $\overline{\bf b}$}
\psfrag{d}{\small $\overline{\hh}(\overline{\bf b})$}
\includegraphics[width=7cm]{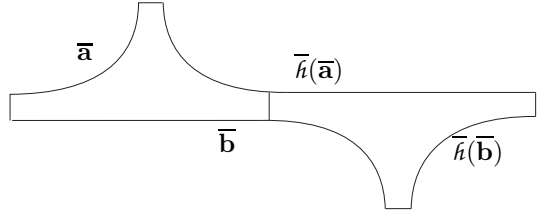}
\end{center}
\caption{Degeneration of the base $B_{-\infty,1}$.}
\label{figure: base degeneration 3}
\end{figure}

The $1$-parameter family $(\overline{W}_{-\infty,1,\tau'}, \overline{J}_{-\infty,1,\tau'})$ induces a map $\Theta_0$ which is chain homotopic to $\Theta_B\circ \Theta_T$, where:
$$\Theta_T:\widetilde{CF}({\bf a}, \hh({\bf a}))\to \widehat{CF}({\bf b},\hh({\bf a})),$$
$$\Theta_B:\widehat{CF}({\bf b},\hh({\bf a}))\to \widehat{CF}({\bf b}, \hh({\bf b}))$$
are defined by tensoring with the top generators $\Theta_{\overline{\bf b},\overline{\bf a}}\in \widehat{CF}(\overline{\bf b}, \overline{\bf a})$ and $\Theta_{\overline{\hh}(\overline{\bf a}),\overline{\hh}(\overline{\bf b})}\in \widehat{CF}(\overline{\hh}(\overline{\bf a}),\overline{\hh}(\overline{\bf b}))$.  There is one undesirable type of breaking which can a priori occur as we vary $\tau'$: a two-level building $\overline{u}_1\cup \overline{u}_2$, where
\begin{itemize}
\item $\overline{u}_1$ is an index $I=-1$, $n^*=|\mathcal{I}|$ curve in $\overline{W}_{-\infty,1,\tau'}$ with ends ${\bf z}, {\bf y}_2', {\bf y}', \overline{\hh}({\bf y_4})$;
\item $\overline{u}_2\in \mathcal{M}_{\overline{J}_{-\infty,2}}^{I=1,n^*=0}({\bf y}_2,{\bf y}_2')$; and
\item ${\bf y}_2$ and ${\bf y}_2'$ differ by replacing one $x_{ik(i)}^\#$ by $x_{ij(i)}^\#$ where $j(i)$ is odd and $k(i)$ is even.
\end{itemize}
One can verify that the only possible component of $\overline{u}_1$ with left end $x^\#_{i1}$ projects to the quadrilateral $Q$ with edges $\overline{a}_i,\overline{b}_i,\overline{\hh}(\overline{b}_i), \overline{\hh}(\overline{a}_i)$ in Figure~\ref{figure: aandb}.  However, the component corresponding to $Q$ has ECH index $I=0$, which is a contradiction. The index calculation basically follows from the fact that $Q$ has three angles smaller than $\pi$ and one larger.
 This implies that undesirable breakings do not exist and that all the breakings contribute toward the chain homotopy.

Both $\Theta_T$ and $\Theta_B$ --- and hence $\Theta_0$ --- induce isomorphisms on the level of homology.
\end{proof}

This completes the proof of Theorem~\ref{thm: chain homotopy part i}, assuming the results from Sections~\ref{subsection: chain homotopy part 1}--\ref{subsection: additional degenerations III}.
\end{proof}

\subsection{Degeneration at $+\infty$}
\label{subsection: chain homotopy part 1}

In this subsection we study the limits of holomorphic maps to $\overline{W}_{\tau}$ as $\tau \to \infty$, i.e., when $\overline{W}_{\tau}$ degenerates into the concatenation of $\overline{W}_+$ with $\overline{W}_-$ along the ECH-type end, in order to prove Lemma~\ref{cherries4}.

We assume that $m\gg 0$; $\varepsilon,\delta>0$ are sufficiently small; and $\{\overline{J}_\tau\}\in \overline{\mathcal{I}}^{reg}$ and $\{\overline{J}_\tau^\Diamond(\varepsilon,\delta,{\frak p}(\tau))\}$ satisfies Lemma~\ref{lemma: codimension one}.  Fix ${\bf y}\in \mathcal{S}_{{\bf a},\hh({\bf a})}$, ${\bf y}'\in\mathcal{S}_{{\bf b},\hh({\bf b})}$ and let
$$\mathcal{M}=\mathcal{M}^{I=2,n^*=m}_{\{\overline{J}_\tau^\Diamond(\varepsilon,\delta,{\frak p}(\tau))\}}({\bf y},{\bf y}';\overline{\frak m}), \quad\mathcal{M}_\tau= \mathcal{M}^{I=2,n^*=m}_{\overline{J}_\tau^\Diamond(\varepsilon,\delta,{\frak p}(\tau))}({\bf y},{\bf y}';\overline{\frak m}(\tau)).$$
We will analyze $\bdry_{\{+\infty\}}\mathcal{M}$.

Let $\overline{u}_i$, $i\in \N$, be a sequence of curves in $\mathcal{M}$ such that $\overline{u}_i\in\mathcal{M}_{\tau_i}$ and $\displaystyle\lim_{i\to\infty} \tau_i=+\infty$, and let
\begin{equation} \label{limit at plus infty}
\overline{u}_\infty = (\overline{v}_{-1,1}\cup\dots \cup \overline{v}_{-1,c}) \cup \overline{v}_-\cup (\overline{v}_{0,1}\cup\dots\cup \overline{v}_{0,b})\cup \overline{v}_+\cup (\overline{v}_{1,1}\cup\dots\cup \overline{v}_{1,a})
\end{equation}
be the limit holomorphic building, where each $\overline{v}_*$ is an SFT level (recall Notation~\ref{notation: sub superscripts} regarding the use of subscripts $*$), the levels are ordered from bottom to top as we go from left to right, $\overline{v}_{-1,j}$, $j=1,\dots,c$, maps to $\overline{W}$; $\overline{v}_-$ maps to $\overline{W}_-$; $\overline{v}_{0,j}$, $j=1,\dots,b$, maps to $\overline{W'}$; $\overline{v}_+$ maps to $\overline{W}_+$; and $\overline{v}_{1,j}$, $j=1,\dots,a$, maps to $\overline{W}$. Here we are allowing the possibility that $a$, $b$, or $c=0$.  For notational convenience, sometimes we will refer to $\overline{v}_+$ as $\overline{v}_{0,b+1}$ or $\overline{v}_{1,0}$ and $\overline{v}_-$ as $\overline{v}_{-1,c+1}$ or $\overline{v}_{0,0}$.

\begin{notation} \label{notation}
 We will be using the conventions established in Section~\ref{subsubsection: convention bambi} (and in particular Notation~\ref{notation: overline u}).
\begin{itemize}
\item We write $\dot F_*$, $\dot F_*'$, $\dot F_*''$ for the domains of $\overline{v}_*$, $\overline{v}_*'$, $\overline{v}_*''$.
\item We write $p_*$ for the covering degree of $\overline{v}'_*$.
\item If $\widetilde{v}_{*_1}$ is a union of components of a level $\overline{v}_{*_1}$ and $\widetilde{v}_{*_2}$ is a union of components of a possibly different level $\overline{v}_{*_2}$, then we write $\widetilde{v}_{*_1}\succ \widetilde{v}_{*_2}$ (resp.\ $\widetilde{v}_{*_1}\succeq \widetilde{v}_{*_2}$) to indicate that the level $\overline{v}_{*_1}$ is above (resp.\ equal to or above) the level $\overline{v}_{*_2}$.
\end{itemize}
\end{notation}

Since ghost components can be eliminated by the discussion in Lemma~I.\ref{P1-lemma: no ghosts} they will not be explicitly mentioned in the rest of the paper.

We have the following two constraints:
\begin{align}
\label{sum of n} n^*(\overline{u}_i) &=\sum_{\overline{v}_*} n^*(\overline{v}_*)=m;\\
\label{sum of I} I(\overline{u}_i) &=\sum_{\overline{v}_*} I(\overline{v}_*)=2,
\end{align}
where the summations are over all the levels $\overline{v}_*$ of $\overline{u}_\infty$.

\s\n
{\em Outline of proof of Lemma~\ref{cherries4}.}
The proof of Lemma~\ref{cherries4} follows the same general outline of Sections~I.\ref{P1-subsection: intersection numbers}--I.\ref{P1-proof of lemma}: First we calculate the contributions to $n^*$ of the ends that limit to multiples of $z_\infty$ or $\delta_0$ in Section~\ref{orange1} and obtain lower bounds on the ECH indices of the levels $\overline{v}_*$ in Section~\ref{bounds on ECH indices}, under the assumption that there are no boundary points at $z_\infty$.  Boundary points at $z_\infty$ are treated in Sections~\ref{boundpunc} and ~\ref{subsub: bounds on ech indices part II}. The main new difficulty is to show that $I(\overline{v}_*)\geq I(\overline{v}'_*)+I(\overline{v}''_*)$ for $\overline{v}_*\succeq \overline{v}_+$; this uses the more complicated version of the ECH index inequality given in Lemma~I.\ref{P1-index inequality for z infinity case}. We then use Equations~\eqref{sum of n} and \eqref{sum of I} to obtain Lemma~\ref{los angeles}, which describes the case when $\overline{v}'_*\cup \overline{v}^\sharp_*=\varnothing$ for all levels $\overline{v}_*$, and Lemma~\ref{hojicha}, which gives a preliminary list when $\overline{v}'_*\cup \overline{v}^\sharp_*\not=\varnothing$ for some $\overline{v}_*$.  The renormalization argument from Sections~I.\ref{P1-subsection: rescaling}--\ref{P1-subsection: theorem complement}, given in Lemma~\ref{vancouver}, eliminates all the possibilities with the exception of Case (2$_1$) of Lemma~\ref{hojicha} when $\overline{v}'_*\cup \overline{v}^\sharp_*\not=\varnothing$ for some $\overline{v}_*$.

\s
{\em At this point the reader is strongly encouraged to review Section~I.\ref{P1-subsection: modified indices at z infty} on holomorphic curves with ends at $z_\infty$.}

\subsubsection{Intersection numbers} \label{orange1}

In this subsection we give the analogs of Lemmas I.\ref{P1-intersezione 1}--I.\ref{P1-lemma: cherimoya2} for $\overline{u}_\infty\in \bdry_{\{+\infty\}} \mathcal{M}$:

\begin{lemma} \label{intersezione}
$\mbox{}$
\begin{enumerate}
\item If $\overline{v}''_*$ has a negative end $\mathcal{E}_-$ that converges to $\delta^p_0$, then $n^*(\mathcal{E}_-)\geq m-p.$
\item If $\overline{v}''_*$ has a positive end $\mathcal{E}_+$ that converges to $\delta^p_0$, then $n^*(\mathcal{E}_+)\geq p.$
\end{enumerate}
\end{lemma}

\begin{proof}
This is analogous to Lemma~I.\ref{P1-intersezione 1} and is proved in the same way.
\end{proof}

\begin{defn}
An end $\mathcal{E}$ is {\em nontrivial} if it is not an end of a trivial cylinder or trivial strip.
\end{defn}

\n
{\em Convention.}  In this paper we assume that an end of a holomorphic curve is connected, unless stated otherwise.

\s
Recall the sequence $\overline{u}_i\in \mathcal{M}_{\tau_i}$ with $\tau_i\to +\infty$ and its limit SFT building $\overline{u}_\infty$ given by Equation~\eqref{limit at plus infty}. Let $\dot G_i$ be the domain of $\overline{u}_i$. Fix $k\in \{1,\dots,2g\}$. The following Lemma~\ref{fact: continuation} describes the breaking/``SFT limit'' of the component of $\overline{u}_i|_{\bdry \dot G_i}$ which maps to $L_{\overline{a}_k}^{\tau_i,+}$, into a sequence of paths $g^1_a,\dots,g^1_1,g_0,g^0_1, \dots, g^0_a$, $g^*_j\subset \op{Im}(\overline{v}_{1,j})$, $j\geq 0$, $*=0,1,\varnothing$, as $\tau_i\to \infty$.

Let $\bdry_+ B_\tau$ (resp.\ $\bdry_+ \overline{W}_\tau$) be the $s>0$ boundary of $B_\tau$ (resp.\ $\overline{W}_\tau$) and let $\mathcal{C}_{1,j}$, $j\geq 0$, be the data for $\overline{v}'_{1,j}$; see Section~I.\ref{P1-subsubsection: multisections} for the definition of the data $\mathcal{C}_{1,j}$. Then we have the following, whose proof is immediate.

\begin{lemma}\label{fact: continuation}
Given a sequence $\overline{u}_i$ and a choice of $k\in \{1,\dots,2g\}$, the components of $\overline{u}_i|_{\bdry \dot G_i}$ which map to $L_{\overline{a}_k}^{\tau_i,+}$ converge uniquely (in the ``SFT sense'') to a sequence of paths
\begin{equation}\label{continuation}
g^1_a,\dots,g^1_1,g_0,g^0_1, \dots, g^0_a
\end{equation}
which satisfies the following:
\begin{enumerate}
\item $g^*_j=\overline{v}_{1,j}(f^*_j)$ where $j\geq 0$, $*=0,1,\varnothing$, and $f^*_j$ is a (connected) component of $\bdry \dot F_{1,j}$, where $\dot F_*$ is as given in Notation~\ref{notation}.
\item If $f^*_j$ is a component of $\bdry\dot F''_{1,j}$, then $g^1_j\subset \R\times\{1\}\times \overline{a}_k$ and $g^0_j\subset \R\times\{0\}\times \overline{\hh}(\overline{a}_k)$ if $j>0$ and $g_0\subset L^{+}_{\overline{a}_k}$ if $j=0$.
\item If $f^*_j$ is a component of $\bdry\dot F'_{1,j}$, then $f^*_j$ {\em and also $g^*_j$} come with extra data $\mathcal{C}_{1,j}$ which assigns:
$$f^1_j, g^1_j\mapsto L_{\overline{a}_{k,l}}= \R\times\{1\}\times \overline{a}_{k,l},\quad f^0_j,g^0_j\mapsto L_{\overline{\hh}(\overline{a}_{k,l'})}=\R\times\{0\}\times \overline{\hh}(\overline{a}_{k,l'}),$$
$$f_0,g_0\mapsto L^+_{\overline{a}_{k,l''}},$$
for some $l,l',l''=0$ or $1$.
\end{enumerate}
\end{lemma}

For convenience we write $g_0=g^1_0=g^0_0$.  We say an element $g^*_j$ is {\em trivial} if the corresponding $f^*_j$ satisfies (3).

\begin{defn}[Continuations] \label{defn: continuation}
If $g^*_j$ is any element of Sequence \eqref{continuation}, then Sequence \eqref{continuation} is the {\em continuation of $g^*_j$ along $\bdry_+ B_\tau$} and the terms to the right of $g^*_j$ in Sequence \eqref{continuation} form the {\em continuation of $g^*_j$ in the direction of $\bdry_+B_\tau$.}
\end{defn}

Lemma~\ref{fact: continuation} and Definition~\ref{defn: continuation} play an important role in the proof of Lemma \ref{intersezione prime}, which is a refined version of Lemma~I.\ref{P1-intersezione 2}, where the ends are considered collectively as well as individually. The proof strategy will usually be referred to as the {\em continuation argument}.

Let $\rho_0>0$ be small and let $\pi_{D^2_{\rho_0}}$ be the projection of
\begin{equation} \label{trader joes}
\mathfrak{D}_1:=\{\rho\leq \rho_0\}\subset \overline{W}-int(W) \ \mbox{ or } \ \mathfrak{D}_2:=\{\rho\leq \rho_0, |s|\geq l(\tau)+1\}\subset \overline{W}_\tau-int(W_\tau)
\end{equation}
to $D^2_{\rho_0}=\{(\rho,\phi)~|~\rho\leq \rho_0\}\subset \overline{S}$ along the stable Hamiltonian vector field $\overline{R}_\tau$ which was defined in Section~\ref{subsubsection: stable Hamiltonian}.

\begin{rmk} \label{rmk: tjs}
$\pi_{D^2_{\rho_0}}$ maps the intersection of $\R\times \delta_{\rho_0,\phi_0^+}$ (which appears in the definition of $n^*$ in Equation~\eqref{n and n alt}) and one of ${\frak D}_1$ or $\frak D_2$ to $m$ equally spaced points on $\bdry D^2_{\rho_0}$.
\end{rmk}

\begin{lemma} \label{intersezione prime}
Suppose $\overline{v}'_{1,j}\cup\overline{v}^\sharp_{1,j}\not=\varnothing$ for some $j>0$. Let $\mathcal{E}_{-,i}$, $i=1,\dots,q$, be the negative ends of $\cup_{j=1}^a\overline{v}_{1,j}^\sharp$ that converge to $z_\infty$ and let $\mathcal{E}_{+,i}$, $i=1,\dots,r$, be the positive ends of $\cup_{j=0}^{a-1}\overline{v}_{1,j}^\sharp$ that converge to $z_\infty$.
\begin{enumerate}
\item For each $i$,
\begin{equation} \label{jasmine}
n^*(\mathcal{E}_{-,i})\geq k_0-1\gg 2g,
\end{equation}
where the constant $k_0$ is as given in Section~I.\ref{P1-coconut}.
\item If there are no boundary points at $z_\infty$, then
\begin{equation} \label{oolong}
n^*( (\cup_{i=1}^q\mathcal{E}_{-,i})\cup (\cup_{i=1}^r\mathcal{E}_{+,i}))\geq m-p_+,
\end{equation}
where $p_+=\deg (\overline{v}_+')$, and $$D^2_{\rho_0}-(\cup_{i=1}^q \pi_{D^2_{\rho_0}}(\mathcal{E}_{-,i}))\cup (\cup_{i=1}^r \pi_{D^2_{\rho_0}}(\mathcal{E}_{+,i}))$$ consists of at most $p_+$ thin sectors.
\end{enumerate}
\end{lemma}

Recall that a sector ${\frak S}$ of $D^2_{\rho_0}$ from $\phi_0$ to $\phi_1$ is the map
$$[0,\rho_0]\times[\phi_0,\phi_1]\to D^2_{\rho_0},\quad (\rho,\phi)\mapsto \rho e^{i\phi},$$
or its image (by abuse of notation).  A {\em thin sector} (also referred to as a {\em thin wedge} in \cite{CGH-I}) is the smallest counterclockwise sector from $\overline{a}_{i,j}\cap D^2_{\rho_0}$ to $\overline{\hh}(\overline{a}_{i,j})\cap D^2_{\rho_0}$ for some $i,j$ and has angle ${2\pi\over m}$.

We also refer the reader to Definition~I.\ref{P1-defn: removable at infty} for the definition of a boundary point at $z_\infty$.

\begin{proof}
We consider the $\pi_{D^2_{\rho_0}}$-projections of the positive and negative ends of $\overline{v}_{1,j}$, $j>0$, and the positive ends of $\overline{v}_{1,0}$ that limit to $z_\infty$.

(1) Given a nontrivial negative end $\mathcal{E}_{-,i}$ limiting to $z_\infty$, its projection can be written as:
$$\pi_{D^2_{\rho_0}}(\mathcal{E}_{-,i})={\frak S}(\overline{\hh}(\overline{a}_{k,l}),\overline{a}_{k',l'}),$$
where ${\frak S}(A,B)$ denotes the smallest counterclockwise sector from the radial ray $A$ to the radial ray $B$.  Then Equation~\eqref{jasmine} follows from Remark~\ref{rmk: tjs} and the fact that the angle between $\overline{\hh}(\overline{a}_{k,l})$ and $\overline{a}_{k',l'}$ is at least ${2\pi(k_0-1)\over m}$ as defined in Section~I.\ref{P1-coconut}. (This is more or less the same as Lemma~I.\ref{P1-intersezione 2}(2); in fact ``$>2g$'' in the statement can be replaced by ``$\geq k_0-1\gg 2g$''.)

(2) We start at a nontrivial negative end of some $\overline{v}_{1,j_1}$ limiting to $z_\infty$, which we call $\mathcal{E}_{-,1}$ without loss of generality. Then
$$\pi_{D^2_{\rho_0}}(\mathcal{E}_{-,1})={\frak S}(\overline{\hh}(\overline{a}_{k_1,l_1}),\overline{a}_{k_2,l_2}).$$

We analyze the continuation
\begin{equation} \label{continuation 1}
g^1_{j_1-1,1},\dots,g^1_{1,1},g_{0,1},g^0_{1,1},\dots,g^0_{a,1}
\end{equation}
of $g^1_{j_1,1}\supset \bdry_1\mathcal{E}_{-,1}$ in the direction of $\bdry_+ B_\tau$. Here $\bdry_k \mathcal{E}_{-,i}$, $k=0,1$, is the $t=k$ boundary of $\mathcal{E}_{-,i}$, and $f^{*_1}_{*_2}$ corresponds to $g^{*_1}_{*_2}$ as in Definition~\ref{defn: continuation}.

\s\n
(i) Suppose there is some $0\leq j<j_1$ such that $g^1_{j,1}$ is nontrivial. Let $j_2\geq 0$ be the first such occurrence in the continuation. Then $\overline{v}_{1,j_2}^\sharp$ has some nontrivial end which we call $\mathcal{E}_{+,1}$, such that
$$\pi_{D^2_{\rho_0}}(\mathcal{E}_{+,1})={\frak S}(\overline{a}_{k_2,l_2},\overline{\hh}(\overline{a}_{k_3,l_3})).$$
This is because all the terms of Equation~\eqref{continuation 1} are assigned $L_{\overline{a}_{k_2}}$ or $L^+_{\overline{a}_{k_2}}$ and all the terms between $g^1_{j_1,1}$ and $g^1_{j_2,1}$ correspond to the same $L_{\overline{a}_{k_2,l_2}}$.

\s\n
(ii) On the other hand, if $g^1_{j,1}$ is trivial for all $0\leq j\leq j_1$, then we set $j_2=0$ and $\overline{\hh}(\overline{a}_{k_2,l_2})=\overline{\hh}(\overline{a}_{k_3,l_3})$ and skip ${\frak S}(\overline{a}_{k_2,l_2},\overline{\hh}(\overline{a}_{k_3,l_3}))$, which is a thin sector.

\s
Next we consider the continuation
$$ g^0_{j_2+1,2},\dots, g^0_{a,2}$$
of $g^0_{j_2,2}$ in the direction of $\bdry_+ B_\tau$. Here $g^0_{j_2,2}\supset\bdry_0 \mathcal{E}_{+,1}$ if (i) holds; and $j_2=0$, $g_{0,2}=g_{0,1}$, and $f_{0,2}=f_{0,1}$ if (ii) holds. There must exist some nontrivial $g^0_{j,2}$, $j>j_2$, and we write $j_3$ for the first such occurrence in the continuation. Then $\overline{v}_{1,j_3}^\sharp$ has some nontrivial end which we call $\mathcal{E}_{-,2}$, such that
$$\pi_{D^2_{\rho_0}}(\mathcal{E}_{-,2})={\frak S}(\overline{\hh}(\overline{a}_{k_3,l_3}),\overline{a}_{k_4,l_4}).$$

Continuing in the same manner, we eventually return to $\mathcal{E}_{-,1}$, and the sectors
\begin{equation} \label{farmers mkt}
{\frak S}(\overline{\hh}(\overline{a}_{k_1,l_1}),\overline{a}_{k_2,l_2}), {\frak S}(\overline{a}_{k_2,l_2},\overline{\hh}(\overline{a}_{k_3,l_3})), {\frak S}(\overline{\hh}(\overline{a}_{k_3,l_3}),\overline{a}_{k_4,l_4}),\dots,
\end{equation}
with some thin sectors of type ${\frak S}(\overline{a}_{k_{2i},l_{2i}},\overline{\hh}(\overline{a}_{k_{2i+1},l_{2i+1}}))$ omitted if they correspond to (ii), sweep out a neighborhood of $z_\infty$ in $D^2_{\rho_0}$, possibly more than once and with the possible exception of $p_+$ thin sectors. This proves (2).
\end{proof}

We remark that, a posteriori, the sectors from Equation~\eqref{farmers mkt} sweep out a neighborhood of $z_\infty$ only once in view of Equation~\eqref{oolong}.

\begin{lemma} \label{cherimoya}
If ${\frak p}\in \bdry F''_*$ is a boundary point of $\overline{v}''_*$ at $z_\infty$, then $n^*(\overline{v}''_*(N({\frak p})))\geq k_0-1\gg 2g$, where $N({\frak p})\subset F''_*$ is a sufficiently small neighborhood of ${\frak p}$.
\end{lemma}

\begin{proof}
The proof is the same as that of Lemma~I.\ref{P1-lemma: cherimoya2}.
\end{proof}

\subsubsection{Bounds on ECH indices} \label{bounds on ECH indices}

The goal of this subsection is to show the nonnegativity of $I(\overline{v}_*)$ except when $\overline{v}_*=\overline{v}_+$, {\em under the assumption that there are no boundary points at $z_\infty$}; see Lemma~\ref{nonnegative ECH indices}. The main new difficulty is to show that $I(\overline{v}_*)\geq 0$ for $\overline{v}_*\succ \overline{v}_+$. If $\overline{v}_*\succ \overline{v}_+$ and $\overline{v}_*'\cup\overline{v}_*^\sharp\not=\varnothing$, then we need to apply the version of the ECH index inequality in the presence of ends that limit to $z_\infty$ (Lemma~I.\ref{P1-index inequality for z infinity case}).  To apply Lemma~I.\ref{P1-index inequality for z infinity case}, we need to verify a certain ``alternating property'' for the ends of $\overline{v}_*$ that limit to $z_\infty$; this is the content of Lemma~\ref{alternate}.

Let $A_\varepsilon=\bdry D^2_\varepsilon\times[0,1]$ for $0<\varepsilon<\rho_0$ small and let $\pi_{[0,1]\times\overline{S}}$ be the projection of $\overline{W}$ or the positive end of $\overline{W}_\tau$ to $[0,1]\times\overline{S}$.
By the proof of Lemma~\ref{intersezione prime}, the intersection ${\frak c}:= \pi_{[0,1]\times\overline{S}}(\cup_i \mathcal{E}_{-,i})\cap A_\varepsilon$ is groomed and the sets $P_0$ and $P_1$ of initial and terminal points of ${\frak c}$ alternate along $(0,2\pi)$.

\begin{defn}
A {\em cycle $\mathcal{Z}=({\frak z}_1\to {\frak z}_2\to \dots\to {\frak z}_k\to {\frak z}_1)$ on $\bdry D^2_\varepsilon=\R/2\pi\Z$} is a sequence of points ${\frak z}_i\in \R/2\pi\Z$, together with chords in $\R/2\pi\Z$ from ${\frak z}_i$ to ${\frak z}_{i+1}$, where the indices are taken modulo $k$.
\end{defn}

The continuation method from Lemma~\ref{intersezione prime} gives a cycle
$$\mathcal{Z}=({\frak z}_{10}\to {\frak z}_{11}\to \dots\to {\frak z}_{k0}\to {\frak z}_{k1}\to{\frak z}_{10}),$$
where $P_i=\{{\frak z}_{1i},\dots,{\frak z}_{ki}\}$ for $i=0,1$, the chords correspond to the sectors, the cycle winds around $\R/2\pi\Z$ once, and each point of $P_i$ appears only once in $\mathcal{Z}$. Note that some of the ${\frak z}_{i0}\to {\frak z}_{i1}$ may correspond to sectors $\frak{S}(\overline{a}_{k',l'},\overline{\hh}(\overline{a}_{k',l'}))$ of the type that are skipped in Step (ii) of Lemma~\ref{intersezione prime}.

Next let $\ar{\mathcal{D}}_{\pm,j}$ be the data at $z_\infty$ for the $\pm$ end of $\overline{v}_{1,j}$  and let $P_{\pm,j,0}$ and $P_{\pm,j,1}$ be the initial and terminal points on $A_\varepsilon$ determined by $\ar{\mathcal{D}}_{\pm,j}$. Then we write
\begin{equation} \label{athena}
P_{\pm,j,i}=P'_{\pm,j,i}\sqcup P''_{\pm,j,i},
\end{equation}
where $P'_{\pm,j,i}$ corresponds to $\overline{v}'_{1,j}$ and $P''_{\pm,j,i}$ corresponds to $\overline{v}''_{1,j}$. Note that
\begin{equation} \label{iris}
P_{+,j-1,i}=P_{-,j,i} \ \mbox{ and } \ P'_{+,j-1,i}=P'_{-,j-1,i}.
\end{equation}

\begin{defn}
Let $P$ be a finite subset of $S^1=\R/2\pi \Z$.  Denoting an element of $S^1$ by an equivalence class $[c]$, where $c\in \R$, if $[a]\not=[b]\in P$, then we write $[a]\lessdot_P[b]$ if there exist $a'\in[a]$, $b'\in[b]$ such that $a'< b'< a'+2\pi$ and there are no representatives of $P$ in the open interval $(a',b')$.
\end{defn}

Observe that the relation $\lessdot_P$ is not symmetric, i.e., $[a]\lessdot_P[b]$ does not necessarily imply that $[b]\lessdot_P[a]$.

\begin{lemma} \label{alternate}
For each $*=(\pm,j)$, $P_{*,0}\subset P_0$ and $P_{*,1}\subset P_1$ and the points of $P_{*,0}$ and $P_{*,1}$ alternate along $(0,2\pi)$.
\end{lemma}

\begin{proof}
The proof is by induction on the level; see Figure~\ref{figure: alternating} for an example.  We will inductively define $P_i^{(0)}\supset P_i^{(1)}\supset P_i^{(2)}\supset \dots$ and the corresponding cycles $\mathcal{Z}^{(0)},\mathcal{Z}^{(1)},\dots$ and show that the following hold for each $j=0,1,\dots$:
\begin{enumerate}
\item[($j0$)] the points of $P_0^{(j)}$ and $P_1^{(j)}$ alternate along $(0,2\pi)$;
\item[($j1$)] $P_{-,j_0-j,i}=P_{+,j_0-j-1,i}\subset P_i^{(j)}$, $i=0,1$;
\item[($j2$)] there is a partition of $P_{-,j_0-j,0}\cup P_{-,j_0-j,1}$ into pairs of type $\{{\frak p}_0,{\frak p}_1\}$, ${\frak p}_i\in P_{-,j_0-j,i}$, such that ${\frak p}_0\lessdot_{P_0^{(j)}\cup P_1^{(j)}}{\frak p}_1;$ in particular, the points of $P_{-,j_0-j,0}$ and $P_{-,j_0-j,1}$ alternate along $(0,2\pi)$;
\item[($j3$)] $P'_{+,j_0-j-1,i}=P'_{-,j_0-j-1,i}\subset P_i^{(j+1)}$, $i=0,1$;
\item[($j4$)] there is a partition of $P'_{+,j_0-j-1,0}\cup P'_{+,j_0-j-1,1}$ into pairs of type $\{{\frak p}_0,{\frak p}_1\}$, ${\frak p}_i\in P'_{+,j_0-j-1,i}$, such that ${\frak p}_0\lessdot_{P_0^{(j+1)}\cup P_1^{(j+1)}}{\frak p}_1;$ in particular, the points of $P'_{+,j_0-j-1,0}$ and $P'_{+,j_0-j-1,1}$ alternate along $(0,2\pi)$.
\end{enumerate}
Note that $P_{-,j_0-j,i}=P_{+,j_0-j-1,i}$ in ($j1$) and $P'_{+,j_0-j-1,i}=P'_{-,j_0-j-1,i}$ in ($j3$) by Equation~\eqref{iris}.

\begin{figure}[ht]
\begin{center}
\psfragscanon
\psfrag{a}{\tiny $P_i^{(0)}$}
\psfrag{b}{\tiny $P_{-,j_0,i}$}
\psfrag{c}{\tiny $P''_{-,j_0,i}$}
\psfrag{d}{\tiny $P_i^{(0)}$}
\psfrag{e}{\tiny $P_{+,j_0-1,i}$}
\psfrag{f}{\tiny $P''_{+,j_0-1,i}$}
\psfrag{g}{\tiny $P_i^{(1)}$}
\psfrag{h}{\tiny $P_{-,j_0-1,i}$}
\psfrag{i}{\tiny $P''_{-,j_0-1,i}$}
\includegraphics[width=13cm]{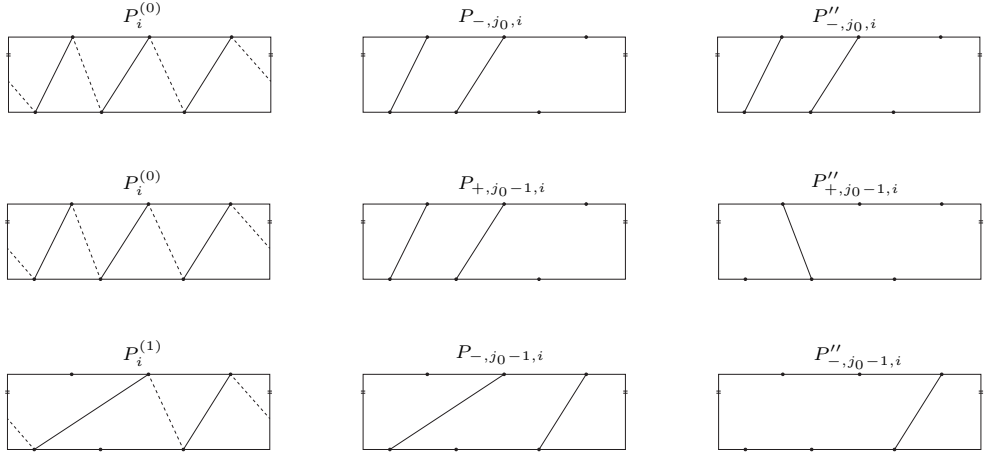}
\end{center}
\caption{Each rectangular box, with the sides identified, is $A_\varepsilon$; the top is $\bdry D^2_\varepsilon\times\{1\}$ and the bottom is $\bdry D^2_\varepsilon\times\{0\}$. For each figure, $P^\star_{*,0}$ and $P^\star_{*,0}$ are the sets of initial and terminal points of the arcs drawn, where the arcs are always oriented from bottom to top.  The extra dots are drawn only for reference. The $k$th row corresponds to Step $k$. The solid and dashed arcs in the left column together indicate $\mathcal{Z}^{(0)}$ and $\mathcal{Z}^{(1)}$. The dashed arcs are potential arcs whose endpoints give pairs of points in $P''_{+,j,0}$ and $P''_{+,j,1}$.}
\label{figure: alternating}
\end{figure}

\s\n {\em Step 1.} Let $\overline{v}_{1,j_0}$ be the highest level which has a negative end at $z_\infty$, let $P_i^{(0)}=P_i$, $i=0,1$, and let $\mathcal{Z}^{(0)}=\mathcal{Z}$. The construction of $\mathcal{Z}^{(0)}$ (following the proof of Lemma~\ref{intersezione prime}) immediately implies ($00$)--($02$).

\s\n {\em Step 2.} We consider the positive ends of $\overline{v}_{1,j_0-1}$ that limit to $z_\infty$. There is a partition of $P_{+,j_0-1,0}''\cup P_{+,j_0-1,1}''$ into pairs  of type $\{{\frak p}_0,{\frak p}_1\}$, ${\frak p}_i\in P_{+,j_0-1,i}''$, such that ${\frak p}_1\lessdot_{P_0^{(0)}\cup P_1^{(0)}}{\frak p}_0$, i.e., there is a chord ${\frak p}_1\to {\frak p}_0$ in $\mathcal{Z}^{(0)}$. Let $P_i^{(1)}=P_i^{(0)}-P_{+,j_0-1,i}''$ and let $\mathcal{Z}^{(1)}$ be obtained from $\mathcal{Z}^{(0)}$ by inductively replacing ${\frak q} \to {\frak p}_1\to {\frak p}_0\to {\frak q}'$ by ${\frak q}\to {\frak q}'$, given by concatenation. This makes sense since each point of $P_i^{(0)}$ appears only once in $\mathcal{Z}^{(0)}$. Then $P_0^{(1)}\cup P_1^{(1)}$ is the set of (alternating) points of $\mathcal{Z}^{(1)}$, each point of $P_i^{(1)}$ appears only once in $\mathcal{Z}^{(1)}$, and $\mathcal{Z}^{(1)}$ winds around $\R/2\pi\Z$ once.  Hence ($10$) follows from the description of $\mathcal{Z}^{(1)}$. Similarly, since $P'_{+,j_0-1,i}=P_{+,j_0-1,i}-P_{+,j_0-1,i}''$ and ($02$) holds, ($03$) and ($04$) follow immediately.

\s\n {\em Step 3.} We consider the negative ends of $\overline{v}_{1,j_0-1}$ that limit to $z_\infty$. There is a partition of $P''_{-,j_0-1,0}\cup P''_{-,j_0-1,1}$ into pairs of type $\{{\frak p}_0,{\frak p}_1\}$, ${\frak p}_i\in P''_{-,j_0-1,i}$, such that ${\frak p}_0\lessdot_{P_0^{(1)}\cup P_1^{(1)}}{\frak p}_1$, i.e., there is a chord ${\frak p}_0\to{\frak p}_1$ in $\mathcal{Z}^{(1)}$.  Since $P_{-,j_0-1,i}=P'_{-,j_0-1,i}\cup P''_{-,j_0-1,i}$ and ($04$) holds, ($11$) and ($12$) follow immediately.

\s
Repeated application of the above then implies the lemma.
\end{proof}

\begin{lemma} \label{nonnegative ECH indices}
If fiber components are removed from $\overline{u}_\infty$ and there are no boundary points at $z_\infty$, then:
\begin{enumerate}
\item  the only components of $\overline{v}_*'$ with negative ECH index are the branched covers of $\sigma_\infty^+$;
\item the ECH index of each level $\overline{v}_*\not= \overline{v}_+$ is nonnegative; and
\item $I(\overline{v}_{1,j})\geq I(\overline{v}'_{1,j}) +I(\overline{v}''_{1,j})$ for $0\leq j\leq a$, with equality for $j>0$.
\end{enumerate}
\end{lemma}

\begin{proof}
The proof is analogous to that of Lemma~I.\ref{P1-claim in proof}.

(1) If a component $\widetilde v$ of $\overline{u}_\infty$ is a branched cover of some $\sigma_\infty^*$ with possibly empty branch locus, then, by Lemmas~I.\ref{P1-lemma: HF index of sections at infinity} and I.\ref{P1-lemma: ECH index of sections at infinity}, $I(\widetilde{v})=0$ with the exception of $I(\widetilde{v})=-\deg(\widetilde{v})$ if $\widetilde{v}$ covers $\sigma_\infty^+$. For all other components $\widetilde v$ (i.e., those that do not branch cover any $\sigma_\infty^*$), the regularity of $\{\overline{J}_\tau\}$ and the ECH-type index inequalities imply that $I(\widetilde v)\geq 0$.

(2) for $\overline{v}_{0,j}$, $1\leq j\leq b$. By \cite[Proposition 7.15(a)]{HT1} and the regularity of $\{\overline{J}_\tau\}$, we have $I(\overline{v}_{0,j})\geq 0$ for $1\leq j\leq b$, where equality holds if and only if $\overline{v}_{0,j}$ is a connector.

(2) and (3) for $\overline{v}_{1,j}$. We claim that $I(\overline{v}_{1,j})\geq 0$ for $0< j \leq a$ and $I(\overline{v}_{1,j})\geq I(\overline{v}'_{1,j}) +I(\overline{v}''_{1,j})$ for $0\leq j\leq a$, with equality for $j>0$.

Suppose $z_\infty$ does not appear at an end of $\overline{v}_{1,j}$. With the possible exception of fiber components, $\overline{v}_{1,j}$ is simply-covered and regular since there is at least one HF end.  Hence the claim follows from the regularity of $\{\overline{J}_\tau\}$ and the usual index inequality (Lemmas~I.\ref{P1-thm: index inequality for HF} and I.\ref{P1-thm: index inequality for W+ and W-}).

On the other hand, if $z_\infty$ appears at an end of $\overline{v}_{1,j}$, then we use the version of the index inequality given by Lemma~I.\ref{P1-index inequality for z infinity case}. Recall that $\mathcal{E}_{-,i}$, $i=1,\dots,q$, are the negative ends of $\cup_{j=1}^a\overline{v}_{1,j}^\sharp$ that converge to $z_\infty$ and $\mathcal{E}_{+,i}$, $i=1,\dots,r$, are the positive ends of $\cup_{j=0}^{a-1}\overline{v}_{1,j}^\sharp$ that converge to $z_\infty$. Let $I_{-,j}$ be the subset of $\{1,\dots,q\}$ such that $\mathcal{E}_{-,i}$ is a negative end of $\overline{v}^\sharp_{1,j}$ if and only if $i\in I_{-,j}$, and let $I_{+,j}$ be the subset of $\{1,\dots,r\}$ such that $\mathcal{E}_{+,i}$ is a positive end of $\overline{v}^\sharp_{1,j}$ if and only if $i\in I_{+,j}$. Then let ${\frak c}_{\pm,j}$ (resp.\ ${\frak c}'_{\pm,j}$, ${\frak c}''_{\pm,j}$) be the groomings corresponding to the $\pm$ ends of $\overline{v}_{1,j}$ (resp.\ $\overline{v}'_{1,j}$, $\overline{v}''_{1,j}$) at $z_\infty$, subject to the following:
\begin{enumerate}
\item[($\alpha_1$)] ${\frak c}_{\pm,j}={\frak c}'_{\pm,j}\cup {\frak c}''_{\pm,j}$;
\item[($\alpha_2$)] ${\frak c}_{-,j}$ has winding $w({\frak c}_{-,j})\geq 0$ and ${\frak c}_{+,j}$ has winding $w({\frak c}_{+,j})\leq 0$;
\item[($\alpha_3$)] ${\frak c}''_{-,j}=\pi_{[0,1]\times\overline{S}}(\cup_{k\in I_{-,j}} \mathcal{E}_{-,k})\cap A_\varepsilon$ and ${\frak c}''_{+,j}= \pi_{[0,1]\times\overline{S}} (\cup_{k\in I_{+,j}} \mathcal{E}_{+,k})\cap A_\varepsilon$.
\end{enumerate}
By Lemma~\ref{alternate}, the sets $P_{\pm,j,0}$ and $P_{\pm,j,1}$ of initial and terminal points of ${\frak c}_{\pm,j}$ alternate along $(0,2\pi)$. Hence, by Lemma~I.\ref{P1-index inequality for z infinity case} and Equation~(I.\ref{P1-second eq}),
$$I(\overline{v}_{1,j})\geq I(\overline{v}'_{1,j})+I(\overline{v}''_{1,j})\geq \op{ind}(\overline{v}''_{1,j})\geq 0$$
with equality for $j>0$  since $\overline{v}'_{1,j}$ and $\overline{v}''_{1,j}$ do not intersect (if they did, then $\sum_{\overline{v}_*} n^*(\overline{v_*})>m$) and we can calculate that $d^+=d^-=0$. This proves the claim.

(2) and (3) for $\overline{v}_{-1,j}$ and $\overline{v}_-$. The case of $\overline{v}_{-1,j}$, $1\leq j\leq c+1$, (this includes $\overline{v}_-$) is similar, and the lemma follows.
\end{proof}

\subsubsection{Boundary points at $z_\infty$} \label{boundpunc}

In this subsection we describe the necessary modifications when $\overline{u}_\infty$ has boundary points at $z_\infty$.

Let us suppose the following:
\begin{enumerate}
\item[(S)] There is only one boundary point ${\frak r}\in F''_{*_0}$ at $z_\infty$ and $\overline{v}''_{*_0}({\frak r})\in \{(0,1)\}\times \overline{\bf a}\subset \overline{W}= \R\times [0,1]\times \overline{S}$.
\end{enumerate}
 We are assuming (S) for notational simplicity; the general case of multiple boundary points at $z_\infty$ is treated in exactly the same way (except for more complicated notation).

For the purposes of computing the Fredholm and ECH indices, we make the following modifications which allow us to use the considerations from Sections~\ref{orange1}, \ref{bounds on ECH indices}, and I.\ref{P1-subsection: modified indices at z infty}. We view the base $B=\R\times[0,1]$ with the puncture $(0,1)$ as a two-level building consisting of a disk $B^\circ_1$ with three boundary punctures and a disk $B^{\circ}_2$ with one boundary puncture. (What we are doing here is bubbling off a neighborhood of the puncture $(0,1)\in B$ when taking the limit $\overline{u}_i\to\overline{u}_\infty$.) All the punctures are viewed as strip-like ends: the punctures of $B^\circ_1$ are called the positive, left negative, and right negative ends, corresponding to the positive end, $(0,1)$, and the negative end of $B$, and the puncture of $B^\circ_2$ is called the positive end and is identified with the left negative end of $B^\circ_1$. See Figure~\ref{figure: puncture}. We write $\bdry B^\circ_1=\sqcup_{i=1}^3 \bdry_i B^\circ_1$, arranged in counterclockwise order, such that $\bdry_1B^\circ_1$, $\bdry_2B^\circ_1$, and $\bdry B^\circ_2$ correspond to $\R\times\{1\}$ and $\bdry_3 B^\circ_1$ corresponds to $\R\times\{0\}$. The cobordism $\overline{W}=B\times\overline{S}$ decomposes into $\overline{W}^\circ_1=B^\circ _1\times\overline{S}$ and $\overline{W}^\circ_2=B^\circ _2\times\overline{S}$, the Lagrangian submanifold $\R\times\{1\}\times\overline{\bf a}$ decomposes into $\bdry_1 B^\circ_1\times\overline{\bf a}$, $\bdry B^\circ_2\times\overline{\bf a}$, and $\bdry_2B^\circ_1\times\overline{\bf a}$, and the Lagrangian submanifold $\R\times\{0\}\times\overline{\hh}(\overline{\bf a})$ corresponds to $\bdry_3 B^\circ_1\times\overline{\hh}(\overline{\bf a})$.  We identify the positive end of $\overline{W}^\circ_1$ with $[0,\infty)\times[0,1]\times\overline{S}$, the left negative end with $(-\infty,0]\times[{1\over 2},1]\times \overline{S}$, and the right negative end with $(-\infty,0]\times[0,{1\over 2}]\times \overline{S}$.

\begin{figure}[ht]
\begin{center}
\psfragscanon
\psfrag{a}{\tiny $\overline{\bf a}$}
\psfrag{b}{\tiny $\overline{\hh}(\overline{\bf a})$}
\psfrag{c}{\tiny $B^\circ_1$}
\psfrag{d}{\tiny $B^\circ_2$}
\includegraphics[width=6cm]{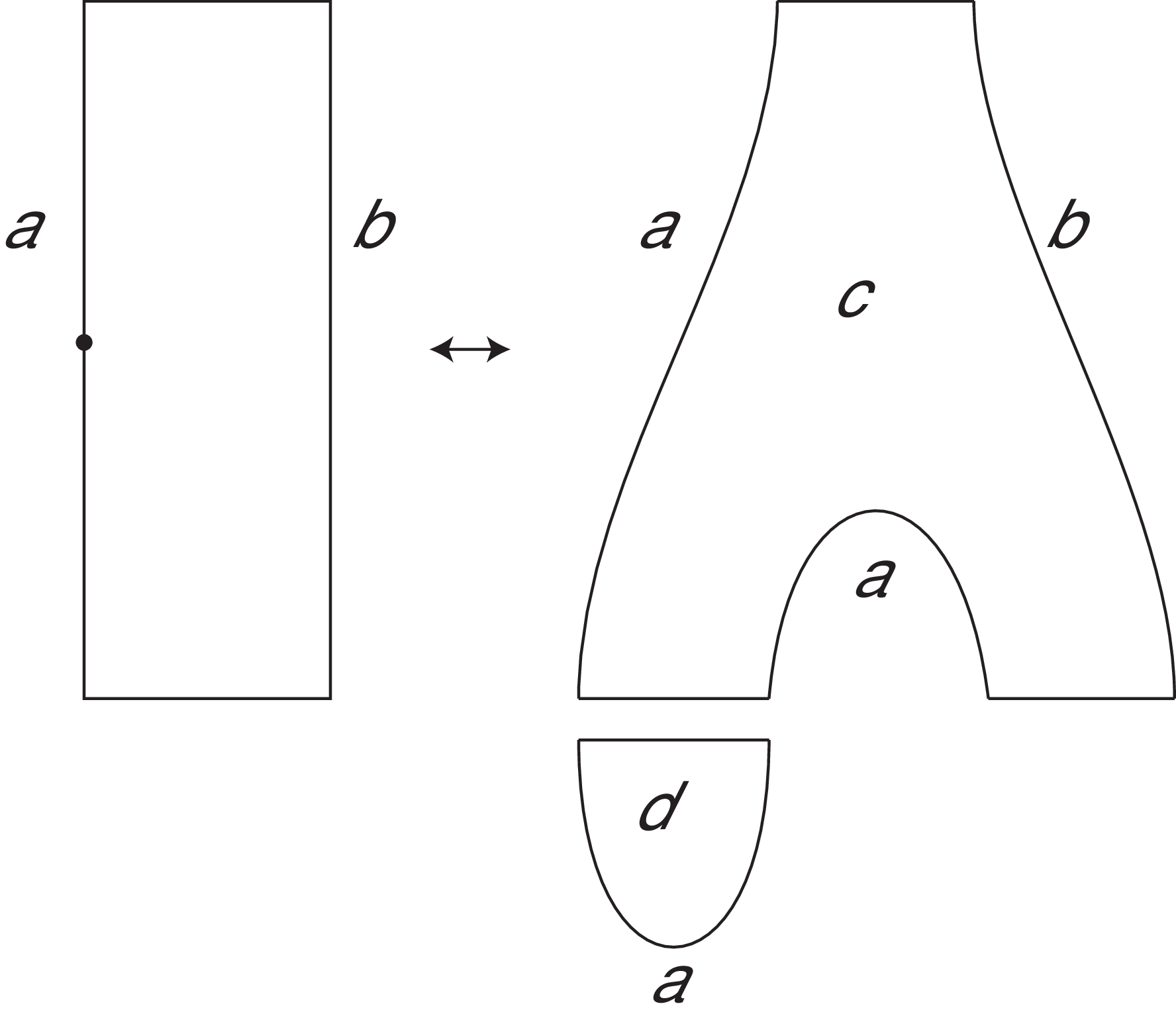}
\end{center}
\caption{}
\label{figure: puncture}
\end{figure}

Denote the sections at infinity of $\overline{W}^\circ_i$, $i=1,2$, by $\sigma_\infty^{\circ,i}= B^\circ_i\times\{z_\infty\}$. A curve $\overline{v}_{*_0}$ in $\overline{W}$ decomposes into $\overline{v}_{*_0}^{\circ,1}\cup \overline{v}_{*_0}^{\circ,2}$, where $\overline{v}_{*_0}^{\circ,i}=\overline{v}_{*_0}^{\circ,i,'}\cup \overline{v}_{*_0}^{\circ,i,''}$, $i=1,2$, is a curve in $\overline{W}^\circ_i$, $\overline{v}_{*_0}^{\circ,i,'}$ is a possibly branched cover of $\sigma_\infty^{\circ,i}$, and $\overline{v}_{*_0}^{\circ,2,''}$ is a union of constant sections $B^\circ_1\times\{pt\}$, $pt\in \widehat{\bf a}$.
\begin{enumerate}
\item[(R)] We view the chords at the positive ends of $\overline{v}_{*_0}^{\circ,2}$ and the left negative ends of $\overline{v}_{*_0}^{\circ,1}$, including $z_\infty$, as negative Morse-Bott Reeb chords. This means that, when calculating Maslov indices at ends that limit to $z_\infty$, we pretend that we have slightly rotated $\{1\}\times \overline{\bf a}$ in the positive $\phi$-direction.
\end{enumerate}
{\em We emphasize that $\overline{v}_{*_0}^{\circ,1}\cup \overline{v}_{*_0}^{\circ,2}$ is essentially the same thing as $\overline{v}_{*_0}$, described in a slightly different way.}

The reader can verify that, with our convention (R),
\begin{align}
\op{ind}(\sigma_\infty^{\circ,i})=0, i=1,2, & \quad I(\overline{v}_{*_0}^{\circ,2,'})=0,\\
\nonumber  I(\overline{v}_{*_0}^{\circ,2,''})=\op{ind}(\overline{v}_{*_0}^{\circ,2,''})=0, & \quad \op{ind}(\overline{v}_{*_0}'')=\op{ind}(\overline{v}_{*_0}^{\circ,1,''}).
\end{align}

We apply the continuation method to $\overline{v}_{1,j}$, $0\leq j\leq a$, $(1,j)\not=*_0=(1,j_0)$, and $\overline{v}_{*_0}^{\circ,i}$, $i=1,2$:  We adapt the definition of a continuation along $\bdry_+ B_\tau$ in a natural way by replacing $g^1_{j_0}$ by $g^1_{j_0,1}, g_{j_0,2}, g^{1/2}_{j_0,1}$, which correspond to $\bdry_1 B^\circ_1$, $\bdry B^\circ_2$, $\bdry_2 B^\circ_1$, respectively, in Equation~\eqref{continuation}. As in Lemma~\ref{intersezione prime}, start with a nontrivial negative end $\mathcal{E}_{1}$ of some $\overline{v}_{1,j_1}$ or $\overline{v}_{*_0}^{\circ,1}$ limiting to $z_\infty$ and continue the corresponding $g^1_{j_1}$, $g^1_{j_0,1}$, or $g^{1/2}_{j_0,1}$ in the direction of $\bdry_+ B_\tau$, until some $g^{\star_1}_{\star_2}\not\subset \sigma_\infty^*$ is reached. Suppose $\mathcal{E}_2$ is the end of some component of $\overline{u}_\infty$ which ``lies between'' $g^{\star_1}_{\star_2}$ and some $g^{\star_1'}_{\star'_2}$. We then switch from $g^{\star_1}_{\star_2}$ to $g^{\star_1'}_{\star_2'}$ and continue.  This gives a sequence of sectors of type ${\frak S}(\overline{a}_{k,l},\overline{a}_{k',l'})$ with $(k,l)\not=(k',l')$ (corresponding to the left negative end of $\overline{W}^\circ_1$), ${\frak S}(\overline{\hh}(\overline{a}_{k,l}),\overline{a}_{k',l'})$, or ${\frak S}(\overline{a}_{k',l'}, \overline{\hh}(\overline{a}_{k,l}))$, and hence a cycle
$$\mathcal{Z}=({\frak z}_1\to {\frak z}_2\to \dots\to {\frak z}_k\to {\frak z}_1)$$
which winds around $\R/2\pi\Z$ once.

We leave it to the reader to make the proper generalizations of $\mathcal{Z}$ when we do not assume (S); $*_0=+$ or $(-1,j_0)$, $1\leq j_0\leq c+1$; and/or the boundary point at $z_\infty$ lies on $L_{\overline{\hh}(\overline{\bf a})}$, $L^+_{\overline{\bf a}}$, $L_{\overline{\bf b}}$, $L_{\overline{\hh}(\overline{\bf b})}$, or $L^-_{\overline{\bf b}}$.

Let $A^{[a,b]}_\varepsilon$ be the annulus $\{\rho=\varepsilon\}\subset [a,b]\times D^2_{\rho_0}$. Writing $*_0=(1,j_0)$, let $\ar{\mathcal{D}}_{+,j_0}$, $\ar{\mathcal{D}}_{L-,j_0}$, $\ar{\mathcal{D}}_{R-,j_0}$ be the data at $z_\infty$ of the positive, left negative, and right negative ends, and let $P_{L-,j_0,1/2}$ and $P_{L-,j_0,1}$ be the initial and terminal points of $A^{[1/2,1]}_\varepsilon$ determined by $\ar{\mathcal{D}}_{L-,j_0}$, etc. We also decompose $P_{*,i}=P'_{*,i}\cup P''_{*,i}$, $*=(L-,j_0)$, $(R-,j_0)$, or $(+,j_0)$, as before so that $P'_{*,i}$ corresponds to $\overline{v}_{*}^{\circ,1,'}$ and $P''_{*,i}$ corresponds to $\overline{v}_{*}^{\circ,1,''}$.

The following definition does not assume (S).

\begin{defn} \label{types of boundary punctures}
A boundary point ${\frak r}\in \bdry F_{*_0}$, $*_0=(1,j_0)$, $0\leq j_0\leq a$, at $z_\infty$ falls into one of three (mutually exclusive) types:
\begin{enumerate}
\item[(P$_1$)] $\overline{v}'_{*_0}=\varnothing$;
\item[(P$_2$)] $\overline{v}'_{*_0}\not=\varnothing$ and all the vertices of $\mathcal{Z}$ correspond to points of type $\overline{a}_{k,l}$ or all the vertices of $\mathcal{Z}$ correspond to points of type $\overline{\hh}(\overline{a}_{k,l})$;
\item[(P$_3$)] $\overline{v}'_{*_0}\not=\varnothing$ and $\mathcal{Z}$ has vertices of both types $\overline{a}_{k,l}$ and $\overline{\hh}(\overline{a}_{k,l})$.
\end{enumerate}
Boundary points ${\frak r}\in \bdry F_{*_0}$, $*_0=(-1,j_0)$, $1\leq j_0\leq c+1$, at $z_\infty$ are classified similarly.
\end{defn}

The following is a strengthening of Lemma~\ref{intersezione prime} when $\overline{v}'_{1,j}\cup\overline{v}^\sharp_{1,j}\not=\varnothing$ for some $j>0$ and there are boundary points of type (P$_3$) at $z_\infty$:

\begin{lemma}\label{intersezione prime revisited}
Suppose $\overline{v}'_{1,j}\cup\overline{v}^\sharp_{1,j}\not=\varnothing$ for some $j>0$. Let $\mathcal{E}_{-,i}$, $i=1,\dots,q$, be the negative ends of $\cup_{j=1}^a\overline{v}_{1,j}^\sharp$ that converge to $z_\infty$, let $\mathcal{E}_{+,i}$, $i=1,\dots,r$, be the positive ends of $\cup_{j=0}^{a-1}\overline{v}_{1,j}^\sharp$ that converge to $z_\infty$, and let $\mathcal{E}'_i$, $i=1,\dots,s$, be the neighborhoods of the boundary points of type (P$_3$). Then
\begin{equation}\label{oolong bis}
n^*( (\cup_{i=1}^q\mathcal{E}_{-,i})\cup (\cup_{i=1}^r\mathcal{E}_{+,i})\cup(\cup_{i=1}^s \mathcal{E}'_i))\geq m-p_+.
\end{equation}
\end{lemma}

\begin{proof}
Similar to that of Lemma~\ref{intersezione prime}.
\end{proof}

\begin{lemma}\label{cherimoya2}
If ${\frak p}_1,\dots,{\frak p}_s$ are the boundary points of type (P$_1$) and (P$_2$) and $N({\frak p}_i)\subset F_*$ is a small neighborhood of ${\frak p}_i$, then $\sum_{i=1}^s n^*(\overline{v}_*(N({\frak p}_i)))\geq m$.
\end{lemma}

\begin{proof}
(P$_2$) follows immediately from the description of $\mathcal{Z}$; (P$_1$) is similar.
\end{proof}

\subsubsection{Bounds on ECH indices, part II} \label{subsub: bounds on ech indices part II}

In this subsection we give bounds on ECH indices in the presence of boundary points at $z_\infty$.

For simplicity we are still assuming (S). The ECH indices of $\overline{v}_{*_0}^{\circ,i}$ are defined in a manner similar to that of Definition~I.\ref{P1-ECH index revisited}.

\begin{lemma} \label{Lisbeth}
If $\overline{v}'_{*_0}, \overline{v}''_{*_0}\not=\varnothing$, (S) holds, and the boundary point at $z_\infty$ is of type (P$_3$),  then there exist contributions $I_+\geq 0$, $I_{L-}= 2$, and $I_{R-}\geq 0$ from the ends that limit to $z_\infty$ at the positive, left negative, and right negative ends, such that
\begin{equation}
I(\overline{v}_{*_0})\geq I(\overline{v}_{*_0}^{\circ,1,'})+ I(\overline{v}_{*_0}^{\circ,1,''})+I_++I_{L-}+I_{R-}.
\end{equation}
\end{lemma}

\begin{proof}
The calculation of $I_{L-}$ is similar to the ECH index calculations of Section~I.\ref{P1-subsection: modified indices at z infty} and in particular that of Lemma~I.\ref{P1-calc of almost sum}.
\begin{figure}[ht]
\begin{center}
\psfragscanon
\psfrag{a}{\tiny ${\frak z}'_1$}
\psfrag{g}{\tiny ${\frak z}'_2 $}
\psfrag{h}{\tiny ${\frak z}''_1$}
\psfrag{i}{\tiny ${\frak z}''_s$}
\psfrag{B}{\tiny ${1\over 2}$}
\psfrag{A}{\tiny $1$}
\includegraphics[width=6.5cm]{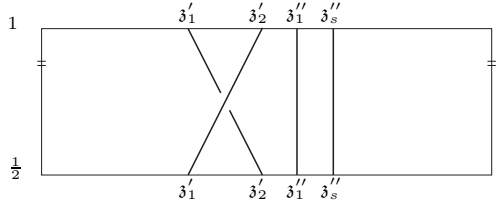}
\end{center}
\caption{${\frak c}'_{L-}\cup {\frak c}''_{L-}$ for ${\frak r}$ of type (P$_3$). Here ${\frak c}''_{L-}$ is the strand in the front.}
\label{figure: puncture3}
\end{figure}

First observe that $P_{L-,j_0,1/2}=P_{L-,j_0,1}$. For simplicity suppose that $\mathcal{Z}$ has only one chord ${\frak z}'_1\to{\frak z}'_2$ corresponding to a sector of type ${\frak S}(\overline{a}_{k,l},\overline{a}_{k',l'})$. Observe that there are no points of $\mathcal{Z}$ between ${\frak z}'_1$ and ${\frak z}'_2$; otherwise $\mathcal{Z}$ winds more than once around $\R/2\pi\Z$.  Hence $P''_{L-,j_0,1/2}=\{{\frak z}'_1\}$, $P'_{L-,j_0,1/2}=\{{\frak z}'_2,{\frak z}''_1,\dots,{\frak z}''_s\}$, $P''_{L-,j_0,1}=\{{\frak z}'_2\}$, and $P'_{L-,j_0,1}=\{{\frak z}'_1,{\frak z}''_1,\dots,{\frak z}''_s\}$, where $P_{L-,j_0,1/2}=P_{L-,j_0,1}$ is written as $\{{\frak z}'_1,{\frak z}'_2, {\frak z}''_1,\dots,{\frak z}''_s\}$ in cyclic order around $\R/2\pi\Z$, and the projection of the left negative end of $\overline{v}_{*_0}^{\circ,1,''}$ that limits to $z_\infty$ intersects $A^{[1/2,1]}_\varepsilon$ along an arc ${\frak c}''_{L-}$ with winding number $w({\frak c}''_{L-})=0$ or $1$, depending on whether ${\frak S}(\overline{a}_{k,l},\overline{a}_{k',l'})$ is a large sector; see Figure~\ref{figure: puncture3}. The left negative ends of $\overline{v}_{*_0}^{\circ,1,'}$ that limit to $z_\infty$ can modified to give a grooming ${\frak c}'_{L-}$ on $A^{[1/2,1]}_{\varepsilon/2}$ such that the winding number $w({\frak c}'_{L-})=0$ or $-1$ and ${\frak c}'_{L-}$ connects ${\frak z}''_i$ to ${\frak z}''_i$, $i=1,\dots,s$, by vertical arcs.  Then the writhe of ${\frak c}'_{L-}\cup {\frak c}''_{L-}$ is $+1$ and resolving the positive crossing of ${\frak c}'_{L-}\cup {\frak c}''_{L-}$ yields a grooming by vertical arcs from $P_{L-,j_0,1/2}$ to $P_{L-,j_0,1}$.

We now consider the disk $\check D$ corresponding to resolving the positive crossing that we ``append'' to the left negative end of $\overline{v}_{*_0}^{\circ,1}$ as in the proof of Lemma~I.\ref{P1-calc of almost sum}. The disk $\check D$ contributes $1,0,1$, and $0$ to $Q, c_1, \mu$, and the discrepancy: $Q$ is equal to the writhe $+1$. The calculations for $\mu$ assume (R). Suppose $w({\frak c}'_{L-})=w({\frak c}''_{L-})=0$ (the case $w({\frak c}'_{L-})=-1$ and $w({\frak c}''_{L-})=1$ is similar and is left to the reader). Then $\mu$ of the positive ends are $1,0,\dots,0$ and $\mu$ of the negative ends are all $0$.  The discrepancy contribution at the positive end is computed separately for ${\frak c}'_{L-}$ and ${\frak c}''_{L-}$ and both are zero.  At the negative end, the discrepancy is zero.  Hence $I_{L-}=2$.

Finally, since $I(\overline{v}_{*_0})=I(\overline{v}_{*_0}^{\circ,1})+ I(\overline{v}_{*_0}^{\circ,2})$ and $I(\overline{v}_{*_0}^{\circ,2})=0$, we have $I(\overline{v}_{*_0})=I(\overline{v}_{*_0}^{\circ,1})= I(\overline{v}_{*_0}^{\circ,1,'})+ I(\overline{v}_{*_0}^{\circ,1,''})+I_++I_{L-}+I_{R-}$. Lemma~\ref{nonnegative ECH indices} implies that $I_+\geq 0$ and $I_{R-}\geq 0$.
\end{proof}

\begin{rmk} \label{Salander}
In general, each collection of boundary points of type (P$_3$) that map to the same point on the base contributes at least $+2$ towards $I$.  The proof is similar to but slightly more complicated than that of Lemma~\ref{Lisbeth}, and is left to the reader.
\end{rmk}

The following is a strengthening of Lemma~\ref{nonnegative ECH indices} in the presence of boundary points at $z_\infty$:

\begin{lemma} \label{nonnegative ECH indices better}
If fiber components are removed from $\overline{u}_\infty$ and the only boundary points are of type (P$_3$), then
\begin{enumerate}
\item  the only components of $\overline{v}_*'$ with negative ECH index are the branched covers of $\sigma_\infty^+$;
\item the ECH index of each level $\overline{v}_*\not= \overline{v}_+$ is nonnegative; and
\item  For $0\leq j\leq a$,
$$ I(\overline{v}_{1,j})\geq  \left\{
\begin{array}{ll}
I(\overline{v}'_{1,j}) +I(\overline{v}''_{1,j}) & \mbox{if } \mathfrak{bp}_{1,j}=0;\\
I(\overline{v}'_{1,j}) +I(\overline{v}''_{1,j})+2  & \mbox{if } \mathfrak{bp}_{1,j}>0.
\end{array}
\right.
$$
\end{enumerate}
Here $\mathfrak{bp}_*$ is the number of boundary points of type (P$_3$) on $\overline{v}_*$.
\end{lemma}

\begin{proof}
We explain the modifications that need to be made when there are boundary points at $z_\infty$ in view of Lemma~\ref{Lisbeth} and Remark~\ref{Salander}. We assume (S) for simplicity.  The cycles $\mathcal{Z}^{(a-j)}$ are defined as before, for $j\geq j_0$. We define $\mathcal{Z}^{(a-j_0,+)}=\mathcal{Z}^{(a-j_0)}$ and $\mathcal{Z}^{(a-j_0,-)}$ as $\mathcal{Z}^{(a-j_0,+)}$ with ${\frak z}_1'\to {\frak z}'_2$ replaced by ${\frak z}'_2$. Then $\mathcal{Z}^{(a-(j_0-1))}$ is obtained using $\mathcal{Z}^{(a-j_0,-)}$ instead of $\mathcal{Z}^{(a-j_0)}$.  Also, $P_{R-,j_0,i}^\star$, $\star=\varnothing,',''$, is obtained from $P_{+,j_0,i}^\star$ by replacing ${\frak z}'_1$ by ${\frak z}'_2$. The rest of the argument of Lemma~\ref{nonnegative ECH indices} carries over.
\end{proof}

\subsubsection{Case of $\overline{v}'_*\cup \overline{v}^\sharp_*=\varnothing$ for all $\overline{v}_*$}

\begin{lemma} \label{los angeles}
If $\overline{u}_\infty\in \bdry_{\{+\infty\}}\mathcal{M}$ and $\overline{v}'_*\cup\overline{v}^\sharp_*=\varnothing$ for all levels $\overline{v}_*$ of $\overline{u}_\infty$, then $a=c=0$; $I(\overline{v}_+)=0$; $I(\overline{v}_-)=2$; $\overline{v}_+$ is a $W_+$-curve; $\overline{v}_-$ is a $\overline{W}_-$-curve; and there may be connectors $\overline{v}_{0,j}$ in between.
\end{lemma}

\begin{proof}
Suppose that $\overline{u}_\infty\in \bdry_{\{+\infty\}}\mathcal{M}$ and $\overline{v}'_*\cup\overline{v}^\sharp_*=\varnothing$ for all levels $\overline{v}_*$ of $\overline{u}_\infty$.  Then there is a point ${\frak q}\in int(F_-)$ and a sufficiently small neighborhood $N({\frak q})\subset \dot F_-$ of ${\frak q}$ such that $\overline{v}_-({\frak q})=\overline{\frak m}(\infty)$ and $n^*(\overline{v}_-(N({\frak q})))\geq m$. By Equation~\eqref{sum of n} and Lemma~\ref{cherimoya}, there are no boundary points at $z_\infty$ and the only possible fiber component passes through $\overline{\frak m}(\infty)$. Hence every level $\overline{v}_*$, $*\not=-$, has image inside $W'$, $W$, or $W_+$ and $\overline{v}_-$ is a $\overline{W}_-$-curve or a {\em degenerate} $\overline{W}_-$-curve by the analog of Lemma~I.\ref{P1-lemma: preliminary restrictions part 1}.

The ECH index of each level $\not=\overline{v}_+$ which has no fiber component is nonnegative by Lemma~\ref{nonnegative ECH indices}. Since $\overline{v}_+$ is a $W_+$-curve, $I(\overline{v}_+)\geq 0$. On other hand, by the previous paragraph, if there is a fiber component, then it is a component of $\overline{v}_-$. We claim that $I(\overline{v}_-)\geq 2$, with equality if and only if $\overline{v}_-$ is not a degenerate $\overline{W}_-$-curve. Indeed, if $\overline{v}_-$ is not degenerate, then $I(\overline{v}_-)\geq 2$ by the point constraint (this is the only place where we use the fact that $\overline{u}_\infty\in \bdry_{\{+\infty\}}\mathcal{M}$) and the ECH index inequality, and if $\overline{v}_-$ is degenerate, then $I(\overline{v}_-)\geq 4$, as computed in the proof of Lemma~I.\ref{P1-claim in proof}.

The lemma then follows from Equation~\eqref{sum of I}. In particular, degenerate $\overline{W}_-$-curves are not allowed.
\end{proof}

\subsubsection{Preliminary restrictions when $\overline{v}'_*\cup \overline{v}^\sharp_*\not=\varnothing$ for some $\overline{v}_*$}

We now consider the case where $\overline{v}'_*\cup\overline{v}^\sharp_*\not =\varnothing$ for some level $\overline{v}_*$.

\begin{lemma} \label{sencha}
If $\overline{v}'_*\cup\overline{v}^\sharp_*\not =\varnothing$ for some level $\overline{v}_*$ of $\overline{u}_\infty$, then:
\begin{enumerate}
\item $p_-=\deg(\overline{v}_-')>0$;
\item $\overline{u}_\infty$ has no boundary point of type (P$_1$) or (P$_2$);
\item $\overline{u}_\infty$ has no fiber components and no components of $\overline{v}''_*$ that intersect the interior of a section at infinity;
\item each component of $\overline{v}_{-1,j}^\sharp$, $j=1,\dots,c$, is a thin strip from $z_\infty$ to some $x_i$ or $x_i'$ with $I=1$;
\item each component of $\overline{v}_-^\sharp$ is a section of $\overline{W}_--W_-$ from $\delta_0$ to some $x_i$ or $x_i'$ with $I=1$;
\item each component of $\overline{v}_{0,j}^\sharp$, $j=1,\dots,b$, which has a positive end at a multiple of $\delta_0$ is a cylinder from $\delta_0$ to $h$ or $e$ with $I=1$ or $2$;
\item the only boundary points at $z_\infty$ are type (P$_3$) points of $\overline{v}_{1,j}^\sharp$, $j=0,\dots,a$.
\end{enumerate}
\end{lemma}

\begin{proof}
The lemma is a consequence of Equation~\eqref{sum of n}.  Suppose $\overline{v}'_*\cup\overline{v}^\sharp_*\not =\varnothing$ for some level $\overline{v}_*$. Then $\overline{v}^\sharp_*\not=\varnothing$ for some $\overline{v}_*$ and each end of $\overline{v}_*^\sharp$ that limits to some multiple of $z_\infty$ or $\delta_0$ contributes positively to $n^*$ by Lemmas~\ref{intersezione} and ~\ref{intersezione prime revisited}.

(1) If $\overline{v}_-'=\varnothing$, then the restriction of $\overline{v}_-''$ to a neighborhood of $\overline{\frak m}(+\infty)$ contributes $m$ towards $n^*(\overline{v}_-)$. This contradicts the discussion from the previous paragraph, proving (1).

Since $\overline{v}_-'\not=\varnothing$, some $\overline{v}_{0,j}^\sharp$, $j=1,\dots,b+1$, or $\overline{v}_{1,j}^\sharp$, $j=1,\dots,a$, has a negative end at $z_\infty$ or a multiple of $\delta_0$. By Lemmas~\ref{intersezione}(1) and ~\ref{intersezione prime revisited}, it follows that:
\begin{equation} \label{first}
\sum_{\overline{v}_*\succ \overline{v}_-}n^*(\overline{v}_*)\geq m-2g,
\end{equation}
where we are counting contributions from the ends and the boundary points of type (P$_3$).

(2)--(7) are consequences of Equation~\eqref{first}. We explain (2) and (3), leaving the rest to the reader. By Lemma~\ref{cherimoya2}, the neighborhoods of the boundary points of type (P$_1$) or (P$_2$) contribute at least $m$ towards $n^*(\overline{v}_*'')$ in total. Also, a non-ghost fiber component or a component of $\overline{v}''_*$ that intersects the interior of a section at infinity contributes $m$ towards $n^*$. They both contradict Equation~\eqref{first}.
\end{proof}

\begin{lemma} \label{a or b}
If $\overline{v}'_*\cup\overline{v}^\sharp_*\not =\varnothing$ for some level $\overline{v}_*$, then the following alternative holds:
\begin{enumerate}
\item[(a)] either some $\overline{v}_{0,j_0}^\sharp$, $j_0=1,\dots,b+1$, has a negative end at a multiple of $\delta_0$, in which case $\overline{v}_{0,j_0}'=\varnothing$, $\overline{v}_*'=\overline{v}_*^\sharp=\varnothing$ for all levels $\overline{v}_*\succ \overline{v}_{0,j_0}$, and $\overline{v}_{0,j}'\not=\varnothing$ for all levels $\overline{v}'_-\preceq \overline{v}_{0,j}\prec \overline{v}'_{0,j_0}$; or
\item[(b)] no $\overline{v}_{0,j}^\sharp$, $j=1,\dots,b+1$, has a negative end at a multiple of $\delta_0$, in which case $\overline{v}'_*\not=\varnothing$ for all levels $\overline{v}'_-\preceq \overline{v}'_*\preceq \overline{v}'_+$ and $\overline{v}_{1,j}'\cup \overline{v}_{1,j}^\sharp\not=\varnothing$ for some $j>0$.
\end{enumerate}
\end{lemma}

\begin{proof}
This follows from Lemmas~\ref{intersezione}(1) and ~\ref{intersezione prime revisited} by observing that either case contributes at least $m-2g$ towards $n^*$ and that it is not possible to have both since $m\gg 2g$. It is also not possible to have multiple occurrences of negative ends of $\overline{v}_{0,j}^\sharp$, $j>0$, that converge to some multiple of $\delta_0$.
\end{proof}

\subsubsection{List of possibilities when $\overline{v}'_*\cup \overline{v}^\sharp_*\not=\varnothing$ for some $\overline{v}_*$}

We first start with a useful definition.

\begin{defn} $\mbox{}$
\begin{enumerate}
\item Let $X_1\cup\dots\cup X_{r}$ be an $r$-level building of almost complex manifolds with cylindrical ends, ordered from bottom to top, and let
$$v= v_1\cup\dots \cup v_r, \quad \op{Im}(v_i)\subset X_i,$$ be a corresponding $r$-level holomorphic building, where each level has finite energy. If $\mathcal{E}$ is some collection of positive ends of $v_r$, then the {\em holomorphic building hanging from $\mathcal{E}$} is the union of irreducible components $\widetilde{v}_{i_0}$ for which there exist irreducible components $\widetilde{v}_i$, $i=i_0+1,\dots,r$, where the limit of a positive end of $\widetilde{v}_i$ agrees with the limit of a negative end of $\widetilde{v}_{i+1}$ for all $i=i_0,\dots,r-1$ and a positive end of $\widetilde{v}_r$ is contained in $\mathcal{E}$.
\item If $X_1\cup\dots\cup X_{r+1}$ is an $(r+1)$-level building, $v=v_1\cup\dots\cup v_{r+1}$ is a corresponding $(r+1)$-level holomorphic building and $\mathcal{E}$ is some collection of negative ends of $v_{r+1}$, then the {\em holomorphic building hanging from $\mathcal{E}$} is defined similarly.
\item If $X_1\cup\dots\cup X_r$ is an $r$-level building, $v=v_1\cup\dots\cup v_r$ is a corresponding $r$-level holomorphic building and $\mathcal{E}$ is some collection of negative ends of $v_1$, then the {\em holomorphic building sitting above $\mathcal{E}$} is the union of irreducible components $\widetilde{v}_{i_0}$ for which there exist irreducible components $\widetilde{v}_i$, $i=1,\dots,i_0-1$, where the limit of a negative end of $\widetilde{v}_i$ agrees with the limit of a positive end of $\widetilde{v}_{i-1}$ for all $i=2,\dots,i_0$ and a negative end of $\widetilde{v}_1$ is contained in $\mathcal{E}$.
\item If $X_0\cup\dots\cup X_{r}$ is an $(r+1)$-level building, $v=v_0\cup\dots\cup v_{r}$ is a corresponding $(r+1)$-level holomorphic building and $\mathcal{E}$ is some collection of positive ends of $v_0$, then the {\em holomorphic building sitting above $\mathcal{E}$} is defined similarly.
\end{enumerate}
\end{defn}

\begin{lemma}  \label{hojicha}
If $\overline{v}'_*\cup\overline{v}^\sharp_*\not =\varnothing$ for some level $\overline{v}_*$, then $\overline{v}_-'\not=\varnothing$, there are no removable points at $z_\infty$, and $\overline{u}_\infty$ contains one of the following subbuildings:
\begin{enumerate}
\item A $3$-level building consisting of a component of $\overline{v}_{0,1}^\sharp$ with $I=1$ from $\bs\gamma$ to $\delta_0\bs\gamma'$; $\overline{v}_-'=\sigma_\infty^-$; and a thin strip.
\item[($2_i$)] A $3$-level building consisting of a component of $\overline{v}_+^\sharp$ with $I=i$, $i=0,1$, from ${\bf y}$ or ${\bf y}''$ to $\delta_0\bs\gamma'$; $\overline{v}_-'=\sigma_\infty^-$; and a thin strip.
\item[(3)] A $4$-level building consisting of a component of $\overline{v}_+^\sharp$ with $I=0$ from ${\bf y}$ to $\delta_0^2\bs\gamma'$; $\overline{v}_{0,1}'=\sigma_\infty'$; a component of $\overline{v}_{0,1}^\sharp$ which is an $I=1$ cylinder from $\delta_0$ to $h$; $\overline{v}_-'=\sigma_\infty^-$; and a thin strip.
\item[(4)] A $4$-level building consisting of a component of $\overline{v}_+^\sharp$ with $I=0$ from ${\bf y}$ to $\delta_0^2\bs\gamma'$; $\overline{v}'_{0,1}$ with $I=0$ and $\deg=2$; $\overline{v}_-'=\sigma_\infty^-$; a component of $\overline{v}_-^\sharp$ which is an $I=1$ curve from $\delta_0$ to some $x_i$ or $x_i'$; and a thin strip.
\item[(5)] A $\geq 3$-level building consisting of a component of $\overline{v}_+^\sharp$ with $I=0$ from ${\bf y}$ to $\delta_0^2\bs\gamma'$; $\overline{v}'_{0,1}$ and $\overline{v}_-'$ with $I=0$ and $\deg=2$, where $\overline{v}'_{0,1}=\varnothing$ is possible; and two thin strips.
\item[($6_i$)] A $\geq 4$-level building consisting of a component of $\overline{v}_{1,1}^\sharp$ with $I=i$, $i=1,2$, from ${\bf y}$ to $\{z_\infty^{r}\}\cup {\bf y'}$; $0<p_{0,0}\leq p_{0,1}\leq \dots\leq p_{0,b+1}\leq r$; $I(\overline{v}_+')=-p_{0,b+1}$; $I(\overline{v}_{0,j}')=0$ for $j=1,\dots,b$; $\overline{v}_+^\sharp$ with $I=2-i$ which has no negative ends at a multiple of $\delta_0$; $p_{0,b+1}-p_{0,1}$ cylinders of $\overline{v}_{0,j}^\sharp$, $j\in\{1,\dots,b\}$, with $I=1$ each from $\delta_0$ to $h$; $p_{0,1}-p_{0,0}$ components of $\overline{v}_-^\sharp$ with $I=1$ each from $\delta_0$ to some $x_i$ or $x_i'$; and $p_{0,0}$ thin strips. The total ECH index of the building hanging from the negative end of $\overline{v}_+'$ is $p_{0,b+1}$.
\end{enumerate}
Here we are omitting levels which are connectors.  If there is more than one thin strip, then the thin strips could be on the same level or on different levels.
\end{lemma}

See Figure~\ref{figure: graphs3bis}.

\begin{figure}[ht]
\begin{center}
\psfragscanon
\psfrag{A}{\tiny $\R\times\overline{N}$}
\psfrag{B}{\tiny $\overline{W}_-$}
\psfrag{C}{\tiny $\overline{W}$}
\psfrag{D}{\tiny $\overline{W}_+$}
\psfrag{1}{\tiny $1$}
\psfrag{2}{\tiny $2$}
\psfrag{0}{\tiny $0$}
\psfrag{z}{\tiny $0,1$}
\psfrag{y}{\tiny $1,2$}
\psfrag{x}{\tiny $-p_{0,b+1}$}
\psfrag{w}{\tiny $0,1$}
\psfrag{v}{\tiny total}
\psfrag{u}{\tiny $=p_{0,b+1}$}
\psfrag{a}{\small (1)}
\psfrag{b}{\small (2$_i$)}
\psfrag{c}{\small (3)}
\psfrag{d}{\small (4)}
\psfrag{e}{\small (5)}
\psfrag{f}{\small (6$_i$)}
\includegraphics[width=12cm]{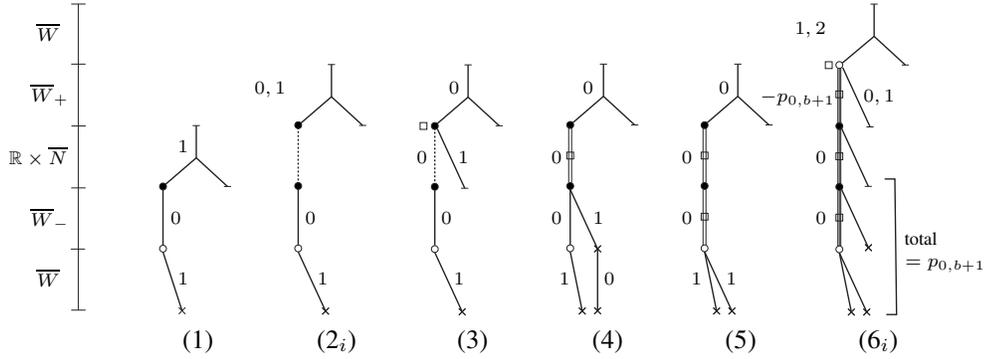}
\end{center}
\vskip.2in
\caption{Schematic diagrams for the possible types of degenerations. Here $\bullet$ represents $\delta_0$, $\circ$ represents $z_\infty$, {\tiny $\square$} represents a possible location of a branch point and $\times$ represents some $x_i$ or $x_i'$. A vertical line indicates the restriction of a trivial cylinder, a dotted vertical line indicates a trivial cylinder, a double vertical line indicates a degree $2$ branched cover of a trivial cylinder or a restriction of a trivial cylinder, and a triple vertical line indicates a degree $p\geq 2$ branched cover of a trivial cylinder or a restriction of a trivial cylinder. The labels on the graphs are ECH indices of each component. All the thin strips of $\overline{v}_{-1,j}^\sharp$ are drawn on the same level for convenience.}
\label{figure: graphs3bis}
\end{figure}

\begin{proof}
The lemma is a consequence of Equations~\eqref{sum of n} and \eqref{sum of I}.  By Lemmas \ref{nonnegative ECH indices better} and \ref{sencha}(2), (3), all the levels $\overline{v}_*$ besides $\overline{v}_+$ have nonzero ECH index, the only components of $\overline{u}_\infty$ which have negative ECH index are the branched covers of $\sigma_\infty^+$, and $I(\overline{v}_{1,j})\geq I(\overline{v}'_{1,j})+I(\overline{v}''_{1,j})$ or $I(\overline{v}'_{1,j})+I(\overline{v}''_{1,j})+2$ for $0\leq j\leq a$, depending on whether $\mathfrak{bp}_{1,j}=0$ or $>0$.  In particular, all the components of $\overline{u}_\infty$ which are left out of the subbuildings in (1)--(6$_i$) and are not drawn in Figure~\ref{figure: graphs3bis} have $I\geq 0$.

By Lemma~\ref{sencha}(1), $\overline{v}'_-\not=\varnothing$.  By Lemma~\ref{sencha}(2), (7), there are no boundary points of types (P$_1$) and (P$_2$) and no boundary points of type (P$_3$) except for $\overline{v}^\sharp_{1,j}$, $j=0,\dots,a$.   Cases (1)--(5) have no removable points at $z_\infty$  (since the ECH index would add up to more than $2$) and Case (6$_i$) may have boundary points of type (P$_3$).

First suppose that some $\overline{v}_{0,j_0}^\sharp$, $j_0\in\{1,\dots,b\}$, has a negative end at $\delta_0^p$ for some $p>0$. This is the situation of Lemma~\ref{a or b}(a) with $j_0\in\{1,\dots,b\}$.  In this case, all the components and levels have nonnegative ECH index.  Each component of $\overline{v}^\sharp_*\prec\overline{v}_{0,j_0}$ described in Lemma~\ref{sencha}(4)--(6) satisfies $I\geq 1$ and there must be $p>0$ such components since $\overline{v}_{0,j_0}$ has multiplicity $p$ at $\delta_0$. We then have the following contributions towards $I$:
\begin{itemize}
\item[(a)] $\sum_{\overline{v}^\sharp_*\prec\overline{v}_{0,j_0}} I(\overline{v}^\sharp_*)\geq p$ and there is at least one thin strip.
\item[(b)] $I(\overline{v}_{0,j_0}^\sharp)\geq 1$ since $\overline{v}_{0,j_0}$ is nontrivial.
\end{itemize}
Then $p=1$ and all the components and levels besides those of (a) and (b) satisfy $I=0$. We are in Case (1).

Next suppose that $\overline{v}_{0,b+1}^\sharp=\overline{v}_+^\sharp$ has a negative end at $\delta_0^p$. This is the situation of Lemma~\ref{a or b}(a) with $j=b+1$. Since $\overline{v}_+'=\varnothing$, all the components and levels satisfy $I\geq 0$. Here $\overline{v}_+^\sharp$ cannot have any multiple of $z_\infty$ at the positive end, since otherwise we contradict Equation~\eqref{sum of n}. Since $I(\overline{v}_+^\sharp)\geq 0$ and $\sum_{\overline{v}^\sharp_*\prec\overline{v}_+}I(\overline{v}^\sharp_*)\geq p$, either $p=1$ and we are in Case ($2_i$) or $p=2$ and we are in Cases (3)--(5).

Finally suppose that no $\overline{v}_{0,j}^\sharp$, $j=1,\dots,b+1$, has a negative end at a multiple of $\delta_0$. This is the situation of Lemma~\ref{a or b}(b). We have the following contributions:

\begin{enumerate}
\item[($\alpha$)] $\deg\overline{v}'_+=p_{0,b+1}$ and $I(\overline{v}'_+)= -p_{0,b+1}$.

\item[($\beta$)] By Lemma~\ref{sencha}(4)--(6), the holomorphic building hanging from the negative end of $\overline{v}'_+$ satisfies $\sum_{\overline{v}_*\prec\overline{v}_+} I(\overline{v}_*'\cup \overline{v}_*^\sharp)\geq p_{0,b+1}$, and equality holds if and only if there is no cylinder from $\delta_0$ to $e$.

\item[($\gamma$)] $I=2$ components of $\overline{v}_{0,j}^\sharp$ from $\delta_0$ to $e$ can be eliminated by observing that it is followed by an $I=-1$ component of $\overline{v}_-^\flat$ from $e$ to some $x_i$ or $x_i'$, which is a contradiction.\footnote{Strictly speaking, $\overline{J}_-$ is the restriction of a Morse-Bott $\overline{J'}$. To give a proper treatment of regularity, we must perturb $\overline{J'}$ using an arbitrarily small Morse function to obtain $(\overline{J'})^\Diamond$ and then restrict $(\overline{J'})^\Diamond$ to $\overline{W}_-$.}

\item[($\delta$)] Some $\overline{v}^\sharp_{1,j_0}$, $j_0\geq 0$, must have an end at $z_\infty$ (or some boundary point of type (P$_3$) must have a neighborhood) which projects to a large sector of $D^2_{\rho_0}$. Since $\overline{v}^\sharp_{1,j_0}$ generically lies on a codimension $1$ stratum of some $\op{ind}(\overline{v}^\sharp_{1,j_0})$-dimensional moduli space, the large sector contributes an additional $+1$ to the Fredholm and ECH indices, cf.\ Lemma~I.\ref{P1-index inequality for z infinity case}.
\item[($\varepsilon$)] If $\mathfrak{bp}$ is the number of boundary points of type (P$_3$) and $\mathfrak{bp}>0$, then the contribution to $I$ is at least $2$.
\end{enumerate}

($\alpha$), ($\beta$), and ($\gamma$) together give:
$$I(\overline{v}'_+)+\sum_{\overline{v}_*\prec\overline{v}_+} I(\overline{v}_*'\cup \overline{v}_*^\sharp)=0.$$
Now, $I(\overline{v}''_{1,j})\geq 1$ for each $j>0$,\footnote{Recall that we are ignoring connectors.} ($\delta$) contributes at least $1$ to $I$, and ($\varepsilon$) contributes at least $2$ to $I$ if $\mathfrak{bp}>0$. Hence $\mathfrak{bp}=0$, $a=1$, and we are in Case ($6_{i}$), $i=1,2$.
\end{proof}

\subsubsection{Elimination of some cases when $\overline{v}'_*\cup \overline{v}^\sharp_*\not= \varnothing$ for some $\overline{v}_*$}

The goal of this subsection is to eliminate all but one of the possibilities (namely Case (2$_1$)) given in Lemma~\ref{hojicha} and to prove Lemma~\ref{cherries4}.

\begin{defn} \label{branch points}
If $\pi:\Sigma\to\Sigma'$ is a branched cover and $\mathcal{B}(\pi)$ is the branch locus of $\pi$, then we define
$${\frak b}(\pi)=\sum_{x\in \mathcal{B}(\pi)} (\op{deg}(x)-1),$$
where $\op{deg}(x)$ is the degree of the branch point $x$.  We also write ${\frak b}(\overline{u})={\frak b}(\overline{\pi}_*\circ \overline{u})$, where $\overline{\pi}_*$ is the projection to the base $B$, $B'$, etc., and refer to it informally as ``the number of branch points of $\overline{u}$.''
\end{defn}

Let $\mathcal{E}_{-,i}$, $i=1,\dots,q$, be the negative ends of $\cup_{j=1}^a \overline{v}_{1,j}^\sharp$ that converge to $z_\infty$ and let $\mathcal{E}_{+,i}$, $i=1,\dots,r$, be the positive ends of $\cup_{j=0}^{a-1}\overline{v}_{1,j}^\sharp$ that converge to $z_\infty$ as in Lemma~\ref{intersezione prime}.

\begin{lemma} \label{banff}
Suppose we are in Case {\em ($6_i$)} of Lemma~\ref{hojicha}.  Then $\sum_{\overline{v}'_*}{\frak b}(\overline{v}'_*)$, including branched points of connector components, is equal to the number of negative ends $\mathcal{E}_{-,i}$. Moreover, if $F$ is the surface obtained by gluing all the domains of $\overline{v}'_*$, then $F$ is a connected planar surface whose compactification (after filling in the interior and boundary punctures) is a closed disk.
\end{lemma}

\begin{proof}
Suppose we are in Case (6$_i$) of Lemma~\ref{hojicha}. We may make the following simplifying assumptions:
\begin{enumerate}
\item $\overline{v}_{1,1}$ is a connector;
\item all the ends $\mathcal{E}_{-,i}$ are negative ends of $\overline{v}_{1,2}^\sharp$;
\item all the ends $\mathcal{E}_{+,i}$ are positive ends of $\overline{v}_{+}^\sharp$;
\item all the branch points of $\overline{v}_*'$ lie on $\overline{v}_{1,1}'$.
\end{enumerate}
(1)--(3) are immediate. Note that in Lemma~\ref{hojicha} we omitted connectors, but in the present analysis we need to keep track of connectors which are branched covers. (4) is obtained by pushing all the branched points from the upper and lower levels to $\overline{v}_{1,1}'$; this operation is possible because $\overline{v}^\sharp_*\preceq \overline{v}_+$ does not have a negative end that limits to $z_\infty$ in Case (6$_i$). This operation does not affect: \begin{itemize}
\item the total Fredholm index $\sum_{\overline{v}_*'} \op{ind}(\overline{v}_*')$; and
\item the topological type of $F$.
\end{itemize}
As a consequence of (1)--(4), each component of $\overline{v}_*'\preceq \overline{v}_{+}'$ is an unbranched cover of the appropriate $\sigma_\infty^*$.

We now compute the Fredholm index of $\overline{v}'_{1,1}:\dot{F}_{1,1}\to \overline{W}$ using the Fredholm index formula
$$\op{ind}(\overline{v}'_{1,1})=-\chi(\dot{F}_{1,1})+p_{1,1} +\mu_\tau(\overline{v}'_{1,1})+2c_1((\overline{v}'_{1,1})^*T\overline{S},\tau);$$
see Equation~(I.\ref{P1-Fredholm index revisited}). Recall we are writing $p_*=\op{deg}(\overline{v}'_*)$. The groomed multivalued trivialization $\tau$ is defined as follows: Let ${\frak S}_{\pm, i}$ be the sector of $D^2_{\rho_0}$ given by $\pi_{D^2_{\rho_0}}(\mathcal{E}_{\pm,i})$. If (ii) and its analogs occur in the proof of Lemma~\ref{intersezione prime}, then we add the thin counterclockwise sectors ${\frak S}(\overline{a}_{k_2,l_2},\overline{\hh}(\overline{a}_{k_2,l_2}))$, etc.\ to the set $\{{\frak S}_{+,i}\}$. The sets $\{{\frak S}_{+,i}\}$ and $\{{\frak S}_{-,i}\}$ correspond to the data $\overrightarrow{\mathcal{D}}_-$ and $\overrightarrow{\mathcal{D}}_+$ at the negative and positive ends of $\overline{v}'_{1,1}$\footnote{Note that the plus and minus signs are switched. This is due to the fact that, for example, $\{{\frak S}_{-,i}\}$ gives the data for the negative ends of $\overline{v}_{1,2}$ that limit to $z_\infty$, which agrees with the data for the positive ends of $\overline{v}_{1,1}$ that limit to $z_\infty$.} and we let $\tau$ be the induced groomed multivalued trivialization.

By a calculation similar to that of Lemma~I.\ref{P1-lemma: HF index of sections at infinity}, $c_1((\overline{v}'_{1,1})^*T\overline{S},\tau)=1$ and the ends of $\overline{v}'_{1,1}$ contribute the following to $\mu_\tau(\overline{v}'_{1,1})$:
\begin{itemize}
\item $0$ if ${\frak S}_{-,i}$ is a small sector;
\item $-1$ if ${\frak S}_{-,i}$ is a large sector;
\item $-1$ if ${\frak S}_{+,i}$ is a small sector; and
\item $-2$ if ${\frak S}_{+,i}$ is a large sector.
\end{itemize}
Hence $\mu_\tau(\overline{v}'_{1,1})=-p_{1,1}-1$. Since $\chi(\dot{F}_{1,1})=p_{1,1}-{\frak b}(\overline{v}'_{1,1})$, we obtain:
\begin{align}  \label{ind}
\op{ind}(\overline{v}'_{1,1})&=({\frak b}(\overline{v}'_{1,1})-p_{1,1})+p_{1,1}+(-p_{1,1}-1)+2\\
\notag &= {\frak b}(\overline{v}'_{1,1})-(p_{1,1}-1).
\end{align}

Next we claim that $\bdry F_{1,1}$ is connected. Indeed, if $\bdry F_{1,1}$ is disconnected, then the method of Lemma~\ref{intersezione prime} implies that the union of all the ${\frak S}_{\pm,i}$ covers $D^2_{\rho_0}$ more than once; this contradicts Equation~\eqref{sum of n}. The claim in turn implies that ${\frak b}(\overline{v}'_{1,1})\geq p_{1,1}-1$, since otherwise $F_{1,1}$ is disconnected and $\bdry F_{1,1}$ will have more than one component. Hence $\op{ind}(\overline{v}'_{1,1})\geq 0$; moreover, if $\op{ind}(\overline{v}'_{1,1})>0$, then $\op{ind}(\overline{v}'_{1,1})\geq 2$.

We claim that $\op{ind}(\overline{v}'_{1,1})\geq 2$ is not possible. Indeed, we add up the Fredholm indices of all the remaining components as in the proof of Lemma~\ref{hojicha}:
\begin{enumerate}
\item[(a)] $\op{ind}(\overline{v}^\sharp_{1,2})\geq 1$ and $\op{ind}(\overline{v}'_+)=-p_+$;
\item[(b)] $\sum_{\widetilde{v}}\op{ind}(\widetilde{v})\geq p_+$, where the summation is over all components $\widetilde{v}$ of $\overline{u}_\infty$ that are hanging from the negative end of $\overline{v}'_+$;
\item[(c)] a large sector of $D^2$ contributes an additional $+1$ to the Fredholm index; and
\item[(d)] all the other components of $\overline{u}_\infty$ have nonnegative Fredholm index.
\end{enumerate}
(a), (b), (d) are clear, and (c) was explained in Lemma~\ref{hojicha}. The total of (a)--(d) is $\geq 2$, which is an index excess of $+2$, and the claim follows. The claim then implies that $\op{ind}(\overline{v}'_{1,1})=0$, ${\frak b}(\overline{v}'_{1,1})=p_{1,1}-1$, and $F_{1,1}$ is a disk.

Finally, since $F$ is diffeomorphic to a surface obtained by gluing $\dot F_{1,1}$ and $p_{1,1}$ surfaces which are diffeomorphic to $B_\tau$ or $B_+$, it must be connected and planar, and its compactification is a closed disk.
\end{proof}

\begin{lemma} \label{vancouver}
Suppose $m\gg 0$. If $\overline{v}'_*\cup\overline{v}^\sharp_*\not =\varnothing$ for some level $\overline{v}_*$, then Cases (1), (2$_0$), (3)--(6$_i$) of Lemma~\ref{hojicha} are not possible.  This leaves us with Case (2$_1$).
\end{lemma}

\begin{proof}
Arguing by contradiction, suppose there exist sequences $m_i\to \infty$, $\varepsilon_i,\delta_i\to 0$, and $\overline{u}_{ij}\to \overline{u}_{i\infty}$, where $\overline{u}_{ij}\in \mathcal{M}^{(m_i)}$, $\overline{u}_{i\infty}\in \bdry_{\{+\infty\}} \mathcal{M}^{(m_i)}$, $\mathcal{M}^{(m_i)}$ is $\mathcal{M}$ with respect to the family $\{\overline{J}^\Diamond_\tau(\varepsilon_i,\delta_i,{\frak p}(\tau);m_i)\}$, and $\overline{u}_{i\infty}$ falls into one of Cases (1), (2$_0$), (3)--(6$_i$).  We consider a diagonal subsequence $\overline{u}_i:= \overline{u}_{ij(i)}$.

Cases (1) and (2$_0$) are eliminated by an argument similar to that of Case (2) of Theorem~I.\ref{P1-thm: complement} and Cases (3)--(5) are eliminated by an argument similar to that of Cases (3)--(6) of Theorem~I.\ref{P1-thm: complement}.

\s\n {\em Case} (6$_i$). Suppose for simplicity that $b=1$, including connectors.

As in the proof of Theorem~I.\ref{P1-thm: complement}, for $i\gg 0$, we consider a truncation $\widetilde{u}_i:\Sigma_i\to \overline{W}_{\tau_i}$ of $\overline{u}_i$ (i.e., a restriction of $\overline{u}_i$ to a neighborhood of $\sigma_\infty^{\tau_i}$) such that there exist real numbers
$$R_{0,i}<R_{1,i}<R_{2,i}< R_{3,i}, \quad R_{0,i}\ll -l(\tau_i), \quad R_{3,i}\gg l(\tau_i)$$
and a map
$$\pi_i: \Sigma_i\to B_{\tau_i}\cap \{R_{0,i}\leq s\leq R_{3,i}\},$$
such that the restriction of $\pi_i$ to $\pi_i^{-1}(\{R_{j,i}\leq s\leq R_{j+1,i}\})$, $j=0,1,2$, is a degree $p_{0,j}$ branched cover. Here $p_{0,0}\leq p_{0,1}\leq p_{0,2}$. We project $\widetilde{u}_i$ to $D^2_{\rho_0}\subset \overline{S}$ for $\rho_0>0$ small using balanced coordinates and then apply the ansatz from Equation~(I.\ref{P1-eqn: ansatz}) to obtain $w_i: \Sigma_i\to\C$.

Applying the method of Section~I.\ref{P1-subsection: rescaling}, we rescale $w_i$ by a positive real constant and take the limit $m_i\to \infty$ to obtain a $3$-level holomorphic building
$$w_\infty=w_-\cup w_{0,1}\cup w_+.$$
We write $w_*:\Sigma_*\to \C\P^1$ for the components of the building $w_\infty$ and $\pi_*: \Sigma_*\to cl(B_*)$ for the corresponding branched covers, where $*=+$, $-$, or $(0,1)$, and $B_{0,1}=\R\times S^1$. We may also use subscripts $(0,0)$ and $(0,2)$ to mean $-$ and $+$. Note that $\deg \pi_{0,j}=p_{0,j}$. By Lemma~\ref{banff}, the surface $\Sigma_\infty$, obtained by gluing the $\Sigma_*$, is connected and planar.  The next few paragraphs are devoted to the description of $w_*$ and $\pi_*$.

Suppose for simplicity that $\Sigma_-$, $\Sigma_{0,1}$, and $\Sigma_+$ are connected.  Starting from the bottom, $w_-$ and $\pi_-$ satisfy the following:
\begin{enumerate}
\item[(i$_-$)] $ w_-(\bdry\Sigma_-)\subset \{\phi=0,\rho>0\}$;
\item[(ii$_-$)] $\pi_-^{-1}(+\infty)$ is a single point and $ w_-(\pi^{-1}_-(+\infty))=\infty$;
\item[(iii$_-$)] $ w_-(z_0)=0$ for one of the points $z_0\in \pi^{-1}_-(\overline{\frak m}^b(\infty))$; and
\item[(iv$_-$)] $ w_-|_{int(\Sigma_-)}$ is a biholomorphism onto its image.
\end{enumerate}
(i$_-$) and (iii$_-$) are clear, (ii$_-$) follows from the fact that $\Sigma_\infty$ (and hence $\Sigma_-$) is connected and planar, and (iv$_-$) is a consequence of Equation~\eqref{sum of n}. Let us write $f_*=\pi_*\circ w_*^{-1}$, where defined. Then $f_-$ maps the asymptotic marker $\dot{\mathcal{R}}_\pi(\infty)$ for $\infty\in \C\P^1$, corresponding to the radial ray $\mathcal{R}_\pi$, to the asymptotic marker $\dot{\mathcal{L}}_{3/2}(+\infty)$ for $+\infty\in cl(B_-)$, corresponding to the half-line $\mathcal{L}_{3/2}$, by the Involution Lemma~I.\ref{P1-lemma: effect of involutions}.

We then move up to $w_{0,1}$, which satisfies the following:
\begin{enumerate}
\item[(i$_{0,1}$)] $w_{0,1}(\pi^{-1}_{0,1}(-\infty))= \{d_1=0,d_2,\dots, d_k\}$, where $d_i\in \R^{\geq 0}$ and $d_i<d_{i+1}$;
\item[(ii$_{0,1}$)] $\pi_{0,1}^{-1}(+\infty)$ is a single point and $w_{0,1}(\pi^{-1}_{0,1}(+\infty))=\infty$;
\item[(iii$_{0,1}$)] $w_{0,1}$ is a biholomorphism; and
\item[(iv$_{0,1}$)] $f_{0,1}: \C\P^1\to cl(B_{0,1})$ maps $\dot{\mathcal{R}}_\pi(0)$ to $\dot{\mathcal{L}}_{3/2}(-\infty)$.
\end{enumerate}
The placement of the points $d_2,\dots,d_k$ in (i$_{0,1}$) follows from observing that $\overline{v}_-^\sharp$ consists of $I=1$ components from $\delta_0$ to $x_i$ or $x_i'$. In particular, the asymptotic eigenfunctions of $\overline{v}_-^\sharp$ at the positive end are close to constant functions with values on $\R^+\subset \C$. (ii$_{0,1}$) follows from the fact that $\Sigma_\infty$ is connected and planar and (iii$_{0,1}$) follows from Equation~\eqref{sum of n}. (iv$_{0,1}$) is a consequence of the fact that $f_-$ maps $\dot{\mathcal{R}}_\pi(\infty)$ to $\dot{\mathcal{L}}_{3/2}(+\infty)$.  Then $f_{0,1}$ maps $\dot{\mathcal{R}}_\pi(\infty)$ to $\dot{\mathcal{L}}_{3/2}(+\infty)$ by the Involution Lemma~I.\ref{P1-lemma: effect 2}.

We now describe the construction of $w_+$ in some detail. By translating
$$\pi_i^{-1}(B_{\tau_i}\cap \{R_{2,i}\leq s\leq R_{3,i}\})$$ down by $T_i$, we obtain the branched cover
$$\pi_i^+: \Sigma_i^+\to B_+\cap \{R_{2,i}-T_i\leq s\leq R_{3,i}-T_i\}$$ and the holomorphic map $w_i^+: \Sigma_i^+\to\C$. We may assume that $R_{2,i}-T_i\to -\infty$ and $R_{3,i}-T_i\to +\infty$ have been chosen so that there is no sequence of branch points of $\pi_i^+$ that limits to $s=\pm\infty$ as $i\to\infty$. Indeed, the branch points that ``escape to $s=\pm\infty$'' properly belong to a different level. Then $w_+$ is the limit of $w_i^+$, after suitably rescaling by positive real constants.

Let $\overrightarrow{\mathcal{D}}=\{(i_k',j_k')\to (i_k,j_k)\}_{k=1}^p$ be the data at the positive end of $w_+$.  For $i\gg 0$, the component of $w_i^+|_{(\pi_i^+)^{-1}(\{s=R_{3,i}-T_i\})}$ corresponding to $(i_k',j_k')\to(i_k,j_k)$ is arbitrarily close to a normalized asymptotic eigenfunction $\phi_k$ from $\overline{\hh}(\overline{a}_{i_k',j_k'})$ to $\overline{a}_{i_k,j_k}$, after multiplying by some positive real constant. (Here $\phi_k$ is an eigenfunction of an asymptotic operator and we are not making any {\em a priori} assumptions on the corresponding eigenvalues.) Now we claim that some $\phi_{k_0}$ must sweep out a large sector of $D^2$, since otherwise $\op{Im}(w_i^+)$ cannot pass through the origin, contradicting the ``continuity'' to $w_{0,1}$. In the limit $i\to\infty$, $\phi_{k_0}$ will sweep out an angle of $2\pi$.

The top level $w_+$ satisfies the following:
\begin{enumerate}
\item[(i$_+$)] $w_+(\pi^{-1}_+(-\infty))= \{d^+_1=0,d^+_2,\dots, d^+_{k_+}\}$, where $d^+_i\in \R^{\geq 0}$ and $d^+_i<d^+_{i+1}$;
\item[(ii$_+$)] $w_+(\bdry \Sigma_+) \subset \{\phi=0, 0<\rho\leq \infty\}$;
\item[(iii$_+$)] $w_+$ maps some point of $\pi_+^{-1}(+\infty)$ to $\infty$;
\item[(iv$_+$)] $f_+$ maps $\dot{\mathcal{R}}_\pi(0)$ to $\dot{\mathcal{L}}_{3/2}(-\infty)$; and
\item[(v$_+$)] $w_+$ is a biholomorphism onto its image.
\end{enumerate}
The placement of the points $d_2^+,\dots,d_{k_+}^+$ in (i$_+$) follows from observing that $\overline{v}_{0,1}^\sharp$ consists of $I=1$ components from $\delta_0$ to $h$ and that $\mathcal{R}_{\phi_h}\to \mathcal{R}_0$ as $m\to \infty$ by our choice of $h$ from Section~\ref{subsubsection: convention bambi}. (ii$_+$) and (iii$_+$) are immediate consequences of the construction of $w_+$ from the previous paragraphs. (iv$_+$) is a consequence of the fact that $f_{0,1}$ maps $\dot{\mathcal{R}}_\pi(\infty)$ to $\dot{\mathcal{L}}_{3/2}(+\infty)$.  We now apply the Involution Lemma~I.\ref{P1-lemma: effect 3} to conclude that $f_+(\infty)=\mathcal{L}_{3/2}\cap \bdry B_+$. Note that the Lemma~I.\ref{P1-lemma: effect 3} applies because the compactification of $\Sigma_+$ is a closed disk by Lemma~\ref{banff}. This contradicts (ii$_+$).  We have eliminated Case (6$_i$).
\end{proof}

We are now in a position to prove Lemma~\ref{cherries4}.

\begin{proof}[Proof of Lemma~\ref{cherries4}]
Suppose $\overline{u}_\infty\in \bdry_{\{+\infty\}}\mathcal{M}$. By Lemma~\ref{los angeles}, if $\overline{v}'_*\cup\overline{v}^\sharp_*=\varnothing$ for all levels $\overline{v}_*$ of $\overline{u}_\infty$, then $\overline{u}_\infty\in A_1$.  If $\overline{v}'_*\cup\overline{v}^\sharp_*\not =\varnothing$ for some level $\overline{v}_*$, then $\overline{u}_\infty$ is as given in Lemma~\ref{hojicha}. Now, by Lemma~\ref{vancouver}, the only possibility left is Case (2$_1$), which implies that $\overline{u}_\infty\in A_2$.
\end{proof}

\subsection{Degeneration at $-\infty$}
\label{subsection: chain homotopy part 3}

In this subsection we study the limit of holomorphic maps to $\overline{W}_{\tau}$ as $\tau \to -\infty$, i.e., when $\overline{W}_{\tau}$ degenerates into $\overline{W}_{-\infty,1}\cup\overline{W}_{-\infty,2}$. This will prove Lemma \ref{cherries5}.

We assume that $m\gg 0$; $\varepsilon,\delta>0$ are sufficiently small; and $\{\overline{J}_\tau\}\in \overline{\mathcal{I}}^{reg}$ and $\{\overline{J}_\tau^\Diamond(\varepsilon,\delta,{\frak p}(\tau))\}$ satisfy Lemma~\ref{lemma: codimension one}.  Fix ${\bf y}\in \mathcal{S}_{{\bf a},\hh({\bf a})}$, ${\bf y}'\in\mathcal{S}_{{\bf b},\hh({\bf b})}$ and let
$$\mathcal{M}=\mathcal{M}^{I=2,n^*=m}_{\{\overline{J}_\tau^\Diamond(\varepsilon,\delta,{\frak p}(\tau))\}}({\bf y},{\bf y}';\overline{\frak m}), \quad \mathcal{M}_\tau= \mathcal{M}^{I=2,n^*=m}_{\overline{J}_\tau^\Diamond(\varepsilon,\delta,{\frak p}(\tau))}({\bf y},{\bf y}';\overline{\frak m}(\tau)).$$
We will analyze $\bdry_{\{-\infty\}}\mathcal{M}$.

Let $\overline{u}_i$, $i\in \N$, be a sequence of curves in $\mathcal{M}$ such that $\overline{u}_i\in \mathcal{M}_{\tau_i}$ and $\displaystyle\lim_{i\to\infty} \tau_i=-\infty$, and let
$$\overline{u}_\infty=\overline{v}_2\cup (\overline{v}_{L,1}\cup\dots\cup \overline{v}_{L,a}) \cup \overline{v}_1
\cup(\overline{v}_{R,1}\cup\dots\cup \overline{v}_{R,b})$$
$$\cup(\overline{v}_{B,1}\cup\dots\cup \overline{v}_{B,c})\cup (\overline{v}_{T,1}\cup\dots\cup \overline{v}_{T,d}),$$
be the limit holomorphic building, where each $\overline{v}_*$ is an SFT-type level (now we are stretching sideways), $\overline{v}_j$ maps to $\overline{W}_{-\infty,j}$, $j=1,2$; $\overline{v}_{L,j}$ and $\overline{v}_{R,j}$ map to $[-2,2]\times\R\times \overline{S}$; $\overline{v}_{B,j}$ and $\overline{v}_{T,j}$ map to $\R\times[0,1]\times\overline{S}$. The levels in the first row are arranged in cyclic order from left to right, the levels in the second row are arranged in order from bottom to top, and $\overline{v}_1$ is between $\overline{v}_{B,c}$ and $\overline{v}_{T,1}$. For notational convenience we refer to $\overline{v}_1$ as $\overline{v}_{L,a+1}$, $\overline{v}_{R,0}$, $\overline{v}_{B,c+1}$, or $\overline{v}_{T,0}$, and $\overline{v}_2$ as $\overline{v}_{L,0}$ or $\overline{v}_{R,b+1}$. As usual, we write $p_*=\deg \overline{v}_*'$.

\s\n
{\em Terminology.} Thin counterclockwise sectors in $\C$ from $\overline{b}_{i,j}$ to $\overline{a}_{i,j}$ and from $\overline{\hh}(\overline{b}_{i,j})$ to $\overline{\hh}(\overline{a}_{i,j})$ will also be referred to as {\em thin sectors}.

\s\n
{\em Outline of proof of Lemma~\ref{cherries5}.} The initial steps of the proof are similar to those of Section~I.\ref{P1-subsection: intersection numbers}--I.\ref{P1-proof of lemma} and Section~\ref{subsection: chain homotopy part 1}.  However, the authors were unable to prove that $I(\overline{v}_{L,j})$ and $I(\overline{v}_{R,j})$ were nonnegative when $\overline{v}_1'\not=\varnothing$, since we could not sufficiently control the groomings in order to apply the ECH index inequality (Lemma~I.\ref{P1-index inequality for z infinity case}). Note that Lemma~\ref{nonnegative ECH indices 3}(4) specifically excludes the case $\overline{v}_1'\not=\varnothing$. Sections~\ref{truncation}--\ref{second case} are intended as a substitute, and involve ideas from tropical geometry (see for example Parker~\cite{Pa}) since we are stretching simultaneously in two directions.

\subsubsection{Continuation argument}

We discuss the continuation argument in the current case; this is similar to but more complicated than those of Sections~\ref{orange1} and \ref{boundpunc}. Suppose that $\overline{v}'_*\cup\overline{v}^\sharp_*\not =\varnothing$ for some level $\overline{v}_*$ of $\overline{u}_\infty$.  For simplicity we assume that there are no boundary points at $z_\infty$; we leave it to reader to make the appropriate modifications when there are boundary points at $z_\infty$.

\s\n {\em Case 1.}
Suppose that $\overline{v}_{T,j}'\cup \overline{v}_{T,j}^\sharp\not=\varnothing$ for some $j>0$. We start at a nontrivial negative end $\mathcal{E}_1$ of some $\overline{v}_{T,j_1}$, $j_1>0$, limiting to $z_\infty$. We consider the continuation
$$g^1_{j_1-1,1},\dots,g^1_{1,1},g_{-a-1,1},\dots,g_{b+1,1},g^0_{1,1},\dots,g^0_{d,1}$$
of the $t=1$ boundary of $\mathcal{E}_1$ in the direction of $\bdry_+B_\tau$. Here Definition~\ref{defn: continuation} needs to be adapted in the obvious way to the breaking of a component of $\overline{u}_i|_{\bdry \dot G_i}$ as $\tau_i\to -\infty$, where $\dot G_i$ is the domain of $\overline{u}_i$. The components $g_{-a-1,1}$ and $g_{b+1,1}$ correspond to $\overline{v}_1$ and the components $g_{j,1}$, $j=-a,\dots,b$, correspond to $\overline{v}_{L,j}$, $j=a,a-1,\dots,0$, and $\overline{v}_{R,j}$, $j=b,b-1,\dots,1$, in that order. There are three possibilities:
\begin{enumerate}
\item[(i)] there is some $g^1_{j,1}$, $0\leq j\leq j_1-1$, which is nontrivial;
\item[(ii)] all the $g^1_{j,1}$ are trivial but some $g_{j,1}$ is nontrivial;
\item[(iii)] all the $g^1_{j,1}$ and $g_{j,1}$ are trivial.
\end{enumerate}
Cases (i) and (iii) have already been treated in the proof of Lemma~\ref{intersezione prime}. If we are in Case (ii), then the nontrivial component $g_{j,1}$ contains the $s=2$ boundary of a right end $\mathcal{E}_2$ that limits to $z_\infty$. We then consider the continuation of the $s=-2$ boundary of $\mathcal{E}_2$ in the direction of $\bdry_-B_\tau$. The details of the continuation are left to the reader, but in the end the sectors $\pi_{D^2_{\rho_0}}(\mathcal{E}_i)$ will sweep out a neighborhood of $z_\infty$ with the exception of thin sectors.

\s\n {\em Case 2.}
Suppose that $\overline{v}_{T,j}'\cup\overline{v}_{T,j}^\sharp=\varnothing$ for all $j>0$. Then $\overline{v}_1'=\varnothing$. Consider the horizontal levels:
\begin{equation} \label{horizontal}
\overline{v}_1=\overline{v}_{R,0},\dots,\overline{v}_{R,b},\overline{v}_2,\overline{v}_{L,1}\dots,\overline{v}_{L,a+1}=\overline{v}_1.
\end{equation}
Suppose that $\overline{v}_*'\cup \overline{v}_*^\sharp\not=\varnothing$ for some horizontal level $\overline{v}_*$. Let $\overline{v}_*$ be the leftmost level in Equation~\eqref{horizontal} such that $\overline{v}_*^\sharp$ has a right end $\mathcal{E}_1$ at $z_\infty$. The sector  $\pi_{D^2_{\rho_0}}(\mathcal{E}_1)$ is not a thin sector and we apply the usual continuation argument to the levels of Equations~\eqref{horizontal}.

\s\n {\em Case 3.}
The case where $\overline{v}_{T,j}'\cup\overline{v}_{T,j}^\sharp=\varnothing$ for all $j>0$, $\overline{v}_*'\cup \overline{v}_*^\sharp=\varnothing$ for all horizontal levels $\overline{v}_*$, and $\overline{v}_{B,j}'\cup\overline{v}_{B,j}^\sharp\not=\varnothing$ for some $j\leq c$ is treated similarly.

\s
Let $\mathcal{Z}=({\frak z}_1\to \dots \to {\frak z}_k\to {\frak z}_1)$ be the cycle constructed using the above continuation argument. Boundary points of type (P$_1$), (P$_2$), and (P$_3$) are defined in the same way.

\subsubsection{Some restrictions on $\overline{u}_\infty$}

The following are analogs of Lemmas~\ref{intersezione prime revisited}, \ref{nonnegative ECH indices better} and \ref{sencha}.

\begin{lemma} \label{intersezione three}
Suppose that $\overline{v}'_*\cup\overline{v}^\sharp_*\not =\varnothing$ for some level $\overline{v}_*$ of $\overline{u}_\infty$.  If $\mathcal{E}_i$, $i=1,\dots,q$, are the ends of all the $\overline{v}^\sharp_*$ that converge to $z_\infty$ and $\mathcal{E}_i'$, $i=1,\dots,r$ are the neighborhoods of the boundary points of type (P$_3$), then
$$\sum_{i=1}^q n^*(\mathcal{E}_i)+\sum_{i=1}^r n^*(\mathcal{E}'_i)\geq m-2g.$$
\end{lemma}

\begin{lemma} \label{nonnegative ECH indices 3}
If fiber components are removed from $\overline{u}_\infty$ and the only boundary points at $z_\infty$ are of type (P$_3$), then the following hold:
\begin{enumerate}
\item the ECH index of each $\overline{v}''_*$ is nonnegative;
\item the only components of $\overline{u}_\infty$ which have negative ECH index are those of $\overline{v}_1'$, i.e., the branched covers of $\sigma_\infty^{-\infty,1}$;
\item the ECH index of each level $\overline{v}_{T,j},\overline{v}_{B,j}\not=\overline{v}_1$ is nonnegative;
\item if $\overline{v}'_1=\varnothing$, then the ECH index of each $\overline{v}_{L,j},\overline{v}_{R,j}\not=\overline{v}_1,\overline{v}_2$ is nonnegative;
\item there is an additional contribution of $\mathfrak{bp}_*+1$ towards $I$, where $\mathfrak{bp}_*$ is the number of boundary points of type (P$_3$) on $\overline{v}_*$.
\end{enumerate}
\end{lemma}

\begin{proof}
(1) and (2) are consequences of the index inequality. (3) and (4) follow from the proof of Lemma~\ref{nonnegative ECH indices}. In the proof of (4), the vertical levels $\overline{v}_{1,a},\dots,\overline{v}_{1,0}$ and the arcs $\overline{\bf a}$, $\overline{\hh}(\overline{\bf a})$ are replaced by the horizontal levels given by Equation~\eqref{horizontal} and the arcs $\overline{\bf b}$, $\overline{\bf a}$. (5) is argued in the same way as Lemma~\ref{nonnegative ECH indices better}.
\end{proof}

\begin{lemma} \label{los angeles 3}
If $\overline{u}_\infty\in \bdry_{\{-\infty\}}\mathcal{M}$ and $\overline{v}'_*\cup\overline{v}^\sharp_*=\varnothing$ for all levels $\overline{v}_*$ of $\overline{u}_\infty$, then $\overline{u}_\infty$ satisfies the following: $a=b=c=d=0$; $I(\overline{v}_1)=\op{ind}(\overline{v}_1)=0$ and $I(\overline{v}_2)=\op{ind}(\overline{v}_2)=2$; and $\overline{v}_1$ is a $W_{-\infty,1}$-curve and $\overline{v}_2$ is a $\overline{W}_{-\infty,2}$-curve.
\end{lemma}

\begin{proof}
Similar to that of Lemma~\ref{los angeles} and is omitted.
\end{proof}

The following is the analog of Lemma~\ref{sencha}.

\begin{lemma} \label{sencha 3}
If $\overline{v}'_*\cup\overline{v}^\sharp_*\not =\varnothing$ for some level $\overline{v}_*$ of $\overline{u}_\infty$, then:
\begin{enumerate}
\item $p_2=\deg(\overline{v}_2')>0$;
\item there is no boundary point of type (P$_1$) or (P$_2$), i.e., the only boundary points at $z_\infty$ are of type (P$_3$);
\item $\overline{u}_\infty$ has no fiber components and no components $\overline{v}''_*$ that intersect the interior of a section at infinity;
\item each of $\overline{v}_{L,j}$, $j=1,\dots,a$, and $\overline{v}_{R,j}$, $j=1,\dots,b+1$, consists of thin strips and trivial strips; in particular, $\overline{v}_{L,j}$, $j=1,\dots,a$, and $\overline{v}_{R,j}$, $j=1,\dots,b+1$ have no boundary points at $z_\infty$.
\end{enumerate}
\end{lemma}

\begin{proof}
(1), (2), (3) We have the following contributions towards $n^*$: (1) the restriction of $\overline{v}_2''$ to a neighborhood of $\overline{\frak m}(-\infty)$ contributes $m$ if  $\overline{v}_2'=\varnothing$; (2) boundary points of type (P$_1$) or (P$_2$) contribute at least $m$ in total; and (3) a fiber component or a component of $\overline{v}''_*$ that intersects a section at infinity contributes at least $m$. All cases contradict Lemma~\ref{intersezione three}.

(4) Arguing by contradiction, suppose that some level $\overline{v}_{R,j_0}$, $1\leq j_0\leq b$, has a component which is not a thin strip or a trivial strip. (The case of $\overline{v}_{L,j_0}$, $0\leq j_0\leq a$, only differs in notation.) Here $\overline{v}_{R,j_0}$ may have a boundary point of type (P$_3$).

We claim that the following hold:
\begin{enumerate}
\item[(a)] $n^*(\overline{v}_{R,j_0})= m$ and $n^{*,alt}(\overline{v}_{R,j_0})\geq m-2g$;
\item[(b)] $\overline{v}^\sharp_{R,j_0}\not=\varnothing$ and some end of $\overline{v}_{R,j_0}^\sharp$ corresponds to a large sector;
\item[(c)] $\overline{v}_1'=\varnothing$ and $\overline{v}_{T,j}^\sharp=\varnothing$ for all $j>0$;
\item[(d)] no left end of $\overline{v}_{R,j_0}$ limits to a multiple of $z_\infty$; in particular $\overline{v}'_{R,j_0}=\varnothing$;
\item[(e)] there are no boundary points of type (P$_3$);
\item[(f)] the ECH index of each $\overline{v}_{L,j}$, $1\leq j\leq a$, and $\overline{v}_{R,j}$, $1\leq j\leq b+1$ is nonnegative and $I(\overline{v}_{R,j_0})\geq 2$.
\end{enumerate}
Recall the definition of $n^{*,alt}$ from Section~\ref{subsubsection: convention bambi}.

We first prove (a). Consider the projection
$$\pi_{\overline{S}}:[-2,2]\times\R\times\overline{S}\to \overline{S}.$$
Since we are only dealing with compactness issues, we may assume without loss of generality that $\overline{J}_{-\infty,2}^\Diamond(\varepsilon,\delta,{\frak p}(-\infty))$ is a product complex structure and the projection $\pi_{\overline{S}}$ is holomorphic. Since $\overline{v}_{R,j_0}$ has a component which is not a thin strip or a trivial strip, $\op{Im}(\pi_{\overline{S}}\circ\overline{v}_{R,j_0})$ must contain the complement of all the thin strips between the $\overline{b}_i$ and the $\overline{a}_i$. This proves (a).

We now prove (b)--(f). If (c) does not hold, then there is some negative end $\mathcal{E}$ of $\cup_{j=1}^a\overline{v}_{T,j}^\sharp$ that limits to $z_\infty$ and $n^*(\mathcal{E})\geq k_0-1\gg 2g$. This is a contradiction of (a). If (d) does not hold, i.e., a left end of $\overline{v}_{R,j_0}$ limits to a multiple of $z_\infty$, then in view of (c) there is some right end $\mathcal{E}$ of $\overline{v}^\sharp_*$ to the left of $\overline{v}_{R,j_0}$ that limits to $z_\infty$ and satisfies $n^*(\mathcal{E})\geq k_0-1\gg 2g$. This is a contradiction of (a). (e) is a consequence of (d). (b) follows from (a) and (e) by excluding some possibilities using (2), (3). Finally we prove (f).  By Lemma~\ref{nonnegative ECH indices 3}(4) and (c), the ECH indices of $\overline{v}_{L,j}$, $j=1,\dots,a$, and $\overline{v}_{R,j}$, $j=1,\dots,b+1$, are nonnegative.  Since $\overline{v}^\sharp_{R,j_0}$ has a large sector by (b), its ECH index is increased by one.  Hence $I(\overline{v}_{R,j_0})\geq 2$.

We now return to the proof of (4).
By (d), the right end of $\overline{v}_{R,j_0}$ limits to a multiple of $z_\infty$. By (a), each component of $\overline{v}_*^\sharp\not=\overline{v}_1$ to the right of $\overline{v}_{R,j_0}$ must be a thin strip and the projection to $\overline{S}$ of the union of all the ends of $\overline{v}_1^\sharp$ limiting to $z_\infty$ is a union of thin wedges of type ${\frak S}(\overline{b}_{i,j},\overline{a}_{i,j})$.  Using Figure~\ref{figure: aandb} one can verify that such a curve $\overline{v}_1^\sharp$ does not exist.  This implies that $\overline{v}_1^\sharp=\varnothing$ and that the ECH index of $\overline{v}_1$ is $\geq 0$. Finally, since the ECH index of each thin strip is $1$, the analog of Equation~\eqref{sum of I} for our case, Lemma~\ref{nonnegative ECH indices 3}, and (f) together imply that no component of $\overline{v}_*\not=\overline{v}_1$ to the right of $\overline{v}_{R,j_0}$ can be a thin strip.  This is a contradiction and (4) follows.
\end{proof}

\subsubsection{Truncations} \label{truncation}

In the rest of this section (until the proof of Lemma~\ref{alishan2}), we consider the case where $\overline{v}'_*\cup \overline{v}_*^\sharp\not=\varnothing$ for some $\overline{v}_*$. We have $\overline{v}'_2\not=\varnothing$ in view of Lemma~\ref{sencha 3}(1).

In the next few subsections we consider the case $\overline{v}'_1=\varnothing$. By Lemma~\ref{sencha 3}(4), $\overline{u}_\infty$ has no boundary points at $z_\infty$.

\begin{defn}
Let $X$ be set and $\varepsilon>0$ be a positive real number. Then two functions $f,g:X\to \C^\times$ are {\em $\varepsilon$-approximate} if
$$|f(x)-g(x)|< \varepsilon\cdot |g(x)| \quad \mbox{ and } \quad |f(x)-g(x)|< \varepsilon\cdot |f(x)|$$ for all $x\in X$.
\end{defn}

We now define the following sequence of truncations, analogous to those which appear in the proof of Theorem~I.\ref{P1-thm: complement}.

\begin{defn} \label{defn: truncation}
With respect to the above assumptions on $\overline{u}_i: \dot G_i\to \overline{W}_{\tau_i}$, an {\em $\varepsilon_i$-truncation of $\overline{u}_i$} with $0<\varepsilon_i< {\rho_0\over 2}$ is the restriction of $\overline{u}_i$ to a subsurface $\Sigma_i\subset \dot G_i$ which satisfies the following:  First write
\begin{align*}
\bdry_v \Sigma_i &= \bdry\Sigma_i-(\pi_{B_{\tau_i}}\circ \overline{u}_i)^{-1}(\bdry B_{\tau_i})\\
\bdry_h \Sigma_i &= \bdry\Sigma_i\cap(\pi_{B_{\tau_i}}\circ \overline{u}_i)^{-1}(\bdry B_{\tau_i}).
\end{align*}
Then
\begin{enumerate}
\item[(T1)] $\overline{u}_i(\Sigma_i)$ is contained in a $2\varepsilon_i$-neighborhood of $\sigma_\infty^{\tau_i}$;
\item[(T2)] if $\overline{u}_i(x)$ is contained in an ${\varepsilon_i\over 2}$-neighborhood of $\sigma_\infty^{\tau_i}$, then $x\in \Sigma_i$;
\item[(T3)] $\pi_{B_{\tau_i}}\circ \overline{u}_i$ maps each component $c$ of $\bdry_v\Sigma_i$ to some $t=\mbox{const}$ and $\bdry_h \Sigma_i$ to $\{s=\pm 2\}$;
\item[(T4)] there exist constants $r(\tau_i)-2>R_i^{(1)}> \dots > R_i^{(\iota+1)}> 2$ and a decomposition $\Sigma_i=\Sigma_i^{(1)}\cup\dots \cup \Sigma_i^{(\iota)}$ such that $$\Sigma_i^{(j)}=(\pi_{B_{\tau_i}}\circ \overline{u}_i)^{-1}([-2,2]\times [R_i^{(j+1)}, R_i^{(j)}])\cap \Sigma_i$$ and each $$\pi_{B_{\tau_i}}\circ \overline{u}_i:\Sigma_i^{(j)} \to [-2,2]\times [R_i^{(j+1)}, R_i^{(j)}]$$ is a branched cover with possibly empty branch locus;
\item[(T5)] for each component $c$ of $\bdry_v\Sigma_i$,  $\pi_{D^2_{\rho_0}}\circ \overline{u}_i|_{c}$ is ${\varepsilon_i\over 10}$-approximate to a positive multiple $\varepsilon_i$ of a {\em normalized} eigenfunction of $\overline{v}_{L,j}^\sharp$, $j=0,\dots,a$, or $\overline{v}_{R,j}^\sharp$, $j=0,\dots,b$; moreover, for all $c$, $\max_c |\pi_{D^2_{\rho_0}}\circ \overline{u}_i|=\varepsilon_i$.
\end{enumerate}
Here
$$\pi_{D^2_{\rho_0}}:\pi_{B_{\tau_i}}^{-1}([-2,2]\times[R_i^{(\iota+1)},R_i^{(1)}])\cap \{\rho\leq \rho_0\}\to D^2_{\rho_0}$$
is obtained by projecting out the $\bdry_s$- and $\overline{R}_{\tau_i}$-directions.
\end{defn}

{\em From now on we assume that $\overline{u}_i|_{\Sigma_i}$ is a sequence of $\varepsilon_i$-truncations, where $\varepsilon_i\to 0$ and $R_i^{(1)}-R_i^{(\iota+1)}\to\infty$ as $i\to\infty$.}

Let $\widetilde\pi_i$ be the map obtained by postcomposing
$$\pi_{B_{\tau_i}}\circ \overline{u}_i: \Sigma_i\to [-2,2]\times[R_i^{(\iota+1)},R_i^{(1)}]$$
with a $-{r(\tau_i)+1\over 2}$-translation in the $t$-direction, and let $\widetilde{w}_i=\pi_{D^2_{\rho_0}}\circ \overline{u}_i|_{\Sigma_i}$. Also let
$$R_i^+=R_i^{(1)}-{r(\tau_i)+1\over 2}, \quad R_i^-=R_i^{(\iota+1)}-{r(\tau_i)+1\over 2}.$$

\s\n
{\em Notation.} When we want to distinguish the $t$-coordinates for $B_{-\infty,1}$ and $B_{-\infty,2}$, we write $t_i$ for the $t$-coordinate for $B_{-\infty,i}$. Note that $\widetilde\pi_i$ can be viewed as a map to $B_{-\infty,2}$ with $(s,t_2)$-coordinates.

\begin{lemma} \label{otter2}
If $i\gg 0$, then $\Sigma_i$ is a disk with $\geq 2p_2$ boundary punctures.
\end{lemma}

\begin{proof}
This is a consequence of (T5) and the following:
\begin{enumerate}
\item[(i)] $n^*(\overline{u}_i|_{\Sigma_i})=m$;
\item[(ii)] $\widetilde{w}_i(z_0)=0$ for some $z_0\in (\pi_{B_{\tau_i}}\circ \overline{u}_i)^{-1}(\overline{\frak m}^b(\tau_i))$;
\item[(iii)] $\widetilde{w}_i$ maps each component of $\widetilde\pi_i^{-1}(\{s=2\})$ to a different $\mathcal{R}_{\phi(\overline{a}_*)}$ and each component of $\widetilde\pi_i^{-1}(\{s=-2\})$ to a different component of $\mathcal{R}_{\phi(\overline{b}_*)}$;
\item[(iv)] $\widetilde{w}_i|_{int(\Sigma_i)}$ is a biholomorphism onto its image.
\end{enumerate}
Here $\phi(\overline{a}_*)$ is the $\phi$-coordinate for $\overline{a}_*$, etc.
\end{proof}

\subsubsection{Large-scale behavior of $\Xi_i$} \label{orange2}

When taking the limits of $\widetilde\pi_i$ and $\widetilde w_i$, we are simultaneously stretching in two directions $t_2$ and $\log \rho$.  In this subsection we study the large-scale behavior of the map
$$\Xi_i=(t_2\circ \widetilde\pi_i,\log\rho\circ \widetilde w_i):\Sigma_i\to [R_i^-,R_i^+]\times [-\infty,\infty)$$
as $i\to \infty$.  We use coordinates $(x'=t_2,y')$ on $[R_i^-,R_i^+]\times[-\infty,\infty)$.
The goal is to construct a ``tropical curve''
$$\overline\Xi_i:\Gamma_i\to [-1,1]\times[0,d_i],$$
with some $d_i\geq 1$ which approximates $\Xi_i$ when viewed from ``far away''; see Figure~\ref{figure: interval-contraction}.  Here $\Gamma_i$ is a finite graph whose topological type is independent of $i\gg 0$.  The analysis is of the same type as that of Parker~\cite{Pa}.

\s\n {\em Step 1.}
We start with the following lemma, which is a consequence of Gromov compactness and which describes the behavior of $\widetilde{w}_i$ and $\widetilde{\pi}_i$  for large $i$.

\begin{lemma} \label{compact sets}
Given $\varepsilon>0$ small, after passing to a subsequence and possibly shrinking $\varepsilon_i>0$ subject to the condition $R_i^{(1)}-R_i^{(\iota+1)}\to\infty$, there exist constants $L>0$, $\kappa,\kappa'\in \Z^+$ such that for each $i$ there exist:
\begin{itemize}
\item disjoint compact subsurfaces $K_{i1},\dots,K_{i\kappa}\subset \Sigma_i$ and
\item components $C_{i1},\dots,C_{i\kappa'}$ of $\Sigma_i-\cup_j K_{ij}$ which are strips
\end{itemize}
such that:
\begin{enumerate}
\item $\widetilde \pi_i|_{K_{ij}}$ is a branched cover over $[-2,2]\times [\tau_{ij},\tau_{ij}+L]$ for some $\tau_{ij}$ and $\widetilde \pi_i$ has no branch points outside $K_{i1},\dots,K_{i\kappa}$;
\item $K_{ij}$ is disjoint from $\bdry_v \Sigma_i$;
\item there is a component $K_{ij_0}$ such that  $(0,-\infty)\in \Xi_i(K_{ij_0})$;
\item for each $j$, the sequence $\{\widetilde w_i|_{K_{ij}}\}_{i=1}^\infty$, after rescaling by positive constants, limits to some $\widetilde w_{\infty j}:K_{\infty j}\to \C$;
\item $\widetilde w_i|_{C_{ij}}$ is $\varepsilon$-approximate to a multiple of $e^{\lambda_{ij} (t-is)}$, where $\lambda_{ij}\in \R-\{0\}$ is $\pm {1\over 4}$ times the angle of a sector of type ${\frak S}(\overline{b}_{k,l},\overline{a}_{k',l'})$ (here $\pm{1\over 4}$ comes from the fact that $s\in[-2,2]$);
\item $\Xi_i(C_{ij})$ is $\varepsilon$-close to a line segment $\{y'=\lambda_{ij} x'+\beta_{ij}~|~ x'\in t_2\circ \widetilde\pi_i(C_{ij})\}$, where $\beta_{ij}$ is a constant;
\item there exists $d_i'\in \R$ such that $y'\circ \Xi_i(c)$ is $\varepsilon$-close to $d_i'$  and $\max_{c} y'\circ \Xi_i(c)=d_i'$ for each component $c$ of $\bdry_v\Sigma_i$.
\end{enumerate}
\end{lemma}

\begin{proof}
(1)--(5) follow from Gromov compactness, once we observe using the ECH compactness theorem (cf.\ Section~I.\ref{P1-subsection: compactness PFH case}) that there is an upper bound on the number of branch points of $\widetilde\pi_i$ which is independent of $i$ and $m\gg 0$. (6) follows from (5).  (7) is a consequence of (6) and (T5) in Definition~\ref{defn: truncation}.
\end{proof}

\s\n {\em Step 2.}  We now construct the ``tropical curve'' $\overline\Xi_i:\Gamma_i\to [-1,1]\times[0,d_i].$
We first define the map $$\Xi'_i: \Gamma_i\to [R_i^-,R_i^+]\times [-\infty,\infty),$$ where $\Gamma_i=(V_{\Gamma_i}, E_{\Gamma_i})$, $V_{\Gamma_i}=V_{\Gamma_i,i}\sqcup V_{\Gamma_i,e}$ is the set of vertices, $E_{\Gamma_i}$ is the set of edges, and the following hold:
\begin{enumerate}
\item $V_{\Gamma_i,i}$ is in one-to-one correspondence with the set $\{K_{i1},\dots, K_{i\kappa}\}$ and $V_{\Gamma_i,e}$ is in one-to-one correspondence with the components of $\bdry_v \Sigma_i$;
\item the set $E_{\Gamma_i}$ is in one-to-one correspondence with the set $\{C_{i1},\dots,C_{i\kappa'}\}$;
\item $\Xi'_i$ maps the vertices corresponding to $K_{ij}$ and the component $c$ of $\bdry_v\Sigma_i$ to
$$(x',y')=(\tau_{ij}, \max( \log \rho\circ \widetilde w_i|_{K_{ij}}))\mbox{ and }(t_2\circ \widetilde \pi_i(c), \max( \log \rho\circ \widetilde w_i|_{c}));$$
and each $e\in E_{\Gamma_i}$ to a straight line segment.
\end{enumerate}
Note that $\Xi'_i$ has image in $\{ y'_{ij_0} \leq y'\leq d_i'\}$, where $y'_{ij_0}= \max( \log \rho\circ \widetilde w_i|_{K_{ij_0}})$. The map $\overline\Xi_i$ is obtained by postcomposing $\Xi'_i: \Gamma_i\to [R_i^-,R_i^+]\times [y'_{ij_0},d_i']$ by an affine transformation
$$ [R_i^-,R_i^+]\times [y'_{ij_0},d_i']\stackrel\sim\to [-1,1]\times[0,d_i],$$
where $d_i\geq 1$ is chosen so that $\max_{e\in E_{\Gamma_i}} |\lambda_i'(e)|=1$, where $\lambda_i': E_{\Gamma_i}\to \R$ maps $e$ to the slope of $\overline\Xi_i(e)$.

Let $(x,y)$ be coordinates on $[-1,1]\times [0,d_i]$.

\begin{lemma} \label{map}
Fix $\varepsilon, \delta>0$ small. Then, after passing to a subsequence, the map $\overline\Xi_i$ satisfies the following:
\begin{enumerate}
\item $\max_{e\in E_{\Gamma_i}} |\lambda_i'(e)|=1$;
\item if $\lambda_i: E_{\Gamma_i}\to \R$ maps $e\mapsto\lambda_{ij}$, where $e$ corresponds to $C_{ij}$ and $\lambda_{ij}$ is as in Lemma~\ref{compact sets}(4),  then $\lambda_i'$ and a constant multiple of $\lambda_i$ are $\varepsilon$-approximate;
\item if $E'_{\Gamma_i}\subset E_{\Gamma_i}$ is the set of edges $e$ such that $|\lambda_i'(e)|<1-\delta$, then $|\lambda_i'(e)|\leq K/m$ for some constant $K>0$ which is independent of $m\gg 0$ and $i$;
\item each vertex of $V_{\Gamma_i,i}$, has the same number $\geq 1$ of adjacent edges whose interiors have larger $x$-coordinate and whose interiors have smaller $x$-coordinate;
\item $y\circ\overline\Xi_{i}(p)=d_i$ for all $p\in V_{\Gamma_i,e}$;
\item there is a vertex $q_{0}\in V_{\Gamma_i,i}$ corresponding to $K_{ij_0}$ (cf.\ Lemma~\ref{compact sets}(3)) which satisfies $y\circ \overline\Xi_i(q_{0})=0$.
\end{enumerate}
\end{lemma}

\begin{proof}
(1), (4), (5), (6) follow from the construction.  (2) follows from Lemma~\ref{compact sets}(5).  (3) follows from (2).
\end{proof}

See Figure~\ref{figure: interval-contraction} for an example.  The following lemma is immediate from the construction.

\begin{figure}[ht]
\begin{center}
\psfragscanon
\psfrag{a}{\tiny $y'$}
\psfrag{b}{\tiny $|\lambda'_i|\approx 1$}
\psfrag{c}{\tiny $x'$}
\psfrag{d}{\tiny $0$}
\psfrag{e}{\tiny $\dots$}
\psfrag{f}{\tiny $-\infty$}
\psfrag{g}{\tiny $d_i'$}
\psfrag{h}{\tiny $\times$}
\psfrag{i}{\tiny $R_i^+$}
\psfrag{j}{\tiny $R_i^-$}
\includegraphics[width=8cm]{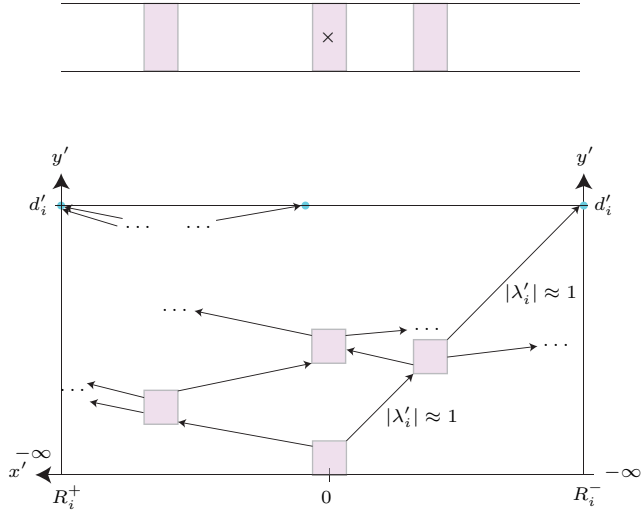}
\end{center}
\caption{The top figure represents $[-2,2]\times [R_i^-,R_i^+]$, where $\times$ indicates the location of $\overline{\frak m}^b(\tau_i)$ after translation and the shaded regions are $K_{ij}$. The bottom is a schematic diagram for the image of $\Xi_i$ which becomes the image of $\overline\Xi_i$ ``when seen from far away''. The dots along the line $y'=d_i'$ correspond to the endpoints of $V_{\Gamma_i,e}$ and the shaded regions are $t_2(K_{ij})\times \log\rho(\widetilde w_i\circ \widetilde \pi_i^{-1}(K_{ij}))$. {\em When viewed in $[-1,1]\times[0,d_i]$,} the edges with the largest slope (in absolute value) satisfy $|\lambda_i'|\approx 1$ and the remaining edges have much smaller slope when $m\gg 0$.}
\label{figure: interval-contraction}
\end{figure}

\begin{lemma}\label{independence}
After passing to a subsequence, we may assume that the following do not depend on the choice of $i$:
\begin{itemize}
\item the graph $\Gamma_i$, the function $\lambda_i$, and the set $E'_{\Gamma_i}$;
\item given an edge $E_{\Gamma_i}$ with endpoints $p,q$, whether $x\circ\overline\Xi_{i}(p)\geq x\circ\overline\Xi_{i}(q)$ and whether $y\circ \overline\Xi_{i}(p)\geq y\circ\overline\Xi_{i}(q)$.
\end{itemize}
\end{lemma}

In view of Lemma~\ref{independence}, we may write $\Gamma=(V_\Gamma,E_\Gamma)$ for $\Gamma_i=(V_{\Gamma_i},E_{\Gamma_i})$.

Finally, we orient the edges $p\stackrel{e}\to q \in E_\Gamma$ so that $y\circ\overline\Xi_{i}(q)-y\circ\overline\Xi_{i}(p)>0$;  we denote the corresponding orientation by ${\frak o}$.  Given $p,q\in V_\Gamma$, we write $p\succeq q$ (resp.\ $p\simeq q$) to mean $y\circ\overline\Xi_{i}(p)-y\circ \overline\Xi_{i}(q)\geq 0$ (resp.\ $=0$) for all $i$.

\subsubsection{The case $\overline{v}'_1=\varnothing$, $\overline{v}'_2\not=\varnothing$}

We continue to assume that $\overline{v}'_1=\varnothing$, $\overline{v}'_2\not=\varnothing$. See Figure~\ref{figure: graphs4bis}(1).  We write $\approx$ to mean ``is close to''.

\begin{figure}[ht]
\vskip.2in
\begin{center}
\psfragscanon
\psfrag{0}{\tiny $0$}
\psfrag{1}{\tiny $1$}
\psfrag{g}{\tiny $\geq 1$}
\psfrag{x}{\tiny $\overline{v}'_2$}
\psfrag{y}{\tiny $\overline{v}''_1$}
\psfrag{z}{\tiny $\overline{v}'_1$}
\psfrag{u}{\tiny $\overline{v}''_{T,1}$}
\psfrag{v}{\tiny $\overline{v}''_{B,1}$}
\psfrag{a}{\small (1)}
\psfrag{b}{\small (2)}
\psfrag{c}{\small (3)}
\includegraphics[width=12cm]{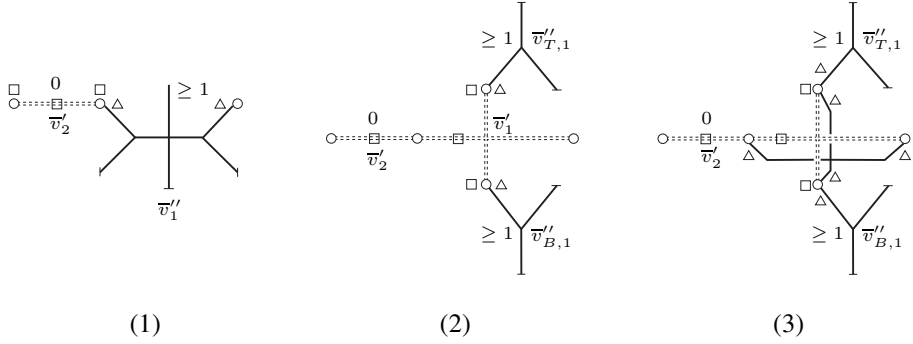}
\end{center}
\vskip.2in
\caption{Schematic diagrams for the possible types of degenerations corresponding to Lemmas~\ref{alishan} and \ref{alishan2}. Here $\circ$ represents $z_\infty$ or $z_\infty^\#$, {\tiny $\square$} represents a branch point, and {\tiny $\triangle$} represents an end with a large sector. Double dotted lines indicate multiple covers of $\sigma_\infty^*$. The labels on the graphs indicate the components and their ECH indices. If there is more than one {\tiny $\square$} or {\tiny $\triangle$} in a diagram, then we interpret it as one of the possible locations for {\tiny $\square$} or {\tiny $\triangle$}.}
\label{figure: graphs4bis}
\end{figure}

We start with the following useful lemma:

\begin{lemma}[Comparison Lemma] \label{comparison lemma}
Suppose $m\gg 0$, $i=i(m)\gg 0$, and $\varepsilon=\varepsilon(m)>0$ is small.
Let $q\stackrel{e}\to q'$ be an edge of $\Gamma$ such that $|\lambda'_i(e)|\approx 1$ and let $$\delta'=(p_1\stackrel{f_1}\to \dots \stackrel{f_{k}}\to p_{k+1})$$ be an oriented path of $\Gamma$ from $p_1$ to $p_{k+1}$ such that $p_1\preceq q\preceq q' \preceq p_{k+1}$ and $0<|\lambda'_i(f_j)|\leq K/m$ for $j=1,\dots,k$.  Then
$$|x\circ\overline\Xi_{i}(q')-x\circ\overline\Xi_{i}(q)|\leq K'/m,$$
where $K'$ is independent of $m$ and $i$.
\end{lemma}

Note that, by Lemma~\ref{map}(3), $0<|\lambda'_i(f_j)|\leq K/m$ if and only if $|\lambda'_i(f_j)|$ is not close to $1$.

\begin{proof}
The horizontal variation $\sum_{j=1}^{k}|x\circ\overline\Xi_{i}(p_{j+1}) -x\circ\overline\Xi_{i}(p_j)|$
is bounded above by $2$ times the maximal covering degree of $\overline{v}'_{L,j}$, $j=1,\dots,a$, and $\overline{v}'_{R,j}$, $j=1,\dots,b+1$, which we denote by $\mathcal{K}$. (Here the $2$ comes from the width of the interval $[-1,1]$.) On the other hand, since $p_1\preceq q\preceq q' \preceq p_{k+1}$, we have:
\begin{align*}
|x\circ\overline\Xi_{i}(q')-x\circ\overline\Xi_{i}(q)| & \approx |y\circ\overline\Xi_{i}(q')-y\circ\overline\Xi_{i}(q)| \\
&\leq  \sum_{j=1}^{k} |\lambda'_i(f_j)|\cdot |x\circ\overline\Xi_{i}(p_{j+1}) -x\circ\overline\Xi_{i}(p_j)|\\
&\leq \left(\max_{j=1,\dots,k}|\lambda'_i (f_j)|\right)(2\mathcal{K})\leq K'/m,
\end{align*}
where $K'=2\mathcal{K} K$.
\end{proof}

Given $(x_0,y_0)\in [-1,1]\times [0,d_i]$, let $$l_{(x_0,y_0)}=\{x=x_0, 0\leq y\leq y_0\}$$ be a line segment oriented in the positive $y$-direction. Next choose an orientation $\widetilde{\frak o}$ of $\Gamma$ so that the oriented edges of $\overline\Xi_i(\Gamma)$ are pointing in the positive $x$-direction and write $\widetilde\Gamma=(\Gamma,\widetilde{\frak o})$.  (This is different from the orientation ${\frak o}$ used previously.)

Consider the {\em weight function}
$$\mathcal{W}_i: ([-1,1]\times[0,d_i])-\overline\Xi_i(\Gamma)\to \Z^{\geq 0},$$
which is defined as follows:
\begin{enumerate}
\item $\mathcal{W}_i$ is locally constant;
\item $\mathcal{W}_i(x,y)$ is the signed intersection number $\langle \overline\Xi_i(\widetilde\Gamma),l_{(x,y)}\rangle$ for generic $(x,y)$, with respect to the orientations $(\bdry_x,\bdry_y)$ for $[-1,1]\times[0,d_i]$ and $\widetilde{\frak o}$.
\end{enumerate}
The function $\mathcal{W}_i$ is well-defined by Lemma~\ref{map}(4).

\begin{lemma} \label{alishan}
If $m\gg 0$, then there is no $\overline{u}_\infty\in \bdry_{\{-\infty\}}\mathcal{M}$ such that $\overline{v}'_1=\varnothing$ and $\overline{v}'_2\not=\varnothing$.
\end{lemma}

\begin{proof}
The proof we give is not the most efficient, but carries over more easily to other situations. We choose $m\gg 0$ so that $0< K'/m\ll 1$.

Let $\delta=(q_{0}\stackrel{e_{0}}\to \dots \stackrel{e_{l-1}}\to q_{l})$ be an oriented path from $q_0$ to $q_{l}\in V_{\Gamma,e}$, such that $|\lambda'_i(e_j)|\approx 1$ for each $i$ and each edge $e_j$, $j=0,\dots,l-1$. We may assume that $x\circ\overline\Xi_{i}(q_l)= \pm 1$ in view of Lemma~\ref{sencha 3}(4), since the only end $\mathcal{E}$ of $\overline{v}^\sharp_{L,j}$, $j=1,\dots,a+1$, or $\overline{v}^\sharp_{R,j}$, $j=1,\dots,b+1$, whose $\pi_{D^2_{\rho_0}}$-projection is a sector with angle $>\pi$ is an end of $\overline{v}_1^\sharp=\overline{v}^\sharp_{L,a+1}$.

We claim that, for each $l_0=0,\dots,l-1$, there is an oriented path
$$\delta'=(p_1\stackrel{f_1}\to \dots \stackrel{f_{k}}\to p_{k+1})$$ such that $p_1\preceq q_{l_0}\preceq q_{l_0+1} \preceq p_{k+1}$ and $0<|\lambda'_i(f_j)|< K'/m$ for $j=1,\dots,k$. Arguing by contradiction, if the claim does not hold, then $\overline\Xi_i(\Gamma)\cap \{y=y_0\}$ consists of only one point $(x_0,y_0)\in \overline\Xi_i(e_{l_0})\cap \{y=y_0\}$ for any constant $y_0$ in the interval $(y\circ\overline\Xi_{i}(q_{l_0}), y\circ\overline\Xi_{i}(q_{l_0+1}))$. This means that there is an integer $\kappa$ such that $\mathcal{W}_i(x,y_0)=\kappa$ for $-1\leq x<x_0$ and $\mathcal{W}_i(x,y_0)=\kappa\pm 1$ for $1\geq x>x_0$.  This contradicts $\mathcal{W}_i(-1,y_0)=\mathcal{W}_i(1,y_0)=0$, which is due to the fact that all the exterior vertices $p\in V_{\Gamma_i,e}$ satisfy $y\circ\overline\Xi_{i}(p)=d_i$ by Lemma~\ref{map}(5).

The claim, together with the Comparison Lemma, implies that
$$|x\circ\overline\Xi_{i}(q_{k+1})-x\circ\overline\Xi_{i}(q_k)|\leq K'/m$$ for each $k=0,\dots,l-1$.  Hence,
$$\sum_{k=0}^{l-1} |x\circ\overline\Xi_{i}(q_{k+1})-x\circ\overline\Xi_{i}(q_k)|\leq lK'/m\ll 1$$ for $m\gg 0$.  This contradicts Lemma~\ref{map}(5) with $d_i\geq 1$.
\end{proof}

\subsubsection{The case $\overline{v}'_1\not=\varnothing$, $\overline{v}'_2\not=\varnothing$} \label{second case}

In this subsection we consider the case $\overline{v}'_1\not=\varnothing$, $\overline{v}'_2\not=\varnothing$. For simplicity assume that there are no boundary points of type (P$_3$). Some of the possibilities are given by Figure~\ref{figure: graphs4bis}(2) and (3).

We outline the necessary modifications in the current case:

(1) We consider the truncation $\overline{u}_i:\Sigma_i\to \overline{W}_{\tau_i}$ of $\overline{u}_i$ so that (T1) and (T2) in Definition~\ref{defn: truncation} hold. (T4) becomes:
\begin{enumerate}
\item[(T4$'$)] there exists a decomposition $\Sigma_i=\Sigma_i^{(1)}\cup\dots\cup \Sigma_i^{(\iota)}$ such that each $\pi_{B_{\tau_i}}\circ \overline{u}_i|_{\Sigma_i^{(j)}}$ is a branched cover with possible empty branch locus over a component of $B_{\tau_i}$ which is cut up by (possibly multiple) arcs of type $t=R$ with $2< R<r(\tau_i)-2$ and $s=R'$ with $R'<-3$ or $R'>3$.
\end{enumerate}
(T5) is slightly modified so that the normalized eigenfunction is that of $\overline{v}^\sharp_*$ for any $*$.

(2) Given $\varepsilon>0$ small, there exists $L>0$ so that the analog of Lemma~\ref{compact sets} holds; here we restrict $\Sigma_i$ to $|s|\leq L$, while keeping the same notation. 
To the list of compact subsets $K_{ij}$ of Lemma~\ref{compact sets} (viewed as subsets of $B_{\tau_i}$ instead of $[-2,2]\times[R_i^-,R_i^+]$), we add the compact subset $K_{ij}$ of the following type, which we call ``type $\overline{v}_1$'':
$$K_{ij}= B_{\tau_i}\cap \{|s|\leq L,|t-1/2|\leq L\}.$$

(3) After suitable translations, contractions of components of $\pi^{-1}_{B_{\tau_i}}(K_{ij})$ to points, and rescalings, we obtain the ``tropical curves''
$$\overline\Xi_i:\Gamma\to [-1,1]/(-1\sim 1)\times [0,d_i],$$
where the equivalence relation $\sim$ is consistent with contracting each component of $\pi^{-1}_{B_{\tau_i}}(K_{ij})$ of type $\overline{v}_1$ to a point and the graph $\Gamma=(V_\Gamma,E_\Gamma)$ is independent of $i$.

(4) The set $V_\Gamma$ of vertices admits the decomposition $V_{\Gamma,i}\sqcup V_{\Gamma,e}$, where $V_{\Gamma,i}$ is in one-to-one correspondence with the set of components of $\pi^{-1}_{B_{\tau_i}}(K_{ij})$ and $V_{\Gamma,e}$ is in one-to-one correspondence with the set of left and right ends of $\overline{v}_{L,j}^\sharp$, $j=1,\dots,a+1$, and $\overline{v}_{R,j}^\sharp$, $j=1,\dots,b+1$, that limit to $z_\infty$. There is a subset $V'_{\Gamma,i}\subset V_{\Gamma,i}$ consisting of vertices which are not initial points of any oriented edge; $V'_{\Gamma,i}$ is obtained by contracting components of $\pi^{-1}_{B_{\tau_i}}(K_{ij})$ of type $\overline{v}_1$ to points.

(5) In Lemma~\ref{map}, (1)--(3) and (5) still hold, in (4) we replace $V_{\Gamma,i}$ by $V_{\Gamma_i}-V'_{\Gamma,i}$, and (6) becomes: there is a vertex $q_{0}\in V_{\Gamma,i}$ corresponding to $K_{ij_0}$ which satisfies $\overline\Xi_i(q_{0})=(0,0)$.

(6) The weight function $\mathcal{W}_i$ is now a function
$$\mathcal{W}_i: (([-1,1]/\sim) \times[0,d_i])-\overline\Xi_i(\Gamma)\to \Z^{\geq 0},$$
i.e., it is periodic in the $x$-direction.

\begin{lemma} \label{alishan2}
If $m\gg 0$, then there is no $\overline{u}_\infty\in \bdry_{\{-\infty\}}\mathcal{M}$ such that $\overline{v}'_1\not =\varnothing$ and $\overline{v}'_2\not=\varnothing$.
\end{lemma}

\begin{proof}
The proof is similar to that of Lemma~\ref{alishan} and uses the Comparison Lemma.

Suppose that there are no boundary points of type (P$_3$). Let $\delta=(q_{0}\stackrel{e_{0}}\to \dots \stackrel{e_{l-1}}\to q_{l})$ be a maximal oriented path which starts from $q_0$, has $|\lambda_i(e)|\approx 1$ for each edge, and ends at some $q_l$ with $x\circ\overline\Xi_{i}(q_l)=\pm 1$. We claim that, for each $l_0=0,\dots,l-1$, there is an oriented path
$$\delta'=(p_1\stackrel{f_1}\to \dots \stackrel{f_{k}}\to p_{k+1})$$ such that $p_1\preceq q_{l_0}\preceq q_{l_0+1} \preceq p_{k+1}$ and $0<|\lambda'_i(f_j)|\leq K'/m$ for $j=1,\dots,k$. Indeed, using the same notation as that of Lemma~\ref{alishan}, there is a point $(x_0,y_0)$ and an integer $\kappa$ such that $\mathcal{W}_i(x,y_0)=\kappa$ for $-1\leq x<x_0$ and $\mathcal{W}_i(x,y_0)=\kappa\pm 1$ for $x_0\leq x\leq  1$, which is impossible by the $x$-periodicity.  The claim gives a contradiction as in the proof of Lemma~\ref{alishan}.

Suppose there are boundary points of type (P$_3$). Then, by Lemma~\ref{sencha 3}(4), there are no boundary points of type (P$_3$) on $\overline{v}_{L,j}$, $j=1,\dots,a$, and $\overline{v}_{R,j}$, $j=1,\dots,b+1$.  The maximal oriented path $\delta$ from the previous paragraph has $y\circ \overline \Xi_{i}(q_l)$ which is much larger than $y\circ \overline\Xi_i$ of the endpoint of another path.  This is a contradiction.
\end{proof}

\begin{proof}[Proof of Lemma~\ref{cherries5}]
Suppose $\overline{u}_\infty\in \bdry_{\{-\infty\}}\mathcal{M}$. If $\overline{v}'_*\cup \overline{v}_*^\sharp=\varnothing$ for all levels $\overline{v}_*$ of $\overline{u}_\infty$, then, by Lemma~\ref{los angeles 3}, $\overline{u}_\infty$ is a $2$-level building $\overline{v}_1\cup\overline{v}_2$, where $\overline{v}_1$ is a $W_{-\infty,1}$-curve with $I=0$ and $\overline{v}_2$ is a $\overline{W}_{-\infty,2}$-curve with $I=2$ which passes through $\overline{\frak m}(-\infty)$. By Lemma~\ref{lemma: ECH for W minus infinity 2}, ${\bf y}_2$ and ${\bf y}_4$ satisfy the conditions of $A_3$.

On the other hand, it is not possible that $\overline{v}'_*\cup \overline{v}_*^\sharp\not=\varnothing$ for some level $\overline{v}_*$ by Lemmas~\ref{alishan} and ~\ref{alishan2}.
\end{proof}

\subsection{Breaking in the middle}
\label{subsection: chain homotopy part 2}

In this subsection we study the limit of holomorphic maps to $\overline{W}_{\tau}$ as $\tau \to T'$ for some $T'\in(-\infty,\infty)$. This will prove Lemma~\ref{cherries3}.

We assume that $m\gg 0$; $\varepsilon,\delta>0$ are sufficiently small; and $\{\overline{J}_\tau\}\in \overline{\mathcal{I}}^{reg}$ and $\{\overline{J}_\tau^\Diamond(\varepsilon,\delta,{\frak p}(\tau))\}$ satisfy Lemma~\ref{lemma: codimension one}.  Fix ${\bf y}\in \mathcal{S}_{{\bf a},\hh({\bf a})}$, ${\bf y}'\in\mathcal{S}_{{\bf b},\hh({\bf b})}$ and let
$$\mathcal{M}=\mathcal{M}^{I=2,n^*=m}_{\{\overline{J}_\tau^\Diamond(\varepsilon,\delta,{\frak p}(\tau))\}}({\bf y},{\bf y}';\overline{\frak m}), \quad \mathcal{M}_\tau= \mathcal{M}^{I=2,n^*=m}_{\overline{J}_\tau^\Diamond(\varepsilon,\delta,{\frak p}(\tau))}({\bf y},{\bf y}';\overline{\frak m}(\tau)).$$
We will analyze $\bdry_{(-\infty,\infty)}\mathcal{M}$.

Let $\overline{u}_i$, $i\in \N$, be a sequence of curves in $\mathcal{M}$ such that $\overline{u}_i\in \mathcal{M}_{\tau_i}$ and $\displaystyle\lim_{i\to\infty} \tau_i=T'$, and such that its limit
$$\overline{u}_\infty= (\overline{v}_{-1,1}\cup\dots\cup \overline{v}_{-1,c})\cup \overline{v}_0 \cup (\overline{v}_{1,1}\cup\dots\cup\overline{v}_{1,a})$$
is a holomorphic building in $\bdry_{\{T'\}}\mathcal{M}$, ordered from bottom to top, where each $\overline{v}_*$ is an SFT level,  $\overline{v}_{-1,j}$, $j=1,\dots,c$, and $\overline{v}_{1,j}$, $j=1,\dots,a$, map to $\overline{W}$ and $\overline{v}_0$ maps to $\overline{W}_{T'}$. Sometimes we will refer to $\overline{v}_0$ as $\overline{v}_{-1,c+1}$ or $\overline{v}_{1,0}$.

The following are analogs of Lemmas~\ref{nonnegative ECH indices better}--\ref{sencha}, stated without proof.

\begin{lemma} \label{nonnegative ECH indices 2}
If fiber components are removed from $\overline{u}_\infty$ and the only boundary points at $z_\infty$ are of type (P$_3$), then the ECH index of each level $\overline{v}_*\not= \overline{v}_0$ is nonnegative,
$$ I(\overline{v}_{1,j})\geq  \left\{
\begin{array}{ll}
I(\overline{v}'_{1,j}) +I(\overline{v}''_{1,j}) & \mbox{if } \mathfrak{bp}_{1,j}=0;\\
I(\overline{v}'_{1,j}) +I(\overline{v}''_{1,j})+2  & \mbox{if } \mathfrak{bp}_{1,j}>0,
\end{array}
\right.
$$
for $0\leq j\leq a$,
and the only components of $\overline{u}_\infty$ which have negative ECH index are the following:
\begin{enumerate}
\item branched covers of $\sigma_\infty^{T'}$; and
\item at most one component $\widetilde{v}$ of $\overline{v}_0''$ with $I(\widetilde{v})=-1$.
\end{enumerate}
Here (2) occurs when $T'\in \mathcal{T}_1$ and $\widetilde{v}\in \mathcal{M}^{\dagger,s,irr,\op{ind}=-1}_{\overline{J}_{T'}^\Diamond(\varepsilon,\delta,{\frak p}(T'))}({\bf z},{\bf z}')$, as described in Lemma~\ref{lemma: codimension one}(1).
\end{lemma}

\begin{lemma}\label{los angeles2}
If $\overline{u}_\infty\in \bdry_{(-\infty,\infty)}\mathcal{M}$ and $\overline{v}'_*\cup\overline{v}^\sharp_*=\varnothing$ for all levels $\overline{v}_*$ of $\overline{u}_\infty$, then $\overline{u}_\infty$ is one of the following:
\begin{enumerate}
\item $a=0$, $c=1$; $\overline{v}_0$ is a $\overline{W}_{T'}$-curve with $I=1$ which passes through $\overline{\frak m}({T'})$; and $\overline{v}_{-1,1}$ is a $W$-curve with $I=1$; or
\item $a=1$, $c=0$; $\overline{v}_{1,1}$ is a $W$-curve with $I=1$; and $\overline{v}_0$ is a $\overline{W}_{T'}$-curve with $I=1$ which passes through $\overline{\frak m}({T'})$.
\end{enumerate}
Here either $T'\in \mathcal{T}_2$ and there is a component of $\overline{v}_0$ which is in
$$\mathcal{M}_{\overline{J}^\Diamond_{T'}(\varepsilon,\delta,{\frak p}(T'))}^{\dagger,s,irr, \op{ind}=1,n^*=m}({\bf z},{\bf z}',\overline{\frak m}(T'))$$
from Lemma~\ref{lemma: codimension one}(2), for some ${\bf z}$, ${\bf z}'$; or $T'\in \mathcal{T}_1$ and there is a component of $\overline{v}_0$ which does not pass through $\overline{\frak m}(T')$ but is in
$$\mathcal{M}_{\overline{J}_{T'}^\Diamond(\varepsilon,\delta,{\frak p}(T'))}^{\dagger,s,irr,\op{ind}=-1,n^*=0}({\bf z},{\bf z'})$$ from Lemma~\ref{lemma: codimension one}(1).
\end{lemma}

\begin{lemma} \label{sencha2}
If $\overline{v}'_*\cup\overline{v}^\sharp_*\not=\varnothing$ for some level $\overline{v}_*$ of $\overline{u}_\infty$, then:
\begin{enumerate}
\item $p_0=\deg(\overline{v}_0')>0$;
\item $\overline{u}_\infty$ has no boundary point of type (P$_1$) or (P$_2$);
\item $\overline{u}_\infty$ has no fiber components and no components of $\overline{v}''_*$ that intersect the interior of a section at infinity;
\item each component of $\overline{v}_{-1,j}^\sharp$, $j=1,\dots,c$, is a thin strip from $z_\infty$ to some $x_i$ or $x_i'$ with $I=1$;
\item the only boundary points at $z_\infty$ are type (P$_3$) points of $\overline{v}^\sharp_{1,j}$, $j=0,\dots,a$.
\end{enumerate}
\end{lemma}

The following is the analog of Lemma~\ref{hojicha}:

\begin{lemma} \label{hojicha2}
If $\overline{v}'_*\cup\overline{v}^\sharp_*\not=\varnothing$ for some level $\overline{v}_*$, then there are no removable points at $z_\infty$ and $\overline{u}_\infty$ contains a subbuilding consisting of $\overline{v}_{1,a}^\sharp$ with $I\geq 1$; $\overline{v}_{1,j}'$, $1\leq j<a$, which branch cover $\sigma_\infty$; $\overline{v}_0'$ with $I=-p$ which is a degree $p$ branched cover of $\sigma_\infty^{T'}$; $\cup_{j=1}^c \overline{v}^\sharp_{-1,j}$ which is a union of $p$ thin strips; and possibly the following:
\begin{itemize}
\item $\overline{v}_0^\sharp$ with positive ends at multiples of $z_\infty$ and no negative ends at multiples of $z_\infty$; and
\item $\overline{v}_{1,j}^\sharp$, $1\leq j<a$, with $I(\overline{v}_{1,j}^\sharp)\geq 1$ and (positive or negative) ends at multiples of $z_\infty$.
\end{itemize}
Here at most one component of $\overline{v}_0''$ satisfies $I=-1$ and the remaining components of $\overline{v}_*''$ satisfy $I\geq 0$.
\end{lemma}

\begin{proof}
 We explain why there are no boundary points of type (P$_3$); the rest of the proof is similar to that of Lemma~\ref{hojicha}.  A boundary point of type (P$_3$) contributes $+2$ towards $I$ by Lemma~\ref{nonnegative ECH indices better}, there is a large sector which contributes $+1$, and some $\overline{v}_{1,j}$, $j>0$, contributes $+1$, for a total of $I=4$.  Since there is at most one component which contributes negatively to $I$, namely a component of $\overline{v}_0''$ with $I=-1$, we have a total of $I\geq 3$, a contradiction.
\end{proof}

See Figure~\ref{figure: graphs2} for some possibilities. Observe that in the current case there is at most one component of $\overline{v}_0''$ with $I=-1$ whereas there are none in Lemma~\ref{hojicha}.

\begin{figure}[ht]
\begin{overpic}[width=12cm]{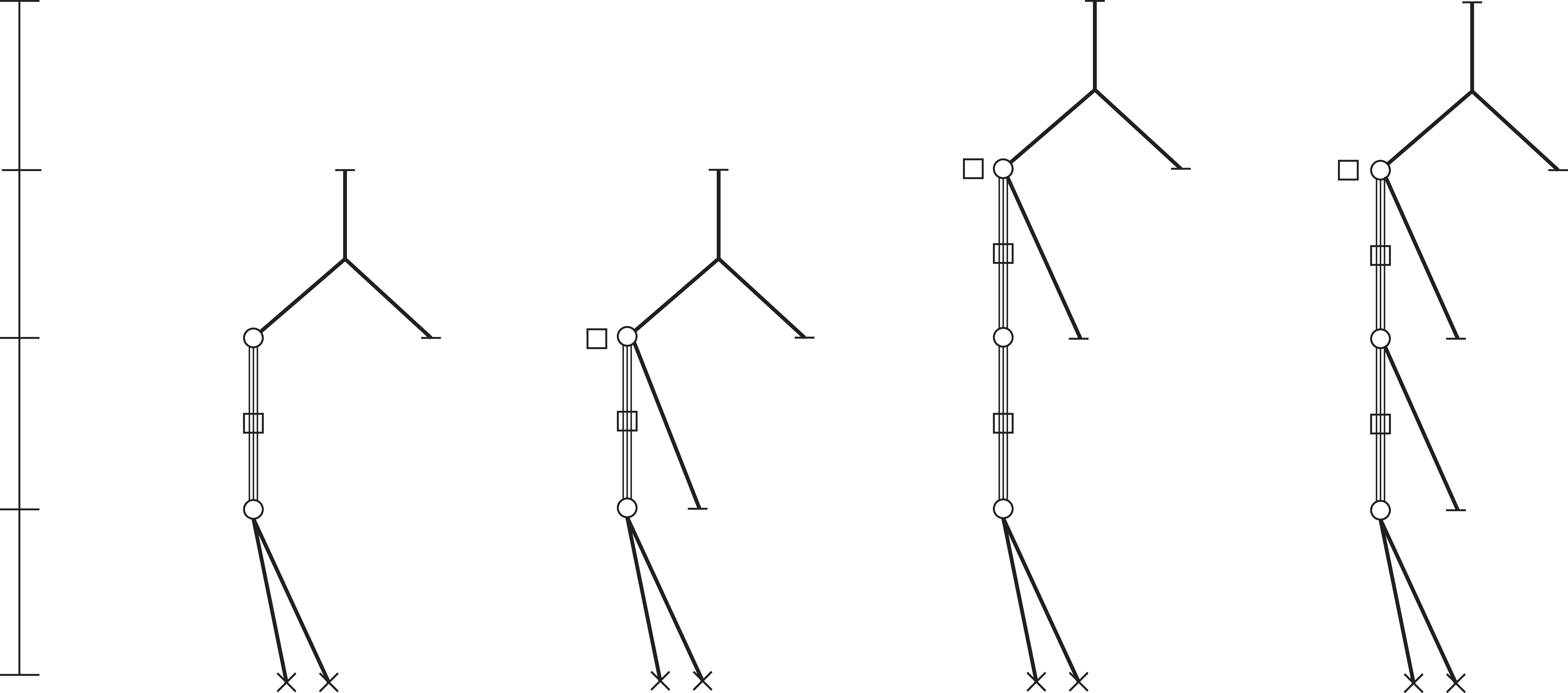}
\put(-3.8,5.5){\tiny $\overline{W}$} \put(-4.7,16) {\tiny $\overline{W}_\tau$} \put(-3.8,26.5) {\tiny $\overline{W}$} \put(-3.8,36.8) {\tiny $\overline{W}$}

\put(14.9,29){\tiny $1,2,3$} \put(17.25,17){\tiny $-p$}\put(19.5,6.5){\tiny $\mbox{total}=p$}

\put(43.65,6.5){\tiny $\mbox{total}=p$} \put(35.5,17){\tiny $-p$} \put(38.85,29){\tiny $1,2,3$} \put(43.7,17){\tiny $\geq -1$}

\put(65,39.7){\tiny $1,2$}\put(68,27.3){\tiny $\geq 1$} \put(61.35,27.3){\tiny $0$}\put(59.3,17){\tiny $-p$} \put(67,6.5){\tiny $\mbox{total}=p$}

\put(89.2,39.7){\tiny $1,2$}\put(85.2,27.3){\tiny $0$} \put(92.2,27.3){\tiny $\geq 1$} \put(83.25,17){\tiny $-p$}\put(92,17){\tiny $\geq -1$}  \put(91.6,6.5){\tiny $\mbox{total}=p$}
\end{overpic}
\vskip.2in
\caption{Schematic diagrams for some possible types of degenerations. Here $\circ$ represents $z_\infty$, {\tiny $\square$} represents one or more branch points and $\times$ represents some $x_i$ or $x_i'$. A vertical line indicates a trivial cylinder or a restriction of a trivial cylinder, and a triple vertical line indicates a degree $p$ branched cover of a trivial cylinder or a restriction of a trivial cylinder. The labels on the graphs are ECH indices of each component.}
\label{figure: graphs2}
\end{figure}
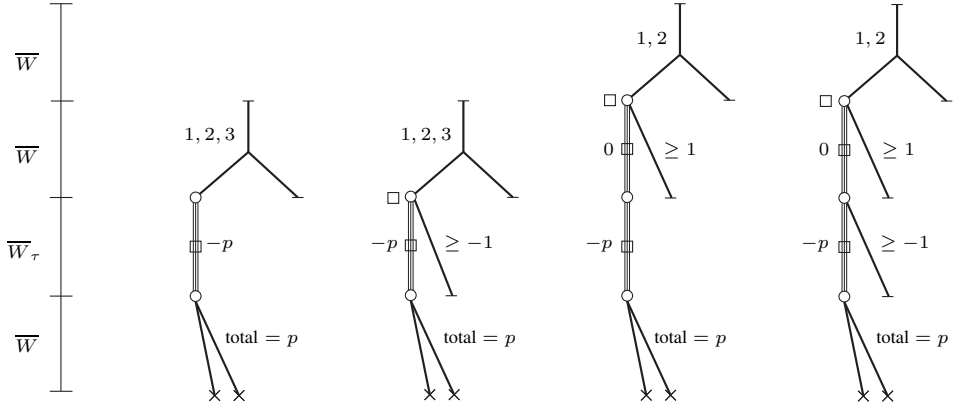

\begin{lemma} \label{vancouver2}
For each interval $[-T,T]$, there exists $m\gg 0$ such that there is no sequence of curves $\overline{u}_i\in \mathcal{M}_{\tau_i}$, $\tau_i\to T'\in[-T,T]$, that limits to $\overline{u}_\infty$ for which $\overline{v}'_*\cup \overline{v}_*^\sharp\not=\varnothing$ for some level $\overline{v}_*$.
\end{lemma}

\begin{proof}
This is similar to Case (6$_i$) of Lemma~\ref{vancouver}.
We apply the usual rescaling argument with $m\to\infty$ and obtain $w_0: \Sigma_0\to \C\P^1$ and a branched cover $\pi_0: \Sigma_0\to cl(B_{T'})$ such that:
\begin{enumerate}
\item[(i)] $w_0(\bdry \Sigma_0) \subset \{ \phi=0, \rho>0\}\cup \{\infty\}$;
\item[(ii)] $w_0(z_0)=\infty$ for some point $z_0\in \pi_0^{-1}(+\infty)$;
\item[(iii)] $w_0(z_1)=0$ for some point $z_1\in \pi_0^{-1}(\overline{\frak m}^b(T'))$;
\item[(iv)] $w_0|_{int(\Sigma_0)}$ is a biholomorphism onto its image.
\end{enumerate}
Let us write $f_0=\pi_0\circ w_0^{-1}$ where defined. We now apply the Involution Lemma I.\ref{P1-lemma: effect 4}. Using the notation of Lemma~I.\ref{P1-lemma: effect 4}, let $\overline\Sigma_1$ be the compact Riemann surface with boundary whose interior is biholomorphic to $w_0(int(\Sigma_0))$ and let $\overline\Sigma_2=cl(B_{T'})$. We then extend $f_0$ to a holomorphic map $\overline\Sigma_1\to \overline\Sigma_2$. By the Involution Lemma~I.\ref{P1-lemma: effect 4} and (i), (iii) and (iv), $f_0$ maps the point on $\overline\Sigma_1$ which corresponds to $\infty\in \C\P^1$ to $\mathcal{L}_{(r(T')+1)/2}\cap \bdry cl(B_{T'})$. This contradicts (ii).
\end{proof}

We are now in a position to prove Lemma~\ref{cherries3}.

\begin{proof}[Proof of Lemma~\ref{cherries3}]
Suppose $\overline{u}_\infty\in \bdry_{(-\infty,+\infty)} \mathcal{M}$.  By Lemma~\ref{los angeles2}, if $\overline{v}'_*\cup\overline{v}^\sharp_*=\varnothing$ for all levels $\overline{v}_*$ of $\overline{u}_\infty$, then $\overline{u}_\infty\in A_4$ or $A_5$. By Lemma~\ref{vancouver2}, for any $T>0$, there exists $m\gg 0$ such that if $\overline{u}_\infty\in \bdry_{[-T,T]}\mathcal{M}$, then $\overline{v}'_*\cup\overline{v}^\sharp_*=\varnothing$ for all $\overline{v}_*$.

It remains to consider the case where there exist sequences $m_i\to\infty$, $\varepsilon_i,\delta_i\to 0$, $T_i\to\infty$, and $\overline{u}_{ij}\to \overline{u}_{i\infty}$, where $\overline{u}_{ij}\in \mathcal{M}^{(m_i)}$, $\overline{u}_{i\infty}\in \bdry_{\{\pm T_i\}}\mathcal{M}^{(m_i)}$, $\mathcal{M}^{(m_i)}$ is $\mathcal{M}$ with respect to the family $\{\overline{J}^\Diamond_\tau(\varepsilon_i,\delta_i,{\frak p}(\tau);m_i)\}$, and $\overline{u}_{i\infty}$ satisfies $\overline{v}'_*\cup \overline{v}^\sharp_*\not=\varnothing$ for some $*$. We then take a diagonal subsequence $\overline{u}_{ij(i)}$. The proofs of Lemmas~\ref{vancouver} and \ref{alishan2} carry over to give a contradiction.
\end{proof}

\subsection{Degeneration at $+\infty$, part II}
\label{subsection: additional degenerations I}

In Sections~\ref{subsection: additional degenerations I}--\ref{subsection: additional degenerations III} we study the limit of holomorphic maps to $\overline{W}_\tau$ whose positive end is of the form ${\bf z}= \{z_{\infty,i}\}_{i\in \mathcal{I}}\cup {\bf y}$, i.e., we are in Step 4 of the proof of Theorem~\ref{thm: chain homotopy part i}.

Let us write
$$\mathcal{M}:=\mathcal{M}^{I=2,n^*\leq m+|\mathcal{I}|}_{\{\overline{J}_\tau^\Diamond(\varepsilon,\delta,{\frak p}(\tau))\}}({\bf z},{\bf y'};\overline{\frak m}), \mbox{ and}$$
$$\mathcal{M}_\tau:=\mathcal{M}^{I=2,n^*\leq m+|\mathcal{I}|}_{\overline{J}_\tau^\Diamond(\varepsilon,\delta,{\frak p}(\tau))}({\bf z},{\bf y'};\overline{\frak m}),$$
where $|\mathcal{I}|\geq 1$.

Let $\overline{u}_\infty\in \bdry_{+\infty}\mathcal{M}$ be the limit of $\overline{u}_i\in \mathcal{M}_{\tau_i}$, where $\tau_i\to +\infty$. We use the notation from Section~\ref{subsection: chain homotopy part 1} for the levels and components of $\overline{u}_\infty$ and write $p_*=\deg(\overline{v}_*')$ as before.

We now have two constraints
\begin{align}
\label{sum of n part 2} n^*(\overline{u}_i) &=\sum_{\overline{v}_*} n^*(\overline{v}_*)=m+|\mathcal{I}|;\\
\label{sum of I part 2} I(\overline{u}_i) &=\sum_{\overline{v}_*} I(\overline{v}_*)=2,
\end{align}
where the summations are over all the levels $\overline{v}_*$ of $\overline{u}_\infty$.

\s\n {\em Outline of proof of Lemma~\ref{kyoho plus infty}.} The proof is similar to that of Lemma~\ref{cherries4}, with the following differences: The proof of Lemma~\ref{eliminate fiber components 2}, which is the analog of Lemma~\ref{sencha}, is more involved and the proof of Lemma~\ref{nonnegative ECH indices 5}, which is the analog of Lemma~\ref{nonnegative ECH indices}, uses a slightly more complicated notion of an ``almost alternating'' pair $(P_{*,0},P_{*,1})$.

\subsubsection{Continuation argument} \label{continuation argument}

First observe that Lemma~\ref{intersezione} carries over verbatim. The analog of Lemmas~\ref{intersezione prime} and \ref{intersezione prime revisited} hinge on the continuation argument:  Let $\mathcal{E}_{-,i}$, $i=1,\dots,q$, be the negative ends of $\cup_{j=1}^a\overline{v}_{1,j}^\sharp$ that converge to $z_\infty$ and let $\mathcal{E}_{+,i}$, $i=1,\dots,r$, be the positive ends of $\cup_{j=0}^{a-1}\overline{v}_{1,j}^\sharp$ that converge to $z_\infty$.  Suppose that $q\not=0$ (i.e., some $\mathcal{E}_{-,i}$ exists) or not all $\mathcal{E}_{+,i}$ project to thin sectors. Also for simplicity we assume that there are no boundary points of type (P$_3$).  By assumption, we may start with an end $\mathcal{E}_{-,i}$ or $\mathcal{E}_{+,i}$ that projects to a non-thin sector. The continuation argument (i.e., the proof of Lemma~\ref{intersezione prime}) carries over with one modification: When we are considering the continuation $$g^0_{j_2+1,2},\dots, g^0_{a,2}$$ of $g^0_{j_2,2}$, it is possible that $g^0_{j,2}$ is trivial for all $j_2+1\leq j\leq a$. In other words, there is no nontrivial negative end $\mathcal{E}_{-,2}$ such that
$$\pi_{D^2_{\rho_0}}(\mathcal{E}_{-,2})={\frak S}(\overline{\hh}(\overline{a}_{k_3,l_3}),\overline{a}_{k_4,l_4})$$ for some $(k_4,l_4)$.  This happens when $(k_3,l_3)\to (k_3,l_3)$ belongs to the data $\overrightarrow{\mathcal{D}}$ at the positive end of $\overline{v}'_{1,a}$. In this case we set $j_3=a$ and $\overline{a}_{k_3,l_3}=\overline{a}_{k_4,l_4}$ and skip ${\frak S}(\overline{\hh}(\overline{a}_{k_3,l_3}),\overline{a}_{k_3,l_3})$. The rest of the argument is the same.

\subsubsection{Bounds on ECH indices} \label{snowman}

The goal of this subsection is to show the nonnegativity of $I(\overline{v}_*)$ except when $\overline{v}_*=\overline{v}_+$, {\em under the assumption that there are no boundary points at $z_\infty$}.

Let $A_\varepsilon=\bdry D^2_\varepsilon\times[0,1]$ for $0<\varepsilon<\rho_0$ small and let $\pi_{[0,1]\times\overline{S}}$ be the projection of $\overline{W}$ or the positive end of $\overline{W}_\tau$ to $[0,1]\times\overline{S}$. Let ${\frak c}'$ be the grooming on $A_\varepsilon$ corresponding to the data $\ar{\mathcal{D}}_{+,a}'$ at $z_\infty$ for the positive end of $\overline{v}'_{1,a}$, such that the winding number $w({\frak c}')$ is zero. Here the data $\ar{\mathcal{D}}_{+,a}'$ satisfies  $(\mathcal{D}'_{+,a})^{to}=(\mathcal{D}'_{+,a})^{from}$. Let ${\frak c}''= \pi_{[0,1]\times\overline{S}}(\cup_i \mathcal{E}_{-,i})\cap A_\varepsilon$.   By Section~\ref{continuation argument}, ${\frak c}''$ is groomed and the sets of initial and terminal points of ${\frak c}''$ alternate along $(0,2\pi)$. We then define $P_0$ and $P_1$ as the initial and terminal points of ${\frak c}'\cup{\frak c}''$.

\begin{rmk} \label{repeat}
It is possible for initial/terminal points of ${\frak c}'$ to also appear as initial/terminal points of ${\frak c}''$.
\end{rmk}

\begin{defn} \label{leopard}
A pair of points $({\frak q}_0,{\frak q}_1)\subset \bdry D^2_\varepsilon$ is a {\em thin pair} if ${\frak q}_0=\overline{\hh}(\overline{a}_{i,j})\cap \bdry D^2_\varepsilon$ and ${\frak q}_1=\overline{a}_{i,j}\cap \bdry D^2_\varepsilon$ for the same $i,j$. A pair $(P_{*,0},P_{*,1})$ consisting of disjoint finite subsets of $(0,2\pi)$ with the same cardinality is {\em almost alternating along $(0,2\pi)$} if each $P_{*,i}$, $i=0,1$, admits a splitting $P_{*,i}=Q_{*,i}\sqcup R_{*,i}$ such that the pair $(R_{*,0},R_{*,1})$ is alternating along $(0,2\pi)$ and there is a partition of $Q_{*,0}\cup Q_{*,1}$ into thin pairs $({\frak q}_0,{\frak q}_1)$, ${\frak q}_i\in Q_{*,i}$.
\end{defn}

By definition, $(P_0,P_1)$ is almost alternating along $(0,2\pi)$. See Figure~\ref{figure: almost-alternating} for an example.

Section~\ref{continuation argument} gives the main cycle
$$\mathcal{Z}_{\tiny\mbox{main}}=({\frak z}_{0}\to {\frak z}_{1}\to \dots\to {\frak z}_{k-1}\to {\frak z}_{0}),$$
where $\{{\frak z}_{1},\dots,{\frak z}_{k}\}$ can be decomposed into an almost alternating pair and the cycle winds around $\R/2\pi\Z$ once. If we apply the continuation method to the positive ends of $\overline{v}_{1,a}'$, then we also obtain a union $\mathcal{Z}_{\tiny\mbox{aux}}$ of auxiliary cycles of the form $({\frak z}_0\to{\frak z}_1\to{\frak z}_0)$, where $({\frak z}_0,{\frak z}_1)$ is a thin pair and the chords are short chords from ${\frak z}_i$ to ${\frak z}_{1-i}$, $i=0,1$.
The sets of initial and terminal points of $\mathcal{Z}_{\tiny\mbox{main}}\cup\mathcal{Z}_{\tiny\mbox{aux}}$ are $P_0$ and $P_1$.

\begin{figure}[ht]
\begin{center}
\psfragscanon
\psfrag{a}{\tiny $Q_0$}
\psfrag{b}{\tiny $Q_1$}
\psfrag{c}{\tiny $R_0$}
\psfrag{d}{\tiny $R_1$}
\includegraphics[width=6.5cm]{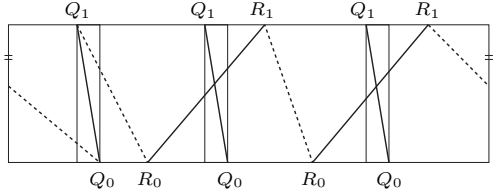}
\end{center}
\caption{The rectangular box, with the sides identified, is $A_\varepsilon$. The grooming ${\frak c}'$ connects $Q_0$ to $Q_1$ and the grooming ${\frak c}''$ connects $R_0$ to $R_1$. The vertical lines corresponding to $Q_1$ (resp.\ $Q_0$) are the intersections $(\overline{a}_{i,j} \cap \bdry D^2_\varepsilon) \times [0,1]$ (resp.\ $(\overline{\hh}(\overline{a}_{i,j}) \cap \bdry D^2_\varepsilon) \times [0,1]$). The dotted lines, together with some solid lines, give the main cycle $\mathcal{Z}_{\tiny\mbox{main}}$.  There are also two auxiliary cycles of $\mathcal{Z}_{\tiny\mbox{aux}}$ corresponding to the middle and right arcs of ${\frak c}'$.}
\label{figure: almost-alternating}
\end{figure}

Let $\ar{\mathcal{D}}_{\pm,j}$ be the data at $z_\infty$ for the $\pm$ end of $\overline{v}_{1,j}$ and let $P_{\pm,j,0}$ and $P_{\pm,j,1}$ be the initial and terminal points on $A_\varepsilon$ determined by $\ar{\mathcal{D}}_{\pm,j}$.  Then we write $$P_{\pm,j,i}=P'_{\pm,j,i}\cup P''_{\pm,j,i},$$ where $P'_{\pm,j,i}$ corresponds to $\overline{v}'_{1,j}$ and $P''_{\pm,j,i}$ corresponds to $\overline{v}''_{1,j}$.  For convenience we write $P_{-,a+1,i}=P'_{-,a+1,i}$, $i=0,1$, for the endpoints of ${\frak c}'$. Also let $Q^\star_*$ and $R^\star_*$ be the thin and alternating parts of $P_*^\star$ as given in Definition~\ref{leopard}.

The following lemmas are analogs of Lemma~\ref{alternate} and \ref{nonnegative ECH indices}:

\begin{lemma} \label{almost alternate}
If $\overline{u}_\infty$ has no boundary point at $z_\infty$, then, for each $*=(\pm,j)$, $P_{*,0}\subset P_0$ and $P_{*,1}\subset P_1$ and the pair $(P_{*,0},P_{*,1})$ is almost alternating along $(0,2\pi)$.
\end{lemma}

\begin{proof}
Similar to but slightly more complicated than that of Lemma~\ref{alternate}. We set $P_i^{(0)}=P_i$ and $\mathcal{Z}_*^{(0)}=\mathcal{Z}_*$, where $*=\mbox{main}$ or $\mbox{aux}$. Then $P_{-,a+1,i}=P'_{-,a+1,i}\subset P_i^{(0)}$ and the pair $(P_{-,a+1,0},P_{-,a+1,1})$ is almost alternating. In ($j0$)--($j4$) we replace ``alternating'' by ``almost alternating'', $j_0$ by $a+1$, and ($j2$) and ($j4$) by:
\begin{enumerate}
\item[($j2$)] there is a partition of $P_{-,a-j+1,0}\cup P_{-,a-j+1,1}$ into pairs of type $\{{\frak p}_0,{\frak p}_1\}$, ${\frak p}_i\in P_{-,a-j+1,i}$, such that ${\frak p}_0\lessdot_{R_0^{(j)}\cup R_1^{(j)}}{\frak p}_1$ or $\{{\frak p}_0,{\frak p}_1\}$ is a thin pair of $Q_0^{(j)}\cup Q_1^{(j)}$; in particular, the points of $P_{-,a-j+1,0}$ and $P_{-,a-j+1,1}$ almost alternate along $(0,2\pi)$;
\item[($j4$)] there is a partition of $P'_{+,a-j,0}\cup P'_{+,a-j,1}$ into pairs of type $\{{\frak p}_0,{\frak p}_1\}$, ${\frak p}_i\in P'_{+,a-j,i}$, such that ${\frak p}_0\lessdot_{R_0^{(j+1)}\cup R_1^{(j+1)}}{\frak p}_1$ or $\{{\frak p}_0,{\frak p}_1\}$ is a thin pair of $Q_0^{(j+1)}\cup Q_1^{(j+1)}$; in particular, the points of $P'_{+,a-j,0}$ and $P'_{+,a-j,1}$ almost alternate along $(0,2\pi)$.
\end{enumerate}

We inductively define $P_i^{(j)}$ and $\mathcal{Z}^{(j)}_*$ as follows: For each pair $\{{\frak p}_0,{\frak p}_1\}$ of $P''_{+,a-j,0}\cup P''_{+,a-j,1}$, if $({\frak p}_1\to {\frak p}_0)$ is a chord of $\mathcal{Z}^{(j-1)}_{\tiny\mbox{aux}}$, then we remove $({\frak p}_1\to{\frak p}_0\to{\frak p}_1)$ from $\mathcal{Z}^{(j-1)}_{\tiny\mbox{aux}}$; otherwise, in $\mathcal{Z}^{(j-1)}_{\tiny\mbox{main}}$, we replace ${\frak q} \to {\frak p}_1\to {\frak p}_0\to {\frak q}'$ by ${\frak q}\to {\frak q}'$, given by concatenation. Then $P_0^{(j)}$ and $P_1^{(j)}$ are the sets of endpoints of $\mathcal{Z}^{(j)}_{\tiny\mbox{main}}\cup\mathcal{Z}^{(j)}_{\tiny\mbox{aux}}$.

The verification of ($j0$)--($j4$) is left to the reader.
\end{proof}

\begin{lemma} \label{nonnegative ECH indices 5}
If  $\overline{u}_\infty$ has no fiber components and there are no boundary points at $z_\infty$, then the ECH index of each level $\overline{v}_*\not= \overline{v}_+$ is nonnegative, the only components of $\overline{u}_\infty$ which have negative ECH index are branched covers of $\sigma_\infty^+$, and $I(\overline{v}_{1,j})\geq I(\overline{v}'_{1,j}) +I(\overline{v}''_{1,j})$ for $0\leq j\leq a$.
\end{lemma}

Before embarking on the proof of Lemma~\ref{nonnegative ECH indices 5} we encourage the reader to review Section~I.\ref{P1-example of I calc} and in particular Lemmas I.\ref{P1-calc of almost sum} and I.\ref{P1-calc of almost sum 2}.

\begin{proof}
Similar to the proof of Lemma~\ref{nonnegative ECH indices}. The levels $\overline{v}_{0,j}$, $1\leq j\leq b$, satisfy $I(\overline{v}_{0,j})\geq 0$ by \cite[Proposition 7.15(a)]{HT1}.

In order to treat the case of $\overline{v}_{1,j}$, $0\leq j\leq a$, we use Lemma~\ref{almost alternate} and Lemmas~I.\ref{P1-calc of almost sum} and I.\ref{P1-calc of almost sum 2} to compute $I(\overline{v}_{1,j})$. Let
$$\pi_{[0,1]\times\overline{S}}: \check {\overline{W}}=[-1,1]\times[0,1]\times\overline{S}\to [0,1]\times\overline{S}$$
be the projection along $[-1,1]$.

We claim that there exist representatives $\check C_1,\check C_2, \check C_3\subset \check{\overline{W}}$ such that  the following hold:
\begin{enumerate}
\item[(0)] $\check C_1\cup \check C_2$ is a representative of $\overline{v}'_{1,j}$, $\check C_3$ is a representative of $\overline{v}''_{1,j}$,  $\check C_1\subset [-1,1]\times[0,1]\times D^2_{\varepsilon/2}$, and $\check C_2\subset[-1,1]\times[0,1]\times D^2_\varepsilon$;
\item ${\frak c}_1^\pm:=\pi_{[0,1]\times\overline{S}}(\check C_1|_{s=\pm 1})\subset A_{\varepsilon/2}$, ${\frak c}_1^+={\frak c}_1^-$ is groomed, and the endpoints of each component of ${\frak c}_1^\pm$ forms a thin pair;
\item the components of $\pi_{[0,1]\times\overline{S}}(\check C_i|_{s=\pm 1})$, $i=2,3$, corresponding to $z_\infty$ are contained in $A_\varepsilon$; let us write ${\frak c}_i^\pm$ for their unions;
\item ${\frak c}_2^+\cup{\frak c}_3^+$ and ${\frak c}_2^-\cup{\frak c}_3^-$ are groomed;
\item $w({\frak c}_2^+\cup {\frak c}_3^+)= 0$ or $-1$, and $w({\frak c}_2^-\cup {\frak c}_3^-)= 0$ or $1$;
\item ${\frak c}_2^+\cup {\frak c}_3^+$ satisfies (G$_3'$) from Section~I.\ref{P1-example of I calc} and ${\frak c}_2^-\cup {\frak c}_3^-$ satisfies (G$_3$) from Section~I.\ref{P1-example of I calc}.
\end{enumerate}
We first choose a representative $\check C_3$ of $\overline{v}''_{1,j}$ such that (2) holds for $i=3$ and
$${\frak c}_3^\pm = \pi_{[0,1]\times\overline{S}}(\cup_k\mathcal{E}_{\pm,k})\cap A_\varepsilon,$$ where $\mathcal{E}_{\pm,k}$ are the $\pm$ ends of $\overline{v}''_{1,j}$ that limit to $z_\infty$. We then choose the representative $\check C_1\cup\check C_2$ of $\overline{v}'_{1,j}$ as follows: Let $\check C_1=[-1,1]\times {\frak c}_1^+$ so that (1) holds and the endpoints of ${\frak c}_1^+={\frak c}_1^-$ are partitioned into thin pairs of $(P'_{+,j,0},P'_{+,j,1})$. Let $\check C_2\subset [-1,1]\times [0,1]\times D^2_\varepsilon$ be a disk of the type used in the proof of Lemma~I.\ref{P1-lemma: HF index of sections at infinity}, such that (2) holds for $i=2$, $\check C_2|_{s=\pm 1}= {\frak c}_2^\pm$, ${\frak c}_2^\pm$ are groomed, $w({\frak c}_2^+)= 0$ or $-1$, $w({\frak c}_2^-)= 0$ or $1$, and the endpoints of ${\frak c}_2^\pm$ are alternating along $(0,2\pi)$.  This is possible by ($j4$) in the proof of Lemma~\ref{almost alternate}.  (3)--(5), in particular the fact that adding ${\frak c}_3^+$ to ${\frak c}_2^+$ --- and similarly adding ${\frak c}_3^-$ to ${\frak c}_2^-$ --- leaves the grooming property invariant, follow from the proof of Lemma~\ref{almost alternate}.

We now prove that $I(\overline{v}_{1,j})\geq 0$ for $j>0$; the verification of $I(\overline{v}_{1,j})\geq I(\overline{v}'_{1,j})+I(\overline{v}''_{1,j})$ for $0\leq j\leq a$ is similar and is left to the reader.  First observe that $I(\check C_1)=0$ and $I(\check C_2)=0$, where the latter follows from the proof of Lemma~I.\ref{P1-lemma: HF index of sections at infinity}. Also $I(\check C_3)\geq 0$ by the ECH index inequality. Next observe that $\langle \check C_1,\check C_2\rangle=-l$, where $l$ is the degree of $\check C_1$. Hence the contribution of $\check C_1\cap \check C_2$ to $I(\overline{v}_{1,j})$ is $-2l$. We now apply Lemmas~I.\ref{P1-calc of almost sum} and I.\ref{P1-calc of almost sum 2}, where we split into $\check C_1$ and $\check C_2\cup \check C_3$ instead of $\overline{v}_{1,j}'$ and $\overline{v}_{1,j}''$. Observe that $q_2=0$ (resp.\ $q_2=1$) in Lemma~I.\ref{P1-calc of almost sum} corresponds to $q_2=-1$ (resp.\ $q_2=0$) in Lemma~I.\ref{P1-calc of almost sum 2}. One can verify that the extra contributions to $I$ from the rightmost terms of Equations~(I.\ref{P1-almost additive}) and (I.\ref{P1-almost additive 2}) add up to at least $2l$. Hence $I(\overline{v}_{1,j})\geq 0$ for $j>0$.

The case of $\overline{v}_{-1,j}$, $1\leq j\leq c+1$, is easier and will be omitted.
\end{proof}

\subsubsection{Boundary points at $z_\infty$} \label{boundpunc2}

In this subsection we describe the necessary modifications when $\overline{u}_\infty$ has boundary points at $z_\infty$.  We use the notation from Section~\ref{boundpunc}, with some modifications: The continuation method gives rise to a main cycle $\mathcal{Z}_{\mbox{\tiny main}}$ which winds around $\R/2\pi\Z$ once and a union $\mathcal{Z}_{\mbox{\tiny aux}}$ of auxiliary cycles. Also, in Definition~\ref{types of boundary punctures} we replace $\mathcal{Z}$ by $\mathcal{Z}_{\mbox{\tiny main}}$.

\begin{lemma} \label{nonnegative ECH indices 5 better}
If $\overline{u}_\infty$ has no fiber components, then the ECH index of each level $\overline{v}_*\not= \overline{v}_+$ is nonnegative, the only components of $\overline{u}_\infty$ which have negative ECH index are branched covers of $\sigma_\infty^+$, and there exist constants $\delta_{1,j}$ such that
\begin{equation}
I(\overline{v}_{1,j})= I(\overline{v}'_{1,j}) +I(\overline{v}''_{1,j})+\delta_{1,j}
\end{equation}
for $0\leq j\leq a$, where:
\begin{enumerate}
\item[(i)] if there is a boundary point of type (P$_3$) on $\overline{v}_{1,j}$, then $\delta_{1,j}\geq 2$;
\item[(ii)] if there is a boundary point of type (P$_2$) on $\overline{v}_{1,j}$, then $\delta_{1,j}\geq 2p_{1,j}$, where $p_*=\deg (\overline{v}'_*)$;
\item[(iii)] otherwise, $\delta_{1,j}=0$.
\end{enumerate}
\end{lemma}

We remark that we have not tried to obtain the best lower bound, just one that suffices for our purposes of eliminating boundary points at $z_\infty$.

\begin{proof}
Let ${\frak r}$ be a boundary point at $z_\infty$ on $\overline{v}_{*_0}$ with $*_0=(1,j_0)$. We compute the extra contribution $I_{L-}$ to $I(\overline{v}_{*_0}^{\circ,1})$ that comes from the left negative end at $z_\infty$ as in Lemma~\ref{Lisbeth}.

\begin{figure}[ht]
\begin{center}
\psfragscanon
\psfrag{a}{\tiny ${\frak z}'_1$}
\psfrag{b}{\tiny ${\frak z}'_2$}
\psfrag{c}{\tiny ${\frak z}'_3$}
\psfrag{d}{\tiny $\dots$}
\psfrag{e}{\tiny ${\frak z}'_s$}
\psfrag{s}{\tiny ${\frak z}''_i$}
\psfrag{f}{\tiny ${\frak z}''_1 $}
\psfrag{g}{\tiny ${\frak z}''_2 $}
\psfrag{h}{\tiny $\dots $}
\psfrag{i}{\tiny ${\frak z}''_{p_{*_0}} $}
\psfrag{B}{\tiny ${1\over 2}$}
\psfrag{A}{\tiny $1$}
\psfrag{C}{\tiny (i)}
\psfrag{D}{\tiny (ii)}
\includegraphics[width=6.5cm]{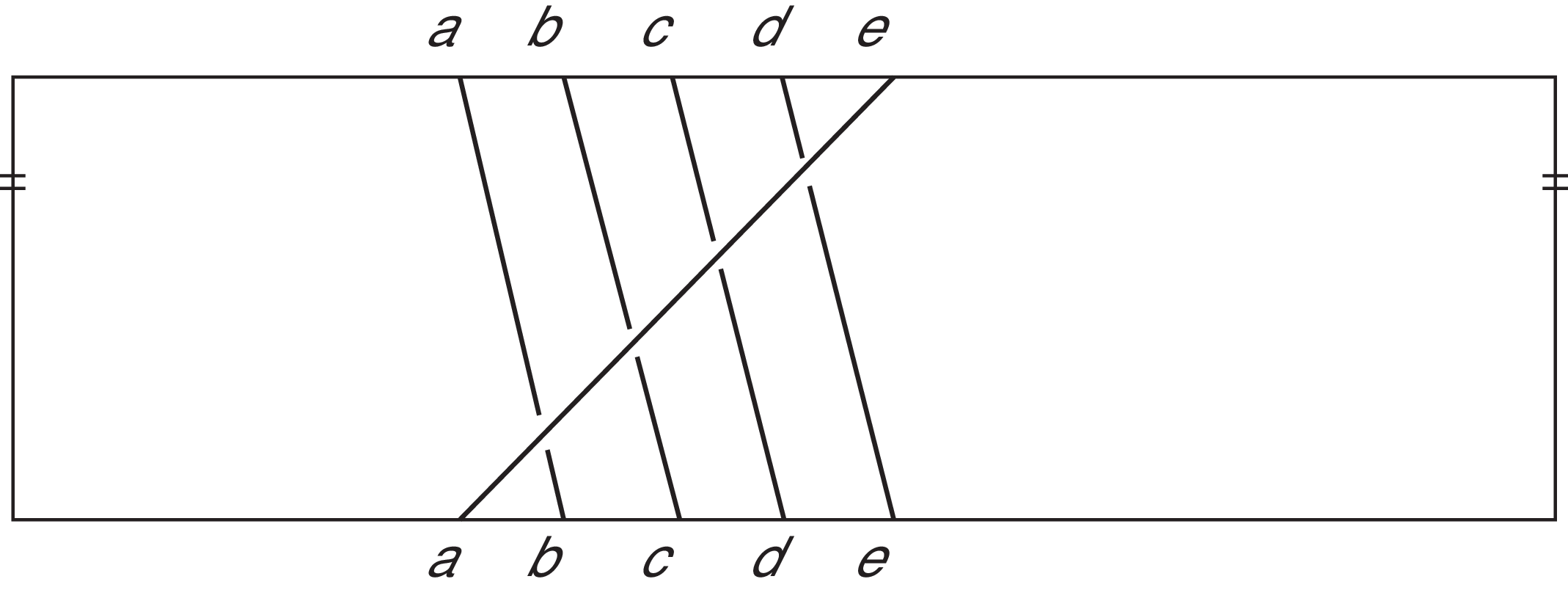}
\end{center}
\caption{${\frak c}'_{L-}\cup {\frak c}''_{L-}$ for ${\frak r}$ of (i) type (P$_3$) and (ii) type (P$_2$). Here ${\frak c}''_{L-}$ is the strand in the front.}
\label{figure: puncture2}
\end{figure}

(i) Suppose ${\frak r}$ is of type (P$_3$).  We assume (S) from Section~\ref{boundpunc} for simplicity.   First observe that $P_{L-,j_0,1/2}=P_{L-,j_0,1}$. By (S), $\mathcal{Z}_{\mbox{\tiny main}}$ has only one chord ${\frak z}'_1\to{\frak z}'_s$ corresponding to a sector of type ${\frak S}(\overline{a}_{k,l},\overline{a}_{k',l'})$. Hence we have:
$$P''_{L-,j_0,1/2}=\{{\frak z}'_1\},\quad P'_{L-,j_0,1/2}=\{{\frak z}'_2,\dots,{\frak z}'_s, {\frak z}''_1,\dots,{\frak z}''_t\},$$
$$P''_{L-,j_0,1}=\{{\frak z}'_s\},\quad P'_{L-,j_0,1}=\{{\frak z}'_1,\dots,{\frak z}'_{s-1}, {\frak z}''_1,\dots,{\frak z}''_t\},$$
where $P_{L-,j_0,1/2}=P_{L-,j_0,1}$ is written as $\{{\frak z}'_1,\dots,{\frak z}'_s, {\frak z}''_1,\dots,{\frak z}''_t\}$ in cyclic order around $\R/2\pi\Z$; see Figure~\ref{figure: puncture2}(i). Note that ${\frak z}'_2,\dots,{\frak z}'_{s-1}$ are points of $\mathcal{Z}_{\mbox{\tiny aux}}$ but not $\mathcal{Z}_{\mbox{\tiny main}}$; otherwise $\mathcal{Z}_{\mbox{\tiny main}}$ winds more than once around $\R/2\pi\Z$. The projection of the left negative end of $\overline{v}_{*_0}^{\circ,1,''}$ that limits to $z_\infty$ intersects $A^{[1/2,1]}_\varepsilon$ along an arc ${\frak c}''_{L-}$ with winding number $w({\frak c}''_{L-})=0$ or $1$, depending on whether ${\frak S}(\overline{a}_{k,l},\overline{a}_{k',l'})$ is a large sector. The left negative ends of $\overline{v}_{*_0}^{\circ,1,'}$ that limit to $z_\infty$ can modified to give a grooming ${\frak c}'_{L-}$ on $A^{[1/2,1]}_{\varepsilon/2}$ such that the winding number $w({\frak c}'_{L-})=0$ or $-1$ and ${\frak c}'_{L-}$ connects ${\frak z}''_i$ to ${\frak z}''_i$, $i=1,\dots,t$, by vertical arcs.

We now consider the disk $\check D$ that we ``append'' to the left negative end of $\overline{v}_{*_0}^{\circ,1}$ as in the proof of Lemma~I.\ref{P1-calc of almost sum}.
The writhe of ${\frak c}'_{L-}\cup {\frak c}''_{L-}$ is equal to $s-1$. Resolving the (positive) crossings of ${\frak c}'_{L-}\cup {\frak c}''_{L-}$ yields a grooming with vertical arcs from ${\frak z}'_i$ to ${\frak z}'_i$; this corresponds to a disk $\check D$ whose contributions to $Q, c_1, \mu$ are  $s-1,0,s-1$.  [The calculations for $\mu$ assume (R) from Section~\ref{boundpunc}. For example, if $w({\frak c}'_{L-})=w({\frak c}''_{L-})=0$, then $\mu$ of the positive ends are $1$ for $s-1$ arcs and $0$ for the rest and $\mu$ of the negative ends are all $0$.] The discrepancy contribution to $I$ is $\geq 0$.  Hence $I_{L-}=2(s-1)\geq 2$.

In general, each collection of boundary points of type (P$_3$) that map to the same point on the base contributes at least $+2$ towards $I$; this is analogous to Remark~\ref{Salander}.

The cycles $\mathcal{Z}^{(a-j)}_{\mbox{\tiny main}}$ and $\mathcal{Z}^{(a-j)}_{\mbox{\tiny aux}}$ are defined as before, for $j\geq j_0$. We define $\mathcal{Z}^{(a-j_0,+)}_*=\mathcal{Z}^{(a-j_0)}_*$, $\mathcal{Z}^{(a-j_0,-)}_{\mbox{\tiny aux}}=\mathcal{Z}^{(a-j_0,+)}_{\mbox{\tiny aux}}$, and $\mathcal{Z}^{(a-j_0,-)}_{\mbox{\tiny main}}$ as $\mathcal{Z}^{(a-j_0,+)}_{\mbox{\tiny main}}$ with ${\frak z}_1'\to {\frak z}'_s$ replaced by ${\frak z}'_s$.  Also, $P_{R-,j_0,i}^\star$, $\star=\varnothing,',''$, is obtained from $P_{+,j_0,i}^\star$ by replacing ${\frak z}'_1$ by ${\frak z}'_s$. The rest of the arguments of Lemmas~\ref{almost alternate} and \ref{nonnegative ECH indices 5} carry over.

(ii) Suppose ${\frak r}$ is the only boundary point of type (P$_2$). In this case ${\frak z}'_1={\frak z}'_s$ and $\mathcal{Z}_{\mbox{\tiny main}}=({\frak z}'_1\to {\frak z}'_1)$.  Then:
$$P''_{L-,j_0,1/2}=\{{\frak z}'_1\}, \quad P'_{L-,j_0,1/2}=\{{\frak z}''_1,\dots,{\frak z}''_{p_{*_0}}\},$$
$$P''_{L-,j_0,1}=\{{\frak z}'_1\}, \quad P'_{L-,j_0,1}=\{{\frak z}''_1,\dots,{\frak z}''_{p_{*_0}}\}.$$ Also ${\frak c}''_{L-}$ is an arc from ${\frak z}'_1$ to itself with $w({\frak c}''_{L-})=1$ and ${\frak c}'_{L-}$ consists of $p_{*_0}$ vertical arcs from ${\frak z}''_i$ to itself; see Figure~\ref{figure: puncture2}(ii).

In order to groom ${\frak c}'_{L-}\cup {\frak c}''_{L-}$ so that the result ${\frak c}$ satisfies $w({\frak c})=0$, we switch the (positive) crossings of ${\frak c}'_{L-}\cup {\frak c}''_{L-}$ while keeping the same endpoints.  This gives rise to $\check D$ which is a union of $p_{*_0}+1$ disks.  The total contributions to $Q,c_1,\mu$, and the discrepancy, are $2p_{*_0}+1, 1,-2,0$. [Again recall (R) from Section~\ref{boundpunc}. The writhe of ${\frak c}'_{L-}\cup {\frak c}''_{L-}$ is $2p_{*_0}$ which contributes $2p_{*_0}$ towards $Q$.  The writhe of ${\frak c}''_{L-}$ and its pushoff is $1$, which contributes an additional $1$ towards $Q$.  $\mu$ of the positive (resp.\ negative) ends are $0,\dots,0,-2$ (resp.\ all $0$).] Hence $I_{L-}=2p_{*_0}$.

We also obtain a lower bound of $2p_{*_0}$ in the general case of more boundary points of type (P$_2$); the details are left to the reader.
\end{proof}

The following analog of Lemma~\ref{intersezione prime revisited} is a consequence of Section~\ref{continuation argument}:

\begin{lemma} \label{intersezione revisited}
Suppose $\overline{v}'_{1,j}\cup\overline{v}^\sharp_{1,j}\not=\varnothing$ for some $j>0$. Let $\mathcal{E}_{-,i}$, $i=1,\dots,q$, be the negative ends of $\cup_{j=1}^a\overline{v}_{1,j}^\sharp$ that converge to $z_\infty$,  let $\mathcal{E}_{+,i}$, $i=1,\dots,r$, be the positive ends of $\cup_{j=0}^{a-1}\overline{v}_{1,j}^\sharp$ that converge to $z_\infty$, and let $\mathcal{E}'_i$, $i=1,\dots,s$, be the neighborhoods of the points of type (P$_3$).  Then
\begin{equation} \label{jasmine 2}
n^+(\mathcal{E}_{-,i})\geq k_0-1\gg 2g
\end{equation}
for each $i$, where the constant $k_0$ is as given in Section~I.\ref{P1-coconut}. If $q\not=0$ (i.e., some $\mathcal{E}_{-,i}$ exists) or not all $\mathcal{E}_{+,i}$ project to thin sectors, then
\begin{equation} \label{oolong 2}
n^+( (\cup_{i=1}^q\mathcal{E}_{-,i})\cup (\cup_{i=1}^r\mathcal{E}_{+,i})\cup (\cup_{i=1}^s \mathcal{E}'_i))\geq m-p_+,
\end{equation}
where $p_+=\deg (\overline{v}_+')$.
Moreover,
$$D^2_{\rho_0}-(\cup_{i=1}^q \pi_{D^2_{\rho_0}}(\mathcal{E}_{-,i}))\cup (\cup_{i=1}^r \pi_{D^2_{\rho_0}}(\mathcal{E}_{+,i}))\cup(\cup_{i=1}^s \pi_{D^2_{\rho_0}}(\mathcal{E}'_i))$$
consists of at most $p_+$ thin sectors.
\end{lemma}

The following is the analog of Lemma~\ref{cherimoya2}:

\begin{lemma}\label{cherimoya3}
If ${\frak p}_1,\dots,{\frak p}_s$ are the boundary points of type (P$_1$) and (P$_2$) and $N({\frak p}_i)\subset F_*$ is a small neighborhood of ${\frak p}_i$, then $\sum_{i=1}^s n^*(\overline{v}_*(N({\frak p}_i)))\geq m$.
\end{lemma}

\subsubsection{Asymptotic eigenfunctions}

Fix $m\gg 0$ and let $\overline{u}_i\to \overline{u}_\infty\in \bdry_{\{+\infty\}} \mathcal{M}$.

\begin{lemma}\label{lemma 2012}
There is no $\overline{u}_\infty\in \bdry_{\{+\infty\}} \mathcal{M}$ such that:
\begin{enumerate}
\item[(a)]  for all $\overline{v}_*\succeq \overline{v}_+$, $\deg(\overline{v}'_*)=p_+\geq 1$, $\overline{v}^\sharp_*=\varnothing$, and $\overline{v}'_*\cap \overline{v}''_*=\varnothing$; and
\item[(b)] $\overline{v}^\sharp_{0,b}$ is a union of $p_+$ cylinders from $\delta_0$ to $h$.
\end{enumerate}
\end{lemma}

\begin{proof}
Arguing by contradiction, we restrict $\overline{u}_i$ to a neighborhood of $\sigma_\infty^{\tau_i}$ and further restrict it to the components that are close to $\overline{v}'_+$. After projecting to $D^2_{\rho_0}$ {\em using balanced coordinates,} applying the ansatz given by Equation~(I.\ref{P1-eqn: ansatz}), rescaling, and taking the SFT limit with $m\gg 0$ fixed, we obtain a holomorphic map $w_+: \Sigma_+\to \C\P^1$ together with the branched cover $\pi_+:\Sigma_+\to cl(B_+)$.  By (a), $w_+$ maps each component of $\bdry \Sigma_+$ to a distinct thin sector ${\frak S}(\phi=\phi(\overline{a}_{i,j})+\tfrac{\pi}{m},\overline{\hh}(\overline{a}_{i,j}))$.\footnote{This complicated expression is the result of using balanced coordinates instead of $\pi_{D^2_{\rho_0}}$.}  By (a) and (b), the image of $w_+$ cannot contain $\infty\in \C\P^1$. Hence, by the open mapping property, $w_+$ maps each component of $\Sigma_+$ to a distinct  ${\frak S}(\phi=\phi(\overline{a}_{i,j})+\tfrac{\pi}{m},\overline{\hh}(\overline{a}_{i,j}))$.

From now on assume without loss of generality that $p_+=1$ and $\Sigma_+=B_+$. We identify $cl(B_+)$ with the closed unit disk $\overline{\D}=\{|z|\leq 1\}$ by sending $+\infty$ to $1$ and $-\infty$ to $0$.

The Lagrangian boundary condition for $\overline{u}_i$, when projected to $D^2_{\rho_0}$ using balanced coordinates, descends to the boundary condition $w_+(e^{i\theta})\in \mathcal{R}_{\varphi(e^{i\theta})}$, where the map $\varphi: \bdry \D-\{1\} \to S^1=\R/2\pi \Z$ satisfies the following:
\begin{itemize}
\item $\displaystyle\lim_{\theta\to 0^+} \varphi(e^{i\theta})=\phi(\overline{a}_{i_0,j_0})+ \tfrac{\pi}{m}$ and $\displaystyle\lim_{\theta\to 0^-}\varphi(e^{i\theta})= \phi(\overline{\hh}(\overline{a}_{i_0,j_0}))=\phi(\overline{a}_{i_0,j_0})+ \tfrac{2\pi}{m}$ for some $(i_0,j_0)$;
\item $\varphi(e^{i\theta})\in(0,2\pi)$ is a nondecreasing function of $\theta\in(0,2\pi)$; and
\item $\varphi\circ i_0 = i_1\circ \varphi$, where $i_0: \bdry \D-\{1\}\stackrel\sim\to \bdry \D-\{1\}$ is the reflection across the $x$-axis and $i_1: S^1\stackrel\sim\to S^1$ is the reflection across $\mathcal{R}_{\phi_0}$ where $\phi_0=\phi(\overline{a}_{i_0,j_0})+\tfrac{3\pi}{2m}$.
\end{itemize}

\begin{lemma} \label{2012}
The holomorphic map $w_+: \overline{\D}\to \C\P^1$ satisfies the following:
\begin{enumerate}
\item[(i)] $w_+(1)=0$ and $w_+(e^{i\theta})\in \mathcal{R}_{\varphi(e^{i\theta})}$ for $e^{i\theta}\not=1$;
\item[(ii)] $w_+(0)\in \mathcal{R}_{\phi_0}$ and $w_+$ maps $\mathcal{R}_0\cup \mathcal{R}_\pi$ to $\mathcal{R}_{\phi_0}$;
\item[(iii)] $\op{Im}(w_+)\subset{\frak S} ( \phi=\phi(\overline{a}_{i_0,j_0})+\tfrac{\pi}{m}, \overline{\hh}(\overline{a}_{i_0,j_0}))$;
\item[(iv)] $w_+$ is a biholomorphism onto its image.
\end{enumerate}
Moreover, $w_+$ is the unique holomorphic map $\overline{\D}\to \C\P^1$ which satisfies (i), (iii) and (iv), up to multiplication by a positive real constant.
\end{lemma}

\begin{proof}
(i), (iii), and (iv) are immediate from the construction. To prove the uniqueness statement, suppose there exists $\widetilde w_+$ satisfying (i), (iii) and (iv). Then we consider ${w_+\over \widetilde w_+}$ as in Section~I.\ref{P1-subsection: involutions}.  It is easy to see that ${w_+\over\widetilde w_+}$ is a positive real constant on $\bdry \D$ and has no zeros or poles on $\D$, hence is a constant map. (ii) then follows from the uniqueness, Observation~I.\ref{P1-obs: involution}, and the fact that  $i_0$ and $i_1$ extend to involutions of $\D$ and $\C\P^1$. Here we are viewing $S^1$ as the unit circle in $\C\P^1$.
\end{proof}

The proof of Lemma~\ref{lemma 2012} now follows from the first part of Lemma~\ref{2012}(ii): Suppose $\overline{v}^\sharp_{0,b}$ is a cylinder from $\delta_0$ to $h$. Then by the choice of $h$ from Section~\ref{subsubsection: convention bambi} we have a contradiction, since the asymptotic eigenfunction of $\overline{v}^\sharp_{0,b}$ at the positive end $\delta_0$ is a constant $c\in \mathcal{R}_{\phi=-2\pi/m}$.
\end{proof}

\subsubsection{Preliminary restrictions}

\begin{lemma}  \label{eliminate fiber components 2}
If $m\gg 0$ and $\overline{u}_\infty\in \bdry_{\{+\infty\}}\mathcal{M}$, then the following hold:
\begin{enumerate}
\item $\overline{u}_\infty$ has no fiber components;
\item there is no level $\overline{v}_*$ such that $\overline{v}'_*\cap \overline{v}''_*\not=\varnothing$ and $\overline{v}'_*\cap \overline{v}''_*\subset int(\overline{v}'_*)$; and
\item $\overline{u}_\infty$ has no boundary point at $z_\infty$.
\end{enumerate}
\end{lemma}

\begin{proof}
First observe that ``not (1)'', ``not (2)'', and ``not (3)'' are mutually exclusive by considerations of $n^*$.

(1) Arguing by contradiction, suppose that $\overline{u}_\infty$ has a fiber component $\widetilde{v}:\dot F\to \overline{W}_*$. The fiber component $\widetilde{v}$ satisfies $n^*(\widetilde{v})\geq m$.

We first claim the following:
\begin{enumerate}
\item[(a)] The only ends of $\overline{v}_*^\sharp$ that limit to multiples of $z_\infty$ or $\delta_0$ are positive ends.
\item[(b)] All the positive ends $\mathcal{E}_+$ of $\overline{v}_{0,j}^\sharp$, $j=0,\dots,b$, that limit to $\delta_0^r$ satisfy $n^*(\mathcal{E}_+)=r$.
\item[(c)] All the positive ends of $\overline{v}_{1,j}^\sharp$, $j=0,\dots,a$, and $\overline{v}_{-1,j}^\sharp$, $j=1,\dots,c$, that limit to $z_\infty$ project to thin sectors under $\pi_{D^2_{\rho_0}}$.
\item[(d)] $p_{*_1}\leq p_{*_2}$ if $\overline{v}_{*_1}\preceq \overline{v}_{*_2}$.
\end{enumerate}
(a) Arguing by contradiction, let $\mathcal{E}_-$ be a negative end of some $\overline{v}_*^\sharp$ that limits to a multiple of $z_\infty$ or $\delta_0$. Then $n^*(\mathcal{E}_-)\gg 2g$ by Lemmas~\ref{intersezione}(1) and \ref{intersezione revisited}. Since $n^*(\widetilde{v})\geq m$, such a negative end $\mathcal{E}_-$ cannot exist by Equation~\eqref{sum of n part 2}. The arguments for (b) and (c) are similar and (d) is a consequence of (a)--(c).

Next we have the following contributions towards $I$:
\begin{enumerate}
\item[($\beta_1$)] $I(\overline{v}_+')=-p_{0,b+1}$.
\item[($\beta_2$)] $\cup_{j=1}^b\overline{v}_{0,j}^\sharp$ is a union of $p_{0,b+1}-p_{0,1}$ cylinders from $\delta_0$ to $h$ or $e$ and
    $$\sum_{j=1}^b I(\overline{v}_{0,j}^\sharp)\geq p_{0,b+1}-p_{0,1}.$$
\item[($\beta_3$)] $\overline{v}_-^\sharp$ is a union of $p_{0,1}-p_{0,0}$ sections from $\delta_0$ to $x_i$ or $x_i'$ and $$I(\overline{v}_-^\sharp)=p_{0,1}-p_{0,0}.$$
\item[($\beta_4$)] $\cup_{j=1}^c\overline{v}_{-1,j}^\sharp$ is a union of $p_{0,0}$ trivial strips and $$\sum_{j=1}^cI(\overline{v}_{-1,j}^\sharp)=p_{0,0}.$$
\item[($\beta_5$)] If the fiber $\widetilde{v}$ is a component of $\overline{v}_*$, then $\widetilde{v}$ and the intersection $\widetilde{v}\cap (\overline{v}_*-\widetilde{v})$ contribute $2g+2\geq 4$ towards $I$.
\item[($\beta_6$)] After removing any fiber components, all the levels $\not=\overline{v}_+$ have nonnegative ECH index, $I(\overline{v}_+'')\geq 0$, and $I(\overline{v}_+)\geq I(\overline{v}_+')+I(\overline{v}_+'')$.
\end{enumerate}
($\beta_1$)--($\beta_4$) are clear.

We now prove ($\beta_5$). Suppose that $\widetilde{v}$ is a component of $\overline{v}_{*_0}$. If $\overline{v}''_{*_0}$ has a boundary point at $z_\infty$, then a neighborhood of the boundary point contributes $n^*\gg 2g$. Since $n^*(\widetilde{v})\geq m$, we have a contradiction of Equation~\eqref{sum of n part 2}.
 Hence $\overline{v}''_{*_0}$ does not have a boundary point at $z_\infty$.

When $\widetilde{v}$ is an interior fiber, ($\beta_5$) follows from Equation~(I.\ref{P1-persimmon}).

Next suppose that $\widetilde{v}$ is a boundary fiber, i.e., $\widetilde{v}(\widetilde F)\subset\{p\}\times\overline{S}$, where $\widetilde F$ is the domain of $\widetilde v$ and $p\in \bdry B_{*_0}$. If $(p,z_\infty)\not \in\widetilde{v}(\bdry {\widetilde F})$, i.e., $\widetilde{v}(\bdry\widetilde F)$ consists of $2g$ slits along $\widehat{\bf a}$, for example, then the Fredholm index of $\widetilde{v}$ is:
$$\op{ind}(\widetilde{v})=-\chi(\widetilde F) +\mu_\tau(\bdry \widetilde F) +2c_1(\widetilde{v}^*T\overline{S},\tau)= -(2-4g)+0 +2 (2-2g)= 2,$$
where $\tau$ is a partially defined trivialization of $T\overline{S}$ along $\widehat{\bf a}$ which directs $T\widehat{\bf a}$.  Since there are $2g$ arcs of $\overline{u}_i$ that are pinched to points when taking the limit, $\widetilde{v}$ and the ``pinched points'' contribute a total of $2g+2$ towards $I$.

Now suppose that $(p,z_\infty)\in \widetilde{v}(\bdry\widetilde F)$.  Suppose that $\overline{v}''_{*_0}=\varnothing$ and that $*_0=(1,j)$ with $j\geq 1$. By considerations of $n^*$,
\begin{enumerate}
\item[(*)] all the components of the data $\overrightarrow{\mathcal{D}}_\pm$ of the positive and negative ends of $\overline{v}_{*_0}$ are of the form $(i,j)\to(i,j)$.
\end{enumerate}
Let $\tau_0$ be a groomed multivalued trivialization which is compatible with $\overrightarrow{\mathcal{D}}_\pm$ and let ${\frak c}_+={\frak c}_-\subset A_\varepsilon=[0,1]\times \bdry D^2_\varepsilon$ be the corresponding groomings.  Then a $\tau_0$-trivial representative $\check C$ of $\widetilde{v}\cup \overline{v}'_{*_0}$ satisfies
$$[\check C]=[[-1,1]\times{\frak c}]+[\overline{S}]\in \pi_2({\bf z}_+,{\bf z}_-,\tau_0),$$
where ${\bf z}_\pm=\{z_\infty^{2g}(\overrightarrow{\mathcal{D}}_\pm)\}$. ($\beta_5$) then follows from the calculation for a closed interior fiber $\widetilde{v}$ when $\overline{v}''_{*_0}=\varnothing$ and $*_0=(1,j)$ with $j\geq 1$.  The general case follows by incorporating considerations of the previous paragraph and is left to the reader.

 ($\beta_6$) follows from Lemma~\ref{nonnegative ECH indices 5} by observing that proof carries over when fibered components are removed from $\overline{u}_\infty$: Removing interior fibers and boundary fibers with $(p,z_\infty)\not \in\widetilde{v}(\bdry {\widetilde F})$ do not affect groomings.  If $\overline{u}_\infty$ has a boundary fiber with $(p,z_\infty)\in\widetilde{v}(\bdry {\widetilde F})$, then (*) holds, which also is sufficient.

Summing ($\beta_1$)--($\beta_6$), the total ECH index is $\geq 4$, which contradicts Equation \eqref{sum of I part 2}. Hence (1) follows.

\s
(2) Arguing by contradiction, suppose there exist sequences $m_l\to\infty$ and $\overline{u}_{li}\to \overline{u}_{l\infty}$, where $\overline{u}_{li}\in\mathcal{M}^{(m_l)}$, $\overline{u}_{l\infty}\in \bdry_{\{+\infty\}}\mathcal{M}^{(m_l)}$, and $\mathcal{M}^{(m_l)}$ is $\mathcal{M}$ with respect to $m_l$, such that each $\overline{u}_{l\infty}$ has some $\overline{v}_{l,*_0}$ such that $\overline{v}'_{l,*_0}\cap \overline{v}''_{l,*_0}\not=\varnothing$.  Unless indicated otherwise, we fix $l\gg 0$ and suppress $l$ from the notation.

The same argument as in (1) implies that (a)--(d) hold. 
Since $n^*(\overline{u}_i)=m+|\mathcal{I}|\leq m+2g$, there is only one intersection point of $\overline{v}'_{*_0}$ and $\overline{v}''_{*_0}$, which we denote by ${\frak r}=({\frak r}^b,z_\infty)$.

\s
(2A) Suppose that $*_0=+$. As in the proof of (1), the holomorphic building hanging from the positive end of $\overline{v}_+'$ has total ECH index $\geq 0$ --- this is obtained by adding all the ECH index contributions from ($\beta_1$)--($\beta_4$). We also have ($\beta_5'$) instead of ($\beta_5$):
\begin{enumerate}
\item[($\beta_5'$)] $\overline{v}'_{*_0}\cap \overline{v}''_{*_0}$ contributes $2\cdot m({\frak r})$ towards $I$, where $m({\frak r})\geq 1$ is the multiplicity of ${\frak r}$.
\end{enumerate}
Hence $\sum_{\overline{v}_*}I(\overline{v}_*)\geq 2 \cdot m({\frak r})$. This implies that $m({\frak r})=1$ and $p_+=\deg \overline{v}'_{+}=1$.
The sum of the ECH indices from ($\beta_1$)--($\beta_4$), ($\beta_5'$), ($\beta_6$), and $I(\overline{v}''_+)\geq 0$ is at least $+2$.

We claim that $p_-\geq 1$ by considerations of $n^*$. Indeed, if $\overline{v}_-'=\varnothing$, then there exist a point ${\frak q}\in int(F_-)$ and a sufficiently small neighborhood $N({\frak q})\subset \dot F_-$ of ${\frak q}$ such that $\overline{v}_-({\frak q})=\overline{\frak m}(+\infty)$ and $n^*(\overline{v}_-(N({\frak q})))\geq m$. This is a contradiction.

Since $p_+=1$ and $p_-\geq 1$, it follows that $p_{0,j}=1$ for all $j=0,\dots,b+1$ by (d). For the purposes of applying the rescaling argument, we may assume that $b=0$. By restricting $\overline{u}_{li}$ to a neighborhood of $\sigma_\infty^{\tau_{li}}$, taking the limit of a diagonal subsequence $\overline{u}_{li(l)}$, and applying the rescaling argument, we obtain a $2$-level holomorphic building $w_+\cup w_-$, where $w_\pm:cl(B_\pm) \to \C\P^1$ satisfy the following:
\begin{enumerate}
\item[(i)] $w_-(\bdry B_-)\subset \{\phi=0,\rho>0\}$;
\item[(ii)] $w_-(\overline{\frak m}^b(+\infty))=0$ and $w_-(+\infty)=\infty$;
\item[(iii)] $w_+(\bdry B_+)\subset \{\phi=0,\rho>0\}$;
\item[(iv)] $w_+(-\infty)=0$ and $w_+({\frak r}^b)=\infty$;
\item[(v)] $w_\pm|_{int(B_\pm)}$ is a biholomorphism onto its image.
\end{enumerate}
Then, by the Involution Lemma~I.\ref{P1-lemma: effect of involutions}, $w_-$ maps the marker $\dot{\mathcal{L}}_{3/2}(+\infty)$ to $\dot{\mathcal{R}}_\pi(\infty)$ and $w_+$ maps the ray $\mathcal{L}_{3/2}$ to the ray $\mathcal{R}_\pi$. This
gives rise to two constraints for $\overline{v}''_+$: ${\frak r}^b$ is restricted on $\mathcal{L}_{3/2}$ and the derivative of $\overline{v}''_+$ at ${\frak r}$ in the $D^2_{\rho_0}$-direction is also restricted.  Hence $I(\overline{v}''_+)\geq 2$ and the total ECH index is at least $4$, a contradiction.

\s
(2B) Suppose that $*_0=(1,j_0)$ for some $1\leq j_0\leq a$. Summing the ECH indices using ($\beta_1$)--($\beta_4$), ($\beta_5'$), ($\beta_6$), and $I(\overline{v}''_{*_0})\geq 1$, the total ECH index is at least $3$, a contradiction. However, we want to do slightly better so that the argument carries over in the case of Lemma~\ref{eliminate fiber components}: The rescaling argument with $m_l\to\infty$ gives $w_+: cl(B_+)\to \C\P^1$ satisfying the following:
\begin{enumerate}
\item[(i)] $w_+(\bdry B_+)\subset \{\phi=0,\rho>0\}$;
\item[(ii)] $w_+(-\infty)=0$ and $w_+(+\infty)=\infty$;
\item[(iii)] $w_+$ maps the marker $\dot{\mathcal{L}}_{3/2}(-\infty)$ to $\dot{\mathcal{R}}_\pi(0)$;
\item[(iv)] $w_+|_{int(B_+)}$ is a biholomorphism onto its image.
\end{enumerate}
By the Involution Lemmas, $w_+^{-1}$ maps $\dot{\mathcal{R}}_0(0)$ to $\dot{\mathcal{L}}_{3/2}(-\infty)$. This contradicts (iii).

\s
(2C) Suppose that $*_0=(0,j_0)$ for some $1\leq j_0\leq b$. We replace ($\beta_2$) by:
\begin{enumerate}
\item[($\beta'_2$)] If $j\geq j_0$, then each positive end of $\overline{v}_{0,j_0}^\sharp$ that limits to $\delta_0$ contributes $+1$ towards $I$.
\item[($\beta''_2$)] All the other components of $\overline{v}_{0,j}^\sharp$, $1\leq j<j_0$, are cylinders from $\delta_0$ to $h$ or $e$ with $I=1$ or $2$.
\end{enumerate}
($\beta'_2$) is a consequence of the argument of Lemma~\ref{lemma 2012}: there is an independent condition imposed on the asymptotic expansion of each positive end of $\overline{v}_{0,j}^\sharp$ with $j\geq j_0$ that limits to $\delta_0$. ($\beta''_2$) is clear.  Summing the ECH indices using ($\beta_1$), ($\beta_2'$),  ($\beta''_2$) ($\beta_3$), ($\beta_4$), ($\beta_5'$), ($\beta_6$), and $I(\overline{v}''_{*_0})\geq 1$, the total ECH index is at least $3$, a contradiction.

\s
(2D) Suppose that $*_0=-$. In a manner similar to (2C) we replace ($\beta_3$) by:
\begin{enumerate}
\item[($\beta_3'$)] Each positive end of $\overline{v}_-^\sharp$ that limits to $\delta_0$ and is not the end of a section from $\delta_0$ to $x_i$ or $x_i'$ contributes  $+1$ towards $I$.
\item[($\beta_3''$)] All the other components of $\overline{v}_-^\sharp$ are sections from $\delta_0$ to $x_i$ or $x_i'$ with $I=1$.
\end{enumerate}
Summing the ECH indices using ($\beta_1$), ($\beta_2$), ($\beta_3'$), ($\beta_3''$), ($\beta_4$), ($\beta_5'$), ($\beta_6$), and $I(\overline{v}''_-)\geq 0$, the total ECH index is at least $2$.

Suppose that ${\frak r}^b\not=\overline{\frak m}^b(+\infty)$. Then the rescaling argument gives $w_-: cl(B_-)\to \C\P^1$ which satisfies the following:
\begin{enumerate}
\item[(i)] $w_-(\bdry B_-)\subset \{\phi=0,\rho>0\}$;
\item[(ii)] $w_-(\overline{\frak m}^b(+\infty))=0$ and $w_-({\frak r}^b)=\infty$;
\item[(iii)] $w_-(+\infty)\subset \{\phi=0,\rho>0\}$; and
\item[(iv)] $w_-|_{int(B_-)}$ is a biholomorphism onto its image.
\end{enumerate}
Here (iii) follows from considering the continuation of $w_{-,m_l}$ to upper levels $w_{*,m_l}$ and $\pi_{*,m_l}$ (here $w_{*,m_l}$ is the limit of $\overline{u}_{li}$ for $m_l\gg 0$, without taking $m_l\to\infty$) and the fact that $w_{+,m_l}:\Sigma_+\to\C\P^1$ must map $\bdry \Sigma_+$ to a thin sector near $\{\phi=0,\rho>0\}$.

By the Involution Lemma~I.\ref{P1-lemma: effect of involutions} and (i)--(iv), ${\frak r}^b$ must lie on $\mathcal{L}_{1/2}\cup\mathcal{L}_{3/2}$ and the map $w_-^{-1}$ maps the marker $\dot{\mathcal{R}}_\pi(\infty)$ to the marker $\dot{\mathcal{L}}_{1/2}$ or $\dot{\mathcal{L}}_{3/2}$ at ${\frak r}^b$, as appropriate. This  gives rise to two constraints for $\overline{v}_-''$. Hence $I(\overline{v}''_-)\geq 2$ and the total ECH index is $\geq 4$, a contradiction.

Now suppose that ${\frak r}^b=\overline{\frak m}^b(+\infty)$. Since passing through $\overline{\frak m}^b(+\infty)$ is a codimension $2$ condition, we have $I(\overline{v}''_-)\geq 2$.  The total ECH index is $\geq 4$, a contradiction.

\s
(2E) Suppose that $*_0=(-1,j_0)$ for some $1\leq j_0\leq c$. Then ($\beta_1$)--($\beta_3$) hold and the rescaling argument with $m_l\to\infty$ gives $w_-: \Sigma_-\to \C\P^1$ and $\pi_-: \Sigma_-\to cl(B_-)$ such that:
\begin{enumerate}
\item[(i)] $w_-(\bdry\Sigma_-)\subset \{\phi=0,\rho>0\}$;
\item[(ii)] $w_-(z_0)=0$ for some $z_0\in \pi_-^{-1}(\overline{\frak m}^b(+\infty))$;
\item[(iii)] $w_-(z_1)=\infty$ for some $z_1\in \pi_-^{-1}(-\infty)$;
\item[(iv)] $w_-(\pi_-^{-1}(+\infty))\subset\{\phi=0,\rho>0\}$;
\item[(v)] $w_-|_{int(\Sigma_-)}$ is a biholomorphism onto its image.
\end{enumerate}
Here (iv) follows from the considerations similar to those of Lemma~\ref{lemma 2012}, together with (a)--(d).
By the Involution Lemmas, $\pi_-\circ w_-^{-1}$, where defined, maps the component of $\{\op{Im}(z)=0\}$ passing through $0$ to $\mathcal{L}_{3/2}$. This contradicts (iii).

\s
(3) Suppose there is a boundary point ${\frak r}\in \bdry F_{*_0}$ at $z_\infty$.  Then the neighborhood of $\overline{v}''_{*_0}({\frak r})$ contributes at least $k_0-1\gg 2g$ towards $n^*$ by Lemma~\ref{cherimoya}.  We remark that a priori there could be more than one boundary point at $z_\infty$.

\s
(3A) Suppose that $\overline{v}_-'=\varnothing$. Then there is a point ${\frak q}\in int(F_-)$ and a sufficiently small neighborhood $N({\frak q})\subset \dot F_-$ of ${\frak q}$ such that $\overline{v}_-({\frak q})=\overline{\frak m}(+\infty)$ and $n^*(\overline{v}_-(N({\frak q})))\geq m$. This is a contradiction.

\s
(3B) Suppose that $\overline{v}_-'\not=\varnothing$ and $\overline{v}'_{*_0}=\varnothing$, i.e., ${\frak r}$ is of type (P$_1$). Then there exist restrictions $\widetilde{u}_{i,-}$ and $\widetilde{u}_{i,*_0}$ of $\overline{u}_i$ such that $\widetilde{u}_{i,-}$ is close to $\overline{v}'_-$ after translation and passes through $\overline{\frak m}(\tau_i)$, $\widetilde{u}_{i,*_0}$ is close to $\overline{v}''_{*_0}$ after translation and nontrivially intersects $\sigma_\infty^{\tau_i}$, and the images of $\widetilde{u}_{i,-}$ and $\widetilde{u}_{i,*_0}$ are disjoint.  This is a contradiction since we have an excess of $n^*$.

\s
(3C) Suppose that $\overline{v}_-'\not=\varnothing$, $\overline{v}'_{*_0}\not =\varnothing$, and $*_0=(1,j)$, $j=0,\dots,a$.

We claim that (a)--(d) from the proof of (1) hold, where we are only considering levels $\overline{v}_*\prec \overline{v}_+$. Indeed, if (a)--(c) do not hold, then the sum of $n^*$ of the ends is $\geq m-2g$ by Lemmas~\ref{intersezione} and \ref{intersezione revisited}, a contradiction.  (The situation of $\overline{v}_{-1,j}^\sharp$ is not explicitly covered by Lemma~\ref{intersezione revisited}, but the same proof works.)   (d) is a consequence of (a)--(c). The claim implies that ($\beta_1$)--($\beta_4$) hold.

Next we claim that there must be a boundary point at $z_\infty$ on $\overline{v}_+$.  Arguing by contradiction, if there is no boundary point at $z_\infty$ on $\overline{v}_+$, then a rescaling argument similar to (2B) gives $w_+: \Sigma_+\to \C\P^1$ and $\pi_+:\Sigma_+\to cl(B_+)$ such that:
\begin{enumerate}
\item[(i)] $w_+(\bdry \Sigma_+)\subset \{\phi=0,\rho>0\}$;
\item[(ii)] $w_+(z_0)=0$ for some $z_0\in \pi_+^{-1}(-\infty)$;
\item[(iii)] $w_+(\pi_+^{-1}(-\infty))\subset \{\phi=0,\rho\geq 0\}$;
\item[(iv)] $w_+(z_1)=\infty$ for some $z_1\in \pi_+^{-1}(+\infty)$;
\item[(v)] $w_+|_{int(\Sigma_+)}$ is a biholomorphism onto its image.
\end{enumerate}
Here (iii) follows from observing that:
\begin{enumerate}
\item[($\beta'''_2$)] There is no component of $\overline{v}^\sharp_{0,j}$ which is a cylinder from $\delta_0$ to $e$.
\end{enumerate}
Indeed, ${\frak r}$ contributes at least $2$ towards $I$ by Lemma~\ref{nonnegative ECH indices 5 better}, $I(\overline{v}''_{*_0})\geq 1$ if $*_0\not=+$, and a component of $\overline{v}^\sharp_{0,j}$ which is a cylinder from $\delta_0$ to $e$ contributes $I=2$. Together with ($\beta_1$)--($\beta_4$), the total ECH index is at least $4$, if there is a cylinder from $\delta_0$ to $e$, a contradiction.

By the Involution Lemmas, $\pi_+\circ w_+^{-1}$, where defined, maps  $\dot{\mathcal{R}}_0(0)$ to $\dot{\mathcal{L}}_{3/2}(-\infty)$, a contradiction. Hence there must be a boundary point at $z_\infty$ on $\overline{v}_+$. A similar argument implies that a boundary point at $z_\infty$ must lie on $\mathcal{L}_{3/2}$ and project to a large sector.

By Lemma~\ref{nonnegative ECH indices 5 better}, the boundary points at $z_\infty$ contribute at least $2$ towards $I$. In addition, the constraint of lying in $\mathcal{L}_{3/2}$ contributes $1$ towards $I$ and the large sector condition contributes $1$ towards $I$. These add up to at least $I=4$, a contradiction.

\s
(3D) Suppose that $\overline{v}_-'\not=\varnothing$, $\overline{v}'_{*_0}\not =\varnothing$, and $*_0=(-1,j)$, $j=1,\dots,c+1$.

 We would like apply the argument from (3C) and (2E).  The slight complication is that we may have ends of $\overline{v}_{-1,j}^\sharp$ that limit to $z_\infty$ but map to non-thin sectors.  In particular, there may be components of $\overline{v}_{-1,j}^\sharp$ have have at least two ends that limit to $z_\infty$ but has $\op{ind}=1$ and it is not a priori clear that ($\beta_4$) holds.

We claim that:
\begin{enumerate}
\item[($\beta_4'$)] Each positive end of $\overline{v}_{-1,j}^\sharp$, $j=1,\dots,c$, that limits to $z_\infty$ contributes $+1$ towards $I$.
\end{enumerate}
For ease of notation assume that $c=1$ and $p_{0,0}=2$ during the proof of the claim. Applying the rescaling argument without taking $m\to \infty$, we obtain $w_-:\Sigma_-\to \C\P^1$ and $\pi_-: \Sigma_-\to cl(B_-)$ such that:
\begin{enumerate}
\item[(i)] $w_-(z_0)=0$ for some $z_0\in \pi_-^{-1}(\overline{\frak m}^b(+\infty))$;
\item[(ii)] $w_-$ maps the negative ends of $\Sigma_-$ to sectors ${\frak S}(\overline{b}_{k,l},\overline{\hh}(\overline{b}_{k',l'}))$.
\end{enumerate}
We now make crucial use of the fact that the sectors ${\frak S}(\overline{b}_{k,l},\overline{\hh}(\overline{b}_{k',l'}))$ corresponding to the negative ends of $\Sigma_-$ have different angles, provided the sectors are not both thin sectors; this follows from the definition of $\overline{\bf a}$ from Section~I.\ref{P1-coconut}. Since each thin sector corresponds to a thin strip of $\overline{v}_{-1,1}^\sharp$ which contributes $I=1$, let us assume that there are no thin sectors in (ii). We use coordinates $(s,t)\in(-\infty, c]\times[0,1]$ for the two negative ends $\mathcal{E}_{-,1}$, $\mathcal{E}_{-,2}$ of $\Sigma_-$  which agree with the coordinates $(s,t)$ on $B_-$ and write $w_-^i:= w_-|_{\mathcal{E}_{-,i}}$. Then $w_-^i(s,t)\approx c_i e^{\lambda_i (s+it)}$, where $\lambda_1>\lambda_2>0$ without loss of generality; $c_1,c_2\not=0$; and the approximation gets better as $s\to -\infty$. On the other hand, suppose that $\op{ind}(\overline{v}_{-1,1}^\sharp)=1$ and $\overline{v}_{-1,1}^\sharp$ is connected and has two positive ends $\mathcal{E}_{+,i}$, $i=1,2$, that limit to $z_\infty$. Here $\mathcal{E}_{+,i}$ is to be paired with $\mathcal{E}_{-,i}$.  We use coordinates $(s',t')\in [c',\infty)\times[0,1]$ for the positive ends $\mathcal{E}_{+,i}$ which agree with the coordinates on $B=\R\times[0,1]$.  Then $\mathcal{E}_{+,1}$ approaches $z_\infty$ much faster than $\mathcal{E}_{+,2}$ does, as $s'\to +\infty$. This is inconsistent with our description of $w_-$, when we change coordinates $s'=s+\tilde c$ for some $\tilde c$. Hence each end $\mathcal{E}_{+,i}$ should be allowed to move independently and each positive end of $\overline{v}^\sharp_{-1,1}$ that limits to $z_\infty$ contributes $+1$ towards $I$.

($\beta_4'$) replaces ($\beta_4$) and the rest of the argument is similar to those of (3C) and (2E).
\end{proof}

\subsubsection{The case $\overline{v}'_-\not=\varnothing$}

\begin{lemma}\label{pooh}
If $m\gg 0$, $\overline{u}_\infty\in \bdry_{\{+\infty\}}\mathcal{M}$, and $\overline{v}_-'\not=\varnothing$, then no negative end of $\overline{v}_{0,j}^\sharp$, $1\leq j\leq b$, is asymptotic to a multiple of $\delta_0$.
\end{lemma}

\begin{proof}
Arguing by contradiction, suppose there exists some $\overline{v}_{0,j_0}^\sharp$, $1\leq j_0\leq b$, which has a negative end asymptotic to a multiple of $\delta_0$, say $\delta_0^{q_{0,j_0}}$. Since the negative end of $\overline{v}_{0,j_0}^\sharp$ contributes at least $m-q_{0,j_0}$ towards $n^*$, the following hold:
\begin{enumerate}
\item[(i)] at most one $\overline{v}_{0,j}^\sharp$, $j>0$, has a negative end asymptotic to a multiple of $\delta_0$;
\item[(ii)] only one negative end of $\overline{v}_{0,j_0}^\sharp$ limits to a multiple of $\delta_0$;
\item[(iii)] $p_{0,0}\leq p_{0,1}\leq \dots\leq p_{0,j_0-1}\leq p_{0,j_0}+q_{0,j_0}$ and $p_{0,j_0}\leq p_{0,j_0+1}\leq \dots \leq p_{0,b+1}$;
\item[(iv)] if $\mathcal{E}$ is a positive end of $\overline{v}_{0,j}^\sharp$, $j\geq 0$, that limits to $\delta_0^r$, then $n^*(\mathcal{E})=r$; and
\item[(v)] $\cup_{j=1}^{c}\overline{v}_{-1,j}^\sharp$ is a union of $p_{0,0}>0$ thin strips.
\end{enumerate}

By Lemma~\ref{nonnegative ECH indices 5 better}, all the components besides $\overline{v}'_+$ have nonnegative ECH index.  We have contributions ($\beta_1$), ($\beta_3$), ($\beta_4$), ($\beta_6$) as in Lemma~\ref{eliminate fiber components 2}, as well as the following variant of ($\beta_2$):
\begin{enumerate}
\item[($\widetilde\beta_2'$)] For $1\leq j<j_0$, $\overline{v}_{0,j}^\sharp$ is a union of cylinders from $\delta_0$ to $h$ or $e$ and
    $$\sum_{j=1}^{j_0-1} I(\overline{v}^\sharp_{0,j})\geq p_{0,j_0}+q_{0,j_0}-p_{0,1}.$$
\item[($\widetilde\beta_2''$)]  For $j_0+1\leq j\leq b$, $\overline{v}_{0,j}^\sharp$  is a union of cylinders from $\delta_0$ to $h$ or $e$ and
    $$\sum_{j=j_0+1}^{b} I(\overline{v}^\sharp_{0,j})\geq p_{0,b+1}-p_{0,j_0+1}.$$
\item[($\widetilde\beta_2'''$)] $I(\overline{v}^\sharp_{0,j_0})\geq p_{0,j_0+1}-p_{0,j_0}+1$.
\end{enumerate}
($\widetilde\beta_2'$) and ($\widetilde\beta_2''$) are immediate from $n^*(\overline{v}^\sharp_{0,j_0})\geq m-q_{0,j_0}$.  ($\widetilde\beta_2'''$) follows from the proof of Lemma~\ref{lemma 2012}, which implies that each positive end of $\overline{v}^\sharp_{0,j_0}$ that limits to $\delta_0$ additionally contributes one constraint.

Since $\overline{v}_{0,j_0}$ does not satisfy the partition condition at $\delta_0^{p_{0,j_0}+q_{0,j_0}}$, a straightforward calculation which takes into account the braiding near $\delta_0^{p_{0,j_0}+q_{0,j_0}}$ gives:
\begin{equation} \label{A1}
I(\overline{v}_{0,j_0})\geq I(\overline{v}'_{0,j_0})+I(\overline{v}''_{0,j_0})+2p_{0,j_0}.
\end{equation}

Summing the contributions of ($\beta_1$), ($\widetilde\beta'_2$)--($\widetilde\beta_2'''$), ($\beta_3$), ($\beta_4$), and Equation \eqref{A1}, we obtain:
\begin{equation} \label{nesbo}
\sum_{\overline{v}_*}I(\overline{v}_*)\geq q_{0,j_0} +2p_{0,j_0}+1.
\end{equation}
Here $q_{0,j_0}\geq 1$ by assumption. If $p_{0,j_0}>0$ or $p_{0,j_0+1}>0$, then $\sum_{\overline{v}_*}I(\overline{v}_*)\geq  4$, a contradiction. On the other hand, if $p_{0,j_0}=\dots=p_{0,b+1}=0$, then $I(\overline{v}_{0,j_0}^\sharp)\geq 2$ by the usual rescaling argument (cf.\ Cases (3)--(6) of Theorem~I.\ref{P1-thm: complement}).  In this case the right-hand side of Equation~\eqref{nesbo} can be increased by one and $\sum_{\overline{v}_*}I(\overline{v}_*)\geq 3$, a contradiction.
\end{proof}

\begin{lemma}\label{jeeves}
If $m\gg 0$, $\overline{u}_\infty\in \bdry_{\{+\infty\}}\mathcal{M}$, $\overline{v}_-'\not=\varnothing$, and $\overline{v}_+^\sharp$ has a negative end asymptotic to a multiple of $\delta_0$, then $\overline{u}_\infty\in A_2'$.
\end{lemma}

Here $A_2'$ is as given in Step 4 of the proof of Theorem~\ref{thm: chain homotopy part i}.

\begin{proof}
This is similar to Lemma~\ref{pooh}. The contributions of ($\beta_1$), ($\widetilde\beta_2'$), ($\beta_3$), ($\beta_4$), and Equation \eqref{A1} with $j_0=b+1$ add up to $q_++2p_+$.

If $p_+>0$, then $\sum_{\overline{v}_*}I(\overline{v}_*)\geq q_+ + 2p_+ \geq 3$, a contradiction.

If $p_+=0$, then $q_++2p_+\geq 1$. By the usual rescaling argument, the negative end of $\overline{v}_+^\sharp$ that limits to $\delta_0$ additionally contributes $+1$ to $I$. Hence if $\sum_{\overline{v}_*}I(\overline{v}_*)=2$, then $q_+=1$.  This implies that $\overline{u}_\infty\in A_2'$.
\end{proof}

\begin{lemma} \label{piglet}
If $m\gg 0$, $\overline{u}_\infty\in \bdry_{\{+\infty\}}\mathcal{M}$, and $\overline{v}'_-\not=\varnothing$, then $\overline{u}_\infty\in A_2'$.
\end{lemma}

\begin{proof}
Consider $\overline{u}_\infty\in \bdry_{\{+\infty\}}\mathcal{M}$ such that $\overline{v}'_-\not=\varnothing$. By Lemma~\ref{eliminate fiber components 2}, there is no boundary point at $z_\infty$, no fiber component and no $\overline{v}_*''$ such that $\overline{v}'_*\cap\overline{v}''_* \not=\varnothing$. By Lemmas~\ref{pooh} and \ref{jeeves}, if $\overline{v}_{0,j_0}^\sharp$, $1\leq j_0\leq b+1$, has a negative end asymptotic to a multiple of $\delta_0$, then $\overline{u}_\infty\in A_2'$.

It remains to consider three possibilities, all of which satisfy $p_{0,0}\leq \dots\leq p_{0,b+1}$:
\begin{enumerate}
\item Some $\overline{v}^\sharp_{-1,j_0}$, $j_0=1,\dots,c+1$, has a positive or negative end $\mathcal{E}$ at $z_\infty$ such that $n^*(\mathcal{E})> m/2$.
\item Some $\overline{v}^\sharp_{1,j_0}$, $j_0=0,\dots,a$, has a positive or negative end $\mathcal{E}$ at $z_\infty$ such that $n^*(\mathcal{E})> m/2$.
\item Some $\overline{v}^\sharp_{0,j_0}$, $j_0=0,\dots,b$, has a positive end $\mathcal{E}$ at a multiple of $\delta_0$ such that $n^*(\mathcal{E})> m/2$.
\end{enumerate}
It is clear that (1), (2), and (3) are mutually exclusive.

\s
(1) is similar to (2E) of Lemma~\ref{eliminate fiber components 2} and (2) is similar to Case (6$_i$) of Lemma~\ref{vancouver}.

\s
(3) In this case, (a)--(c) in the proof of Lemma~\ref{eliminate fiber components 2} hold, with the exception of the level $\overline{v}_{0,j_0}$.
Observe that $\overline{v}_{0,j_0}'\not=\varnothing$ and $\overline{v}_{0,j_0}^\sharp$ consists of:
\begin{itemize}
\item cylinders from $\delta_0$ to $e$ or $h$ which contribute $I=1$ or $2$ each; and
\item a curve $\overline{v}_{0,j_0}^{\sharp_1}$ with one positive end $\mathcal{E}^\sharp_1$ which is asymptotic to $\delta_0^{c_1}$ and satisfies $n^*(\mathcal{E}^\sharp_1)=m+c_1$ and other positive ends which in total are asymptotic to $\delta_0^{c_2}$ and satisfy $n^*=c_2$.
\end{itemize}

We apply the rescaling argument with $m\gg 0$ fixed to obtain $w_{0,j_0+1}: \Sigma_{0,j_0+1}\to \C\P^1$ and $\pi_{0,j_0+1}:\Sigma_{0,j_0+1}\to cl(B_{0,j_0+1})$ with $B_{0,j_0+1}\simeq B'$ and $\Sigma_{0,j_0+1}$ connected, such that the following hold:
\begin{enumerate}
\item[(i)] $w_{0,j_0+1}(z_0)=0$ for some $z_0\in \pi^{-1}_{0,j_0+1}(-\infty)$;
\item[(ii)] $w_{0,j_0+1}(z_1)=\infty$ for some $z_1\not=z_0\in \pi^{-1}_{0,j_0+1}(-\infty)$;
\item[(iii)] $w_{0,j_0+1}$ is a biholomorphism onto its image; and
\item[(iv)] in particular, $w_{0,j_0+1}(z)\not= 0,\infty$ for $z\in \pi^{-1}_{0,j_0+1}(\pm\infty)-\{z_0,z_1\}$.
\end{enumerate}

We now compute the contributions to the ECH index from the ends of $\overline{v}'_{0,j_0} \cup \overline{v}_{0,j_0}^{\sharp}$ that limit to multiples of $\delta_0$ at the positive end. Here is the list of such ends:
\begin{itemize}
\item the union $\mathcal{E}'$ of positive ends of $\overline{v}'_{0,j_0}$;
\item $\mathcal{E}^\sharp_1$ satisfying $n^*(\mathcal{E}^\sharp_1)=m+c_1$;
\item the union $\mathcal{E}^\sharp_2$ of positive ends of $\overline{v}^\sharp_{0,j_0}$ that correspond to punctures of $w_{0,j_0+1}$;
\item the union $\mathcal{E}^\sharp_3$ of all other positive ends.
\end{itemize}
We use the formula:
\begin{equation}
\op{ind}(\overline{u}) +(\widetilde{\mu}_\tau(\overline{u})-\mu_\tau(\overline{u})-w_\tau(\overline{u}))\leq I(\overline{u}),
\end{equation}
where the notation is as in Section~I.\ref{P1-subsection: compactness PFH case} and the equation follows from the adjunction inequality and \cite[Lemma~4.20]{Hu2}.
The end $\mathcal{E}^\sharp_1$ has contributions $\widetilde{\mu}_\tau=c_1$, $\mu_\tau=1$, and $w_\tau=1-c_1$. Moreover, since the vanishing of the leading asymptotic eigenfunction corresponding to the end $\mathcal{E}^\sharp_1$ is a codimension two condition which contributes $2$ to the Fredholm index, the extra contribution to $I$ from $\mathcal{E}^\sharp_1$ is $c_1-1-(1-c_1)+2=2c_1$. The contributions to $I$ from $\mathcal{E}'$ and $\mathcal{E}^\sharp_2$, arising from a writhe computation, is $2(\deg \mathcal{E}'+\deg \mathcal{E}^\sharp_2)=2(p_{0,j_0} +\deg\mathcal{E}^\sharp_2)$.  Finally, the contributions to $I$ from $\mathcal{E}^\sharp_3$ is $\deg\mathcal{E}^\sharp_3$ by the argument of Lemma~\ref{lemma 2012}.

The total ECH index of all the levels is $\geq  2c_1+ 2p_{0,j_0} \geq4$, a contradiction.
\end{proof}

\subsubsection{The case $\overline{v}'_-=\varnothing$}

\begin{lemma}\label{eeyore}
If $m\gg 0$, then there is no $\overline{u}_\infty\in \bdry_{\{+\infty\}}\mathcal{M}$ such that $\overline{v}'_-=\varnothing$.
\end{lemma}

\begin{proof}
Arguing by contradiction, suppose there exists $\overline{u}_\infty\in \bdry_{\{+\infty\}}\mathcal{M}$ such that $\overline{v}'_-=\varnothing$. By Lemma~\ref{eliminate fiber components 2}, there is no boundary point at $z_\infty$, no fiber component, and no $\overline{v}_*''$ such that $\overline{v}'_*\cap\overline{v}''_* \not=\varnothing$.

The point constraint gives $n^*(\overline{v}''_-)\geq m$ and $I(\overline{v}''_-)\geq 2$. This immediately implies that (a)--(d) in the proof of Lemma~\ref{eliminate fiber components 2} hold. In particular, ($\beta_2$) from Lemma~\ref{eliminate fiber components 2} holds.

If $p_{0,b+1}>0$, then the proof of Lemma~\ref{lemma 2012} gives us a contradiction. Hence $p_{0,b+1}=0$. If a positive end of $\overline{v}''_+$ limits to $z_\infty$, then a component $\widetilde{v}$ of $\overline{v}''_+$ must be a section of $\overline{W}_+$ from $z_\infty$ to $h$ or $e$. However, by the choice of $h$ from Section~\ref{subsubsection: convention bambi} and the proof of Lemma~I.\ref{P1-lemma: value of widetilde Phi}, sections of $\overline{W}_+$ from $z_\infty$ to $h$ cannot exist, leaving us with the possibility that $\widetilde{v}$ is a section from $z_\infty$ to $e$.  Hence $I(\widetilde{v})\geq 1$ and the total ECH index is $\geq 3$, a contradiction. On the other hand, if $\overline{v}''_+=\overline{v}_+^\flat$, then $I(\overline{v}_{1,j})\geq 1$ for some $j>0$, which is again a contradiction.
\end{proof}

\begin{proof}[Proof of Lemma~\ref{kyoho plus infty}]
Lemma~\ref{kyoho plus infty} follows from Lemmas~\ref{piglet} and \ref{eeyore}.
\end{proof}

\subsection{Degeneration at $-\infty$, part II}
\label{subsection: additional degenerations II}

Let $\overline{u}_\infty\in \bdry_{\{-\infty\}}\mathcal{M}$ be the limit of $\overline{u}_i\in \mathcal{M}_{\tau_i}$, where $\tau_i\to -\infty$.

\begin{lemma} \label{nonnegative ECH indices 6}
If $\overline{u}_\infty$ has no fiber components and $\overline{v}'_1=\varnothing$, then the ECH index of each level is nonnegative.
\end{lemma}

\begin{proof}
The proof is similar to that of Lemma~\ref{nonnegative ECH indices 3} and uses the considerations of Lemma~\ref{nonnegative ECH indices 5 better}.
\end{proof}

\subsubsection{The case $\overline{v}_2'=\varnothing$}

\begin{lemma} \label{owl5}
If $\overline{u}_\infty\in \bdry_{\{-\infty\}}\mathcal{M}$ and $\overline{v}'_2=\varnothing$, then:
\begin{enumerate}
\item $\overline{u}_\infty$ is a $2$-level building consisting of $\overline{v}_1$ with $I(\overline{v}_1)=0$, $\overline{v}_2$ with $I(\overline{v}_2)=2$, and $\overline{v}'_j=\varnothing$ for $j=1,2$;
\item the left and right ends of $\overline{v}_1$ ($=$ left and right ends of $\overline{v}_2$) do not limit to $z_\infty$;
\item $\overline{u}_\infty$ has no fiber component;
\item $\overline{u}_\infty$ has no boundary point at $z_\infty$.
\end{enumerate}
\end{lemma}

\begin{proof}
Suppose that $\overline{v}'_2=\varnothing$. First observe that
\begin{equation} \label{bi luo chun}
n^*(\overline{v}_2(N({\frak q})))\geq m,
\end{equation}
where ${\frak q}\in \dot F_2$ is the point such that $\overline{v}_2({\frak q})=\overline{\frak m}(-\infty)$.

We also have $I(\overline{v}_2)\geq 2$: If $\overline{u}_\infty$ has a fiber component $\widetilde{v}$, then $\widetilde{v}$ must pass through $\overline{\frak m}(-\infty)$ and $I(\overline{v}_2)\geq 2g+2\geq 4$ by an argument similar to that of Lemma~I.\ref{P1-claim in proof}. On the other hand, if $\overline{u}_\infty$ does not have a fiber component, then  $I(\overline{v}_2)\geq 2$, since passing through $\overline{\frak m}(-\infty)$ is a generic codimension two condition.

(1) We claim that $\overline{v}'_1=\varnothing$.  Arguing by contradiction, if $\overline{v}'_1\not=\varnothing$, then at least one of $\overline{v}_{L,j}^\sharp$, $j=0,\dots,a$, has a right end $\mathcal{E}_R$ that limits to $z_\infty$. On the other hand, the end $\mathcal{E}_R$ satisfies
\begin{equation}\label{dragonwell}
n^*(\mathcal{E}_R)\geq k_0-1\gg 2g,
\end{equation}
which contradicts Equation~\eqref{bi luo chun}. This proves the claim.

The claim and Lemma~\ref{nonnegative ECH indices 6} imply that each level of $\overline{u}_\infty$ has nonnegative ECH index. Hence $I(\overline{v}_2)=2$, $I(\overline{v}_1)=0$, and $a=b=c=d=0$ since a level $\overline{v}_{L,j}$, $j=1,\dots,a$, is not a connector if and only if $I(\overline{v}_{L,j})>0$ (and the same holds for $\overline{v}_{R,j}$, $j=1,\dots,b$, $\overline{v}_{B,j}$, $j=1,\dots,c$, and $\overline{v}_{T,j}$, $j=1,\dots,d$).

(2) If $\overline{v}_2$ has a right end $\mathcal{E}_R$ that limits to $z_\infty$, then $\mathcal{E}_R$ satisfies Equation~\eqref{dragonwell}, and we have a contradiction of Equation~\eqref{bi luo chun}. If $\overline{v}_2$ has a left end that limits to $z_\infty$, then $\overline{v}_1$ has a right end $\mathcal{E}_R$ that limits to $z_\infty$ and satisfies Equation~\eqref{dragonwell} (since $\overline{v}_1'=\varnothing$), which is again a contradiction.

(3) Since $I(\overline{v}_2)=2$, we cannot have a fiber component.

(4) A boundary point at $z_\infty$ contradicts Equation~\eqref{bi luo chun}.
\end{proof}

\subsubsection{The case $\overline{v}'_1=\varnothing$, $\overline{v}_2'\not=\varnothing$}

\begin{lemma} \label{alishan3}
If $m\gg 0$, then there is no $\overline{u}_\infty\in \bdry_{\{-\infty\}}\mathcal{M}$ such that $\overline{v}'_1=\varnothing$ and $\overline{v}_2'\not=\varnothing$.
\end{lemma}

\begin{proof}
Arguing by contradiction, suppose there exists $\overline{u}_\infty\in \bdry_{\{-\infty\}}\mathcal{M}$ such that $\overline{v}'_1=\varnothing$ and $\overline{v}_2'\not=\varnothing$. The analog of Lemma~\ref{intersezione three} holds, i.e.,
$$\sum_{i=1}^q n^*(\mathcal{E}_i)+\sum_{i=1}^r n^*(\mathcal{E}'_i)=m,\quad \sum_{i=1}^q n^{*,alt}(\mathcal{E}_i)+\sum_{i=1}^r n^{*,alt}(\mathcal{E}'_i)\geq m-2g,$$
where $\mathcal{E}_i$, $i=1,\dots,q$, are the ends that limit to $z_\infty$ and $\mathcal{E}'_i$, $i=1,\dots,r$, are the neighborhoods of boundary points of type (P$_3$) of all the $\overline{v}^\sharp_{L,j}$, $j=0,\dots,a$, and $\overline{v}^\sharp_{R,j}$, $j=0,\dots,b$. The conclusion of Lemma~\ref{sencha 3} then holds by an almost identical argument.  The proof of Lemma~\ref{alishan} then carries over without modification.
\end{proof}

\subsubsection{The case $\overline{v}'_1\not=\varnothing$, $\overline{v}_2'\not=\varnothing$}

\begin{lemma}\label{alternative}
Suppose $\overline{v}'_1\not=\varnothing$ and $\overline{v}_2'\not=\varnothing$.  Let $\mathcal{E}_i$, $i=1,\dots,q$, be the ends of all the $\overline{v}_*^\sharp$ that limit to $z_\infty$ and let $\mathcal{E}'_i$, $i=1,\dots,r$, be the neighborhoods of the boundary points of type (P$_3$).  Then one of the following holds:
\begin{enumerate}
\item[(a)] $\sum_{i=1}^q n^*(\mathcal{E}_i)+\sum_{i=1}^r n^*(\mathcal{E}'_i)\geq m-2g$.
\item[(b)] Each end $\mathcal{E}_i$ is an end of $\overline{v}^\sharp_{B,j}$, $j=1,\dots,c$, or $\overline{v}_{T,j}^\sharp$, $j=0,\dots,d$, that projects to a thin sector of type ${\frak S}(\overline{a}_{i',j'},\overline{\hh}(\overline{a}_{i',j'}))$ or ${\frak S}(\overline{b}_{i',j'},\overline{\hh}(\overline{b}_{i',j'}))$. In particular, $\mathcal{E}_i$ is not a left or right end of $\overline{v}^\sharp_*$, where $*=(L,j)$, $j=0,\dots,a$, or $(R,j)$, $j=0,\dots,b$.
\end{enumerate}

\end{lemma}

\begin{proof}
We apply the continuation argument. If some $\mathcal{E}_i$, $i=1,\dots,q$, does not project to a thin sector (of type ${\frak S}(\overline{a}_{i',j'},\overline{\hh}(\overline{a}_{i',j'}))$, ${\frak S}(\overline{b}_{i',j'},\overline{\hh}(\overline{b}_{i',j'}))$, or  ${\frak S}(\overline{b}_{i',j'},\overline{a}_{i',j'})$), then the sectors $\pi_{D^2_{\rho_0}}(\mathcal{E}_i)$, $i=1,\dots,q$, and $\pi_{D^2_{\rho_0}}(\mathcal{E}'_i)$, $i=1,\dots,r$, will sweep out a neighborhood of $z_\infty$ with the exception of some thin sectors, implying (a).

On the other hand, suppose that all the ends $\mathcal{E}_i$, $i=1,\dots,r$, project to thin sectors. We claim that there are no thin sectors of type ${\frak S}(\overline{b}_{i',j'},\overline{a}_{i',j'})$. Indeed, the number of left and right ends of all the $\overline{v}_*^\sharp$ that limit to $z_\infty$ must be equal and the right ends cannot map to thin sectors. This gives (b).
\end{proof}

\begin{lemma}\label{eliminate fiber components 3}
If $\overline{u}_\infty\in \bdry_{\{-\infty\}}\mathcal{M}$, $\overline{v}'_1\not=\varnothing$, and $\overline{v}'_2\not=\varnothing$, then:
\begin{enumerate}
\item $\overline{u}_\infty$ has no fiber components;
\item if $\overline{v}''_*$ satisfies $\overline{v}'_*\cap \overline{v}''_*\not=\varnothing$ and $\overline{v}'_*\cap \overline{v}''_*\subset int(\overline{v}'_*)$, then $*=1$ and $p_*=1$;
\item $\overline{u}_\infty$ has no boundary point of type (P$_1$) or (P$_2$).
\end{enumerate}
\end{lemma}

\begin{proof}
This is similar to the proof of Lemma~\ref{eliminate fiber components 2} and uses Equations~\eqref{sum of n part 2} and \eqref{sum of I part 2}.

First observe that if $\overline{u}_\infty$ has a fiber component, a boundary point of type (P$_1$) or (P$_2$), or some $\overline{v}'_*\cap\overline{v}''_*\not=\varnothing$ with $\overline{v}'_*\cap \overline{v}''_*\subset int(\overline{v}'_*)$, then we are in the situation of Lemma~\ref{alternative}(b). It is not hard to verify that $I(\overline{v}_*)\geq 0$ for all $*$ with the exception of $\overline{v}_1$.

(1) Suppose $\overline{u}_\infty$ has a fiber component $\widetilde{v}$.  Then we have the following contributions towards $I$:
\begin{enumerate}
\item[($\beta_1$)] $I(\overline{v}_1')=-p_1=-\deg (\overline{v}'_1)$.
\item[($\beta_2$)] $\sum_{j=1}^c I(\overline{v}^\sharp_{B,j})= p_1$.
\item[($\beta_3$)] If the fiber $\widetilde{v}$ is a component of $\overline{v}_*$, then $\widetilde{v}$ and the intersection $\widetilde{v}\cap (\overline{v}_*-\widetilde{v})$ contribute $2g+2\geq 4$ towards $I$.
\end{enumerate}
The argument is similar to that of Lemma~\ref{eliminate fiber components 2}(1).  This gives a total of $I>2$, which is a contradiction.

(2) Suppose $\overline{v}'_{*_0}\cap \overline{v}''_{*_0}\not=\varnothing$ and $\overline{v}'_{*_0}\cap \overline{v}''_{*_0}\subset int(\overline{v}'_{*_0})$. If $*_0\not=1$, then $I(\overline{v}''_{*_0})\geq 1$ and the intersection points contribute at least $2p_{*_0}\geq 2$. If $*_0=1$ and $p_{*_0}>1$, then $I(\overline{v}''_{*_0})\geq 0$ and the intersection points contribute at least $2p_{*_0}\geq 4$.  Combined with the ECH contributions from ($\beta_1$) and ($\beta_2$), we have a contradiction.

(3) If $\overline{u}_\infty$ has a boundary point of type (P$_1$), then argument from (3B) of Lemma~\ref{eliminate fiber components 2} implies an excess of $n^*$.  If $\overline{u}_\infty$ has a boundary point of type (P$_2$), then the analog of Lemma~\ref{nonnegative ECH indices 5 better} implies that the boundary points at $z_\infty$ contribute $2p_{*_0}\geq 2$; there is also a large sector which contributes an additional $+1$.   Combined with ($\beta_1$) and ($\beta_2$) we obtain a total of $I>2$, a contradiction.
\end{proof}

\begin{lemma} \label{b not possible} \label{eliminate fiber components 3 part 2}
If $m\gg 0$, $\overline{u}_\infty\in \bdry_{\{-\infty\}}\mathcal{M}$, $\overline{v}'_1\not=\varnothing$, and $\overline{v}'_2\not=\varnothing$, then there is no level $\overline{v}_*$ such that $\overline{v}'_*\cap \overline{v}''_*\not=\varnothing$ and $\overline{v}'_*\cap \overline{v}''_*\subset int(\overline{v}'_*)$.
\end{lemma}

\begin{proof}
Arguing by contradiction, suppose that $\overline{v}'_{*_0}\cap \overline{v}''_{*_0}\not=\varnothing$ and $\overline{v}'_{*_0}\cap \overline{v}''_{*_0}\subset int(\overline{v}'_{*_0})$ for some $\overline{v}_{*_0}$. Then $*_0=1$ and $p_1=1$ by Lemma~\ref{eliminate fiber components 3}(2). By Lemma~\ref{alternative}(b), we have $p_{L,j}=1$ and $p_{R,j}=1$ for all $j$. Hence we may assume that $a=b=0$. Equation~\eqref{sum of I part 2} then implies that $c=1$ and $d=0$.

By the rescaling argument with $m\to \infty$, we obtain a $2$-level building $w_1\cup w_2$, where
$$ w_2: cl(B_{-\infty,2})\to \C\P^1, \quad  w_1: cl(B_{-\infty,1})\to \C\P^1,$$
and (i)--(viii), given below, hold. Here $cl(B_{-\infty,2})$ is obtained from $B_{-\infty,2}$ by adding the left and right points at infinity, denoted $t=\pm\infty$; similarly $cl(B_{-\infty,1})$ is obtained from $B_{-\infty,1}$ by adding $s=\pm\infty$ and $t=\pm\infty$.  The map $w_2$ satisfies the following:
\begin{enumerate}
\item[(i)] $w_2(\overline{\frak m}^b(-\infty))=0$;
\item[(ii)] $w_2(\bdry cl(B_{-\infty,2}))\subset \{\phi=0, \rho>0\}\cup\{\infty\}$;
\item[(iii)] $w_2(t=+\infty)=\infty$ and $w_2(t=-\infty)\subset \{\phi=0,\rho>0\}$ (or vice versa);
\item[(iv)] $w_2|_{int(B_{-\infty,2})}$ is a biholomorphism onto its image.
\end{enumerate}
The map $w_1$ satisfies the following:
\begin{enumerate}
\item[(v)] $w_1(t=-\infty)=0$ (or $w_1(t=+\infty)=0$; for simplicity assume the former);
\item[(vi)] $w_1(\bdry cl(B_{-\infty,1}))\subset \{\phi=0, \rho>0\}\cup\{\infty\}$;
\item[(vii)] $w_1({\frak r})=\infty$, where ${\frak r}\in \overline{v}'_1\cap \overline{v}''_1$;
\item[(viii)] $w_1|_{int(B_{-\infty,1})}$ is a biholomorphism onto its image.
\end{enumerate}

In view of the above, we consider the ``tropical curves''
$$\overline\Xi_i:\Gamma\to [-1,1]/(-1\sim 1)\times[0,d_i]$$ as in Section~\ref{second case}. The proof strategy of Lemmas~\ref{alishan} and \ref{alishan2} carry over to give us a contradiction; this is due to the disparity in the growth rates of the left and right ends of $\overline{u}_i$ when restricted to $[-2,2]\times [1+L,r(\tau_i)-L]$.
\end{proof}

The combination of Lemmas~\ref{eliminate fiber components 3} and \ref{eliminate fiber components 3 part 2} give the following, which is the analog of Lemma~\ref{sencha 3}(2),(3):

\begin{cor}\label{cor 1}
If $m\gg 0$, $\overline{u}_\infty\in \bdry_{\{-\infty\}}\mathcal{M}$, $\overline{v}'_1\not=\varnothing$, and $\overline{v}'_2\not=\varnothing$, then there is no fiber component, no boundary point of type (P$_1$) or (P$_2$), and no intersection point ${\frak r}\in \overline{v}'_*\cap \overline{v}''_*$ such that ${\frak r}\in\overline{v}'_*\cap \overline{v}''_*\subset int(\overline{v}'_*)$.
\end{cor}

Before embarking on the proof of Lemma~\ref{alishan4}, we prove Lemma~\ref{every comp trivial strip}, which is the analog of Lemma~\ref{sencha 3}(4) and is a bit involved.

First we recall and modify some notation from Section~I.\ref{P1-subsection: modified indices at z infty}. We use the convention that $\star=(L,j)$, $j=1,\dots,a$, or $(R,j)$, $j=1,\dots,b+1$. The data $\ar{\mathcal{D}}_{\star,\pm}$ at $z_\infty$ for $\overline{v}_\star$ is given by a $p=p_{\star,\pm}$-tuple of matchings
\begin{equation} \label{matchings}
\{(i_1',j_1')\to (i_1,j_1), \dots,(i_p',j_p')\to (i_p,j_p)\},
\end{equation}
where $i_k,i_k'\in \{1,\dots,2g\}$, $j_k,j_k'\in\{0,1\}$ for $k=1,\dots,p$ and
$i_k\not=i_l$, $i_k'\not=i_l'$ for $k\not=l$. Here the subscript $+$ (resp.\ $-$) in $\ar{\mathcal{D}}_{\star,\pm}$ refers to the left (resp.\ right) end, $(i_k',j_k')$ corresponds to $\overline{b}_{i_k',j_k'}$, and $(i_k,j_k)$ corresponds to $\overline{a}_{i_k,j_k}$. We write $$\ar{\mathcal{D}}_{\star,\pm}=\ar{\mathcal{D}}'_{\star,\pm}\cup \ar{\mathcal{D}}''_{\star,\pm},$$ where $\ar{\mathcal{D}}'_{\star,\pm}$ and $\ar{\mathcal{D}}''_{\star,\pm}$ correspond to the ends of $\overline{v}'_\star$, $\overline{v}''_\star$, respectively. 

\begin{lemma} \label{every comp trivial strip}
If $\overline{u}_\infty\in \bdry_{\{-\infty\}}\mathcal{M}$, $\overline{v}'_1\not=\varnothing$ and $\overline{v}'_2\not=\varnothing$, then every component of $\overline{v}_\star^\sharp$, $\star=(L,j)$, $j=1,\dots,a$, or $(R,j)$, $j=1,\dots,b+1$,  is a thin strip.
\end{lemma}

\begin{proof}
Arguing by contradiction, suppose there exists $\overline{v}_{\star_0}^\sharp$ which is not a union of thin strips; we take $\star_0=(R,1)$ without loss of generality. Then the following hold:
\begin{enumerate}
\item[(a)] $n^*(\overline{v}_{\star_0}^\sharp)= m$ and $n^{*,alt}(\overline{v}_{\star_0}^\sharp)\geq m-2g$;
\item[(b)] each component $\widetilde{v}$ of $\overline{v}_{T,j}^\sharp$, $j=1,\dots,d$, or $\overline{v}_{B,j}^\sharp$, $j=1,\dots,c$, is a thin strip with $z_\infty$ at the positive end and $I(\widetilde{v})=1$;
\item[(c)] each component $\widetilde{v}$ of $\overline{v}_\star^\sharp$, $\star\not=\star_0$, is a thin strip with $z_\infty$ at the left end and $I(\widetilde{v})=1$;
\item[(d)] each component $\widetilde{v}$ of $\overline{v}_1^\sharp$ limits to $z_\infty$ only at the left end and the top end, the left and top ends project to thin sectors, $\deg(\widetilde{v})=1$, and $I(\widetilde{v})=1$.
\end{enumerate}
(a) follows from the argument of Lemma~\ref{sencha 3}(4)(a) and (b) and (c) follow immediately from (a). (d) The assertions about the ends of $\widetilde{v}$ follow from (a). This then implies that $\widetilde{v}$ projects to the domain bounded by $\overline{b}_i$, $\overline{\hh}(\overline{b}_i)$, and $\overline{\hh}(\overline{a}_i)$ in $\overline{S}$, in view of the positions of the arcs $\overline{b}_{i,j}$, $\overline{a}_{i,j}$, $\overline{\hh}(\overline{b}_{i,j})$, and $\overline{\hh}(\overline{a}_{i,j})$ from Figure~\ref{figure: aandb}.

We now analyze the ends of $\overline{v}_{\star_0}^\sharp$ at $z_\infty$ in more detail. Let $l_+$ and $l_-$ be the number of left and right ends of $\overline{v}_{\star_0}^\sharp$ at $z_\infty$ and let $\mathfrak{bp}$ be the number of boundary points of type (P$_3$) on $\overline{v}_{\star_0}$. In view of (c) and (d), the total number of left ends of $\overline{v}_*$ that limit to $z_\infty$ besides those of $\overline{v}_{\star_0}^\sharp$ is $l_--l_+$.  By (a slight modification of) the continuation method, we obtain a cycle $\mathcal{Z}$ consisting of $2l_-+\mathfrak{bp}$ chords, where $l_-$ of the chords are of type ${\frak S}(\overline{a}_{k,l},\overline{b}_{k',l'})$, $l_-$ of the chords are of type ${\frak S}(\overline{b}_{k,l},\overline{a}_{k',l'})$, $\mathfrak{bp}$ of the chords are of type ${\frak S}(\overline{a}_{k,l},\overline{a}_{k',l'})$ or ${\frak S}(\overline{b}_{k,l},\overline{b}_{k',l'})$. 

Let us write $p_*=\deg \overline{v}'_*$, $q_*=\deg \overline{v}^\sharp_*$, and $r_*=\deg\overline{v}^\flat_*$. We have $\op{ind}(\overline{v}_{\star_0}^\sharp)\geq 2$ since there is an end that limits to $z_\infty$ or a boundary point of type (P$_3$) that projects to a large sector.  Then $I(\overline{v}_{\star_0}^\sharp)\geq 2$ by the ECH index inequality (cf.\ Lemma~I.\ref{P1-index inequality for z infinity case}) and considerations of $\mathcal{Z}$.

\s
(1) Suppose that $\overline{v}'_{\star_0}\not=\varnothing$.    We claim the following:

\begin{claim} \label{claim: summing ECH indices}
$I(\overline{v}_{\star_0})\geq I(\overline{v}'_{\star_0})+I(\overline{v}''_{\star_0})+\delta_{\star_0},$ where $\delta_{\star_0}\geq 2$ if $\mathfrak{bp}>0$ and $\delta_{\star_0}=2p_{\star_0}$ if $\mathfrak{bp}=0$.
\end{claim}

\begin{proof}[Proof of Claim~\ref{claim: summing ECH indices}]
Each collection of boundary points of type (P$_3$) that map to the same point on the base contributes at least $+2$ towards $I$ by the argument from Lemma~\ref{nonnegative ECH indices 5 better}. The inequality for $\mathfrak{bp}>0$ then follows.

Next assume that $\mathfrak{bp}=0$.  Also assume that $l_+=l_-$, since otherwise there exist thin strips of $\overline{v}_*$ with $*\not= \star_0$, which contribute $>0$ to $I$, a contradiction. Let $(s,t_2)$ be coordinates on $[-2,2]\times[-1,1]$. Let
$$\check{C}'_{[-1,1]},\check{C}''_{[-1,1]}\subset [-2,2]\times[-1,1]\times \overline{S}$$ be representatives of $\overline{v}'_{\star_0}$ and $\overline{v}''_{\star_0}$ and let ${\frak c}_{\pm 1}'$, ${\frak c}_{\pm 1}''$ be groomings on $A_{\varepsilon/2}=\bdry D^2_{\varepsilon/2}\times[-2,2]$ and $A_\varepsilon=\bdry D^2_\varepsilon\times[-2,2]$ corresponding to  $\check{C}'_{[-1,1]}|_{t_2=\pm 1}$ and $\check{C}''_{[-1,1]}|_{t_2=\pm 1}$, such that the following hold:
\begin{itemize}
\item ${\frak c}_{+1}'={\frak c}_{-1}'$ and $w({\frak c}'_{\pm 1})=0$;
\item ${\frak c}''_{\pm 1}$ is obtained by intersecting $\pi_{[-2,2]\times\overline{S}}$-projections of the ends of $\overline{v}''_{\star_0}$ with $A_\varepsilon$.
\end{itemize}
Here $\pi_{[-2,2]\times\overline{S}}$ is the projection $[-2,2]\times\R\times\overline{S}\to [-2,2]\times\overline{S}$. We also remark that, for sign consistency with the computations in Section~I.\ref{P1-subsection: modified indices at z infty}, we use the standard orientations on $A_{\varepsilon/2}$ and $A_\varepsilon$ but view ${\frak c}'_{+1}$ and ${\frak c}''_{+1}$ to be at the {\em negative end} and ${\frak c}'_{-1}$ and ${\frak c}''_{-1}$ to be at the {\em positive end} during the proof of this claim. By the above description of $\mathcal{Z}$:
\begin{itemize}
\item $w({\frak c}''_{+1})=0$ or $-1$ and $w({\frak c}''_{-1})=0$ or $1$;
\item the endpoints of ${\frak c}''_{+1}$ and ${\frak c}''_{-1}$ agree and are alternating.
\end{itemize}

We now extend $\check C'_{[-1,1]},\check C''_{[-1,1]}$ by concatenating with
\begin{align*}
\check C'_{[1,2]},\check C''_{[1,2]}&\subset [-2,2]\times[1,2]\times\overline{S},\\
\check C'_{[-2,-1]},\check C''_{[-2,-1]}&\subset [-2,2]\times[-2,-1]\times\overline{S}
\end{align*}
such that the ends of $\check C'_{[-2,2]}\cup \check C''_{[-2,2]}$ are groomed and have zero winding number. Writing $\mathfrak{w}$ for the writhe, Lemmas~I.\ref{P1-calc of almost sum} and I.\ref{P1-calc of almost sum 2} imply that we have an additional contribution of at least
$$2p_{\star_0}=2(\mathfrak{w}({\frak c}'_{+1}\cup{\frak c}''_{+1})-\mathfrak{w}({\frak c}'_{-1}\cup{\frak c}''_{-1}))$$
towards $I$.
\end{proof}

By Claim~\ref{claim: summing ECH indices},  $I(\overline{v}_{\star_0})\geq 4$.  Since the ECH indices of the other levels add up to at least $0$ in view of (b), (c), and (d), we have a contradiction.

\s
(2) Suppose that $\overline{v}'_{\star_0}=\varnothing$. Then there are no boundary points of type (P$_3$) by definition. Since $I(\overline{v}_{\star_0})=I(\overline{v}''_{\star_0})\geq 2$, the only nontrivial levels besides $\overline{v}_{\star_0}$ and $\overline{v}_1$ are $\overline{v}_{B,j}$, $j=1,\dots,c$, which consist of thin strips and trivial strips. Let us write $l_{\star_0}:=l_+=l_-$.
We claim the following:

\begin{claim} \label{calc of I}
$I(\overline{v}_{\star_0}'')\geq 2+ l_{\star_0}$.
\end{claim}

\begin{proof}[Proof of Claim~\ref{calc of I}]
Let us first consider the case where:
\begin{enumerate}
\item $\pi_{\overline{S}}\circ \overline{v}_{\star_0}^\sharp|_{\dot F^\sharp_{\star_0}}$ is a diffeomorphism onto its image, which contains $\overline{S}-\overline{\bf a}-\overline{\bf b}$;
\item $\overline{v}_{\star_0}^\sharp$ limits to $z_\infty^{l_{\star_0}}\cup {\bf y}'$ to the left and to $z_\infty^{l_{\star_0}}\cup {\bf y}''$ to the right, where ${\bf y}'={\bf y}''$.
\end{enumerate}
If we choose a multivalued trivialization $\tau$ to be compatible with $\ar{\mathcal{D}}''_{\star_0,\pm}$ (cf.\ Section~I.\ref{P1-subsubsubsection: multivalued trivialization}), then
\begin{align*}
\op{ind}(\overline{v}_{\star_0}^\sharp) & = -\chi(F^\sharp_{\star_0}) + \deg \overline{v}^\sharp_{\star_0} +\mu_\tau(\overline{v}_{\star_0}^\sharp) +2c_1((\overline{v}_{\star_0}^\sharp)^*T\overline{S},\tau)\\
& = -(1-2g-(\deg \overline{v}^\sharp_{\star_0}-l_{\star_0})) + \deg \overline{v}^\sharp_{\star_0} +(l_{\star_0}+1) +2(1-2g)\\
& =  2-2g+2\deg \overline{v}^\sharp_{\star_0}.
\end{align*}
Hence $I(\overline{v}_{\star_0}^\sharp)\geq 2-2g+2\deg \overline{v}^\sharp_{\star_0}$. We also have $I(\overline{v}_{\star_0}^\flat)\geq 0$ and $\langle\overline{v}_{\star_0}^\sharp,\overline{v}_{\star_0}^\flat\rangle=\deg \overline{v}^\flat_{\star_0}$, where $\langle,\rangle$ is the algebraic intersection number.
Summing the contributions, we obtain:
\begin{align} \label{ari is sick today}
I(\overline{v}_{\star_0})&=I(\overline{v}_{\star_0}^\sharp)+I(\overline{v}_{\star_0}^\flat)+2\langle\overline{v}_{\star_0}^\sharp,\overline{v}_{\star_0}^\flat\rangle\\
\nonumber &\geq 2-2g+2\deg \overline{v}^\sharp_{\star_0}+2\deg \overline{v}^\flat_{\star_0} \geq 2 +\deg \overline{v}^\sharp_{\star_0}.
\end{align}

In general, the cases with smallest ECH indices occur when elements of ${\bf y}'$ are of type $x_{i1}^\#$ or $x_{i3}^\#$ and elements of ${\bf y}''$ are of type $x_{i2}^\#$. This has the effect of decreasing the lower bound in Equation~\eqref{ari is sick today} by the cardinality of ${\bf y}'$. Hence $I(\overline{v}_{\star_0})\geq 2+l_{\star_0}$.
\end{proof}

Since $l_{\star_0}\geq 1$, we have $I(\overline{v}_{\star_0})\geq 3$, which is a contradiction.
\end{proof}

\begin{lemma} \label{alishan4}
If $m\gg 0$, then there is no $\overline{u}_\infty\in \bdry_{\{-\infty\}}\mathcal{M}$ such that $\overline{v}'_1\not=\varnothing$ and $\overline{v}_2'\not=\varnothing$.
\end{lemma}

\begin{proof}
Corollary~\ref{cor 1} and Lemma~\ref{every comp trivial strip} imply the analog of Lemma~\ref{sencha 3}. The proof of Lemma~\ref{alishan2} then carries over with no change.
\end{proof}

\begin{proof}[Proof of Lemma~\ref{kyoho minus infty}]
Suppose $\overline{u}_\infty\in \bdry_{\{-\infty\}}\mathcal{M}$. If $\overline{v}'_2=\varnothing$, then we are in the situation of Lemma~\ref{kyoho minus infty} by Lemma~\ref{owl5}. On the other hand, Lemmas~\ref{alishan3} and \ref{alishan4} imply that $\overline{v}'_2=\varnothing$.
\end{proof}

\subsection{Breaking in the middle, part II}
\label{subsection: additional degenerations III}

Let $\overline{u}_\infty\in \bdry_{(-\infty,+\infty)}\mathcal{M}$ be the limit of $\overline{u}_i\in \mathcal{M}_{\tau_i}$, where $\tau_i\to T'$.

\begin{lemma} \label{nonnegative ECH indices 4}
If  $\overline{u}_\infty$ has no fiber components, then the ECH index of each level $\overline{v}_*\not= \overline{v}_0$ is nonnegative and the only components of $\overline{u}_\infty$ which have negative ECH index are the following:
\begin{itemize}
\item branched covers of $\sigma_\infty^{T'}$; and
\item at most one component $\widetilde{v}$ of $\overline{v}_0''$ with $I(\widetilde{v})=-1$.
\end{itemize}
\end{lemma}

\begin{proof}
Similar to the proofs of Lemmas~\ref{nonnegative ECH indices 5} and ~\ref{nonnegative ECH indices 5 better}.
\end{proof}

The following lemma is analogous to Lemma~\ref{eliminate fiber components 2}.

\begin{lemma}  \label{eliminate fiber components}
If $m\gg 0$ and $\overline{u}_\infty\in \bdry_{(-\infty,+\infty)}\mathcal{M}$, then the following hold:
\begin{enumerate}
\item $\overline{u}_\infty$ has no fiber components;
\item there is no level $\overline{v}_*$ such that $\overline{v}'_*\cap \overline{v}''_*\not=\varnothing$ and $\overline{v}'_*\cap \overline{v}''_*\subset int(\overline{v}'_*)$; and
\item $\overline{u}_\infty$ has no boundary point at $z_\infty$.
\end{enumerate}
\end{lemma}

\begin{proof}
(1) Arguing by contradiction, suppose that $\overline{u}_\infty$ has a fiber component $\widetilde{v}$. Then $n^*(\widetilde{v})\geq m$ and $\widetilde{v}$ can be eliminated by arguing as in Lemma~\ref{eliminate fiber components 2}(1), in view of the following contributions towards $I$:
\begin{enumerate}
\item[($\gamma_1$)] $I(\overline{v}_0')=-p_0$.
\item[($\gamma_2$)] $\cup_{j=1}^c\overline{v}_{-1,j}^\sharp$ is a union of $p_0$ trivial strips and $\sum_{j=1}^cI(\overline{v}_{-1,j}^\sharp)=p_0$.
\item[($\gamma_3$)] If $\widetilde{v}$ is a component of $\overline{v}_*$, then $\widetilde{v}$ and the intersection $\widetilde{v}\cap (\overline{v}_*-\widetilde v)$ contribute $2g+2\geq 4$ towards $I$.
\item[($\gamma_4$)] After removing any fiber components $\widetilde{v}$, all the levels $\not=\overline{v}_0$ have nonnegative ECH index and $I(\overline{v}_0'')\geq -1$ by Lemma~\ref{nonnegative ECH indices 4}.
\end{enumerate}
The sum of the above ECH indices is at least $3$, a contradiction.

(2), (3) Similar to (2) and (3) of Lemma~\ref{eliminate fiber components 2}. Observe that the sum of ECH indices of the levels was at least $4$ in all the cases of Lemma~\ref{eliminate fiber components 2} that were not ruled out by other means; in the present case the sum will be at least $3$. (Note that the ECH indices added up to only $3$ in (2C) of Lemma~\ref{eliminate fiber components 2}, but we do not have levels $\overline{v}_{0,j}$, $j=1,\dots,b$, in the current situation.)
\end{proof}

\begin{lemma} \label{owl4}
If $\overline{u}_\infty\in \bdry_{(-\infty,+\infty)}\mathcal{M}$ and $\overline{v}'_0=\varnothing$, then $\overline{u}_\infty$ is one of the following $2$-level buildings:
\begin{enumerate}
\item $\overline{v}_{1,1}$ with $I=1$ from ${\bf z}$ to some ${\bf z}'$ consisting of
\begin{enumerate}
\item[(i)] one thin strip and trivial strips or
\item[(ii)] one nontrivial component of $\overline{v}_{1,1}''$ with image in $W$ and trivial strips;
\end{enumerate}
and $\overline{v}_0=\overline{v}''_0$ with $I=1$ and $n^*\leq m+|\mathcal{I}|$ from ${\bf z}'$ to ${\bf y}'$.
\item $\overline{v}_0=\overline{v}''_0$ with $I=1$ and $n^*\leq m+|\mathcal{I}|$ from ${\bf z}$ to some ${\bf y}''$; and $\overline{v}_{-1,1}$ with $I=1$ from ${\bf y}''$ to ${\bf y}'$.
\end{enumerate}
\end{lemma}

\begin{proof}
Suppose that $\overline{v}'_0=\varnothing$. Since passing through $\overline{\frak m}(T')$ is a codimension two condition, we have $I(\overline{v}_0)\geq 1$. Hence there can be only one other nontrivial level --- either $\overline{v}_{1,1}$ or $\overline{v}_{-1,1}$ --- and $I(\overline{v}_{\pm 1,1})\geq 1$. By Lemma~\ref{nonnegative ECH indices 4}, $I(\overline{v}_0)=1$ and $I(\overline{v}_{\pm 1,1})=1$. Moreover, $n^*(\overline{v}_{\pm 1,1})\leq  |\mathcal{I}|$ since $n^*(\overline{v}_0)\geq m$. This means that there are no components $\overline{v}_*^\sharp$ that limit to $z_\infty$ at the negative end and all the positive ends $\mathcal{E}_+$ that limit to $z_\infty$ satisfy $n^*(\mathcal{E}_+)=1$. The lemma follows.
\end{proof}

\begin{lemma} \label{owl3}
For each interval $[-T,T]$, there exists $m\gg 0$ such that there is no sequence of curves $\overline{u}_i\in \mathcal{M}_{\tau_i}$, $\tau_i\to T'\in[-T,T]$, that limits to $\overline{u}_\infty$ for which $\overline{v}_0'\not =\varnothing$.
\end{lemma}

\begin{proof}
This is similar to Lemma~\ref{vancouver2} and will be omitted.
\end{proof}

\begin{proof}[Proof of Lemma~\ref{kyoho middle}]
This is argued in the same way as Lemma~\ref{cherries3}. Suppose $\overline{u}_\infty\in \bdry_{(-\infty,+\infty)}\mathcal{M}$. If $\overline{v}'_0=\varnothing$, then Lemma~\ref{owl4} implies that $\overline{u}_\infty$ is as described in Lemma~\ref{kyoho middle}. On the other hand, by Lemma~\ref{owl3}, for any $T>0$ there exists $m\gg 0$ such that $\overline{v}_0'=\varnothing$ in $\overline{u}_\infty\in \bdry_{[-T,+T]}\mathcal{M}$. The possibilities where $m_i\to \infty$, $T_i\to\infty$, $\overline{u}_{ij}\to \overline{u}_{i\infty}\in \bdry_{\{\pm T_i\}} \mathcal{M}$ are treated in the same way as in Lemmas~\ref{eeyore}, \ref{alishan3}, and \ref{alishan4}.
\end{proof}

\section{Homotopy of cobordisms II}
\label{section: homotopy of cobordisms II}

In this section we consider the homotopy of cobordisms corresponding to $\Phi\circ\Psi$. Many of the constructions for $\overline{W}^\mp_\tau$ are similar to those of $\overline{W}^\pm_\tau$ with minor modifications.

We first give a brief description of $\overline{W}_\tau=\overline{W}^\mp_\tau$. (If $\mp$ is understood, at it will be in the rest of this section, then it will be omitted.) The base $B_\tau$ of $\overline{W}_\tau$ is biholomorphic to an infinite cylinder with a disk removed.  As $\tau\to +\infty$, $\overline{W}_\tau$ degenerates to the stacking of $\overline{W}_-$ ``on top of''
$\overline{W}_+$.  On the other hand, as $\tau\to -\infty$, $\overline{W}_\tau$ degenerates to $\overline{W}_{-\infty}$, whose base $B_{-\infty}$ is given by:
\begin{equation} \label{eqn: base B sub E}
B_{-\infty}= \left((\R\times (\R/2\Z))\sqcup E\right)/\sim,
\end{equation}
where $E=\{|z|\leq 1\}\subset \C$ and $\sim$ identifies $(0,{3\over 2})\in \R\times(\R/2\Z)$ with $0\in E$.

\subsection{Construction of the homotopy of cobordisms for $\Phi\circ \Psi$}

\subsubsection{Definition of the family $\overline{W}_\tau$}

Let $l\in(0,\infty)$ and $r\in(0,1]$.  Consider the fibration
$$\pi: \R\times \overline{N} \to \R\times(\R/2\Z),$$
where $\overline{N}$ is viewed as $(\overline{S}\times[0,2])/(x,2)\sim (\overline{\hh}(x),0)$ and $(s,t)$ are coordinates on $\R\times (\R/2\Z)$. We define $\overline{W}_{l,r}=\pi^{-1}(B_{l,r})$, where the base $B_{l,r}$ is obtained by smoothing the corners of
$$(\R\times (\R/2\Z))- ((-l/2,l/2)\times ((3-r)/2,(3+r)/2)).$$

Next choose a function
\begin{equation} \label{l and r part 2}
\eta = (l,r): \R\to(0,\infty)\times(0,1],
\end{equation}
which is obtained by smoothing
$$\eta_0(\tau)= \left\{
\begin{array}{ll}
(\tau+1,1), & \mbox{for } \tau\geq 0;\\
(e^\tau,e^\tau), & \mbox{for } \tau\leq 0;
\end{array}\right.$$
near $\tau=0$.

We then define $\overline{W}_\tau= \overline{W}_{\eta(\tau)}$ and $B_\tau= B_{\eta(\tau)}$. Let $\pi_{B_\tau}: \overline{W}_\tau\to B_\tau$ be the projection along $\{(s,t)\}\times\overline{S}$.

As $\tau\to +\infty$, the cobordism $\overline{W}_\tau$ approaches the concatenation of $\overline{W}_-$ and $\overline{W}_+$. See Figure~\ref{figure: base degeneration 2bis}. On the other hand, as $\tau\to -\infty$, the cobordism $\overline{W}_\tau$ can be viewed as degenerating to $\overline{W}_{-\infty}=(\overline{W}_{-\infty,1}\sqcup \overline{W}_{-\infty,2})/\sim$, which we describe now:
Consider the base $B_{-\infty}=(B_{-\infty,1}\sqcup B_{-\infty,2})/\sim$, where $B_{-\infty,1}= \R\times(\R/2\Z)$, $B_{-\infty,2}=E=\{|z|\leq 1\}$, and $\sim$ identifies $(0,{3\over 2})\in B_{-\infty,1}$ with $0\in B_{-\infty,2}$. We also identify the asymptotic markers $\bdry_s$ at $(0,{3\over 2})$ and $\{x=0,y>0\}$ at $0$. We then set $\overline{W}_{-\infty,1}= \pi^{-1}(B_{-\infty,1})$ and $\overline{W}_{-\infty,2}=B_{-\infty,2}\times \overline{S}$, and identify $(0,{3\over 2},x)\sim (0,x)$, where $(0,{3\over 2},x)\in \overline{W}_{-\infty,1}$ and $(0,x)\in \overline{W}_{-\infty,2}$.  We write $\pi_{B_{-\infty,i}}: \overline{W}_{-\infty,i}\to B_{-\infty,i}$ for the projection along $\overline{S}$.

\begin{figure}[ht]
\begin{overpic}[width=12cm]{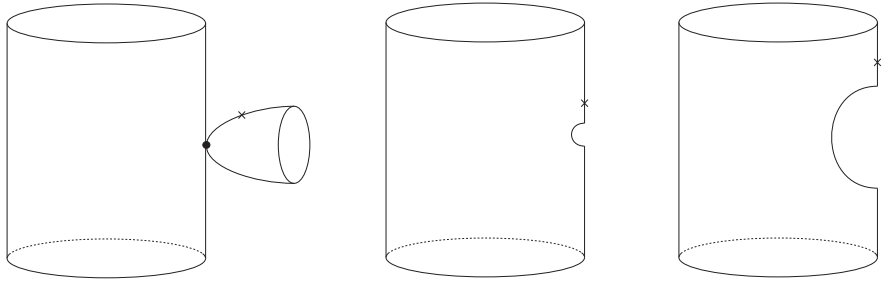}
\end{overpic}
\caption{The bases of the family $\overline{W}^\mp_\tau$. The leftmost diagram represents $B_{-\infty}$, i.e., $\tau=-\infty$, and the parameter $\tau$ increases as we go to the right. The location of $\overline{\frak m}^b(\tau)$ is indicated by $\times$. }
\label{figure: base degeneration 2bis}
\end{figure}

\subsubsection{Marked points}

We choose a $1$-parameter family of marked points
$$\overline{\frak m}(\tau)=(\overline{\frak m}^b(\tau), \overline{\frak m}^f(\tau))=(\overline{\frak m}^b(\tau),z_\infty)\in \overline{W}_\tau$$ for $\tau\in\R$, such that the following hold:
\begin{enumerate}
\item[(i)] $\overline{\frak m}^b(\tau)$ is on the segment $\{s>{l(\tau)\over 2}, t={3\over 2}\}$;
\item[(ii)] as $\tau\to +\infty$, $\overline{\frak m}(\tau)$ limits to $\overline{\frak m}(+\infty)=(\overline{\frak m}^b(+\infty),\overline{\frak m}^f(+\infty))$, where $\overline{\frak m}^b(+\infty)=(0,{3\over 2})\in B_-$ and $\overline{\frak m}^f(+\infty)=z_\infty$;
\item[(iii)] as $\tau\to -\infty$, $\overline{\frak m}(\tau)$ limits to $\overline{\frak m}(-\infty)=(\overline{\frak m}^b(-\infty),\overline{\frak m}^f(-\infty))$, where $\overline{\frak m}^b(-\infty)={i\over 2}\in B_{-\infty,2}$ and $\overline{\frak m}^f(-\infty)=z_\infty$.
\end{enumerate}

\subsubsection{Stable Hamiltonian structures and symplectic forms}

We first consider $\overline{W}_\tau$. The stable Hamiltonian structure on $\overline{N}=(\overline{S}\times[0,2])/\sim$ is obtained from $(dt,\overline\omega)$ on $\overline{S}\times[0,2]$ by passing to the quotient. The $2$-plane field is $\xi_\tau=T\overline{S}$ and the Hamiltonian vector field is $R_\tau=\bdry_t$. The symplectic form $\overline{\Omega}_\tau$ is the restriction to $\overline{W}_\tau$ of $ds\wedge dt+ \overline\omega$, defined on
$\R\times\overline{S}\times[0,2]$ and sent to the quotient $\R\times\overline{N}$.

Next we consider $\overline{W}_{-\infty}$. For $\overline{W}_{-\infty,1}=\overline{N}$ the Hamiltonian structure is $(dt,\overline\omega)$ and the symplectic form is $\overline{\Omega}_{-\infty,1}=ds\wedge dt+\overline\omega$.  The symplectic form on $\overline{W}_{-\infty,2}=B_{-\infty,2}\times\overline{S}$ is a product symplectic form $\overline{\Omega}_{-\infty,2}=\omega_{-\infty,2}\oplus \overline\omega$, where $\omega_{-\infty,2}$ is the standard area form on the unit disk.

\subsection{Holomorphic curves and moduli spaces}

\subsubsection{Lagrangian boundary conditions}

Consider the cobordism $(\overline{W}_\tau,\overline{\Omega}_\tau)$. We place a copy of $\overline{\mathbf{a}}$ on $\pi^{-1}_{B_\tau}({l(\tau)\over 2},{3\over 2})$ and parallel transport along the vertical boundary $\bdry_v\overline{W}_\tau:= \pi^{-1}(\bdry B_\tau)$ using the symplectic connection $\overline{\Omega}_\tau$ to obtain $L_{\overline{\mathbf{a}}}^\tau$.  Similarly, we place a copy of $\overline{\mathbf{a}}$ over $1\in B_{-\infty,2}$ and parallel transport along $\bdry B_{-\infty,2}$ to obtain $L_{\overline{\mathbf{a}}}^{-\infty,2}=\bdry B_{-\infty,2}\times \overline{\mathbf{a}}$. The Lagrangian submanifolds $L^\tau_{\widehat{\bf a}}$, $L^\tau_{\widehat{a}_i}$, $L^\tau_{\overline{a}_i}$ are defined similarly.

\subsubsection{Almost complex structures}

\begin{defn} \label{defn: ac str for mp} $\mbox{}$
\begin{enumerate}
\item An almost complex structure $\overline{J}_\tau$ on $\overline{W}_\tau$ is {\em admissible} if $\overline{J}_\tau$ is the restriction of some $\overline{J'}\in \mathcal{J}_{\overline{W'}}$. In this case $\overline{J}_\tau$ is said to be {\em compatible} with $\overline{J'}$.
\item An almost complex structure $\overline{J}_{-\infty,1}$ on $\R\times \overline{N}$ is {\em admissible} if $\overline{J}_{-\infty,1}\in \mathcal{J}_{\overline{W'}}$.
\item An almost complex structure $\overline{J}_{-\infty,2}$ on $B_{-\infty,2}\times \overline{S}$ is {\em admissible} if it is a product complex structure.
\end{enumerate}
The space of admissible $\overline{J}_\tau$, $\overline{J}_{-\infty,1}$ and $\overline{J}_{-\infty,2}$ will be denoted by $\mathcal{J}_{\overline{W}_\tau}$, $\mathcal{J}_{\overline{W}_{-\infty,1}}$ and  $\mathcal{J}_{\overline{W}_{-\infty,2}}$, respectively.
\end{defn}

\begin{defn}
A family $\{\overline{J}_\tau\in \mathcal{\overline{J}}_\tau\}_{\tau\in\R}$ of almost complex structures is {\em admissible} if there exist $\overline{J'}\in \mathcal{J}_{\overline{W'}}$, $\overline{J}\in \mathcal{J}_{\overline{W}}$, $\overline{J}_+\in \mathcal{J}_{\overline{W}_+}$, $\overline{J}_-\in \mathcal{J}_{\overline{W}_-}$, $\overline{J}_{-\infty,1}\in\mathcal{J}_{\overline{W}_{-\infty,1}}$ and $\overline{J}_{-\infty,2}\in \mathcal{J}_{\overline{W}_{-\infty,2}}$ such that the following hold:
\begin{enumerate}
\item $\overline{J}_\tau$ converges to $(\overline{J}_{-\infty,1}, \overline{J}_{-\infty,2})$ as $\tau\to -\infty$;
\item $\overline{J}_\tau$ converges to $(\overline{J}_+,\overline{J}_-)$ as $\tau\to +\infty$;
\item $\overline{J}_+$ and $\overline{J}_-$ are compatible with $\overline{J'}$ and $\overline{J}$; and
\item $\overline{J}_\tau$ is compatible with $\overline{J'}=\overline{J}_{-\infty,1}$.
\end{enumerate}
The space of all admissible $\{\overline{J}_\tau\in \mathcal{\overline{J}}_\tau\}_{\tau\in\R}$ will be denoted by $\overline{\mathcal{I}}$.

\end{defn}

\subsubsection{Notation and conventions}

We will be using the notation and conventions from Section~\ref{subsubsection: convention bambi}, with the exception of intersection numbers $n^*(\overline{u})$.

\s\n
{\em Intersection numbers.} Let $\delta_{\rho_0,\phi_0}$ be a closed orbit of the Hamiltonian vector field $\bdry_t$ which lies on the torus $\{\rho=\rho_0\}\subset \overline{S}\times[0,2]/\sim$ for $\rho_0>0$ sufficiently small and which passes through the point $(t,\phi)=(0,\phi_0)$. 
We assume additionally that $\delta_{\rho_0,\phi_0}$ does not intersect the projections of the Lagrangians of $\overline{W}_\tau$, $\overline{W}_+$ and $\overline{W}_-$ to $\overline{N}$.

We then write $(\sigma_\infty^*)^\dagger$ for the restriction of $\R\times\delta_{\rho_0,\phi_0}$ to $\overline{W}_*$, where $*=\varnothing, ', \tau,+$, $-$ or $(-\infty,1)$. For $\overline{W}_{-\infty,2}$ we write
$$(\sigma_\infty^{-\infty,2})^\dagger=B_{-\infty,2}\times \{ \rho=\rho_0, \phi=\phi_0+3\pi/ 2m+ 2\pi k/m, k\in \Z\}.$$
Finally we define $n^*(\overline{u})= \langle \overline{u},(\sigma_\infty^*)^\dagger\rangle$,
where $*=\varnothing, ', \tau,+$, $-$ or $(-\infty,1)$.

\subsubsection{Holomorphic maps to $\overline{W}_\tau$}

Let $(F,j)$ be a compact Riemann surface, possibly disconnected, with $2g$ boundary components and two sets of interior punctures $\mathbf{p}^+= \{p_1^+,\dots,p_{k_+}^+\}$ and $\mathbf{p}^-= \{p_1^-,\dots,p_{k_-}^-\}$ such that each component of $F$ has nonempty boundary, at least one puncture from $\mathbf{p}^+$, and at least one puncture from $\mathbf{p}^-$. We write $\dot F= F-\mathbf{p}^+-\mathbf{p}^-$.

\begin{defn} \label{defn: W tau curves in mp}
Let $\overline{J}_\tau\in \mathcal{J}_{\overline{W}_\tau}$ and let $\bs\gamma=\prod \gamma_i^{m_i},\bs\gamma'=\prod (\gamma'_i)^{m_i'}\in \widehat{\mathcal{O}}_k$.

A {\em degree $k$ multisection of $(\overline{W}_\tau,\overline{J}_\tau)$ from $\bs\gamma$ to $\bs\gamma'$} is a holomorphic map
$$\overline{u}: (\dot F,j)\to (\overline{W}_\tau, \overline{J}_\tau)$$
which is a degree $k$ multisection of $\pi_{B_\tau}:\overline{W}_\tau\to B_\tau$ and which additionally satisfies the following:
\begin{enumerate}
\item $\overline{u}(\bdry F )\subset L^\tau_{\widehat{\bf a}}$;
\item $\overline{u}$ maps each connected component of $\bdry F$ to a different $L_{\widehat{a}_i}^\tau$;
\item $\displaystyle\lim_{w\to p_i^+} \pi_{\R}\circ \overline{u}(w) =+\infty$ and $\displaystyle\lim_{w\to p_i^-} \pi_{\R}\circ \overline{u}(w)=-\infty$;
\item $\overline{u}$ converges to a cylinder over a multiple of some $\gamma_j$ near each puncture $p_i^+$ so that the total multiplicity of $\gamma_j$ over all the $p_i^+$'s is $m_j$ (and similarly for $p_i^-$).
\end{enumerate}
Here $\pi_{\R}: \overline{W}_\tau\to \R$ is the projection to the $s$-coordinate.

A {\em $(\overline{W}_\tau,\overline{J}_\tau)$-curve from $\bs\gamma$ to $\bs\gamma'$} is a degree $2g$ multisection of $(\overline{W}_\tau,\overline{J}_\tau)$ satisfying $n^*(\overline{u})=m$.
\end{defn}

Let $\mathcal{M}_{\overline{J}_\tau}(\bs\gamma,\bs\gamma')$ be the moduli space of degree $k$ multisections of $(\overline{W}_\tau,\overline{J}_\tau)$ from $\bs\gamma$ to $\bs\gamma'$ and let $\mathcal{M}_{\overline{J}_\tau}(\bs\gamma,\bs\gamma';\overline{\frak m}(\tau))\subset \mathcal{M}_{\overline{J}_\tau}(\bs\gamma,\bs\gamma')$ be the subset of curves that pass through $\overline{\frak m}(\tau)$. We write
$$\mathcal{M}_{\{\overline{J}_\tau\}}(\bs\gamma,\bs\gamma')=\coprod_{\tau\in\R} \mathcal{M}_{\overline{J}_\tau}(\bs\gamma,\bs\gamma'),$$
$$\mathcal{M}_{\{\overline{J}_\tau\}}(\bs\gamma,\bs\gamma';\overline{\frak m})=\coprod_{\tau\in\R}\mathcal{M}_{\overline{J}_\tau}(\bs\gamma,\bs\gamma';\overline{\frak m}(\tau)).$$

\subsubsection{Holomorphic maps to $\overline{W}_{-\infty}$}

We now describe the curves in $\overline{W}_{-\infty,i}$.

\begin{defn}\label{defn: curves in overline W -infty 1}
Let $\overline{J}_{-\infty,1}\in\mathcal{J}_{\overline{W}_{-\infty,1}}$ and let $\bs\gamma=\prod \gamma_i^{m_i},\bs\gamma'=\prod (\gamma'_i)^{m_i'}\in \widehat{\mathcal{O}}_k$.  A {\em degree $k$ multisection of $(\overline{W}_{-\infty,1},\overline{J}_{-\infty,1})$ from $\bs\gamma$ to $\bs\gamma'$} is a holomorphic map
$$\overline{u}:(\dot F,j)\to (\overline{W}_{-\infty,1},\overline{J}_{-\infty,1})$$
which is a degree $k$ multisection of $\pi_{B_{-\infty,1}}:\overline{W}_{-\infty,1}\to B_{-\infty,1}$ and which is asymptotic to $\bs\gamma$ and $\bs\gamma'$ at the positive and negative ends. Here $(\dot F,j)$ is a punctured Riemann surface. A {\em $(\overline{W}_{-\infty,1},\overline{J}_{-\infty,1})$-curve from $\bs\gamma$ to $\bs\gamma'$} is a degree $2g$ multisection of $(\overline{W}_{-\infty,1},\overline{J}_{-\infty,1})$ satisfying $n^*(\overline{u})=0$.
\end{defn}

\begin{defn} \label{defn: curves in E times overline S}
Let $\overline{J}_{-\infty,2}\in\mathcal{J}_{\overline{W}_{-\infty,2}}$. A {\em degree $k$ multisection of $(\overline{W}_{-\infty,2},\overline{J}_{-\infty,2})$} is holomorphic map
$$\overline{u}: (F,j)\to (\overline{W}_{-\infty,2},\overline{J}_{-\infty,2})$$
which is a degree $k$ multisection of $\pi_{B_{-\infty,2}}:\overline{W}_{-\infty,2}\to B_{-\infty,2}$ and which additionally satisfies the following:
\begin{enumerate}
\item $(F,j)$ is a compact Riemann surface with $k$ boundary components;
\item $\overline{u}$ maps each component of $\bdry F$ to a different $L^{-\infty,2}_{\widehat{a}_i}$.
\end{enumerate}
A {\em $(\overline{W}_{-\infty,2},\overline{J}_{-\infty,2})$-curve} is a degree $2g$ multisection of $(\overline{W}_{-\infty,2},\overline{J}_{-\infty,2})$ satisfying $n^*(\overline{u})=m$. A {\em degenerate $(\overline{W}_{-\infty,2},\overline{J}_{-\infty,2})$-curve} consists of $2g$ copies of $B_{-\infty,2}\times\{pt\}$ and a singular fiber $\{0\}\times \overline{S}$.
\end{defn}

Let $\mathcal{M}_{\overline{J}_{-\infty,1}}(\bs\gamma,\bs\gamma')$ be the moduli space of degree $k$ multisections of $(\overline{W}_{-\infty,1}$, $\overline{J}_{-\infty,1})$ from $\bs\gamma$ to $\bs\gamma'$ and let $\mathcal{M}_{\overline{J}_{-\infty,2}}$ be the moduli space of degree $k$ multisections of $(\overline{W}_{-\infty,2},\overline{J}_{-\infty,2})$.

Let ${\frak z}$ be a formal product $\zeta_0^{r_0} \zeta_1^{r_1}\dots \zeta_l^{r_l}$, where $r_0,\dots,r_l\in\Z^{\geq 0}$, $r_0+\dots+r_l=k$,  $\zeta_0=z_\infty$, and $\zeta_1,\dots,\zeta_l\in \overline{S}-\{z_\infty\}$.  We write ${\frak Z}$ for the set of all formal products ${\frak z}$ on $\overline{S}$.
A degree $k$ multisection $\overline{u}\in \mathcal{M}_{\overline{J}_{-\infty,1}}(\bs\gamma,\bs\gamma')$ (resp.\ $\mathcal{M}_{\overline{J}_{-\infty,2}}$) {\em passes through ${\frak z}$} if it passes through $((0,{3\over 2}),\zeta_i)$ (resp.\ $(0,\zeta_i)$), $i=0,\dots,l$, with multiplicities $r_i$. Here the multiplicity $r_i$ is the total degree of the projection $\pi_{B_{-\infty,i}}\circ \overline{u}$, restricted to neighborhoods of points in the domain $\dot F$ (resp.\ $F$) which map to $((0,{3\over 2}),\zeta_i)$ (resp.\ $(0,\zeta_i)$) under $\overline{u}$.

Let
$$\mathcal{M}_{\overline{J}_{-\infty,1}}(\bs\gamma,\bs\gamma',{\frak z})\subset \mathcal{M}_{\overline{J}_{-\infty,1}}(\bs\gamma,\bs\gamma'),\quad\mathcal{M}_{\overline{J}_{-\infty,2}}({\frak z})\subset \mathcal{M}_{\overline{J}_{-\infty,2}}$$ be subsets of curves that pass through ${\frak z}$.
Also let $$\mathcal{M}_{\overline{J}_{-\infty,2}}({\frak z},\overline{\frak m}(-\infty))\subset\mathcal{M}_{\overline{J}_{-\infty,2}}({\frak z})$$ be the subset of curves that pass through $\overline{\frak m}(-\infty)$.

\subsubsection{Indices}
\label{subsubsection: indices part 2}

We now briefly discuss the Fredholm index $\op{ind}(\overline{u})$ and the ECH index $I(\overline{u})$ of a $\overline{W}_\tau$-curve $\overline{u}: \dot F \to \overline{W}_\tau$ from $\bs\gamma$ to $\bs\gamma'$. This is similar to Section~\ref{subsubsection: indices part 1}.

Let $\check{\overline{W}}_\tau=\overline{W}_\tau-\{s>{l(\tau)\over 2}+1\}-\{s< -{l(\tau)\over 2}-1\}$, where $l(\tau)$ is given in Equation~\eqref{l and r part 2}. Let $\check{\overline{u}}: \check F\to \check{\overline{W}}_\tau$ be the compactification of $\overline{u}$, where $\check F$ is obtained by performing a real blow-up of $F$ at its boundary punctures.

The trivialization $\tau^\star$ of $T\overline{S}$ along $L^\tau_{\hat{\bf a}}$ is defined as in Section~I.\ref{P1-subsubsection: Fredholm index second version}: We pick a point $p\in \bdry B_\tau$, define $\tau^\star$ of $T\overline{S}|_{\pi^{-1}_{B_\tau}(p)}$ along $\widehat{\bf a}$ by choosing a nonsingular tangent vector field along $\widehat{\bf a}$, and then parallel transport $\tau^\star$ along $\bdry \overline{W}_\tau$. We also choose $\tau^\star$ along $\bs\gamma$ and $\bs\gamma'$.

Let $Q_{\tau^\star}(\check{\overline u})$ be the relative intersection form given by normalizing $\check{\overline u}$ near $s=\pm({l(\tau)\over 2}+1)$ as in \cite[Definition 2.4]{Hu1} and intersecting $\check{\overline u}$ and a pushoff of $\check{\overline{u}}$.  Here the pushoff along $\bdry \check{\overline{W}}_\tau$ is in the direction of $J\tau^\star$. Then
\begin{equation}
I({\overline u})=c_1(\check{\overline u}^*T\check{\overline{W}}_\tau,(\tau^\star,\bdry_t))
+Q_{\tau^\star}(\check{\overline u})+ \widetilde\mu_{\tau^\star}(\bs\gamma)-\widetilde\mu_{\tau^\star}(\bs\gamma')-2g,
\end{equation}
\begin{equation}
\op{ind}(\overline{u})=-\chi(\dot F)+ 2c_1(\check{\overline u}^*T\check{\overline{W}}_\tau,(\tau^\star,\bdry_t))+
\mu_{\tau^\star}(\bs\gamma)-\mu_{\tau^\star}(\bs\gamma')-2g,
\end{equation}
and the index inequality holds as usual:
\begin{equation}
\op{ind}(\overline{u})+2\delta(\overline{u})\leq I(\overline{u}),
\end{equation}
where $\delta(\overline{u})\geq 0$ and equals zero if and only if $\overline{u}$ is an embedding. We also have
\begin{equation} \label{index for sigma infty tau part 2}
I(\sigma_\infty^\tau)=\op{ind}(\sigma_\infty^\tau)=-1.
\end{equation}

The Fredholm and ECH indices for $\overline{W}_{-\infty,i}$-curves can be defined and computed similarly.

\begin{rmk}
The Fredholm and ECH indices for $\overline{W}_\tau$ and $\overline{W}_{-\infty,i}$ do not take into account the point constraint $\overline{\frak m}(\tau)$ and the condition ``passing through $\overline{\frak m}(\tau)$'' is a codimension $2$ condition.  Moreover, the indices for $\overline{W}_{-\infty,i}$, $i=1,2$, do not take into account the constraints ${\frak z}$.
\end{rmk}

We now have the following:

\begin{lemma} \label{lemma: curves in overline W minus infty 2}
Let $\overline{u}$ be a $\overline{W}_{-\infty,2}$-curve. Then
$$[\overline{u}]=2g[B_{-\infty,2}\times\{pt\}]+[\{pt\}\times \overline{S}]\in H_2(\overline{W}_{-\infty,2},\bdry B_{-\infty,2}\times\overline{S}),$$
$\overline{u}$ consists of $0\leq k<2g$ copies of $B_{-\infty,2}\times\{pt\}$ together with an irreducible component in the class $(2g-k)[B_{-\infty,2}\times\{pt\}]+[\{pt\}\times\overline{S}]$, and $I(\overline{u})=4g+2$.
\end{lemma}

\begin{proof}
Let $\pi_{\overline{S}}: \overline{W}_{-\infty,2}\to \overline{S}$ be the projection along $B_{-\infty,2}$.  Since $\overline{J}_{-\infty,2}$ is a split complex structure, $\pi_{\overline{S}}$ is a holomorphic map.  Hence if $\overline{u}$ is a $\overline{W}_{-\infty,2}$-curve, then $\pi_{\overline{S}}\circ \overline{u}$ either maps to a point on some $\widehat{a}_i$ or to all of $\overline{S}$.  The lemma follows by listing all the possibilities.
\end{proof}

We also have the following, which is stated without proof.

\begin{lemma} \label{lemma: curves in overline W minus infty 3}
Let $\overline{u}$ be a degree $2g$ multisection of $(\overline{W}_{-\infty,2},\overline{J}_{-\infty,2})$ satisfying $n^*(\overline{u})=0$. Then $\overline{u}=\overline{u}'\cup\overline{u}''$ consists of $\overline{u}'$ which is a degree $k$ cover of $\sigma_\infty^{-\infty,2}$ and $\overline{u}''$ which is the union of $2g-k$ copies of $B_{-\infty,2}\times\{pt\}$, and $I(\overline{u})=2g$.
\end{lemma}

\subsubsection{Regularity}

\begin{defn} \label{defn: regularity for W tau part 2}
The family $\{\overline{J}_\tau\}_{\tau\in\R}\in\overline{\mathcal{I}}$ is {\em regular} if:
\begin{enumerate}
\item $\mathcal{M}^{\dagger}_{\{\overline{J}_\tau\}}(\delta_0^p\bs\gamma,\delta_0^q\bs\gamma')$ is transversely cut out for all $\delta_0^p\bs\gamma,\delta_0^q\bs\gamma'\in \widehat{\mathcal{O}}_k$, $k\leq 2g$;
\item the restriction $\overline{J'}$ of $\overline{J}_\tau$ to the positive and negative ends (the restriction is independent of $\tau$) is regular; and
\item $\overline{J}_-$ and $\overline{J}_+$ in the limit $\tau\to+\infty$ are regular.
\end{enumerate}
\end{defn}

Let $\overline{\mathcal{I}}^{reg}$ be the space of regular $\{\overline{J}_\tau\}\in \overline{\mathcal{I}}$.  As usual we have:

\begin{lemma}
The generic $\{\overline{J}_\tau\}\in \overline{\mathcal{I}}$ is regular.
\end{lemma}

We also introduce the perturbations of $\{\overline{J}_\tau\}$ to ensure that passing through $\overline{\frak m}(\tau)$ is a generic condition. Let $\frak p(\tau)\in B_\tau$ be a family of points such that:
$$ \lim_{\tau\to +\infty} \frak p (\tau) =\frak p (+\infty) \in B_-, \quad \lim_{\tau\to -\infty} \frak p(\tau)= \frak p(-\infty)\in B_{-\infty,1}$$
and ${\frak p}(\tau)\not= \overline{\frak m}^b(\tau)$ for all $\tau\in [-\infty,\infty]$. We then define the families $\{U_\tau=U_{\varepsilon,\delta,{\frak p}(\tau)}\}$, $\{K_\tau=K_{{\frak p}(\tau),\delta}\}$, and $\{\overline{J}_\tau^\Diamond=\overline{J}_\tau^\Diamond(\varepsilon,\delta,{\frak p}(\tau))\}$ as in Section~\ref{subsubsection: regularity tomodachi}. Also, the modifier $\{K_\tau\}$ means ``passing through $K_\tau$ for an appropriate $\tau$.''

\begin{defn}
The family $\{\overline{J}_\tau^\Diamond\}$ is {\em $\{K_\tau\}$-regular with respect to $\overline{\frak m}$} if the moduli spaces $\mathcal{M}^{\dagger,\{K_\tau\}}_{\{\overline{J}_\tau\}}(\delta_0^p\bs\gamma,\delta_0^q\bs\gamma';\overline{\frak m})$ are transversely cut out.
\end{defn}

The following lemmas are analogous to Lemmas~\ref{lemma: regularity of W tau family} and \ref{lemma: codimension one}:

\begin{lemma}
\label{lemma: regularity of W tau family part 2}
A generic $\{\overline{J}_\tau^\Diamond\}_{\tau\in\R}$ is $\{K_\tau\}$-regular with respect to $\overline{\frak m}$.
\end{lemma}

\begin{lemma} \label{lemma: codimension one part 2}
If $\{\overline{J}_\tau\}$ is a generic family, then for $\varepsilon,\delta>0$ sufficiently small, there exist a generic family $\{\overline{J}_\tau^\Diamond(\varepsilon,\delta,{\frak p}(\tau))\}$ which is $\{K_{{\frak p}(\tau),\delta}\}$-regular with respect to $\overline{\frak m}$ and disjoint finite subsets $\mathcal{T}_1,\mathcal{T}_2\subset \R$ with the following properties:
\begin{enumerate}
\item $\tau\in \mathcal{T}_1$ if and only if there exists $\overline v_\tau\in \mathcal{M}_{\overline{J}_\tau}^{\dagger,s,irr,\op{ind}=-1} (\delta_0^p\bs\gamma,\delta_0^q\bs\gamma')$ for some $\delta_0^p\bs\gamma$ and $\delta^q\bs\gamma'$.
\item $\tau\in \mathcal{T}_2$ if and only if there exists $\overline v_\tau\in \mathcal{M}^{\dagger,s,irr,\{K_\tau\},\op{ind}=1}_{\overline{J}_\tau^\Diamond(\varepsilon,\delta,{\frak p}(\tau))} (\delta_0^p\bs\gamma,\delta_0^q\bs\gamma'; \overline{\frak m})$ for some $\delta_0^p\bs\gamma$ and $\delta^q\bs\gamma'$.
\end{enumerate}
Moreover, for each $\tau\in\mathcal{T}_i$ there is a unique such irreducible curve $\overline v_\tau$.
\end{lemma}

\subsection{Proof of the other half of Theorem~\ref{thm: isomorphism}}
\label{subsection: proof of thm chain homotopy two}

In this subsection we prove the ``other half'' of Theorem~\ref{thm: isomorphism}, i.e., $\Phi\circ\Psi=id$ on the level of homology,  assuming the results of Sections~\ref{subsection: chain homotopy part 4}--\ref{subsection: chain homotopy part 6}.

We make the following simplifying assumption, which is possible by Theorem~I.\ref{P1-thm: elimination}.
\begin{enumerate}
\item[($\dagger\dagger$)] All the elliptic orbits of the stable Hamiltonian vector field corresponding to $h:S\stackrel\sim\to S$ that intersect $S\times\{0\}$ at most $2g$ times have linearized first return maps which are $-\varepsilon$-rotations with $\varepsilon>0$ small.
\end{enumerate}

\begin{thm} \label{thm: chain homotopy ii}
Suppose ($\dagger\dagger$) holds. If $m\gg 0$, then there exists a chain homotopy
$$K: PFC_{2g}(N)\to PFC_{2g}(N),$$
such that the following holds:
\begin{equation} \label{eqn: chain homotopy part ii}
\bdry_{ECH}\circ K + K\circ \bdry_{ECH} = \Phi\circ \Psi + id.
\end{equation}
\end{thm}

\begin{proof}
Suppose $m\gg 0$. Fix ${\frak p}(\tau)$ and choose $\{\overline{J}_\tau\}\in \overline{\mathcal{I}}^{reg}$ satisfying Lemma~\ref{lemma: codimension one part 2}. For sufficiently small $\varepsilon,\delta>0$ (which depend on the choices of $m$ and $\{\overline{J}_\tau\}$), there exists $\{\overline{J}_\tau^\Diamond(\varepsilon,\delta,{\frak p}(\tau))\}$ so that Lemma~\ref{lemma: codimension one part 2} holds.

Fix $\bs\gamma,\bs\gamma'\in \widehat{\mathcal{O}}_{2g}$ and abbreviate $$\mathcal{M}=\mathcal{M}^{I=2,n^*=m}_{\{\overline{J}_\tau^\Diamond(\varepsilon,\delta,{\frak p}(\tau))\}}(\bs\gamma,\bs\gamma';\overline{\frak m}), ~\mathcal{M}^{\{K_{{\frak p}(\tau),\delta}\}}=\mathcal{M}^{I=2,n^*=m,\{K_{{\frak p}(\tau),\delta}\}}_{\{\overline{J}_\tau^\Diamond(\varepsilon,\delta,{\frak p}(\tau))\}}(\bs\gamma,\bs\gamma';\overline{\frak m}).$$
Let $\overline{\mathcal{M}}$ be the SFT compactification of $\mathcal{M}$ and let $\bdry \mathcal{M}= \overline{\mathcal{M}}-\mathcal{M}$ be the boundary of $\mathcal{M}$. If $U\subset [-\infty,+\infty]$, then we write $\bdry_U\mathcal{M}$ for the set of $\overline{u}_\infty\in  \bdry M$ where $\overline{u}_\infty$ is a building which corresponds to some $\tau\in U$.  By Lemma~\ref{lemma: regularity of W tau family part 2}, we may take $\mathcal{M}^{\{K_{{\frak p}(\tau),\delta}\}}$ to be regular.

{\em In view of the considerations from Claim~\ref{elmwood}, we assume that all of $\mathcal{M}$ is regular.}

\s\n {\em Step 1 (Breaking at $+\infty$).}  The following is proved in Section~\ref{subsection: chain homotopy part 4}.

\begin{lemma} \label{cherries}
$\bdry_{\{+\infty\}}\mathcal{M}\subset A_1$, where
$$A_1=\coprod_{{\bf y}\in\mathcal{S}_{{\bf a},\hh({\bf a})}}\left(\mathcal{M}^{I=2,n^*=m}_{\overline{J}_-^\Diamond(\varepsilon,\delta,{\frak p}(+\infty))}(\bs\gamma,{\bf y};\overline{\frak m}(+\infty))\times \mathcal{M}_{J_+}^{I=0}({\bf y},\bs\gamma')   \right).$$
\end{lemma}

Gluing the pairs in $A_1$ accounts for the term $\Phi\circ\Psi$ in Equation~\eqref{eqn: chain homotopy part ii}.

\s\n {\em Step 2 (Breaking at $-\infty$).} The following lemmas are proved in Section~\ref{subsection: chain homotopy part 6}.

\begin{lemma} \label{cherries2}
$\bdry_{\{-\infty\}}\mathcal{M}\subset A_2\sqcup A_3$, where
$$A_2=\coprod_{\bs\gamma\in \widehat{\mathcal{O}}_{2g},{\frak z}\in{\frak Z},r_0=0} \left( \mathcal{M}^{I=0,n^*=0}_{\overline{J}_{-\infty,1}^\Diamond(\varepsilon,\delta,{\frak p}(-\infty))}(\bs\gamma,\bs\gamma,{\frak z})\times \mathcal{M}_{\overline{J}_{-\infty,2}}^{I=4g+2,n^*=m}({\frak z};\overline{\frak m}(-\infty))\right);$$
$$A_3=\coprod_{\bs\gamma,\bs\gamma'\in \widehat{\mathcal{O}}_{2g},\frak z\in {\frak Z}, r_0=1}\left(\mathcal{M}^{I=2g+2,n^*=m}_{\overline{J}_{-\infty,1}^\Diamond(\varepsilon,\delta,{\frak p}(-\infty))}(\bs\gamma,\bs\gamma',{\frak z})\times \mathcal{M}_{\overline{J}_{-\infty,2}}^{I=2g,n^*=0}({\frak z})   \right).$$
\end{lemma}

Note that if $\overline{v}_2\in \mathcal{M}_{\overline{J}_{-\infty,2}}^{I=2g,n^*=0}({\frak z})$, then $\overline{v}_2= B_{-\infty,2}\times {\frak z}$.

\begin{lemma} \label{A2}
Gluing the pairs in $A_2$ accounts for the term $id$ in Equation~\eqref{eqn: chain homotopy part ii}.
\end{lemma}

\begin{lemma}\label{A3}
Gluing the pairs in $A_3$ gives a total of $0$ mod $2$.
\end{lemma}

\s\n {\em Step 3 (Breaking in the middle).} The following is proved in Section~\ref{subsection: chain homotopy part 5}.

\begin{lemma} \label{cherries new3}
$\bdry_{(-\infty,+\infty)}\mathcal{M}\subset A_4\sqcup A_5$, where
$$A_4=\coprod_{\bs\gamma''\in \widehat{\mathcal{O}}_{2g}}\left(\mathcal{M}^{I=1,n^*=m}_{\{\overline{J}_\tau^\Diamond(\varepsilon,\delta,{\frak p}(\tau))\}}(\bs\gamma,\bs\gamma'';\overline{\frak m})\times \mathcal{M}_{J'}^{I=1}(\bs\gamma'',\bs\gamma')   \right);$$
$$A_5=\coprod_{\bs\gamma''\in \widehat{\mathcal{O}}_{2g}}\left(  \mathcal{M}_{J'}^{I=1}(\bs\gamma,\bs\gamma'')  \times\mathcal{M}^{I=1,n^*=m}_{\{\overline{J}_\tau^\Diamond(\varepsilon,\delta,{\frak p}(\tau))\}}(\bs\gamma'',\bs\gamma';\overline{\frak m})  \right);$$
\begin{align*}
A_6=&\coprod_{\bs\gamma'', \bs\gamma'''\in \widehat{\mathcal{O}}_{2g-1} } \left(  \mathcal{M}_{J'}^{I=2, n^*=m-1,f_{\delta_0}}(\bs\gamma,\delta_0\bs\gamma'')  \times\mathcal{M}^{I=-2,n^*=0}_{\{\overline{J}_\tau^\Diamond(\varepsilon,\delta,{\frak p}(\tau))\}}(\delta_0\bs\gamma'',\delta_0\bs\gamma''')\right.\\
& \qquad\qquad
\times \left. \mathcal{M}^{I=2, n^*=1}_{J'}(\delta_0\bs\gamma''', e\bs\gamma''') \right);
\end{align*}
\begin{align*}
A_7=&\coprod_{\bs\gamma'', \bs\gamma'''\in \widehat{\mathcal{O}}_{2g-2} } \left(  \mathcal{M}_{J'}^{I=2, n^*=m-2,f_{\delta_0}}(\bs\gamma,\delta_0^2\bs\gamma'')  \times\mathcal{M}^{I=-3,n^*=0}_{\{\overline{J}_\tau^\Diamond(\varepsilon,\delta,{\frak p}(\tau))\}}(\delta_0^2\bs\gamma'',\delta_0^2\bs\gamma''')\right.\\
& \qquad\qquad
\times \left. \mathcal{M}^{I=3, n^*=2}_{J'}(\delta_0^2\bs\gamma''', eh\bs\gamma''') \right).
\end{align*}
\end{lemma}

Gluing the pairs in $A_4$ and $A_5$ accounts for the terms $\bdry_{ECH}\circ K$ and $K\circ \bdry_{ECH}$ in Equation~\eqref{eqn: chain homotopy part ii}.

\begin{lemma}\label{A6 A7}
Gluing the triples in $A_6$ and $A_7$ gives a total of $0$ mod $2$.
\end{lemma}

\s
This completes the proof of Theorem~\ref{thm: chain homotopy ii}, modulo the results from Sections~\ref{subsection: chain homotopy part 4}--\ref{subsection: chain homotopy part 5}.
\end{proof}

\subsection{Degeneration at $+\infty$}
\label{subsection: chain homotopy part 4}

In this subsection we study the limit of holomorphic maps to $\overline{W}_\tau$ as $\tau\to +\infty$, i.e., when $\overline{W}_\tau$ degenerates into a concatenation of $\overline{W}_-$ with $\overline{W}_+$ along the HF-type end.  This will prove Lemma~\ref{cherries}.

We assume that $m\gg 0$; $\varepsilon,\delta>0$ are sufficiently small; and $\{\overline{J}_\tau\}\in \overline{\mathcal{I}}^{reg}$ and $\{\overline{J}_\tau^\Diamond(\varepsilon,\delta,{\frak p}(\tau))\}$ satisfy Lemma~\ref{lemma: codimension one part 2}. Fix $\bs\gamma,\bs\gamma'\in \widehat{\mathcal{O}}_{2g}$ and let
$$\mathcal{M}=\mathcal{M}^{I=2,n^*=m}_{\{\overline{J}_\tau^\Diamond(\varepsilon,\delta,{\frak p}(\tau))\}}(\bs\gamma,\bs\gamma';\overline{\frak m}),\quad\mathcal{M}_\tau= \mathcal{M}^{I=2,n^*=m}_{\overline{J}_\tau^\Diamond(\varepsilon,\delta,{\frak p}(\tau))}(\bs\gamma,\bs\gamma';\overline{\frak m}).$$
We will analyze $\bdry_{\{+\infty\}}\mathcal{M}$.

Let $\overline{u}_i$, $i\in \N$, be a sequence of curves in $\mathcal{M}$ such that $\overline{u}_i\in\mathcal{M}_{\tau_i}$ and $\displaystyle\lim_{i\to\infty} \tau_i=+\infty$, and let
$$\overline{u}_\infty = (\overline{v}_{-1,1}\cup\dots \cup \overline{v}_{-1,c}) \cup \overline{v}_+\cup (\overline{v}_{0,1}\cup\dots\cup \overline{v}_{0,b})\cup \overline{v}_-\cup (\overline{v}_{1,1}\cup\dots\cup \overline{v}_{1,a})$$
be the limit holomorphic building in order from bottom to top, where each $\overline{v}_*$ is an SFT level, $\overline{v}_{-1,j}$, $j=1,\dots,c$, maps to $\overline{W'}$; $\overline{v}_+$ maps to $\overline{W}_+$; $\overline{v}_{0,j}$, $j=1,\dots,b$, maps to $\overline{W}$; $\overline{v}_-$ maps to $\overline{W}_-$; and $\overline{v}_{1,j}$, $j=1,\dots,a$ maps to $\overline{W'}$. Here we are allowing the possibility that $a$, $b$, or $c=0$. For notational convenience, sometimes we refer to $\overline{v}_+$ as $\overline{v}_{-1,c+1}$ or $\overline{v}_{0,0}$ and $\overline{v}_-$ as $\overline{v}_{0,b+1}$ or $\overline{v}_{1,0}$.

As before, we have the following constraints:
\begin{equation}\label{sum of n part 3}
n^*(\overline{u}_i)=\sum_{\overline{v}_*}n^*(\overline{v}_*)=m;
\end{equation}
\begin{equation} \label{sum of I part 3}
I(\overline{u}_i)=\sum_{\overline{v}_*}I(\overline{v}_*)=2,
\end{equation}
where the summations are over all the levels $\overline{v}_*$ of $\overline{u}_\infty$.

The following is the analog of Lemma~\ref{los angeles}, with a similar proof (omitted):

\begin{lemma} \label{los angeles 4}
If $\overline{v}'_*\cup\overline{v}^\sharp_*=\varnothing$ for all levels $\overline{v}_*$ of $\overline{u}_\infty$, then $a=b=c=0$; $I(\overline{v}_+)=0$; $I(\overline{v}_-)=2$; $\overline{v}_+$ is a $W_+$-curve; and $\overline{v}_-$ is a $\overline{W}_-$-curve.
\end{lemma}

\begin{lemma}\label{sencha alt}
If $\overline{v}'_*\cup\overline{v}^\sharp_*\not=\varnothing$ for some level $\overline{v}_*$ of $\overline{u}_\infty$, then:
\begin{enumerate}
\item $p_-=\deg(\overline{v}'_-)>0$;
\item some $\overline{v}_{1,j_0}$, $j_0>0$, has a negative end $\mathcal{E}_-$ that limits to $\delta_0^p$ for some $p>0$ and satisfies $n^*(\mathcal{E}_-)\geq m-p$;
\item $\overline{u}_\infty$ has no boundary point at $z_\infty$;
\item $\overline{u}_\infty$ has no fiber components and no components of $\overline{v}''_*$ that intersect the interior of a section at infinity;
\item each component of $\overline{v}^\sharp_{0,j}$, $1\leq j\leq b$, is a thin strip;
\item each component of $\overline{v}^\sharp_+$ is an $n^*=1$, $I=0$ or $1$ section from $z_\infty$ to $h$ or $e$ which is contained in $\overline{W}_+-W_+$;
\item each component of $\overline{v}^\sharp_{-1,j}$, $0\leq j\leq c$, is an $n^*=1$, $I=1$ or $2$ cylinder from $\delta_0$ to $h$ or $e$ which is contained in $\R\times(\overline{N}-N)$;
\item $h$ appears at most once at the negative end of $\overline{v}^\sharp_{-1,1}$.
\end{enumerate}
\end{lemma}

\begin{proof}
The proof is based on Equation~\eqref{sum of n part 3}.  First observe that:
\begin{enumerate}
\item[(*)] either there is a negative end $\mathcal{E}_-$ that limits to $\delta_0^p$ for some $p>0$ and satisfies $n^*(\mathcal{E}_-)\geq m-p$ by Lemma~\ref{intersezione}; or
\item[(**)] there is a negative end $\mathcal{E}_-$ that limits to $z_\infty$ and the sum of $n^*(\mathcal{E}_i)$ over all the ends $\mathcal{E}_i$ that limit to $z_\infty$ and $n^*(\mathcal{E}_i')$ over all the neighborhoods of the boundary points at $z_\infty$ is $\geq m-2g$.
\end{enumerate}

(1) If (*) or (**) holds, then $\overline{v}'_-\not=\varnothing$, since otherwise the neighborhood of $\overline{\frak m}(+\infty)$ contributes $m$ towards $n^*(\overline{v}_-)$.

(2) is a consequence of (1) and is a subcase of (*). In particular (**) does not hold. (3)--(7) follow from (2).  (8) follows from the definition of the ECH differential.
\end{proof}

\begin{lemma} \label{eliminate end h}
If $\overline{v}'_*\cup\overline{v}^\sharp_*\not=\varnothing$ for some level $\overline{v}_*$, then $\overline{u}_\infty$ cannot have the following subbuildings:
\begin{enumerate}
\item a degree one component of $\overline{v}^\sharp_+$ from $z_\infty$ to $h$; and
\item $\overline{v}'_+$ which has degree $p_+$ and $\overline{v}^\sharp_{-1,1}$ which is a union of $p_+$ cylinders from $\delta_0$ to $h$.
\end{enumerate}
\end{lemma}

\begin{proof}
This is due to the positioning of $h$, given in Section~\ref{subsubsection: convention bambi}.  (1) was proved in Lemma~I.\ref{P1-lemma: value of widetilde Phi}. (2) is due to Lemma~\ref{lemma 2012}: the usual rescaling procedure with fixed $m\gg 0$, together with Lemma~\ref{sencha alt}, gives rise to an SFT limit $w_+: \Sigma_+\to \C\P^1$, $\pi_+:\Sigma_+\to cl(B_+)$, where $\Sigma_+$ consists of $p_+$ copies of $cl(B_+)$ (and hence $\pi_+$ is a trivial branched cover) and the restriction of $w_+$ to each component of $\Sigma_+$ satisfies the conditions of Lemma~\ref{2012}.
\end{proof}

\begin{lemma} \label{options}
If $\overline{v}'_*\cup\overline{v}^\sharp_*\not=\varnothing$ for some level $\overline{v}_*$, then $\overline{v}_-'\not=\varnothing$ and $\overline{u}_\infty$ contains one of the following subbuildings:
\begin{enumerate}
\item A $3$-level building consisting of $\overline{v}_{1,1}^\sharp$ with $I=1$ and a negative end $\delta_0\bs\gamma'$; $\overline{v}_-'=\sigma_\infty^-$; and a thin strip of $\overline{v}^\sharp_{0,1}$.
\item[(2)] A $3$-level building consisting of $\overline{v}_{1,1}^\sharp$ with $I=1$ and a negative end $\delta_0\bs\gamma'$; $\overline{v}_-'=\sigma_\infty^-$; and a component of $\overline{v}_+^\sharp$ with $I=1$ from $z_\infty$ to $e$.
\item[(3)] A $4$-level building consisting of $\overline{v}_{1,1}^\sharp$ with $I=1$ and a negative end $\delta_0\bs\gamma'$; $\overline{v}_-'=\sigma_\infty^-$; $\overline{v}_+'=\sigma_\infty^+$; and a cylinder component of  $\overline{v}^\sharp_{-1,1}$ from $\delta_0$ to $e$.
\end{enumerate}
Here we are omitting levels which are connectors.
\end{lemma}

\begin{proof}
The lemma is a consequence of Lemmas~\ref{sencha alt} and \ref{eliminate end h}, subject to the conditions given by Equations~\eqref{sum of n part 3} and \eqref{sum of I part 3}.
\end{proof}

\begin{lemma} \label{apricot}
If $m\gg 0$, then there is no $\overline{u}_\infty\in \bdry_{\{+\infty\}} \mathcal{M}$ such that $\overline{v}'_*\cup\overline{v}^\sharp_*\not=\varnothing$ for some level $\overline{v}_*$.
\end{lemma}

\begin{proof}
The proofs  to eliminate Cases (1)--(3) of Lemma~\ref{options} are similar to those of Theorem~I.\ref{P1-thm: complement} and will be omitted.
\end{proof}

\begin{proof}[Proof of Lemma~\ref{cherries}]
This is a combination of Lemmas~\ref{los angeles 4} and \ref{apricot}.
\end{proof}

\subsection{Degeneration at $-\infty$}
\label{subsection: chain homotopy part 6}

In this subsection we study the limit of holomorphic maps to $\overline{W}_\tau$ as $\tau \to -\infty$.  This will prove Lemma~\ref{cherries2}.

We assume that $m\gg 0$; $\varepsilon,\delta>0$ are sufficiently small; and $\{\overline{J}_\tau\}\in \overline{\mathcal{I}}^{reg}$ and $\{\overline{J}_\tau^\Diamond(\varepsilon,\delta,{\frak p}(\tau))\}$ satisfy Lemma~\ref{lemma: codimension one part 2}. Fix $\bs\gamma,\bs\gamma'\in \widehat{\mathcal{O}}_{2g}$ and let
$$\mathcal{M}=\mathcal{M}^{I=2,n^*=m}_{\{\overline{J}_\tau^\Diamond(\varepsilon,\delta,{\frak p}(\tau))\}}(\bs\gamma,\bs\gamma';\overline{\frak m}),\quad\mathcal{M}_\tau= \mathcal{M}^{I=2,n^*=m}_{\overline{J}_\tau^\Diamond(\varepsilon,\delta,{\frak p}(\tau))}(\bs\gamma,\bs\gamma';\overline{\frak m}).$$
We will analyze $\bdry_{\{-\infty\}}\mathcal{M}$.

Let $\overline{u}_i$, $i\in \N$, be a sequence of curves in $\mathcal{M}$ such that $\overline{u}_i\in\mathcal{M}_{\tau_i}$ and $\displaystyle\lim_{i\to\infty} \tau_i=-\infty$, and let
$$\overline{u}_\infty = (\overline{v}_{-1,1}\cup\dots \cup \overline{v}_{-1,c}) \cup \overline{v}_{1}\cup \overline{v}_{2}\cup (\overline{v}_{1,1}\cup\dots\cup \overline{v}_{1,a})$$
be the limit holomorphic building in order from bottom to top, where each $\overline{v}_*$ is an SFT level, $\overline{v}_{-1,j}$, $j=1,\dots,c$, maps to $\overline{W'}$; $\overline{v}_{j}$ maps to $\overline{W}_{-\infty,j}$; and $\overline{v}_{1,j}$, $j=1,\dots,a$ maps to $\overline{W'}$. Sometimes we refer to $\overline{v}_{1}$ as $\overline{v}_{-1,c+1}$ or $\overline{v}_{1,0}$.

We write $\overline{v}_*=\overline{v}'_*\cup \overline{v}''_*$, $\overline{v}''_*= \overline{v}^\sharp_*\cup\overline{v}^\flat_*$, where:
\begin{itemize}
\item $\overline{v}'_*$ is the union of branched covers of a section at infinity;
\item $\overline{v}^\sharp_*$ is the union of components that are not in $\overline{v}'_*$ and are asymptotic to some multiple of $\delta_0$ or pass through ${\frak z}=z_\infty^{r_0}\zeta_1^{r_1}\cdots\zeta_l^{r_l}$ with $r_0>0$; and
\item $\overline{v}^\flat_*$ is the union of the remaining components of $\overline{v}_*$.
\end{itemize}

\begin{rmk}
Strictly speaking, in the limit $\tau\to-\infty$, there exist levels between $\overline{v}_1$ and $\overline{v}_2$, i.e., levels that map to $\R\times S^1\times \overline{S}$. By considerations of $n^*$, these levels are connectors (i.e., map to $\R\times S^1\times\{pt\}$) and will be ignored until we consider gluing.
\end{rmk}

The following are the analogs of Lemmas~\ref{los angeles 4} and \ref{sencha alt}:

\begin{lemma} \label{los angeles 5}
If $\overline{v}'_*\cup\overline{v}^\sharp_*=\varnothing$ for all levels $\overline{v}_*$ of $\overline{u}_\infty$, then $a=c=0$; $I(\overline{v}_1)=0$; $I(\overline{v}_2)=4g+2$; and $\overline{v}_2$ is a $\overline{W}_2$-curve.
\end{lemma}

\begin{proof}
Suppose $\overline{v}'_*\cup\overline{v}^\sharp_*=\varnothing$ for all levels $\overline{v}_*$ of $\overline{u}_\infty$.  The following are immediate from considerations of $n^*$:
\begin{enumerate}
\item $\overline{v}_2$ is a $\overline{W}_2$-curve or a degenerate $\overline{W}_2$-curve; and
\item $\overline{v}_*=\overline{v}''_*$ and $n^*(\overline{v}_*)=0$ for $*=(-1,j)$, $1$, and $(1,j)$.
\end{enumerate}
(1) implies that $I(\overline{v}_2)=4g+2$ by Lemma~\ref{lemma: curves in overline W minus infty 2}. Since there are $2g$ codimension two gluing conditions between $\overline{v}_1$ and $\overline{v}_2$, it follows that
\begin{equation}
\sum_{*} I(\overline{v}_*)=4g+2,
\end{equation}
where the summation is over all the levels $*$. Hence $a=c=0$, $I(\overline{v}_1)=0$,  and each component $\widetilde{v}$ of $\overline{v}_1$ is a branched cover of a trivial cylinder with possibly empty branch locus. (Here we are assuming without loss of generality that the almost complex structure on $\overline{W}_1$ is $\overline{J}_1$.)

We eliminate degenerate $\overline{W}_2$-curves as follows: Observe that each non-fiber component of a degenerate $\overline{W}_2$-curve is of the type $B_{-\infty,2}\times \{q_l\}$ for some $q_l\in \widehat{a}_l$ and has Fredholm index one. Since we may assume that $\widehat{\mathcal{O}}_{2g}$ is disjoint from $\overline{\bf a}\times\{1\}\subset \overline{S}\times \{1\}$ by genericity, the gluing condition for $\overline{v}_1$ and $\overline{v}_2$ is not satisfied.
\end{proof}

\begin{lemma}\label{sencha alt 3}
If $\overline{v}'_*\cup\overline{v}^\sharp_*\not=\varnothing$ for some level $\overline{v}_*$ of $\overline{u}_\infty$, then:
\begin{enumerate}
\item $p_2=\deg(\overline{v}'_2)>0$;
\item $\overline{u}_\infty$ has no boundary point at $z_\infty$;
\item $\overline{u}_\infty$ has no fiber components and no components of $\overline{v}''_*$ that intersect the interior of a section at infinity; and
\item $\overline{v}''_2$ is the union of $2g-p_2$ copies of $B_{-\infty,2}\times\{pt\}$, and $I(\overline{v}_2)=2g$.
\end{enumerate}
\end{lemma}

\begin{proof}
(1)--(3) are analogs of Lemma~\ref{sencha alt} and (4) is a consequence of Lemma \ref{lemma: curves in overline W minus infty 3}.
\end{proof}

The following is the analog of Lemma~\ref{options}.

\begin{lemma} \label{options minus infty}
If $\overline{v}'_*\cup\overline{v}^\sharp_*\not =\varnothing$ for some level $\overline{v}_*$, then $\overline{v}_{2}'\not=\varnothing$ and $\overline{v}_2''$ is a union of components $B_{-\infty,2}\times \{pt\}$, and $\overline{u}_\infty$ contains one of the following subbuildings:
\begin{enumerate}
\item[(1$_i$)] A $2$-level building consisting of $\overline{v}_{1}^\sharp$ with $I=i+(\deg (\overline{v}_1^\sharp)-1)$, $i=2,3$, which passes through ${\frak z}$ with $r_0=1$; and $\overline{v}_{2}'=\sigma_\infty^{-\infty,2}$ with $I=1$.
\item[(2)] A $2$-level building consisting of $\overline{v}_{1}^\sharp$ with $I=4+(\deg(\overline{v}_1^\sharp)-2)$ which passes through ${\frak z}$ with $r_0=2$; and $\overline{v}_{2}'$ with $I=2$ which is a branched double cover of $\sigma_\infty^{-\infty,2}$.
\item[(3$_i$)] A $4$-level building consisting of $\overline{v}_{1,1}^\sharp$ with $I=1$ or $2$ and a negative end at $\delta_0$; $\overline{v}_{1}'=\sigma_\infty^{-\infty,1}$; $\overline{v}_{2}'=\sigma_\infty^{-\infty,2}$; and a cylinder of $\overline{v}_{-1,1}^\sharp$ from $\delta_0$ to $h$ or $e$ with $I=1$ or $2$.
\item[(4)] A $4$-level building consisting of $\overline{v}_{1,1}^\sharp$ with $I=1$ and a negative end at $\delta_0^2$;  $\overline{v}_{1}'=\sigma_\infty^{-\infty,1}$; a cylinder of $\overline{v}_{1}^\sharp$ from $\delta_0$ to $h$ or $e$ with $I=1,2$; $\overline{v}_{2}'=\sigma_\infty^{-\infty,2}$; a cylinder of $\overline{v}_{-1,1}^\sharp$ from $\delta_0$ to $h$ with $I=1$; and a cylinder of $\overline{v}_{-1,1}''$ from $h$ or $e$ to $e$ with $I=1,0$.
\item[(5)] A $4$-level building consisting of $\overline{v}_{1,1}^\sharp$ with $I=1$ and a negative end at $\delta_0^2$; $\overline{v}_{1}'$ with $I=0$ which is a branched double cover of $\sigma_\infty^{-\infty,1}$; $\overline{v}_{2}'$ with $I=2$ which is a branched double cover of $\sigma_\infty^{-\infty,2}$; and two cylinder components of $\cup_{j=1}^c\overline{v}^\sharp_{-1,j}$ from $\delta_0$ to $h$ or $e$, each with $I=1$ or $2$.
\item[(6)] A $5$-level building consisting of $\overline{v}_{1,2}^\sharp$ with $I=1$ and a negative end at $\delta_0^2$; $\overline{v}_{1,1}'=\R\times\delta_0$; a component of $\overline{v}_{1,1}^\sharp$ with $I=1$ which is a cylinder from $\delta_0$ to $h$; $\overline{v}_{1}'=\sigma_\infty^{-\infty,1}$; $\overline{v}_{2}'=\sigma_\infty^{-\infty,2}$; and a cylinder of $\overline{v}_{-1,1}^\sharp$ from $\delta_0$ to $h$ with $I=1$.
\end{enumerate}
We are omitting levels which are branched covers of trivial cylinders. Moreover, each gluing condition reduces the sum of ECH indices by $2$.
\end{lemma}

See Figure~\ref{figure: graphs7ver2}.

\begin{figure}[ht]
\begin{center}
\psfragscanon
\psfrag{0}{\small $0$}
\psfrag{1}{\small $1$}
\psfrag{2}{\small $2$}
\psfrag{4}{\small $4$}
\psfrag{f}{\small $I=1,2$}
\psfrag{g}{\small $I=2,3$}
\psfrag{G}{\small $-4$}
\psfrag{H}{\small $-2$}
\psfrag{e}{\small $e$}
\psfrag{h}{\small $h$}
\psfrag{i}{\small $h,e$}
\psfrag{j}{\small $1,2$}

\psfrag{A}{(1$_i$)}
\psfrag{B}{(2)}
\psfrag{C}{(3$_i$)}
\psfrag{D}{(4)}
\psfrag{5}{(5)}
\psfrag{6}{(6)}
\psfrag{7}{(7)}

\psfrag{v}{\small $\overline{v}_1$}
\psfrag{w}{\small $\overline{v}_2$}
\includegraphics[width=10cm]{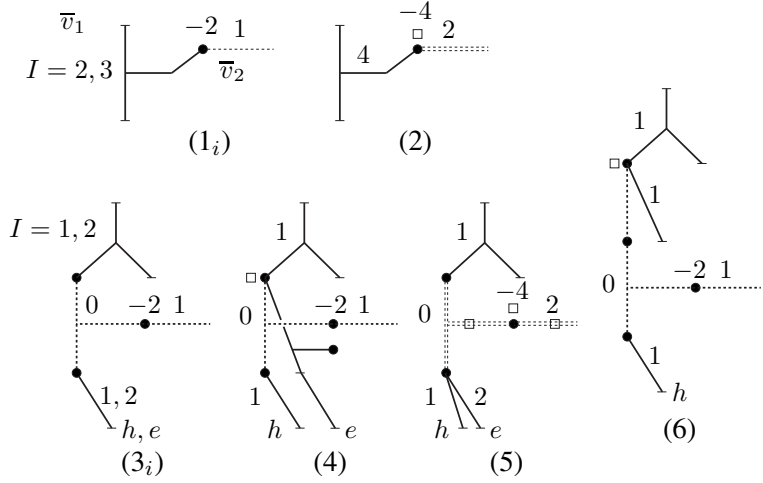}
\end{center}
\caption{Schematic diagrams for the possible types of degenerations. The dotted lines indicate the section at infinity. The quantities $-2$ and $-4$ along the intersections of $\overline{v}_1$ and $\overline{v}_2$ indicate the reduction of $I$ due the gluing conditions. For simplicity we are assuming that $\deg \overline{v}_1^\sharp=1$ in (1$_i$) and $\deg \overline{v}^\sharp_1=2$ in (2).}
\label{figure: graphs7ver2}
\end{figure}

\begin{proof}
The proof is similar to that of Lemma~\ref{options} and we only give a sketch.

We have $\overline{v}'_2\not=\varnothing$ by Lemma~\ref{sencha alt 3}(1). By Lemma~\ref{sencha alt 3}(4), $\overline{v}_{2}''$ is a union of components $B_{-\infty,2}\times\{q_l\}$, $q_l\in \widehat{a}_l$, and $I(\overline{v}_2)=2g$.

First suppose that $\overline{v}_1'=\varnothing$. Then $\overline{v}_1^\sharp$ passes through $((0,{3\over 2}),z_\infty)$ and $\pi_{B_{-\infty,1}}\circ \overline{v}_1^\sharp$ has a $k$-fold branch point at $(0,{3\over 2})$ for some $k>0$, where $\pi_{B_{-\infty,1}}: \overline{W}_{-\infty,1}\to B_{-\infty,1}$ is the projection along $\overline{S}$. The $k$-fold branch point at $((0,{3\over 2}),z_\infty)$ contributes $2k$ towards  $\op{ind}(\overline{v}_1^\sharp)$ and each of the remaining $2g-k$ intersection points with $\{(0,{3\over 2})\}\times \overline{S}$ passes through some $\widehat{a}_l$ and contributes $+1$ towards $\op{ind}(\overline{v}_1^\sharp)$.  Hence $I(\overline{v}_1)\geq \op{ind}(\overline{v}_1)\geq 2k+(2g-k)$. Since the gluing conditions contribute $-4g$, the total ECH contribution of $I(\overline{v}_1)$, $I(\overline{v}_2)$, and the gluing conditions is $\geq k$.  This implies that $k\leq 2$, giving us (1$_i$), $i=2,3$, or (2).

Next suppose that $\overline{v}_1'\not=\varnothing$. Then $I(\overline{v}_1')=0$, $I(\overline{v}_2')=k$ for some $k>0$, and the gluing conditions give $-2k$. Hence (3$_i$), $i=1,2$, and (4)--(6) follow from enumerating all the possibilities, subject to the condition that $h$ appears only once at the negative end of $\overline{v}_{-1,1}$.
\end{proof}

\begin{lemma} \label{apricot minus infty}
If $m\gg 0$, $\overline{u}_\infty\in \bdry_{\{-\infty\}} \mathcal{M}$, and $\overline{v}'_*\cup \overline{v}^\sharp_*\not=\varnothing$ for some level $\overline{v}_*$, then the only possibility is Case (1$_3$).
\end{lemma}

\begin{proof}
{\em Cases (1$_2$) and (2).} We will treat Case (1$_2$); Case (2) is similar. We apply the usual rescaling argument with $m\gg 0$ fixed to obtain a function $w_{2}: B_{-\infty,2}\to \C\P^1$ satisfying the following:
\begin{enumerate}
\item[(i$_2$)] $w_{2}(0)=\infty$;
\item[(ii$_2$)] $w_{2}({i\over 2})=0$;
\item[(iii$_2$)] $w_{2}(\bdry B_{-\infty,2})\subset \R^+\cdot e^{i\phi(\overline{a}_{k,l})}$ for some $(k,l)$;
\item[(iv$_2$)] $w_{2}$ is a biholomorphism away from $\bdry B_{-\infty,2}$.
\end{enumerate}
We now observe that $w_{2}$ is uniquely determined by (i$_2$)--(iv$_2$), up to multiplication by a positive real constant.
This implies that:
\begin{enumerate}
\item[(v$_2$)] $w_{2}$ maps the asymptotic marker $\bdry_y$ at $0\in B_{-\infty,2}$ to the asymptotic marker $\dot{\mathcal{R}}_{\pi+\phi(\overline{a}_{k,l})}(\infty)$.
\end{enumerate}
Here $z=x+iy$ is the complex coordinate on $B_{-\infty,2}$.
(v$_2$) translates into an asymptotic condition for $\overline{v}_1^\sharp$ at $((0,{3\over 2}),z_\infty)$. Hence $I(\overline{v}_1^\sharp)$ is at least $3$.

\s\n
{\em Cases (3$_i$), (4)--(6).} We will treat Case (3$_i$). The rescaling argument gives $w_{2}\cup w_{1}$, where $w_{2}$ is as before and $w_{1}: cl(B_{-\infty,1})\to \C\P^1$ satisfies the following:
\begin{enumerate}
\item[(i$_1$)] $w_{1}(0,{3\over 2})=0$ and $w_{1}(+\infty)=+\infty$;
\item[(ii$_1$)] $w_1$ is a biholomorphism.
\end{enumerate}
(iv$_2$) implies the following asymptotic condition for $w_1$:
\begin{enumerate}
\item[(iii$_1$)] $w_{1}$ maps the marker $\bdry_s$ at $(0,{3\over 2})$ to the marker $\dot{\mathcal{R}}_{\pi+\phi(\overline{a}_{k,l})}(0)$.
\end{enumerate}
Hence $w_1$ is uniquely determined by (i$_1$)--(iii$_1$) up to multiplication by a positive real constant.

As a consequence of the uniqueness of $w_{1}$ up to multiplication by a positive real constant, the following are uniquely determined:
\begin{enumerate}
\item[(a)] the asymptotic eigenfunction of $\overline{v}_{1,1}^\sharp$ at the negative end $\delta_0$;
\item[(b)] the asymptotic eigenfunction of $\overline{v}_{-1,1}^\sharp$ at the positive end $\delta_0$;
\end{enumerate}
(a) is determined by the image of the asymptotic marker $\dot{\mathcal{L}}_{3/2}(+\infty)$ at $+\infty\in cl(B_{-\infty,1})$ and (b) by the radial ray that contains $w_{1}(-\infty)$. 
(a) and (b) give rise to one constraint each on $\overline{v}_{1,1}^\sharp$ and $\overline{v}_{-1,1}^\sharp$.  Hence $I(\overline{v}_{1,1}^\sharp)\geq 2$ and $I(\overline{v}_{-1,1}^\sharp)\geq 2$, which is a contradiction.
\end{proof}

\begin{proof}[Proof of Lemma~\ref{cherries2}]
This is a combination of Lemmas~\ref{los angeles 5} and \ref{apricot minus infty}.
\end{proof}

\begin{proof}[Proof of Lemma~\ref{A2}]
We first claim that the mod $2$ count of
$$\mathcal{M}_{\overline{J}_{-\infty,2}}^{I=4g+2,n^*=m}({\frak z};\overline{\frak m}(-\infty))$$
is $1$ when ${\frak z}$ is generic and $r_0=0$. This is proved by reducing to the calculation of Theorem~\ref{thm: calc of G sub 2} as follows: Degenerate $B_{-\infty,2}$ into a sphere $B_{-\infty,21}$ and a disk $B_{-\infty,22}$ which are identified at one point and degenerate $\overline{W}_{-\infty,2}=B_{-\infty,2}\times \overline{S}$ into $(B_{-\infty,21}\times\overline{S})\cup (B_{-\infty,22}\times\overline{S})$. We assume that the marked point is in $B_{-\infty,21}\times\overline{S}$. Then a curve $$\overline{v}_2\in \mathcal{M}_{\overline{J}_{-\infty,2}}^{I=4g+2,n^*=m}({\frak z};\overline{\frak m}(-\infty))$$
degenerates into a pair $(\overline{v}_{21},\overline{v}_{22})$, where $\overline{v}_{2i}$, $i=1,2$, maps to $B_{-\infty,2i}\times\overline{S}$ and $\overline{v}_{22}$ is a union of constant sections $B_{-\infty,22}\times\{q_l\}$, where $q_l\in \widehat{a}_l$.  Hence $\overline{v}_{21}$ is a curve in $B_{-\infty,21}\times\overline{S}$ with exactly the same type of constraints as in Theorem~\ref{thm: calc of G sub 2}. This implies the claim.

In the rest of the proof we discuss how to glue pairs $(\overline{v}_1,\overline{v}_2)\in A_2$. For simplicity we work with $\overline{J}_{-\infty,1}$ instead of $\overline{J}_{-\infty,1}^\Diamond(\varepsilon,\delta,{\frak p}(-\infty)))$. Since $I(\overline{v}_1)=0$, each component $\widetilde{v}$ of $\overline{v}_1$ is a branched cover of a trivial cylinder with possibly empty branch locus. If $\widetilde{v}$ is simple, then it glues to $\overline{v}_2$ in the usual manner. On the other hand, if $\widetilde{v}$ is multiply covered, then it is the cylinder over an elliptic orbit which is the same type as $e$ by ($\dagger\dagger$). Suppose $\widetilde{v}$ covers $\R\times e$ with multiplicity $k$. By the partition conditions, the incoming partition is $(1,\dots,1)$ and the outgoing partition is  $(k)$. Recall that the outgoing partition condition applies to the positive end of $\widetilde{v}$ and the incoming partition condition applies to the negative end of $\widetilde{v}$; this is the same partition condition as that of $\overline{u}_i$ for $i\gg 0$ where $\overline{u}_i\to \overline{u}_\infty$. Hence $\widetilde{v}$ has  $1$ positive end, $k$ negative ends, and ${\frak b}\geq k-1$ branch points (in the sense of Definition~\ref{branch points}). For consistency with $\overline{v}_2$, we must have a branch point at $(0,{3\over 2})$ which contributes $k-1$ towards ${\frak b}$.

Next we define $\op{ind}'(\overline{v}_1)$ as the Fredholm index when:
\begin{itemize}
\item $\overline{v}_1$ is viewed as a curve in the fibration $$\overline{W}_{-\infty,1}-(\{(0,3/2)\}\times\overline{S})\to B_{-\infty,1}-\{(0,3/2)\};$$
\item $\{(0,{3\over 2})\}\times\overline{S}$ is in one-to-one correspondence with a Morse-Bott family of orbits, viewed as orbits at the {\em positive end of $\overline{v}_1$;}
\item after perturbing the Morse-Bott family, the ends of $\overline{v}_1$ corresponding to $\zeta\in \{(0,{3\over 2})\}\times\overline{S}$ are viewed as limiting to elliptic orbits with Conley-Zehnder index $1$ with respect to the framing coming from the Morse-Bott fibration.
\end{itemize}
By the usual Fredholm index calculation, we obtain $\op{ind}'(\widetilde{v})\geq 0$, with equality if and only if ${\frak b}=k-1$. We can similarly define $\op{ind}'$ for $\overline{v}_2$. Then $\op{ind}'(\overline{v}_2)=2$. When we use $\op{ind}'$ instead of $\op{ind}$, the gluing conditions  between $\overline{v}_1$ and $\overline{v}_2$ are $0$-dimensional instead of $2g$-dimensional. This implies that $\op{ind}'(\widetilde{v})=0$, ${\frak b}=k-1$, and $\widetilde{v}$ has no branch points besides $(0,{3\over 2})$.

Finally we discuss the automatic transversality of $\widetilde{v}$. Recall from \cite[Theorem 1]{We3} that automatic transversality holds if
\begin{equation} \label{mr wendl}
\op{ind}'(\widetilde{v})\geq 2g+\#\Gamma_0-1,
\end{equation}
where $g$ is the genus of $\widetilde{v}$ and $\#\Gamma_0$ is the number of punctures (of $\widetilde{v}$ with $(0,{3\over 2})$ removed) with even Conley-Zehnder index. Since the right-hand side of Equation~\eqref{mr wendl} is equal to $-1$, automatic transversality holds and $\widetilde{v}$ glues in the usual manner (without any concerns of inserting branch-covered cylinders) to $\overline{v}_2$.  This implies the lemma.
\end{proof}

\begin{proof}[Sketch of proof of Lemma~\ref{A3}]
We use some considerations of Section~I.\ref{P1-subsection: gluing for psi}. For simplicity we assume we are gluing degree one curves. Fix $\bs\gamma,\bs\gamma'\in \widehat{\mathcal{O}}_{1}$, $k\in\{1,\dots,2g\}$, and $l\in\{0,1\}$.  We consider the gluing parameter space
$${\frak P}_{k,l}:= \coprod_{\frak z_0\in \overline{{\frak u}}_{k,l} } \left((-\infty,-r]\times \mathcal{M}_1({\frak z}_0)\times\mathcal{M}_2({\frak z}_0) \right),$$
where ${\frak u}_{k,l}\subset \vec{a}_{k,l}$ is a small open interval containing $z_\infty$,
\begin{align*}
\mathcal{M}_1({\frak z}_0)&:= \mathcal{M}^{I=3,n^*=m}_{\overline{J}_{-\infty,1}^\Diamond(\varepsilon,\delta,{\frak p}(-\infty))}(\bs\gamma,\bs\gamma',{\frak z}_0),\\
\mathcal{M}_2({\frak z}_0)&:= \mathcal{M}_{\overline{J}_{-\infty,2}}^{I=1,n^*=0,ext}({\frak z}_0),
\end{align*}
$ext$ means the boundary of the holomorphic curve is mapped to $L_{\overline{a}_k\cup \vec{a}_{k,l}}^{-\infty,2}$, and $(-\infty,-r]$ is viewed as a subset of $(-\infty,\infty)$ with parameter $\tau$. For each ${\frak z}_0\in {\frak u}_{k,l}$ there is a covering map
$$\pi_{{\frak z}_0}:\mathcal{M}_1({\frak z}_0)\to S^1=\R/2\pi\Z$$
which maps $\overline{v}_1$ to its asymptotic marker at ${\frak z}_0$ and is continuous in ${\frak z}_0$. Note that $\mathcal{M}_2({\frak z}_0)$ consists of a single point $B_{-\infty,2}\times\{{\frak z}_0\}$.

Let $$ G_{k,l}: {\frak P}_{k,l}\to \coprod_{\tau\in(-\infty,r]}\mathcal{M}^{I=2,n^*=m,ext}_{\{\overline{J}_\tau^\Diamond(\varepsilon,\delta,{\frak p}(\tau))\}}(\bs\gamma,\bs\gamma'),\quad {\frak d}=(\tau,\overline{v}_1,\overline{v}_2)\mapsto \overline{u}({\frak d})$$
be the gluing map, defined as usual. Also let ${\frak P}_{k,l}'\subset {\frak P}_{k,l}$ be the subset of gluing parameters ${\frak d}$ such that $\displaystyle G_{k,l}({\frak d})\in\coprod_{\tau\in(-\infty,r]}\mathcal{M}^{I=2,n^*=m}_{\{\overline{J}_\tau^\Diamond(\varepsilon,\delta,{\frak p}(\tau))\}}(\bs\gamma,\bs\gamma')$.

Let $\widetilde{\mathcal{N}}$ be the space of holomorphic maps $w_{2}: B_{-\infty,2}=\overline{\D}\to \C\P^1$ satisfying (i$_2$) and (iv$_2$) of Lemma~\ref{apricot minus infty} and
\begin{enumerate}
\item[(iii$'_2$)] $w_{2}(\bdry B_{-\infty,2})\subset \R \cdot e^{i\phi(\overline{a}_{k,l})}$ for some $(k,l)$.
\end{enumerate}
For simplicity we assume that $\phi(\overline{a}_{k,l})=0$. We leave it to the reader to verify the following:

\begin{claim}
$\widetilde{\mathcal{N}}=\{ w_2(z)= az+ b+{\overline{a}\over z}~|~ a\in \C^\times, b\in\R\}$. Hence $\dim \widetilde{\mathcal{N}}=3$ and $\widetilde{\mathcal{N}}$ admits an $\R^\times$-action which is multiplication by $c\in \R^\times$.
\end{claim}

For $\tau_0\ll 0$, the restrictions of $\overline{u}({\frak d})$ to a neighborhood of the section at infinity are approximated by elements of $\widetilde{\mathcal{N}}$.  More precisely, for $\tau_0\ll 0$ we define:
$$g_{\tau_0,k,l}: {\frak P}_{k,l}\cap \{\tau=\tau_0\} \to \widetilde{\mathcal{N}},$$
as follows: Given ${\frak d}\in {\frak P}_{k,l}\cap \{\tau=\tau_0\} $, restrict $G_{k,l}({\frak d})$ to a neighborhood of the section at infinity so that the domain of $G_{k,l}({\frak d})$ is $\overline{\D}^\varepsilon:=\{ \varepsilon \leq |z|\leq 1\}\subset \overline{\D}$ for $\varepsilon>0$ small.  Let us write $\pi_{\overline{S}}\circ G_{k,l}({\frak d})|_{\overline{\D}^\varepsilon}$ as a Laurent series $\sum_{i=-\infty}^\infty c_i({\frak d}) z^i$. Then we set
$$g_{\tau_0,k,l}({\frak d})=\overline{c_{-1}({\frak d})}\cdot z + \op{Re}(c_0({\frak d})) + {c_{-1}({\frak d})\over z}.$$
The definition makes sense in view of the following (proof omitted):

\begin{claim} \label{haruchan} $\mbox{}$
\begin{enumerate}
\item $\displaystyle \lim_{\tau\to-\infty} {c_i({\frak d})\over c_{-1} ({\frak d})}=0$ if $i>1$ or $i<-1$ and $\displaystyle \lim_{\tau\to-\infty} {\overline{c_1({\frak d})}\over c_{-1} ({\frak d})}=1$, where the convergence is uniform in $\displaystyle\coprod_{\frak z_0\in \overline{{\frak u}}_{k,l} } \left(\mathcal{M}_1({\frak z}_0)\times\mathcal{M}_2({\frak z}_0) \right)$.
\item If $\op{Im}(\overline{v}_2)=\{{\frak z}_0\}\not=\{z_\infty\}$, then
$$\displaystyle\lim_{\tau_0\to -\infty}{c_{-1}({\frak d})\over c_0({\frak d})}= 0 ~~\mbox{ and } ~~\displaystyle\lim_{\tau_0\to -\infty}{\op{Im}(c_0({\frak d}))\over \op{Re}(c_0({\frak d}))}= 0,$$
where the convergence is uniform in $\mathcal{M}_1({\frak z}_0)$.
\end{enumerate}
\end{claim}

Now define the evaluation map:
$$ev_{\tau_0,k,l}: {\frak P}_{k,l}\cap \{\tau=\tau_0\} \to \C,$$
which sends ${\frak d}$ to $g_{\tau_0,k,l}({\frak d}) ({i\over 2})$. Using Claim~\ref{haruchan} we can verify that
the local degree of $(ev_{\tau_0,k,l})|_{{\frak P}'_{k,l}\cap \{\tau=\tau_0\}}$ near $0\in \C$ is $\deg \pi_{{\frak z}_0}$. We will informally say that ``$\mathcal{M}$ has $\deg \pi_{{\frak z}_0}$ ends near ${\frak P}'_{k,l}$''.

Finally observe that an element of $\mathcal{M}_2(z_\infty)$ can be viewed as a map to $\overline{a}_{i,0}$ or to $\overline{a}_{i,1}$. Hence $\mathcal{M}$ has $\deg \pi_{{\frak z}_0}$ ends near ${\frak P}'_{k,l}$ and $\deg \pi_{{\frak z}_0}$ ends near ${\frak P}'_{k,1-l}$ for a total of $2\cdot \deg \pi_{{\frak z}_0}$ ends mod $2$. This proves Lemma~\ref{A3}.
\end{proof}

\subsection{Breaking in the middle}
\label{subsection: chain homotopy part 5}

In this subsection we study the limit of holomorphic maps to $\overline{W}_\tau$ as $\tau \to T'$ for some $T'\in(-\infty,\infty)$.  This will prove Lemma~\ref{cherries new3}.

We assume that $m\gg 0$; $\varepsilon,\delta>0$ are sufficiently small; and $\{\overline{J}_\tau\}\in \overline{\mathcal{I}}^{reg}$ and $\{\overline{J}_\tau^\Diamond(\varepsilon,\delta,{\frak p}(\tau))\}$ satisfy Lemma~\ref{lemma: codimension one part 2}. Fix $\bs\gamma,\bs\gamma'\in \widehat{\mathcal{O}}_{2g}$ and let
$$\mathcal{M}=\mathcal{M}^{I=2,n^*=m}_{\{\overline{J}_\tau^\Diamond(\varepsilon,\delta,{\frak p}(\tau))\}}(\bs\gamma,\bs\gamma';\overline{\frak m}),\quad\mathcal{M}_\tau= \mathcal{M}^{I=2,n^*=m}_{\overline{J}_\tau^\Diamond(\varepsilon,\delta,{\frak p}(\tau))}(\bs\gamma,\bs\gamma';\overline{\frak m}).$$
We will analyze $\bdry_{(-\infty,\infty)}\mathcal{M}$.

Let $\overline{u}_i$, $i\in \N$, be a sequence of curves in $\mathcal{M}$ such that $\overline{u}_i\in\mathcal{M}_{\tau_i}$ and $\displaystyle\lim_{i\to\infty} \tau_i=T'$, and let
$$\overline{u}_\infty= (\overline{v}_{-1,1}\cup\dots\cup \overline{v}_{-1,c})\cup \overline{v}_0 \cup (\overline{v}_{1,1}\cup\dots\cup\overline{v}_{1,a})$$
be the limit holomorphic building in order from bottom to top, where each $\overline{v}_*$ is an SFT-type level, $\overline{v}_{-1,j}$ and $\overline{v}_{1,j}$ map to $\overline{W'}$ and $\overline{v}_0$ maps to $\overline{W}_{T'}$. Sometimes we refer to $\overline{v}_0$ as $\overline{v}_{-1,c+1}$ or $\overline{v}_{1,0}$.

The following is the analog of Lemma~\ref{los angeles 4}, and is stated without proof.

\begin{lemma} \label{los angeles 6}
If $\overline{v}'_*\cup\overline{v}^\sharp_*=\varnothing$ for all levels $\overline{v}_*$ of $\overline{u}_\infty$, then $\overline{u}_\infty$ is one of the following:
\begin{enumerate}
\item $a=0$, $c=1$; $\overline{v}_0$ is a $\overline{W}_{T'}$-curve with $I=1$ which passes through $\overline{\frak m}(T')$; and $\overline{v}_{-1,1}$ is a $W'$-curve with $I=1$; or
\item $a=1$, $c=0$; $\overline{v}_{1,1}$ is a $W'$-curve with $I=1$; and $\overline{v}_0$ is a $\overline{W}_{T'}$-curve with $I=1$ which passes through $\overline{\frak m}(T')$.
\end{enumerate}
Here either $T'\in \mathcal{T}_2$ and a component of $\overline{v}_0$ is in
$$\mathcal{M}_{\overline{J}^\Diamond_{T'}(\varepsilon,\delta,{\frak p}(T'))}^{\dagger,s,irr,\op{ind}=1,n^*=m}(\bs\gamma_1,\bs\gamma_2;\overline{\frak m}(T'))$$ from Lemma~\ref{lemma: codimension one part 2}(2), for some $\bs\gamma_1$, $\bs\gamma_2$; or $T'\in \mathcal{T}_1$ and a component of $\overline{v}_0$ does not pass through $\overline{\frak m}(T')$ but is in
$$\mathcal{M}_{\overline{J}^\Diamond_{T'}(\varepsilon,\delta,{\frak p}(T'))}^{\dagger,s,irr,\op{ind}=-1,n^*=0}(\bs\gamma_1,\bs\gamma_2)$$ from Lemma~\ref{lemma: codimension one part 2}(1).
\end{lemma}

The following is the analog of Lemma~\ref{sencha alt} and is stated without proof.

\begin{lemma}\label{sencha alt 2}
If $\overline{v}'_*\cup\overline{v}^\sharp_*\not=\varnothing$ for some level $\overline{v}_*$ of $\overline{u}_\infty$, then:
\begin{enumerate}
\item $p_0=\deg(\overline{v}'_0)>0$;
\item some $\overline{v}_{1,j_0}$, $j_0>0$, has a negative end $\mathcal{E}_-$ that limits to $\delta_0^p$ for some $p>0$ and satisfies $n^*(\mathcal{E}_-)\geq m-p$;
\item $\overline{u}_\infty$ has no boundary point at $z_\infty$;
\item $\overline{u}_\infty$ has no fiber components and no components of $\overline{v}''_*$ that intersect the interior of a section at infinity;
\item each component of $\overline{v}^\sharp_{-1,j}$, $0\leq j\leq c$, is an $n^*=1$, $I=1$ or $2$ cylinder from $\delta_0$ to $h$ or $e$ which is contained in $\R\times(\overline{N}-N)$;
\item $h$ appears at most once at the negative end of $\overline{v}^\sharp_{-1,1}$.
\end{enumerate}
\end{lemma}

The following is the analog of Lemma~\ref{options} and is stated without proof.

\begin{lemma} \label{options2}
If $\overline{v}'_*\cup\overline{v}^\sharp_*\not=\varnothing$ for some level $\overline{v}_*$, then $\overline{v}_0'\not=\varnothing$ and $\overline{u}_\infty$ contains one of the following subbuildings, subject to two conditions:
\begin{itemize}
\item the sum of the ECH indices of the components of the subbuildings is at most $3$;
\item the total multiplicity of $h$ at the negative end of $\overline{v}_{-1,1}^\sharp$ is at most $1$.
\end{itemize}
\begin{enumerate}
\item[($1_i$)] A $3$-level building consisting of $\overline{v}_{1,1}^\sharp$ with $I=i$, $i=1,2$, and a negative end $\delta_0\bs\gamma'$; $\overline{v}_0'=\sigma_\infty^\tau$; and a cylinder component of $\overline{v}_{-1,1}^\sharp$ with $I=1$ or $2$ from $\delta_0$ to $h$ or $e$.
\item[($2_i$)] A $4$-level building consisting of $\overline{v}_{1,2}^\sharp$ with $I=i$, $i=1,2$, and a negative end $\delta_0^2\bs\gamma'$; $\overline{v}_{1,1}'=\R\times \delta_0$; a cylinder component of $\overline{v}_{1,1}^\sharp$ with $I=1$ or $2$ from $\delta_0$ to $h$ or $e$;
    $\overline{v}_0'=\sigma_\infty^\tau$; and a cylinder component of $\overline{v}_{-1,1}^\sharp$ with $I=1$ or $2$ from $\delta_0$ to $h$ or $e$.
\item[($3_i$)] A $3$-level building consisting of $\overline{v}_{1,1}^\sharp$ with $I=i$, $i=1,2$, and a negative end $\delta_0^2\bs\gamma'$;
    $\overline{v}_0'=\sigma_\infty^\tau$; a component of $\overline{v}_0^\sharp$ with $I=0$ or $1$ from $\delta_0$ to $h$ or $e$; and a cylinder component of $\overline{v}_{-1,1}^\sharp$ with $I=1$ or $2$ from $\delta_0$ to $h$ or $e$.
\item[($4_i$)] A $3$-level building consisting of $\overline{v}_{1,1}^\sharp$ with $I=i$, $i=1,2$, and a negative end $\delta_0^2\bs\gamma'$; $\overline{v}_0'$ with $I=-2$ which is a degree $2$ branched cover of $\sigma_\infty^\tau$; and two cylinder components of $\cup_{j=1}^c\overline{v}^\sharp_{-1,j}$ from $\delta_0$ to $h$ or $e$, each with $I=1$ or $2$.
\item[(5)] A $5$-level building consisting of $\overline{v}_{1,3}^\sharp$ with $I=1$ and a negative end $\delta_0^3\bs\gamma'$; $\overline{v}_{1,2}'$ which is a degree $2$ branched cover of $\R\times\delta_0$; a cylinder component of $\overline{v}_{1,2}^\sharp$ with $I=1$ from $\delta_0$ to $h$;
    $\overline{v}_{1,1}'=\R\times\delta_0$; a cylinder component of $\overline{v}_{1,1}^\sharp$ with $I=1$ from $\delta_0$ to $h$;
    $\overline{v}_0'=\sigma_\infty^\tau$; and a cylinder component of $\overline{v}_{-1,1}^\sharp$ with $I=1$ from $\delta_0$ to $h$.
\item[(6)] A $4$-level building consisting of $\overline{v}_{1,2}^\sharp$ with $I=1$ and a negative end $\delta_0^3\bs\gamma'$;
    $\overline{v}_{1,1}'=\R\times\delta_0$; two cylinder components of $\overline{v}_{1,1}^\sharp$ from $\delta_0$ to $h$, each with $I=1$;
    $\overline{v}_0'=\sigma_\infty^\tau$; and a cylinder component of $\overline{v}_{-1,1}^\sharp$ with $I=1$ from $\delta_0$ to $h$.
\item[(7)] A $4$-level building consisting of $\overline{v}_{1,2}^\sharp$ with $I=1$ and a negative end $\delta_0^3\bs\gamma'$; $\overline{v}_{1,1}'$ which is a degree $2$ branched cover of $\R\times\delta_0$; a cylinder component of $\overline{v}_{1,1}^\sharp$ with $I=1$ or $2$ from $\delta_0$ to $h$ or $e$; $\overline{v}_0'=\sigma_\infty^\tau$; a component of $\overline{v}_0^\sharp$ with $I=0$ or $1$ from $\delta_0$ to $h$ or $e$; and a cylinder component of $\overline{v}_{-1,1}^\sharp$ with $I=1$ or $2$ from $\delta_0$ to $h$ or $e$.
\item[(8)] A $4$-level building consisting of $\overline{v}_{1,2}^\sharp$ with $I=1$ and a negative end $\delta_0^3\bs\gamma'$; $\overline{v}_{1,1}'$ which is a degree $2$ branched cover of $\R\times\delta_0$; a cylinder component of $\overline{v}_{1,1}^\sharp$ with $I=1$ from $\delta_0$ to $h$;
    $\overline{v}_0'$ with $I=-2$ which is a degree $2$ branched cover of $\sigma_\infty^\tau$; and two cylinder components of $\cup_{j=1}^c\overline{v}^\sharp_{-1,j}$ from $\delta_0$ to $h$ or $e$, each with $I=1$ or $2$.
\item[(9)] A $3$-level building consisting of $\overline{v}_{1,1}^\sharp$ with $I=1$ and a negative end $\delta_0^3\bs\gamma'$; $\overline{v}_0'=\sigma_\infty^\tau$; two components of $\overline{v}_0^\sharp$ from $\delta_0$ to $h$ or $e$, each with $I=0$ or $1$; and a cylinder component of $\overline{v}_{-1,1}^\sharp$ with $I=1$ or $2$ from $\delta_0$ to $h$ or $e$.
\item[(10)] A $3$-level building consisting of $\overline{v}_{1,1}^\sharp$ with $I=1$ and a negative end $\delta_0^3\bs\gamma'$; $\overline{v}_0'$ with $I=-2$ which is a degree $2$ branched cover of $\sigma_\infty^\tau$; a component of $\overline{v}_0^\sharp$ with $I=0$ or $1$ from $\delta_0$ to $h$ or $e$; and two cylinder components of $\cup_{j=1}^c\overline{v}^\sharp_{-1,j}$ from $\delta_0$ to $h$ or $e$, each with $I=1$ or $2$.
\item[(11)] A $3$-level building consisting of $\overline{v}_{1,1}^\sharp$ with $I=1$ and a negative end $\delta_0^3\bs\gamma'$; $\overline{v}_0'$ with $I=-3$ which is a degree $3$ branched cover of $\sigma_\infty^\tau$; and three cylinder components of $\cup_{j=1}^c\overline{v}^\sharp_{-1,j}$ from $\delta_0$ to $h$ or $e$, each with $I=1$ or $2$.
\end{enumerate}
We are omitting levels which are connectors.
\end{lemma}

See Figure~\ref{figure: graphs6}.  We will write ($1_{i,e}$), ($1_{i,h}$), etc.\ to indicate that we are in Case ($1_i$) and the negative ends of the lowest level are $e$, $h$, etc.

\begin{figure}[ht]
\vskip.2in
\begin{overpic}[width=12cm]{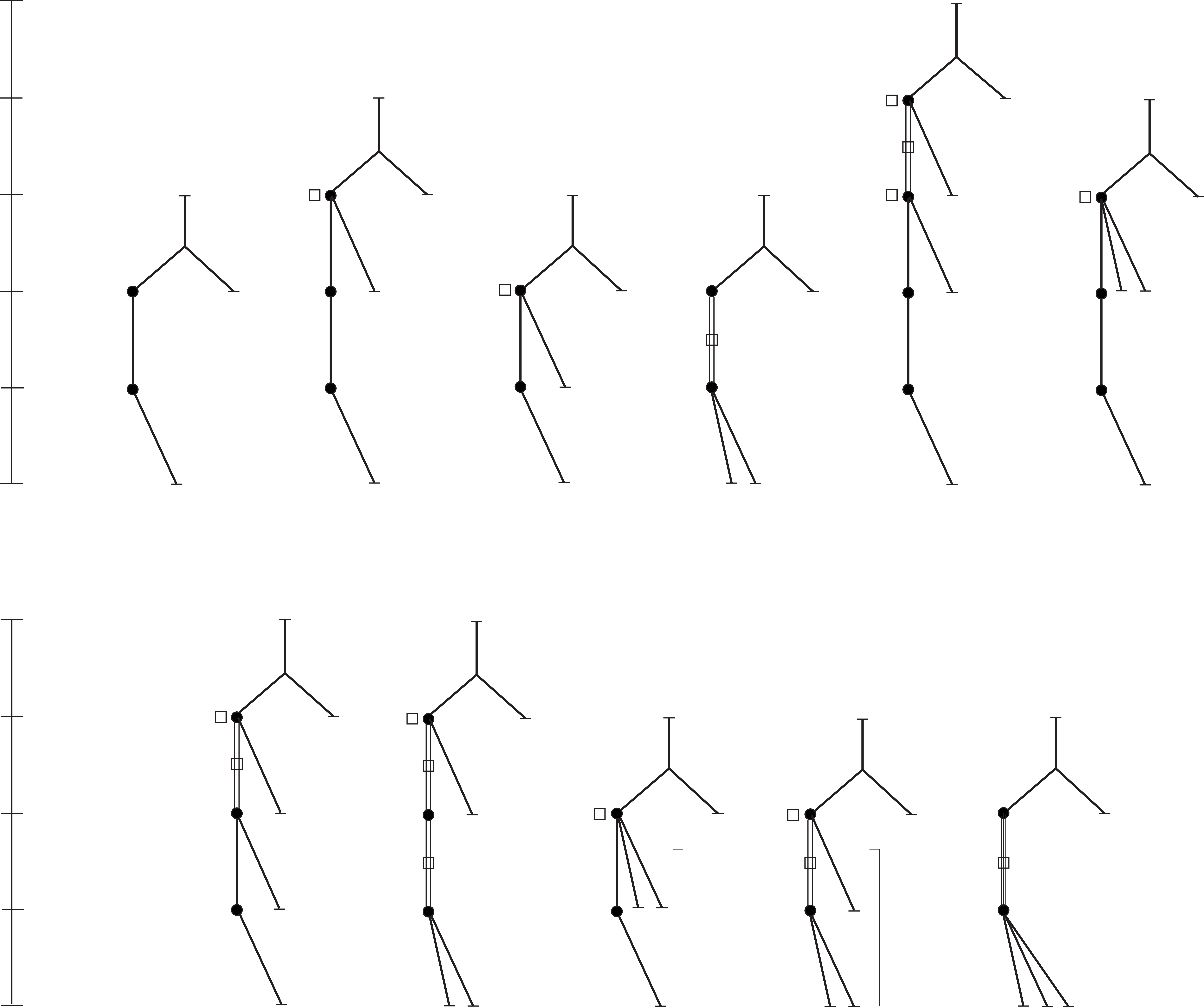}

\put(-7,79){\tiny $\R\times \overline{N}$} \put(-7,71){\tiny $\R\times \overline{N}$} \put(-7,62.8){\tiny $\R\times \overline{N}$} \put(-4,54.75){\tiny $\overline{W}_\tau$} \put(-7,46.75){\tiny $\R\times \overline{N}$}

\put(-7,27.3) {\tiny $\R\times \overline{N}$} \put(-7,19.2){\tiny $\R\times \overline{N}$} \put(-4,11.5){\tiny $\overline{W}_\tau$} \put(-7,3.2){\tiny $\R\times\overline{N}$}

\put(12.6,38){ (1$_i$)} \put(29.5,38) {(2$_i$)} \put(46,38){(3$_i$)} \put(61.6,38){(4$_i$)} \put(77.5,38){(5)} \put(93.7,38){(6)}

\put(7.15,64){\tiny $i=1,2$} \put(7,55){\tiny $-1$} \put(13.8,47){\tiny $1,2$} \put(15.7,42.5) {\tiny $h,e$} 

\put(23,72){\tiny $i=1,2$} \put(25.3,63){\tiny $0$} \put(30,63){\tiny $1,2$} \put(23.3,55){\tiny $-1$} \put(30.33,47){\tiny $1,2$} \put(32.3,42.5){\tiny $h,e$} 

\put(39,64){\tiny $i=1,2$} \put(39.1,55){\tiny $-1$} \put(46.12,55){\tiny $0,1$} \put(46,47){\tiny $1,2$} \put(48,42.5){\tiny $h,e$} 

\put(55,64){\tiny $i=1,2$} \put(54.83,55){\tiny $-2$} \put(56,42.5){\tiny $h,e$} \put(63.9,42.5){\tiny $e$} \put(55.5,47){\tiny $1,2$} \put(62.5,47){\tiny $2$} 

\put(77,79){\tiny $1$} \put(72.8,71){\tiny $0$}  \put(78,71) {\tiny $1$} \put(73.2,63) {\tiny $0$} \put(78,63){\tiny $1$} \put(71.75,55){\tiny $-1$} \put(78.13,47){\tiny $1$} \put(80.5,42.5){\tiny $h$} 

\put(92.9,71){\tiny $1$} \put(89.15,63) {\tiny $0$} \put(94.5,63){\tiny $1,1$} \put(87.88,55){\tiny $-1$} \put(94.14,47){\tiny $1$} \put(96.42,42.5){\tiny $h$} 

\put(22,-4){(7)} \put(37.55,-4){(8)} \put(53.6,-4){(9)} \put(69.2,-4){(10)} \put(85.7,-4){(11)}

\put(21,28){\tiny $1$} \put(17.15,20){\tiny $0$} \put(22.5,20){\tiny $1,2$} \put(15.5,12){\tiny $-1$} \put(22.5,12){\tiny $0,1$} \put(22.5,4){\tiny $1,2$} \put(25,-.5){\tiny $h,e$} 

\put(37,28){\tiny $1$} \put(33,20){\tiny $0$} \put(38,20){\tiny $1$} \put(31.4,12){\tiny $-2$} \put(31.3,4){\tiny $1,2$} \put(32.3,-.5){\tiny $h,e$} \put(38.65,4){\tiny $2$} \put(40.3,-.5){\tiny $e$} 

\put(52.8,20){\tiny $1$} \put(58,4.7){\tiny $\mbox{total}=2$} 

\put(69,20){\tiny $1$}\put(74,4.7){\tiny $\mbox{total}=2$}  

\put(85,20){\tiny $1$} \put(79.25,11.8){\tiny $-3$} \put(87,4.7){\tiny $\mbox{total}=5$} 

\end{overpic}
\vskip.2in
\caption{Schematic diagrams for the possible types of degenerations. For simplicity we have drawn only one level $\overline{v}_{-1,1}$ to indicate the cylindrical components of $\cup_{j=1}^c\overline{v}^\sharp_{-1,j}$ from $\delta_0$ to $h$ or $e$.}
\label{figure: graphs6}
\end{figure}
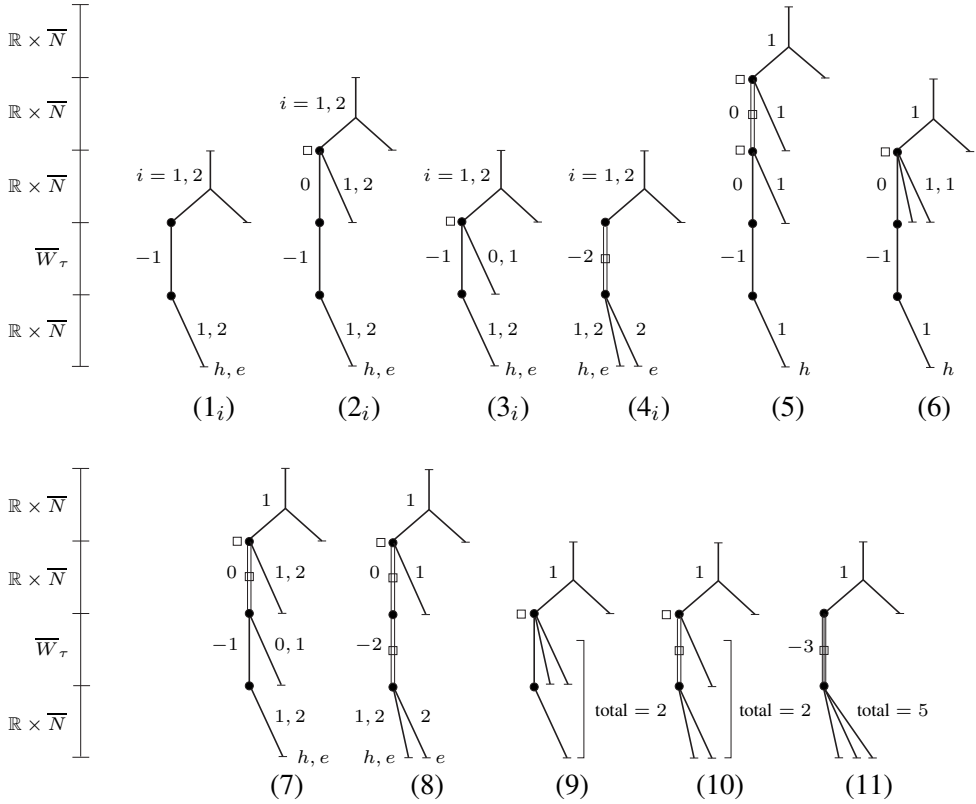

\begin{lemma} \label{apricot 2}
If $m\gg 0$, $\overline{u}_\infty\in \bdry_{(-\infty,+\infty)} \mathcal{M}$, and $\overline{v}'_*\cup\overline{v}^\sharp_*\not=\varnothing$ for some level $\overline{v}_*$, then the only possibilities are (1$_2$) with a cylinder component of $\overline{v}^\sharp_{-1,1}$ from $\delta_0$ to $e$ and (4$_2$) with two cylinder components of $\overline{v}^\sharp_{-1,1}$, one from $\delta_0$ to $h$ and another from $\delta_0$ to $e$.
\end{lemma}

\begin{proof}
{\em Cases (1$_i$), (2$_i$), (3$_i$), (5), (6), (7) and (9).}  We will treat Case (1$_i$); the rest of the cases are similar and can be eliminated. The key observation here is that $\deg (\overline{v}'_0)=1$ and $\overline{v}^\sharp_{-1,1}$ is a cylinder from $\delta_0$ to $h$ or $e$. Applying the usual rescaling argument with $m\gg 0$ fixed, we obtain a holomorphic map $w_0:cl(B_{T'})\to \C\P^1$ which satisfies the following:
\begin{enumerate}
\item[(i)] $w_0(+\infty)=\infty$ and $w_0(\overline{\frak m}^b(T'))=0$;
\item[(ii)] $w_0(s,t)\in int(\mathcal{R}_{\eta(t)})$ for all $(s,t)\in \bdry B_{T'}$;
\item[(iii)] $\deg(w_0)=1$ away from ${\frak S}(\overline{a}_{i_1,j_1},\overline{\hh}(\overline{a}_{i_1,j_1}))$ for some $(i_1,j_1)$.
\end{enumerate}
Here  $\eta(t)=\phi_0+{\pi\over m}(t-1)$, where $\phi_0$ is the $\phi$-coordinate of $\overline{a}_{i_1,j_1}$. (Recall that we are projecting to $\pi_{D^2_{\rho_0}}$ using balanced coordinates.)

We now observe that $w_0: cl(B_{T'})\to \C\P^1$ is uniquely determined by (i)--(iii), up to multiplication by a positive real constant; this is argued in the same way as in Lemma~\ref{2012}. Using the same method as in Case (3$_i$) of Lemma~\ref{apricot minus infty}, we obtain that $I(\overline{v}_{1,1}^\sharp)\geq 2$ and $I(\overline{v}_{-1,1}^\sharp)\geq 2$. Hence the only possibility is Case (1$_2$) with a cylinder component of $\overline{v}^\sharp_{-1,1}$ from $\delta_0$ to $e$.

\s\n
{\em Cases (4$_i$), (8), (10), and (11).} We will treat Case (4$_i$).  In this case we first apply the rescaling argument with $m\to\infty$ to obtain a holomorphic map $w_0: \Sigma_0\to \C\P^1$ and a branched double cover $\pi_0: \Sigma_0\to cl(B_{T'})$ such that:
\begin{enumerate}
\item[(i)] $w_0(z_0)=\infty$, where $\pi_0^{-1}(+\infty)=\{z_0\}$;
\item[(ii)] $w_0(z_1)=0$ for some $z_1\in \pi_0^{-1}(\overline{\frak m}^b(T'))$;
\item[(iii)] $w_0(\pi_0^{-1}(\bdry B_{T'}))\subset \{\phi=0, \rho>0\}$;
\item[(iv)] $w_0|_{int(\Sigma_0)}$ is a biholomorphism onto its image $\C\P^1-([a_1,a_2]\cup [a_3,a_4])$ with $0<a_1<a_2\leq a_3<a_4$.
\end{enumerate}
Here an interval $[a,b]$ stands for $\{\phi=0, a\leq \rho\leq b\}$.

By the Involution Lemma~I.\ref{P1-lemma: effect of involutions},
\begin{enumerate}
\item[(v)]$\pi_0\circ w_0^{-1}$ maps both $(-\infty,a_1]$ and $[a_4,\infty)$ to $\mathcal{L}_{3/2}\cap \{s\geq \tfrac{l(T')}{2}\}$ and $[a_2,a_3]$ to $\mathcal{L}_{1/2}\cup (\mathcal{L}_{3/2}\cap \{s\leq \tfrac{-l(T')}{2}\})$.
\end{enumerate}
In particular this constrains the asymptotic behavior of $\overline{v}_{1,1}^\sharp$ and hence $I(\overline{v}_{1,1}^\sharp)\geq 2$.  Hence the only possibility is Case  (4$_2$) with two cylinder components of $\overline{v}^\sharp_{-1,1}$, one from $\delta_0$ to $h$ and another from $\delta_0$ to $e$.
\end{proof}

\begin{proof}[Proof of Lemma~\ref{cherries new3}]
This is a combination of Lemmas~\ref{los angeles 6} and \ref{apricot 2}.
\end{proof}

\begin{proof}[Proof of Lemma~\ref{A6 A7}]
Gluing triples in $A_6$, namely Case (1$_2$) with a cylinder component of $\overline{v}^\sharp_{-1,1}$ from $\delta_0$ to $e$, is similar to the gluing problem treated in Theorem~I.\ref{P1-thm: transversality of ev map}.  The proof will be omitted, but we note that the gluing occurs in pairs $(i_1,0)$ and $(i_1,1)$, where the section at infinity $\sigma_\infty^\tau$ is viewed as having boundary mapping to $L^\tau_{\overline{a}_{i_1,0}}$ or $L^\tau_{\overline{a}_{i_1,1}}$.  Hence the count corresponding to $A_6$ is $0$ mod $2$.

Next we treat $A_7$, namely Case (4$_2$) with two cylinder components of $\overline{v}^\sharp_{-1,1}$, one from $\delta_0$ to $h$ and another from $\delta_0$ to $e$.   Consider the moduli space $\mathcal{M}_{\infty}$ of pairs $(w_0,\pi_0)$ consisting of a holomorphic map $w_0: \Sigma_0\to \C\P^1$ and a branched double cover $\pi_0:\Sigma_0\to cl(B_{T'})$ satisfying (i)--(iv) in Case (4$_2$) of Lemma~\ref{apricot 2} for some $a_1,\dots,a_4$. It is left to the reader to verify that:
\begin{itemize}
\item $\dim \mathcal{M}_{\infty}=2$ and that the two parameters are given by multiplication of $w_0$ by a positive real constant and moving a branched point of $w_0$ along $\mathcal{L}_{1/2}\cup (\mathcal{L}_{3/2}\cap \{s\leq \tfrac{-l(T')}{2}\})$.
\item either $w_0(\pi_0^{-1}(-\infty))\subset [a_2,a_3]$ or is on a line orthogonal to $[a_2,a_3]$ at $a^*\in[a_2,a_3]$, which maps to a branch point of $\pi_0$.
\end{itemize}
Next let $\widetilde{\mathcal{M}}_\infty$ be the set of $(w_0,\pi_0)\in \mathcal{M}_\infty$ together with a labeling of the points of $\pi^{-1}(-\infty)$; in other words, $\pi^{-1}(-\infty)$ is viewed as an ordered pair of points, sometimes with duplicates.  The forgetful map $\widetilde{\mathcal{M}}_\infty\to \mathcal{M}_\infty$ is generically a double cover.  Then consider the following evaluation map:
$$ev_\infty: \widetilde{\mathcal{M}}_\infty/\R^+\to \R^2,$$
$$ (w_0,\pi_0)\mapsto \op{arg}(w_0(\pi^{-1}(-\infty))).$$
Here $\widetilde{\mathcal{M}}_\infty/\R^+$ is the quotient of $\widetilde{\mathcal{M}}_\infty$ by the $\R^+$-action given by multiplication of $w_0$ by a positive real constant, and $\op{arg}$ refers to projection to the $\phi$-coordinate. Let $Z=\{x_1=\phi_0\}\subset \R^2$. Then $ev_\infty$ has intersection number $1$ (mod $2$) with $Z$.

Next we apply the rescaling argument again with $m\gg 0$ fixed to obtain a holomorphic map $w_0: \Sigma_0\to \C\P^1$ and a branched double cover $\pi_0: \Sigma_0\to cl(B_{T'})$ such that:
\begin{enumerate}
\item[(i)] $w_0(z_0)=\infty$ for some $z_0\in \pi_0^{-1}(+\infty)$;
\item[(ii)] $w_0(z_1)=0$ for some $z_1\in \pi_0^{-1}(\overline{\frak m}^b(T'))$;
\item[(iii)] $w_0(\pi_0^{-1}(s,t))\in int(\mathcal{R}_{\eta_1(t)}\cup \mathcal{R}_{\eta_2(t)})$ for all $(s,t)\in \bdry B_{T'}$;
\item[(iv)] $\deg(w_0)=1$ away from ${\frak S}(\overline{a}_{i_k,j_k},\overline{\hh}(\overline{a}_{i_k,j_k}))$, $k=1,2$, for some $(i_k,j_k)$.
\end{enumerate}
Here $\eta_k(t)=\phi_k+{\pi\over m}(t-1)$, $k=1,2$, where $\phi_k$ is the $\phi$-coordinate of $\overline{a}_{i_k,j_k}$. If we analogously define $\widetilde{\mathcal{M}}_m$ and the map
$$ev_m:\widetilde{\mathcal{M}}_m/\R^+\to \R^2,$$
then the intersection number of $ev_m$ with $Z$ is also $1$ (mod $2$). The well-definition of the intersection number is left to the reader. We simply note that the required properness of $ev_m$ is a consequence of Gromov compactness.

The gluing of triples in $A_7$ is similar to the gluing problem treated in Theorem~I.\ref{P1-thm: transversality of ev map}.  Again we note that the branched double cover of $\sigma_\infty^\tau$ can be viewed as having boundary mapping to $L^\tau_{\overline{a}_{i_k,0}}$ or $L^\tau_{\overline{a}_{i_k,1}}$ for $k=1,2$. Hence the count corresponding to $A_7$ is $0$ mod $2$.
\end{proof}

\section{Stabilization}
\label{section: stabilization}

The goal of this section is to prove Theorem~\ref{thm: stabilization}.
Let $N=N_{(S,\hh)}$ be the mapping torus of $(S,\hh)$, where $S$ is a bordered surface of genus $g$ with connected boundary and $\hh:(S,\omega)\stackrel\sim\to (S,\omega)$ is a symplectomorphism which has zero flux and restricts to the identity on $\bdry S$.

The strategy of the proof is to apply two positive stabilizations to $(S,\hh)$ --- corresponding to the connected sum with a trefoil knot --- to obtain $(S',\hh')$, where $S'$ has connected boundary and genus $g+1$. We then compare $\Phi_{(S,\hh)}$ and $\Phi_{(S',\hh')}$, which both induce isomorphisms on the level of homology. Here $\Phi_{(S,\hh)}$ is the $\Phi$ map for $(S,\hh)$.

\subsection{The setup}

Let $T$ be a genus one surface with connected boundary and let $\eta_0$, $\eta_1$ be two essential simple closed curves on $T$ which intersect transversely in one point.  A positive Dehn twist along a closed curve $\eta$ will be denoted by $\tau_\eta$. Let $\hh_T:T\stackrel\sim\to T$ be the first return map of a Reeb vector field on the mapping torus
$$N_T=N_{(T,\hh_T)}=(T\times[0,1])/((x,1)\sim(\hh_T(x),0))$$
so that the following hold:
\begin{enumerate}
\item $\hh_T$ is isotopic to $\tau_{\eta_0} \circ \tau_{\eta_1}$ relative to the boundary;
\item all the Reeb orbits in the interior of $N_T$ which intersect $T\times\{0\}$ at most $2g+2$ times are nondegenerate;
\item $\hh_T|_{\bdry T}=id$ and $\bdry N_T$ is foliated by a negative Morse-Bott family of slope $\infty$.
\end{enumerate}
In view of the discussion in Section~I.\ref{P1-section: periodic Floer homology}, $\hh_T$ will be viewed interchangeably as (i) the first return map of a Reeb vector field or (ii) the time-1 map of a stable Hamiltonian vector field with zero flux.

Now let $S'$ be the boundary connected sum of $S$ and $T$.  More precisely, if $P$ is a pair-of-pants with boundary $\bdry P= \bdry_1 P\sqcup \bdry_2 P\sqcup \bdry_3 P$, then $S'$ is obtained from $S\sqcup T\sqcup P$ by identifying $\bdry S\simeq -\bdry_1 P$ and $\bdry T\simeq -\bdry_2 P$.  Observe that $\bdry S'=\bdry_3 P$ is connected and $g(S')=g+1$. See Figure~\ref{fig: sprime}.

We now define a symplectomorphism $\hh':(S',\omega')\stackrel\sim\to (S',\omega')$ with zero flux (with respect to some $\omega'$) as follows:  First set $\hh'|_{S}=\hh$ and $\hh'|_T=\hh_T$. Then define $\hh'|_P$ as the first return map of a Reeb flow $R_\alpha$ on
$$N_P=N_{(P,id)}=(P\times[0,1])/((x,1)\sim(x,0)).$$
The contact form $\alpha$ is given by $f_0dt+\beta$, where $f_0$ is a function and $\beta$ is a $1$-form on $P$, and both $f_0$ and $\beta$ do not depend on $t$. We choose a Morse-Bott function $f_0:P\to \R$ which satisfies the following:
\begin{enumerate}
\item $f_0$ is $C^k$-close to $1$ for $k\gg 0$;
\item $f_0$ attains its minimum along the Morse-Bott family $\bdry P$;
\item the critical points of $f_0$ in $int(P)$ are isolated and consist of the maximum $e_P$ and two saddles $h_{1P}$, $h_{2P}$.
\end{enumerate}
The Reeb orbits corresponding to $e_P$, $h_{1P}$, $h_{2P}$ will also be denoted by $e_P$, $h_{1P}$, $h_{2P}$.

Let $\mathcal{N}_i$, $i=1,2,3$, be the negative Morse-Bott family of Reeb orbits corresponding to $\bdry_i P$. If $f_0$ is $C^k$-close to $1$ for $k\gg 0$, then the only orbits that intersect $P\times\{0\}$ at most $2g+2$ times are: $e_P$, $h_{iP}$ and the orbits of $\mathcal{N}_i$. We pick two orbits in each $\mathcal{N}_i$ and label them $e_i$, $h_i$; they will become elliptic and hyperbolic when the Morse-Bott function $f_0$ is perturbed into a Morse function, which we call $f$. For convenience we write $N_S$ for the mapping torus of $(S,\hh)$, $N_{S'}$ for the mapping torus of $(S',\hh')$, etc.

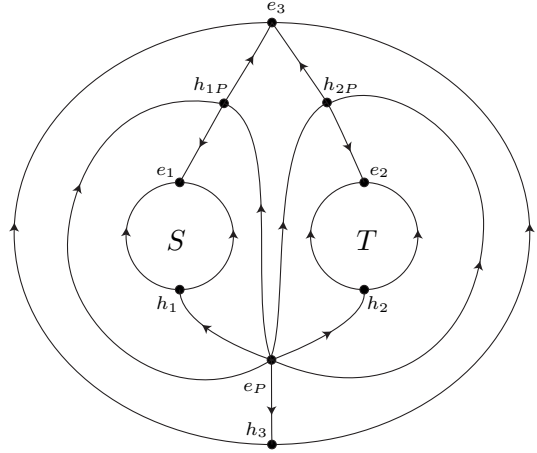
\begin{figure}[ht]
\begin{overpic}[width=7cm]{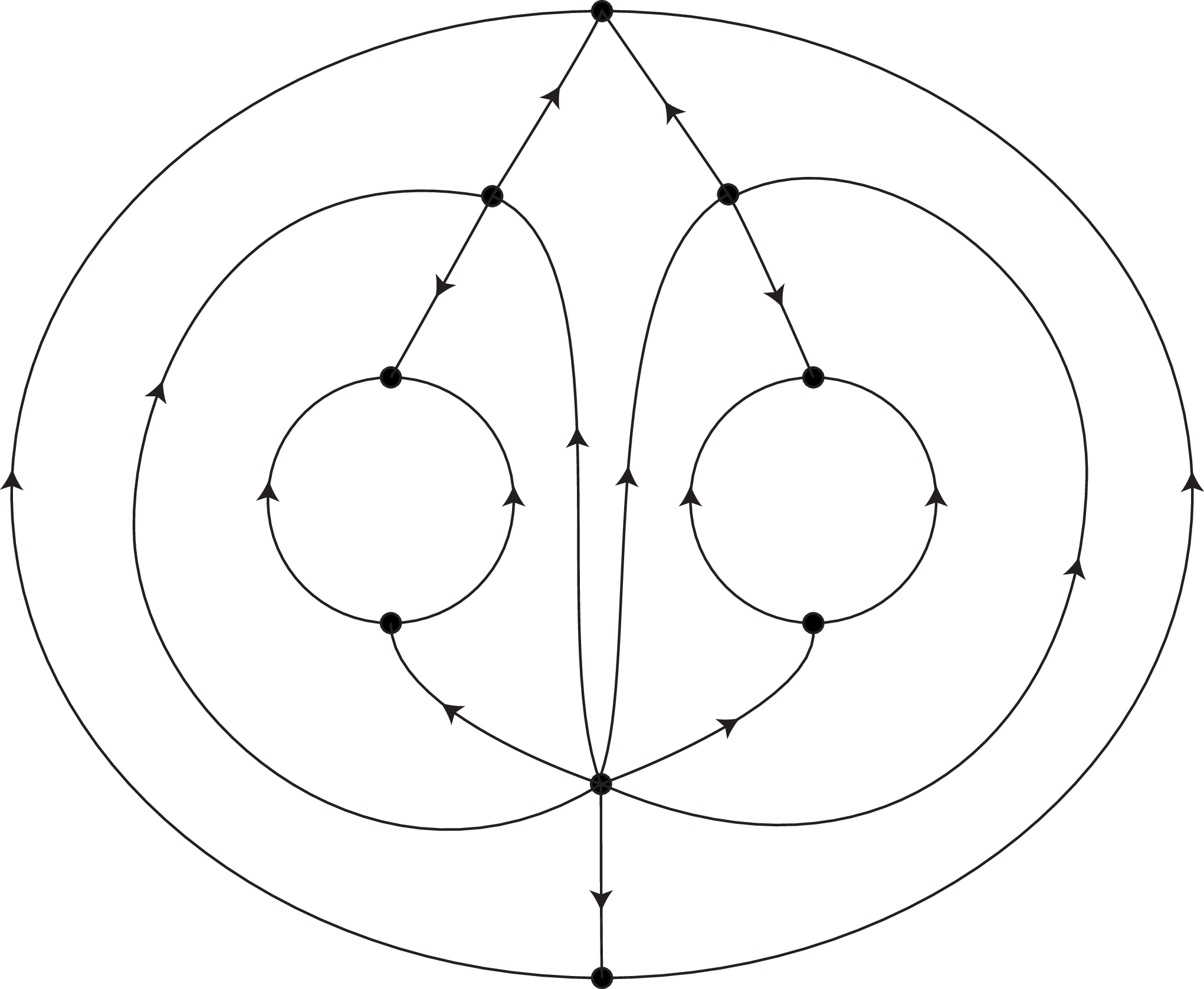}
\put(28,52.3){\tiny $e_1$} \put(68.5,52.3){\tiny $e_2$}
\put(49,83.5){\tiny $e_3$} \put(44.5,11.2){\tiny $e_P$}
\put(35,68.3){\tiny $h_{1P}$} \put(59.7,68.3){\tiny $h_{2P}$}
\put(45,2.8){\tiny $h_3$} \put(28,26.7){\tiny $h_1$}
\put(68,26.7){\tiny $h_2$} \put(30,38){$S$} \put(66,38){$T$}
\end{overpic}
\caption{The page $S'$. The gradient trajectories of the Morse function $f$ are given.}
\label{fig: sprime}
\end{figure}

Since $(S',\hh')$ is obtained from $(S,\hh)$ by applying two positive stabilizations, the corresponding contact structures $\xi_{(S,\hh)}$ and $\xi_{(S',\hh')}$ are isomorphic. As a special case, observe that if $(S,\hh)=(D^2,id)$, then $(S',\hh')=(T,\hh_T)$ and $(T,\hh_T)$ is an open book decomposition for the standard tight contact structure on $S^3$.

\subsection{Morse-Bott theory}

Let $J$ be an adapted almost complex structure on $\R\times N_{S'}$ and let $\pi: \R\times N_{S'}\to N_{S'}$ be the projection onto the second factor. Also let $u$ be a finite energy Morse-Bott building in $\R\times N_{S'}$ from $\bs\gamma$ to $\bs\gamma'$, where both orbit sets intersect $S'\times\{0\}$ in $2g+2$ points.

The following lemma partitions the irreducible components of $u$ into three regions.  When $u$ is a Morse-Bott building, then we say that $u$ is ``irreducible'' if the holomorphic curve, obtained by perturbing the Morse-Bott contact form to a nondegenerate one and correspondingly gluing up the levels of the Morse-Bott building, is irreducible.

\begin{lemma}
\label{lemma: cloistering}
Every irreducible component of the Morse-Bott building $u$ in $\R\times N_{S'}$ has image in one of $\R\times N_S$, $\R\times N_T$, or $\R\times N_{P}$.
\end{lemma}

\begin{proof}
The lemma is an application of the winding number $\op{wind}_\pi$ of \cite{HWZ1} and the positivity of intersections. (For example, see \cite[Lemma~5.5.1]{CGH-0}.)

Suppose without loss of generality that $u: \dot F\to \R\times N_{S'}$ is a single-level Morse-Bott building and that $u(\dot F)$ nontrivially intersects $\R\times N_P$; the case of a multiple-level Morse-Bott building only differs in notation. If $u$ corresponds to a gradient trajectory from $h_i$ to $e_i$, then the lemma holds. Otherwise let $P_\varepsilon\subset P$ be a slight retraction of $P$ so that $\bdry P_\varepsilon= \bdry_1^\varepsilon P \sqcup \bdry_2^\varepsilon P\sqcup \bdry_3 P$.  Let $\Pi: N_P=P\times S^1\to P$ be the projection onto the first factor, let $T_i^\varepsilon= \Pi^{-1}(\bdry_i^\varepsilon P)$, $i=1,2$, and let $N_{P_\varepsilon}=\Pi^{-1}(P_\varepsilon)$. We may assume that $T_i^\varepsilon$ is foliated by (not necessarily closed) Reeb orbits.  We consider the intersection $\delta_i=\pi(u(\dot F))\cap T_i^\varepsilon$, where $\delta_i$ is given the boundary orientation of $\pi(u(\dot F))\cap N_{P_\varepsilon}$. The curve $\delta_i$ is transverse to the Reeb vector field away from a finite number of points by \cite[Lemma~5.5.1]{CGH-0}.  If $\delta_i$ is not homologous to a multiple of $\{pt\}\times S^1$ in $H_1(T_i^\varepsilon)$, then $\Pi(\delta_i)$ is in the homology class $k[\bdry_i P_\varepsilon]\in H_1(P_\varepsilon)$, $k>0$, by the positivity of intersections in dimension four.  This implies that $\Pi(\pi(u(\dot F))\cap \bdry N_{S'})$ is in the class $k[\bdry_3 P]\in H_1(P)$ for $k>0$.  Since this is a contradiction, we must have $[\delta_i]=l[\{pt\}\times S^1]\in H_1(T_i^\varepsilon)$, for some $l\geq 0$. This implies that $u$ has negative ends along $\mathcal{N}_i$, $i=1,2$. Hence $u$ has image in $\R\times N_P$.
\end{proof}

We now use Morse-Bott theory (cf.\ Bourgeois~\cite{Bo1,Bo2}) to analyze holomorphic curves on $\R\times N_P$. In particular, we consider a perturbation of the Morse-Bott family of orbits on $N_P$, perturbed by the Morse function $f:P\to \R$ as described above.

\begin{lemma}
\label{lemma: list of ECH index one curves in P}
Let $f:P\to \R$ be $C^k$-close to $1$ for $k\gg 0$. Then there is an adapted almost complex structure $J$ on $\R\times N_P$ such that every $I_{ECH}=1$ finite energy $J$-holomorphic curve $u$ whose image is in $\R\times N_P$, is a simply-covered cylinder which corresponds to a gradient flow line between critical points of $f$ of adjacent index. The complete list is as follows:
\begin{enumerate}
\item one cylinder from $e_P$ to $h_i$, $i=1,2,3$, and two cylinders from $e_P$ to $h_{iP}$, $i=1,2$;
\item two cylinders from $h_i$ to $e_i$, $i=1,2,3$;
\item one cylinder from $h_{1P}$ to $e_1$ and one cylinder from $h_{1P}$ to $e_3$;
\item one cylinder from $h_{2P}$ to $e_2$ and one cylinder from $h_{2P}$ to $e_3$.
\end{enumerate}
Moreover, all the $I_{ECH}=1$ curves above are regular.
\end{lemma}

See Figure~\ref{fig: sprime} for the gradient trajectories of $f$ which correspond to the above holomorphic cylinders.

\begin{proof}
The proof is very similar to that of \cite[Lemma~9.2.2]{CGH-0}.  We use the following fact, which can be proved using Morse-Bott theory or by a direct computation:

\s\n {\bf Fact.} There is an adapted almost complex structure $J$ on $\R\times N_P$ such that there is a one-to-one correspondence between (parametrized) gradient trajectories $\delta:\R\to P$ of $f$ and finite energy $J$-holomorphic cylinders $Z_\delta$ in $\R\times N_P$ which intersect each $\{(s,t)\}\times P$ exactly once and project to $\op{Im}(\delta)$ under the projection $\Pi\circ \pi:\R\times N_P\to P$,  $(s,x,t)\mapsto x$.  Moreover, the cylinders $Z_\delta$, together with the trivial cylinders over the orbits corresponding to the Morse critical points, give a finite energy foliation of $\R\times N_P$.

\s Fix $f$ and $J$ as above. We will use the notation $s\delta:\R\to P$ for the translation $(s\delta)(\tau)=\delta(\tau+s)$.

Let $u:\dot F\to \R\times N_P$ be a finite energy $J$-holomorphic curve. Let $D_\varepsilon\subset P$ be an arbitrarily small disk centered at the point $e_P$ and let $N(e_P)=\Pi^{-1}(D_\varepsilon)$ be a solid torus neighborhood of the orbit $e_P$.  We assume that $\bdry N(e_P)$ is foliated by (not necessarily closed) orbits of the Reeb vector field. We identify $\bdry N(e_P)\simeq \R^2/ \Z^2$ so that the meridian has slope zero and a fiber $\{pt\}\times S^1$ has slope $\infty$.  Consider $\eta=\pi(u(\dot F))\cap \bdry N(e_P)$, where $\eta$ is given the boundary orientation of $\pi(u(\dot F))\cap \Pi^{-1}(P-int(D_\varepsilon))$. If the projection $[\Pi(\eta)]\in H_1(P-int(D_\varepsilon))$ is nonzero, then $[\Pi(\eta)]=-k[\bdry D_\varepsilon]$, $k>0$, by the positivity of intersections. This is a contradiction as in the proof of Lemma~\ref{lemma: cloistering}. Hence $[\eta]= l[\{pt\}\times S^1]\in H_1(\bdry N(e_P))$ for some $l\geq 0$, i.e., has slope $\infty$.  In other words, $u$ cannot intersect $\R\times e_P$ and can only have $e_P$ at the positive end. 
Similar considerations hold for $N(e_i)$, $i=1,2,3$, where $D_{i,\varepsilon}\subset P$ is a half-disk centered at $e_i$ and $N(e_i)=\Pi^{-1}(D_{i,\varepsilon})$.

We now claim that $u$ is some multiple cover of some $Z_\delta$ with multiplicity $\geq 1$. Arguing by contradiction, suppose $u$ does not multiply cover any $Z_\delta$. Let us first consider the case where $\Pi\circ \pi(u(\dot F))$ does not equal $\op{Im}(\delta)$ for any $\delta$. Then there is some $Z_\delta$ from $e_P$ to some $e_i$ such that the intersection $u(\dot F)\cap Z_\delta$ is nonempty; moreover, in view of the asymptotics on $N(e_P)$ and $N(e_i)$, we may assume that $K=\pi(u(\dot F)\cap Z_\delta)$ is compact.  This implies that $u(\dot F)$ and $Z_{s\delta}$ do not intersect for sufficiently large $s$.  On the other hand, since the intersection pairing $\langle u(\dot F), Z_{s\delta}\rangle$ is a homological quantity and does not depend on $s$ due to the asymptotics, it follows that $K=\varnothing$, which is a contradiction. This implies that $\Pi\circ \pi(u(\dot F))=\op{Im}(\delta)$ for some $\delta$.  Now $\R\times \Pi^{-1}(\delta)$ is a $3$-manifold which is foliated by $Z_{s\delta}$, $s\in \R$, and if $u$ does not multiply cover any $Z_{s\delta}$, then $u$ intersects some $Z_{s\delta}$ along a $1$-manifold, a contradiction. We conclude that $u$ is a multiple cover of some $Z_\delta$.

Now Fredholm index one cylinders $Z_\delta$ --- i.e., those that correspond to $\delta$ connecting two Morse critical points of adjacent index  --- are regular by the automatic transversality results of Wendl~\cite{We2,We3}.  Moreover, $u$ cannot multiply cover $Z_\delta$ with multiplicity $>1$ by \cite[Proposition~7.15]{HT1}, since otherwise $I(u)>1$. This implies that $u$ is equal to some $Z_\delta$, thereby completing the proof of the lemma.
\end{proof}

\subsection{A commutative diagram}

We assume that $\hh$ has no elliptic periodic points of period $\leq 2g+2$ in $int(S)$.

Consider the following (not commutative) diagram of chain complexes:
\begin{equation} \label{equation: diagram}
\begin{diagram}
\widehat{CF} (S,\mathbf{a},\hh(\mathbf{a}))  & \rTo^{\Theta} &
\widehat{CF} (S',\mathbf{a}', \hh'(\mathbf{a}'))\\
\dTo(0,4)^{\Phi_{(S,\hh)}} &  &
\dTo_{\Phi_{(S',\hh')}} \\
PFC_{2g} (S) & \rTo^{i\circ j} & PFC_{2g+2} (S')
\end{diagram}
\end{equation}
Here we are writing $PFC_j(S)$ for $PFC_j(N_S)$ and $PFC_j(S')$ for $PFC_j(N_{S'})$. The chain maps $\Theta$, $j$, and $i$ which appear Diagram~(\ref{equation: diagram}) are defined as follows:

\s\n {\em The map $\Theta$.} The chain map
$$\Theta:\widehat{CF} (S,\mathbf{a},\hh(\mathbf{a})) \to \widehat{CF} (S',\mathbf{a}', \hh'(\mathbf{a}'))$$
is defined as follows: Let $\mathbf{a}=\{a_1,\dots,a_{2g}\}$ be a basis of $S$.  We extend $\mathbf{a}$ to
$\overline{\mathbf{a}}=\{\overline{a}_1,\dots,\overline{a}_{2g}\}$ so that $\overline{a}_i$ is a properly embedded arc with boundary on $\bdry S'$ and then complete $\overline{\mathbf{a}}$ to a basis $\mathbf{a}'$ of $S'$ by adding the extensions $\overline{a}_{2g+1}$ and $\overline{a}_{2g+2}$ of the basis arcs $a_{2g+1}$ and
$a_{2g+2}$ of $T$, subject to the following conditions:
\begin{itemize}
\item $\overline{a}_i$ and $\hh'(\overline{a}_i)$ have two extra pairs of canceling intersections in $S'-(S\cup T)$ for $i=1,\dots,2g+2$;
\item $\overline{a}_i-a_i$ and $\hh'(\overline{a}_j-a_j)$ are disjoint for $1\leq i\not=j\leq 2g+2$.
\end{itemize}
Let $x_i\sim x_i'$, $i=2g+1,2g+2$, be the intersections of $\hh'(\overline{a}_i)$ and $\overline{a}_i$ on $\bdry S'=\bdry_3 P$, which comprise the contact class. We then define
$$\Theta(\mathbf{y})= \mathbf{y}\cup\{x_{2g+1},x_{2g+2}\}.$$
It is easy to see that $\Theta$ is a chain map. (Compare with the gluing map from \cite{HKM2}.)

\s\n {\em The map $j$.}
Next define the map
$$j: PFC_{2g} (S) \rightarrow PFC_{2g+2} (S),$$
$$j(\bs\gamma ) =e_1^2\bs\gamma,$$
where we are using multiplicative notation for orbit sets. The map $j$ is the composition $\mathfrak{I}_{2g+1}\circ \mathfrak{I}_{2g}$, where ${\frak I}_i$ is the map
$$\mathfrak{I}_i: PFC_{i}(S)\to PFC_{i+1}(S),$$
$$\bs\gamma\mapsto e_1\bs\gamma.$$ The elliptic orbit $e_1$ can only appear at the negative asymptotic end of a holomorphic curve, provided the curve is not a trivial cylinder. Hence $j$ is a chain map in view of Lemma~\ref{lemma: cloistering}.

\s\n {\em The map $i$.}
The map $i$ is the inclusion map on the chain level:
$$i: PFC_{2g+2} (S) \rightarrow PFC_{2g+2} (S'),$$
$$\bs\gamma\mapsto \bs\gamma.$$
The map $i$ is chain map by Lemma~\ref{lemma: cloistering}. The analogous map for the inclusion $T\subset S'$ is also a chain map.

\s Although Diagram~(\ref{equation: diagram}) is not quite commutative on the chain level, we have the following:

\begin{lemma}
Diagram~(\ref{equation: diagram}) induces the following commutative diagram on the level of homology:
\begin{equation} \label{equation: diagram on level of homology}
\begin{diagram}
\widehat{HF} (S,\mathbf{a},\hh(\mathbf{a}))  & \rTo^{\Theta_*} &
\widehat{HF} (S',\mathbf{a}', \hh'(\mathbf{a}'))\\
\dTo(0,4)^{(\Phi_{(S,\hh)})_*} &  &
\dTo_{(\Phi_{(S',\hh')})_*} \\
PFH_{2g} (S) & \rTo^{(i\circ j)_*} & PFH_{2g+2} (S')
\end{diagram}
\end{equation}
\end{lemma}

\begin{proof}
We consider the $W_+$-curves for $(S',\hh')$ that limit to the tuple $\Theta({\bf y})=\mathbf{y}\cup\{x_{2g+1},x_{2g+2}\}$ at the positive end. Since $x_{2g+1}$ and $x_{2g+2}$ are components of the contact class, there are no holomorphic curves, besides restrictions of trivial cylinders, that limit to $x_{2g+1}$ or $x_{2g+2}$ at the positive end. The argument is the same as that of Lemma~I.\ref{P1-lemma: trivial cylinder when x i involved}. The trivial cylinder over $x_i$, $i=2g+1,2g+2$, ``maps'' the component $x_i$ to a generic point $p_i$ of the Morse-Bott family $\mathcal{N}_3$ and is concatenated with a cylinder corresponding to a gradient trajectory from $p_i$ to $e_3$ to give a Morse-Bott building.  Moreover, by automatic transversality~\cite[Theorem~4.5.36]{We2} and the discussion from Lemma~I.\ref{P1-lemma: regularity of curve at infinity}, the above Morse-Bott building is Morse-Bott regular.  Hence the pair $\{x_{2g+1},x_{2g+2}\}$ is ``mapped'' to $e_3^2$. By Lemma~\ref{lemma: cloistering} we have:
$$\Phi_{(S',\hh')} (\mathbf{y}\cup\{x_{2g+1} ,x_{2g+2}\})= e_3^2\cdot\Phi_{(S,\hh)}(\mathbf{y}).$$
On the other hand, $e_3^2\bs\gamma$ is homologous to $e_1^2\bs\gamma$ in $PFH_{2g+2}(S')$, since there is one cylinder each from $h_{1P}$ to $e_1$ and $e_3$. This proves the commutativity of the diagram.
\end{proof}

\begin{cor} \label{cor: i circ j is isomorphism}
The composition $i_* \circ j_*$ is an isomorphism.
\end{cor}

\begin{proof}
This follows from the commutativity of Diagram~(\ref{equation: diagram on level of homology}). Since $(S,\hh)$ and $(S',\hh')$ are both open book decompositions for $M$ and we may assume that $\hh$ and $\hh'$ both have no periodic points of period $\leq 2g+2$ in $int(S)$ or $int(S')$, the cobordism maps $(\Phi_{(S,\hh)})_*$ and $(\Phi_{(S',\hh')})_*$ are both isomorphisms. The map $\Theta_*$ is also an isomorphism since $\widehat{HF}(S',\mathbf{a}', \hh'(\mathbf{a}'))$ can be computed as the tensor product of the $S$ and $T$ sides and the $T$ side gives $\widehat{HF}(S^3)$, which is generated by $\{x_{2g+1},x_{2g+2}\}$. Since three sides of the diagram are isomorphisms, it follows that $i_* \circ j_*$ is an isomorphism.
\end{proof}

\subsection{An isomorphism}

To simplify notation let us write $V_i=PFH_i(S)$ and $W_j=PFH_j(T)$. The goal of this subsection is to prove the following key proposition.

\begin{prop} \label{prop: isomorphism of j}
The induced map
$$j_*= (\mathfrak{I}_{2g+1})_*\circ (\mathfrak{I}_{2g})_* :V_{2g} \rightarrow V_{2g+2} $$
is an isomorphism.
\end{prop}

\subsubsection{Description of the differential}
\label{subsubsection: part 0}

Given two orbit sets $\bs\gamma'=\prod\gamma_i^{m_i'}$ and $\bs\gamma=\prod \gamma_i^{m_i}$, we set $\bs\gamma/\bs\gamma'= \prod\gamma_i^{m_i-m_i'}$ if $m_i'\leq m_i$ for all $i$; otherwise we set $\bs\gamma/\bs\gamma'=0$.

The chain group $PFC_{2g+2}(S')$ can be written as:
$$PFC_{2g+2}(S')=\bigoplus_{m+i+j=2g+2} \F[h_{1P}, h_{2P}, h_3, e_P, e_3]_m\otimes PFC_i(S)\otimes PFC_j(T).$$
$\F[h_{1P}, h_{2P}, h_3, e_P, e_3]$ is a polynomial ring where $h_{1P}$, $h_{2P}$, $h_3$ (resp.\ $e_P$, $e_3$) are considered as Grassmann variables of odd degree (resp.\ even degree) and the subscript $m$ indicates the subspace spanned by monomials 
with total exponent $m$.

Let us write a generator of $PFC_{2g+2}(S')$ as $\bs\gamma\otimes \bs\Gamma_1\otimes\bs\Gamma_2$, where $\bs\Gamma_1\in PFC_i(S)$, $\bs\Gamma_2\in PFC_j(T)$, and $\bs\gamma$ is constructed from orbits passing through $S'-S-T$.  Using Lemma~\ref{lemma: cloistering} and the description of ECH index one curves in $\R\times N_P$ from Lemma~\ref{lemma: list of ECH index one curves in P}, we write the differential $\bdry$ of $PFC_{2g+2}(S')$ as follows:
\begin{align*}
\bdry(\bs\gamma\otimes \bs\Gamma_1\otimes\bs\Gamma_2) = & \bs\gamma\otimes (\bdry_S\bs\Gamma_1)\otimes \bs\Gamma_2 +
\bs\gamma\otimes \bs\Gamma_1\otimes (\bdry_T\bs\Gamma_2) \\
& + (\bs\gamma/e_P) (h_3 \otimes\bs\Gamma_1\otimes \bs\Gamma_2 + 1\otimes h_1\bs\Gamma_1\otimes \bs\Gamma_2+ 1\otimes\bs\Gamma_1\otimes h_2\bs\Gamma_2)\\
& + (\bs\gamma/h_{1P}) (e_3\otimes \bs\Gamma_1\otimes \bs\Gamma_2+ 1\otimes e_1\bs\Gamma_1\otimes\bs\Gamma_2)\\
& + (\bs\gamma/h_{2P}) (e_3\otimes \bs\Gamma_1\otimes \bs\Gamma_2+ 1\otimes \bs\Gamma_1\otimes e_2\bs\Gamma_2).
\end{align*}
Here $\bdry_S$ and $\bdry_T$ are the differentials on $PFC(S)$ and $PFC(T)$.

\subsubsection{Spectral sequence calculation}
\label{subsubsection: part 1}

In this subsection we use spectral sequences to interpret $PFH_{2g+2}(S')$ in terms of $V_i$ and $W_j$. More precisely, we prove the following:

\begin{lemma}
\label{lemma: spectral sequence calculation}
There is a natural inclusion of $\left(\bigoplus_{i+j=2g+2} V_i\otimes W_j\right)/\sim$ into $PFH_{2g+2}(S')$,
where the equivalence relation $\sim$ is given by $e_1\bs\Gamma_1\otimes \bs\Gamma_2\sim \bs\Gamma_1\otimes e_2\bs\Gamma_2$. Here $\bs\Gamma_1\in V_i$ and $\bs\Gamma_2\in W_j$ and the equivalence $\sim$ is generated by terms in adjacent $V_{i+1}\otimes W_j$ and $V_i\otimes W_{j+1}$.
\end{lemma}

\begin{proof}
Let $\mathcal{F}$ be a filtration on $(PFC_{2g+2}(S'),\bdry)$ which, on the generators, counts the multiplicity of $h_{1P}$.  This means that $\mathcal{F}$ takes values in $\{0,1\}$.  We write $(E^r(\mathcal{F}),\bdry_r)$ for the $E^r$-term and the $E^r$-differential of the spectral sequence for $\mathcal{F}$. Each page $E^r(\mathcal{F})$ has a grading coming from $\mathcal{F}$ and $E^r_k(\mathcal{F})$ is the degree $k$ component of $E^r(\mathcal{F})$ with respect to this grading.  We remark that the spectral sequence associated to $\mathcal{F}$ is nothing but a long exact sequence in homology, and its use is motivated by our wish to give a parallel treatment of the cases where we filter by the multiplicity of a hyperbolic orbit or by the multiplicity of an elliptic orbit.

Next let $\mathcal{G}$ be a filtration on $(E^0(\mathcal{F}),\bdry_0)$ which counts the multiplicity of $h_{2P}$.  Again, $\mathcal{G}$ takes values in $\{0,1\}$.  We write $(E^r(\mathcal{G}),\bdry_{0r})$ for the $E^r$-term and the $E^r$-differential of the spectral sequence for $\mathcal{G}$.  Finally, let $\mathcal{H}$ be the filtration on $(E^0(\mathcal{G}),\bdry_{00})$ which counts the multiplicity of $e_P$, and let $(E^r(\mathcal{H}),\bdry_{00r})$ be the $E^r$-term and the $E^r$-differential of the spectral sequence for $\mathcal{H}$.

We first consider $(E^0(\mathcal{H}),\bdry_{000})$, where:
$$\bdry_{000}(\bs\gamma\otimes \bs\Gamma_1\otimes \bs\Gamma_2)= \bs\gamma\otimes (\bdry_S\bs\Gamma_1)\otimes \bs\Gamma_2+ \bs\gamma\otimes \bs\Gamma_1\otimes (\bdry_T\bs\Gamma_2).$$
By the K\"unneth formula, we have:
$$E^1(\mathcal{H})= \bigoplus_{m+i+j=2g+2}  \F[h_{1P}, h_{2P}, h_3, e_P, e_3]_m \otimes V_i\otimes W_j .$$

Next consider $(E^1(\mathcal{H}),\bdry_{001})$, where:
\begin{align*}
\bdry_{001}(e_P^{n}\bs\gamma\otimes\bs\Gamma_1\otimes \bs\Gamma_2) = & e_P^{n-1}h_3 \bs\gamma\otimes \bs\Gamma_1\otimes \bs\Gamma_2 +  e_P^{n-1} \bs\gamma\otimes (h_1\bs\Gamma_1)\otimes\bs\Gamma_2\\
& + e_P^{n-1}\bs\gamma\otimes \bs\Gamma_1\otimes (h_2\bs\Gamma_2),
\end{align*}
Here $\bs\Gamma_1\in V_i$, $\bs\Gamma_2\in W_j$, and $\bs\gamma$ has no $e_P$ term. Note that any $e_P^{n-1}h_3 \bs\gamma\otimes \bs\Gamma_1\otimes \bs\Gamma_2$ is homologous (with respect to the differential $\bdry_{001}$) to a linear combination of $e_P^{n-1}\bs\gamma'\otimes \bs\Gamma_1'\otimes\bs\Gamma_2'$, where $\bs\gamma'$ does not have any $e_P$ and $h_3$ terms, $\bs\Gamma_1'$ is in some $V_i$, and $\bs\Gamma_2'$ is in some $W_j$. Hence every element of $E^1(\mathcal{H})/\op{Im}(\bdry_{001})$ can be represented by $w=\sum_{i\geq 0} e_P^iw_i$, where $w_i$ has no $e_P$ and $h_3$ terms, and we can write
$$\bdry_{001}(w)= \sum_{i>0} e_P^{i-1} h_3 w_i +w',$$
where $w'$ has no terms which contain $h_3$. If $\bdry_{001}(w)=0$, then all the $w_i$, $i>0$, must be zero.  We therefore obtain:
$$E^2(\mathcal{H})= \bigoplus_{m+i+j=2g+2}  \F[h_{1P}, h_{2P}, e_3]_m \otimes V_i\otimes W_j .$$
Since $E^2(\mathcal{H})$ is supported in degree $0$, the spectral sequence for $\mathcal{H}$ converges at the $E^2$-term and we have $E^1(\mathcal{G})\simeq E^2(\mathcal{H})$. Moreover, $E^2(\mathcal{H})=E^2_0(\mathcal{H})$ is naturally isomorphic to $E^1(\mathcal{G})$ since $0$ is the lowest filtration level.

Now consider $(E^1(\mathcal{G}),\bdry_{01})$, where
$$\bdry_{01}(\bs\gamma\otimes \bs\Gamma_1\otimes \bs\Gamma_2)=e_3(\bs\gamma/h_{2P}) \otimes \bs\Gamma_1\otimes \bs\Gamma_2+(\bs\gamma/h_{2P}) \otimes \bs\Gamma_1\otimes e_2\bs\Gamma_2.$$
The calculation of $E^2(\mathcal{G})$ is similar to the calculation of $E^2(\mathcal{H})$ in the previous paragraph.  Any $e_3^{n+1} \bs\gamma\otimes \bs\Gamma_1\otimes \bs\Gamma_2$ is homologous to a linear combination of $e_3^{n}\bs\gamma'\otimes\bs\Gamma_1'\otimes \bs\Gamma_2'$, where $\bs\gamma$, $\bs\gamma'$ do not have any $h_{2P}$ and $e_3$ terms. Hence every element of $E^1(\mathcal{G})/\op{Im}(\bdry_{01})$ can be represented by $w=w'+\sum_{i=0}^k h_{2P} e_3^iw_i$, where $w'$ and $w_i$ have no $h_{2P}$ or $e_3$ terms, and we can write
$$\bdry_{01} w= \sum_{i=0}^k e_3^{i+1}w_i +\sum_{i=1}^k e_3^i e_2 w_i.$$
If $\bdry_{01} w=0$, then all the $w_i$ must be zero.  Only the $w'$ term remains, and we have:
$$ E^2(\mathcal{G})= \bigoplus_{m+i+j=2g+2} \F[h_{1P}]_m \otimes V_i\otimes W_j ,$$
which we can write as a direct sum $\mathfrak{L}_0\oplus \mathfrak{L}_1$, where:
$$\mathfrak{L}_0= \bigoplus_{i+j=2g+2} V_i\otimes W_j,$$
$$\mathfrak{L}_1=\bigoplus_{i+j=2g+1} \F\{h_{1P}\}\otimes V_i \otimes W_j.$$
Since $E^2(\mathcal{G})$ is supported in the lowest degree $0$, the spectral sequence for $\mathcal{G}$ converges at the $E^2$-term and $E^2(\mathcal{G})=E^2_0(\mathcal{G})$ is naturally isomorphic to $E^1(\mathcal{F})$.

Finally, $E^1(\mathcal{F})\simeq \mathfrak{L}_0\oplus \mathfrak{L}_1$ has differential $\bdry_1$ given by:
\begin{equation}
\bdry_1(\bs\Gamma_1\otimes \bs\Gamma_2)=0,
\end{equation}
\begin{equation} \label{eqn: connecting}
\bdry_1(h_{1P}\otimes\bs\Gamma_1\otimes\bs\Gamma_2)= e_1\bs\Gamma_1\otimes \bs\Gamma_2+ \bs\Gamma_1\otimes e_2\bs\Gamma_2,
\end{equation}
since $\bdry (h_{1P}\otimes\bs\Gamma_1\otimes\bs\Gamma_2)= e_1\bs\Gamma_1\otimes \bs\Gamma_2+ e_3\otimes\bs\Gamma_1\otimes\bs\Gamma_2$ and $\bdry (h_{2P}\otimes \bs\Gamma_1\otimes \bs\Gamma_2)= \bdry_0 (h_{2P}\otimes\bs\Gamma_1\otimes \bs\Gamma_2) = e_3\otimes\bs\Gamma_1\otimes \bs\Gamma_2 + \bs\Gamma_1\otimes e_2\bs\Gamma_2$. Viewing the spectral sequence as a long exact sequence with connecting homomorphism $\mathfrak{L}_1\stackrel\delta\to
\mathfrak{L}_0$, where $\delta$ is given by Equation~(\ref{eqn: connecting}), we see that $E^2(\mathcal{F})$ is isomorphic to $\mathfrak{L}_0/\op{Im}(\delta) = \mathfrak{L}_0/\sim$, where the equivalence relation is given by $e_1\bs\Gamma_1\otimes\bs\Gamma_2\sim \bs\Gamma_1\otimes e_2\bs\Gamma_2$. As argued previously, the spectral sequence converges at the $E^2$-term and $E^2(\mathcal{F})$ is naturally isomorphic to $PFH_{2g+2}(S')$.
\end{proof}

\subsubsection{Isomorphism of $j_*$}
\label{subsubsection: part 2}

We now complete the proof of Proposition~\ref{prop: isomorphism of j}.

\begin{proof}[Proof of Proposition~\ref{prop: isomorphism of j}]
Since $i_*\circ j_*: V_{2g}\to PFH_{2g+2}(S')$ is an isomorphism by Corollary~\ref{cor: i circ j is isomorphism}, $j_*: V_{2g}\to V_{2g+2}$ is injective. It remains to show that $j_*$ is surjective.

First note that $W_0=\F\{1\}$. We also have $W_2=\F\{e_2^2\}$, since $\widehat{HF}(S^3)\simeq \F$, generated by the contact class $c(\xi_{std})$ of the standard tight contact structure $\xi_{std}$ on $S^3$, and $e_2^2$ is the image of $c(\xi_{std})$ under the isomorphism
$$(\Phi_{(T,\hh_T)})_*:\widehat{HF}(S^3)=\widehat{HF}(T,\hh_T)\stackrel\sim\to PFH_2(T).$$
Now, since the isomorphism $W_0\stackrel\sim\to W_2$, $1\mapsto e_2^2$, factors as $W_0\to W_1\to W_2$, $1\mapsto e_2\mapsto e_2^2$, it follows that $W_1=\F\{e_2\}\oplus W_1'$ for some $\F$-vector space $W_1'$.

By Lemma~\ref{lemma: spectral sequence calculation},
$$PFH_{2g+2}(S')\simeq \left(\bigoplus_{i+j=2g+2} V_i\otimes W_j\right)/\sim.$$
Moreover, $V_{2g}\otimes W_2= V_{2g}\otimes \F\{e_2^2\}$ is the image of $V_{2g}$ under the map $i_*\circ j_*$. Since $i_*\circ j_*$ is an isomorphism, every element of $V_{2g+2}\otimes W_0= V_{2g+2}\otimes \F\{1\}$ is equivalent to some element of $V_{2g}\otimes W_2$, i.e., if $v_{2g+2}\in V_{2g+2}$, then $v_{2g+2}\otimes 1 \sim v_{2g}\otimes e_2^2$ for some $v_{2g}\in V_{2g}$. More explicitly,
\begin{align*}
v_{2g+2}\otimes 1 + v_{2g}\otimes e_2^2 =& (e_1 v''_{2g+1}\otimes 1 +v''_{2g+1}\otimes e_2) + (e_1 v''_{2g}\otimes e_2 + v''_{2g}\otimes e_2^2)\\
& \quad + \sum_i (e_1 v_{2g,i}''\otimes w'_{1,i}+v_{2g,i}''\otimes e_2 w'_{1,i}) + \dots,
\end{align*}
where $v''_{2g+1}\in V_{2g+1}$; $v''_{2g},v_{2g,i}''\in V_{2g}$; $\{w'_{1,i}\}$ is a basis for $W_1'$; and $\dots$ is a linear combination of terms which are not in $V_{2g+2}\otimes W_0$ and $V_{2g+1}\otimes W_1$. A term-by-term comparison gives $v_{2g+2}=e_1v''_{2g+1}$ and $v''_{2g+1}=e_1 v''_{2g}$. Hence $v_{2g+2}=e_1^2v''_{2g}$ and $v_{2g+2}\in \op{Im}(j_*)$. This completes the proof of Proposition~\ref{prop: isomorphism of j}.
\end{proof}

\subsection{Proof of Theorem~\ref{thm: stabilization}} \label{modifications}

We use the notation from Section~I.\ref{P1-direct limits}.

We first consider the ECH cobordism map
$${\frak K}''_j: ECC_j(N,f_j\alpha)\to ECC_{j}(N,f_{j+2}\alpha),$$ defined via Seiberg-Witten Floer homology as in \cite[Theorem~2.4]{HT3}.

More precisely, let $\delta_k>0$ be constants satisfying $\displaystyle\lim_{k\to\infty}\delta_k=0$.  Let $V_{j+2}^{\delta_k}$ be a $\delta_k$-tubular neighborhood of the elliptic orbits of $R_{f_j\alpha}$ with $\mathcal{F}= j+1$ or $j+2$, where $\delta_k$ is measured with respect to a fixed Riemannian metric on $N$, and let $f_{j+2,k}:N\to (0,+\infty)$ be a function such that the following hold:
\begin{enumerate}
\item $f_{j+2,k}=f_j$ on $N-V_{j+2}^{\delta_k}$;
\item $R_{f_{j+2,k}\alpha}$ has no elliptic orbits with $\mathcal{F}\leq j+2$; and
\item $\displaystyle\lim_{k\to\infty}f_{j+2,k}=f_j$ in $C^1$.
\end{enumerate}

Let $J_j$ be an $f_j\alpha$-adapted almost complex structure on $\R\times N$ and let $J_{j+2,k}$ be $f_{j+2,k}\alpha$-adapted almost complex structures on $\R\times N$ which agree with $J_j$ on $\R\times (N-V_{j+2}^{\delta_k})$ and satisfy $J_{j+2,k}\to J_j$ in $C^0$ as $k\to \infty$.  Also let $\widetilde J_{j,k}$ be an almost complex structure on $\R\times N$ which adapted to an exact symplectic cobordism $\omega_{j,k}$ from $f_j\alpha$ to $f_{j+2,k}\alpha$, agrees with $J_j$ on $\R\times (N-V_{j+2}^{\delta_k})$, and satisfies $\widetilde J_{j,k}\to J_j$ as $k\to \infty$.

Observe that the generators of $ECC_j(N,f_j\alpha)$ and $ECC_{j}(N,f_{j+2,k}\alpha)$ agree for $k\gg 0$.  Let $\bs\gamma$ and $\bs\gamma'$ denote orbit sets for either $f_j\alpha$ or $f_{j+2,k}\alpha$.

\begin{lemma} \label{limit of curves}
Let $u_k$ either be a sequence of $J_{j+2,k}$-holomorphic curves from $\bs\gamma$ to $\bs\gamma'$ or a sequence of $\widetilde J_{j,k}$-holomorphic curves from $\bs\gamma$ to $\bs\gamma'$.  Then, after passing to a subsequence, $u_k$ converges to a $J_j$-holomorphic building $u$ from $\bs\gamma$ to $\bs\gamma'$.
\end{lemma}

\begin{proof}[Sketch of proof.]
This follows from Gromov compactness for $C^0$-convergence of almost complex structures, due to Ivashkovich-Shevchishin~\cite{IS}.

Alternatively, we can argue as follows, using the usual Gromov compactness for $C^\infty$-convergence of almost complex structures: We treat the case where $u_k$ is a $J_{j+2,k}$-holomorphic curve. By the compactness argument from Section~I.\ref{P1-subsection: compactness PFH case}, we may assume that $[u_k]\in H_2(N,\bs\gamma,\bs\gamma')$ and the topological type of the domains $F_k$ of $u_k$ are fixed. We then restrict $u_k$ to the preimage $G_k$ of $\R\times (N-V_{j+2}^{\delta_k})$. We may assume that $G_k$ is obtained from a compact Riemann surface with boundary by removing interior punctures. (This is possible by taking $\delta_k$ to be generic.)

We claim that $|\chi(G_k)|$ is bounded above. This is equivalent to an upper bound on the number of disk components of $F_k-G_k$. Let $\bs\gamma''$ be the union of core orbits of $V_{j+2}^{\delta_k}$, which is independent of $k$. The number of disk components is bounded above by the intersection number with $\R\times\bs\gamma''$, which in turn is controlled by the homology class $[u_k]\in H_2(N,\bs\gamma,\bs\gamma')$.

Since $J_{j+2,k}=J_j$ on $\R\times (N-V_{j+2}^{\delta_k})$, we can take the limit of $u_k|_{G_k}$ to obtain $u|_G$. For simplicity assume that $u$ has only one level. Then $\bdry G$ is mapped to $\R\times \bs\gamma''$, which contradicts the positivity of intersections unless $\bdry G=\varnothing$.  Hence $\bdry G=\varnothing$ and all the interior punctures of $G$ which do not to limit to $\bs\gamma$ or $\bs\gamma'$ are removable.
\end{proof}

\begin{lemma}
For $k\gg 0$, the ECH cobordism map $${\frak K}''_j:ECC_j(N,f_j\alpha)\to ECC_{j}(N,f_{j+2,k}\alpha)$$ satisfies ${\frak K}''_j(\bs\gamma)=\bs\gamma$ and induces an isomorphism on the level of homology.
\end{lemma}

\begin{proof}
We use the Holomorphic Curves Property in \cite[Theorem~2.4]{HT3}, which states that $\langle {\frak K}''_j(\bs\gamma),\bs\gamma'\rangle$ is a sum of weights $w(v_i)$ over all $I=0$  $\widetilde J_{j,k}$-holomorphic buildings $v_1,\dots,v_l$ from $\bs\gamma$ to $\bs\gamma'$. By Lemma~\ref{limit of curves}, for $k\gg 0$, the only $I=0$ $\widetilde J_{j,k}$-buildings are curves supported on trivial cylinders.  Moreover, the Holomorphic Curves Property also guarantees that $w(v)=1$ if $v$ is supported on trivial cylinders.  This implies that ${\frak K}''_j(\bs\gamma)=\bs\gamma$.

Finally, since the chain map ${\frak K}''_j$ is the identity, the differentials on $ECC_j(N,f_j\alpha)$ and $ECC_{j}(N,f_{j+2,k}\alpha)$ must agree.  This implies the second assertion of the lemma.
\end{proof}

We now let $f_{j+2}=f_{j+2,k}$ for $k\gg 0$.  In view of Theorem~I.\ref{P1-thm: direct limit}, we can write
$$\displaystyle\widehat{ECH}(M)\simeq \lim_{j\to\infty} ECH_{2j}(N,f_{2j}\alpha),$$
where the map $ECH_{2j}(N,f_{2j}\alpha)\to ECH_{2j+2}(N,f_{2j+2}\alpha)$ is given by the top row of:
\begin{equation}
\begin{diagram}
ECH_{2j}(N,f_{2j}\alpha) & \rTo^{(\frak K''_{2j})_*} & ECH_{2j}(N,f_{2j+2}\alpha) & \rTo^{j'_*} & ECH_{2j+2}(N,f_{2j+2}\alpha)  \\
 & & \dTo^{\simeq} & & \uTo^{\simeq} \\
 & & PFH_{2j}(N,f_{2j+2}\alpha) & \rTo^{j_*} & PFH_{2j+2}(N,f_{2j+2}\alpha) \\
\end{diagram}
\end{equation}
Here $j'_*$ and $j_*$ are maps on homology induced by the inclusions $\bs\gamma \mapsto e^2\bs\gamma$, and the square commutes.  Finally, Proposition~\ref{prop: isomorphism of j} 
implies that $j'_*$ is an isomorphism. This proves Theorem~\ref{thm: stabilization}.

\vskip0.5cm

\n {\em Acknowledgements.} We are indebted to Michael Hutchings for many helpful conversations and for our previous collaboration which was a catalyst for the present work. We also thank Denis Auroux, Tobias Ekholm, Dusa McDuff, Ivan Smith and Jean-Yves Welschinger for illuminating exchanges.  Part of this work was done while KH and PG visited MSRI during the academic year 2009--2010. We are extremely grateful to MSRI and the organizers of the ``Symplectic and Contact Geometry and Topology'' and the ``Homology Theories of Knots and Links'' programs for their hospitality; this work probably would never have seen the light of day without the large amount of free time which was made possible by the visit to MSRI.  KH also thanks the Simons Center for Geometry and Physics for their hospitality during his visit in May 2011.

\newpage
\printnomenclature[7em]

\end{document}